\documentclass[reqno]{amsart}

\usepackage{amsmath,amssymb}
\usepackage{mathtools}

\usepackage{color}
\definecolor{darkgreen}{rgb}{0.00,0.50,0.10}
\definecolor{lightgreen}{rgb}{0.20,0.70,0.30}

\usepackage[bookmarks,colorlinks=true,citecolor=black,linkcolor=black,pagecolor=black,urlcolor=black,pagebackref]{hyperref}

\usepackage{enumitem}
\usepackage{booktabs}
\usepackage{stmaryrd}

\newtheorem{theorem}{Theorem}
\newtheorem{lemma}[theorem]{Lemma} 
\newtheorem{corollary}[theorem]{Corollary} 
\newtheorem{proposition}[theorem]{Proposition}

\newtheorem{definition}[theorem]{Definition}

\newtheorem{conjecture}[theorem]{Conjecture}

\newcommand{\By}[2]{\overset{\mbox{\tiny{#1}}}{#2}}

\newcommand{\Z}{\mathbb{Z}}
\newcommand{\F}{\mathbb{F}}

\newcommand{\Q}{\mathbb{Q}}

\newcommand{\upc}{\mathrm{c}}

\newcommand{\upp}{\mathrm{p}}

\newcommand{\upB}{\mathrm{B}}
\newcommand{\upC}{\mathrm{C}}

\newcommand{\upE}{\mathrm{E}}

\newcommand{\upV}{\mathrm{V}}
\newcommand{\upX}{\mathrm{X}}
\newcommand{\upZ}{\mathrm{Z}}

\newcommand{\Cir}{\mathrm{Cir}}
\newcommand{\Perm}{\mathrm{Perm}}

\newcommand{\dom}{\mathrm{dom}}
\newcommand{\Dom}{\mathrm{Dom}}
\newcommand{\supp}{\mathrm{supp}}
\newcommand{\Supp}{\mathrm{Supp}}

\newcommand{\Ra}{\mathrm{Ra}}

\newcommand{\BG}{\mathrm{BG}}

\newcommand{\rk}{\mathrm{rk}}

\newcommand{\Fix}{\mathrm{Fix}}

\newcommand{\bal}{\mathrm{bal}}

\newcommand{\ul}{\mathrm{ul}}

\newcommand{\lcf}{\mathrm{lcf}}

\newcommand{\chio}{\mathrm{chio}}
\newcommand{\Chio}{\mathrm{Chio}}

\newcommand{\leftsubscript}[2]{{\vphantom{#2}}_{#1}{#2}}

\newcommand{\upXul}{\leftsubscript{\mathrm{ul}}{\upX}}

\newcommand{\Col}{\mathrm{Col}}

\newcommand{\Prob}{\textup{\textsf{P}}}

\DeclareMathOperator{\im}{im}

\DeclareMathOperator{\Aut}{Aut}

\newcommand{\Sym}{\mathrm{Sym}}

{\begin{list}{}%
         {\setlength{\leftmargin}{#1}}%
         \item[]%
}
{\end{list}}

\author{Peter Heinig}

\address{Zentrum Mathematik, M9, Technische Universit{\"a}t M{\"u}nchen, 
Boltzmannstra{\ss}e~3, D-85747 Garching bei M{\"u}nchen, Germany} 
\email{heinig@ma.tum.de}
\date{\today}
\thanks{The author is partially supported by DFG grant TA 309/2-2, 
by the ENB graduate program TopMath and by TUM Graduate School.}

\title[Chio condensation and random sign matrices]{
Chio condensation \\ and random sign matrices
}

\begin{document}

\begin{abstract}
This is to suggest a new approach to the old and open problem of 
counting the number $f_n$ of $\Z$-singular $n\times n$ matrices with entries 
from $\{-1,+1\}$: comparison of two measures, none of them the uniform measure, 
one of them closely related to it, the other asymptotically under control by a 
recent theorem of Bourgain, Vu and Wood.

We will define a measure $\Prob_{\chio}$ on the set $\{-1,0,+1\}^{[n-1]^2}$ of 
all $(n-1)\times (n-1)$-matrices with entries from $\{-1,0,+1\}$ which 
(owing to a determinant identity published by M. F. Chio in 1853) is closely 
related to the uniform measures on $\{-1,+1\}^{[n]^2}$ and $\{0,1\}^{[n-1]^2}$ and at 
the same time it intriguingly mimics the so-called 
\emph{lazy coin flip distribution} $\Prob_{\lcf}$ on $\{-1,0,+1\}^{[n-1]^2}$, with 
the resemblance fading more and more as the events get smaller. This is relevant 
in view of a recent theorem of J. Bourgain, V. H. Vu and P. M. Wood 
(J. Funct. Anal. 258 (2010), 559--603) which proves that if the entries of 
an $n\times n$ matrix  whose $\{-1,0,+1\}$-entries are governed by $\Prob_{\lcf}$ 
and fully independent (they are not when governed by $\Prob_{\chio}$), then an 
asymptotically optimal bound on the singularity probability over $\Z$ can be proved. 
We will characterize $\Prob_{\chio}$ graph-theoretically and use the 
characterization to prove that given a $B\in \{-1,0,+1\}^{[n-1]^2}$, deciding whether 
$\Prob_{\chio}[B] = \Prob_{\lcf}[B]$ is equivalent to deciding an evasive graph 
property, hence the time complexity of this decision is $\Omega(n^2)$. 
Moreover, we will prove $k$-wise independence properties of $\Prob_{\chio}$. 
Many questions suggest themselves that call for further work. In particular, 
the present paper will close with more constrained equivalent formulations of 
the conjecture $f_n/2^{n^2} \sim (\frac12 + o(1))^n$.

\medskip
\noindent
{\it Keywords:} Asymptotic enumeration, Betti number, Chio condensation, 
coboundary space, cut space, cycle space, determinant identities, 
discrete random matrices, lazy coin flip distribution, $k$-wise independence, 
signed graphs
\end{abstract}

\maketitle
{\scriptsize
\tableofcontents
}
\section{Introduction}

For a commutative ring $R$ and a finite subset $U\subseteq R$ one may ask how many 
among the $\lvert U \rvert^{n^2}$ matrices $A\in U^{[n]^2}$ have $\det(A) = 0$. 
Much is known precisely when $R$ is the finite field $\F_{q}$ ($q$ power of a prime) 
and $U=R$. For instance, it follows from elementary linear algebra that the number 
of singular $n\times n$ matrices with entries from $\F_q$ is 
precisely $q^{n^2} - \prod_{0\leq i \leq n-1} (q^n-q^i)$. As an advanced example, very 
precise statements can be proved for matrices over finite fields even if the 
entries are i.i.d. according to (quite) arbitrary distributions (cf. the work 
of Kahn--Koml{\'o}s \cite{MR1833067} and Maples \cite{arXiv:1012.2372v1}. 

In stark contrast, if $R=\Z$ and $U=\{-1,+1\}$, the 
correct order of decay of the density of singular matrices is still not known, 
but there is an old and plausible, yet still unproved conjecture of uncertain 
origin, on which the last two decades have brought several remarkable advances:
\begin{conjecture}\label{conj:socalledfolkloreconjecture}
For $n\rightarrow \infty$, 
$\frac{\lvert \{A\in \{-1,+1\}^{[n]^2}\colon \det(A) = 0 \in \Z\} \rvert}{2^{n^2}} 
\sim (\frac12 + o(1))^n$.
\end{conjecture}

Let us employ the abbreviations $-:=-1$, $+:=+1$, 
$\{\pm\}:=\{-1,+1\}$, $\{0,\pm\}:=\{-1,0,+1\}$, 
$\Prob[\Ra_{<n}(\{\pm\}^{[n]^2})] := \{A\in\{\pm\}^{[n]^2}\colon\det(A) = 0\}$, and 
let $\Prob[\cdot]$ denote the uniform measure on $\{\pm\}^{[n]^2}$, (so that in 
particular Conjecture \ref{conj:socalledfolkloreconjecture} now 
reads $\Prob[\{A\in\{\pm\}^{[n]^2}\colon\det(A) = 0\}] \sim (\frac12 + o(1))^n$). 
Moreover, let us introduce one of the two main protagonists of the present paper:

\begin{definition}[the measure $\Prob_{\lcf}$]
\label{def:definitionofthelazycoinflipmeasure}
For $(s,t)\in \Z_{\geq 2}^2$ and $\emptyset\subseteq I\subseteq [s-1]\times [t-1]$ 
let $\Prob_{\lcf}$ denote the \emph{lazy coin flip distribution} on $\{0,\pm\}^I$, 
i.e. the probability measure on $\{0,\pm\}^I$ defined by considering the values of 
a $B\in \{0,\pm\}^I$ as independent identically distributed random variables, each 
governed by the symmetric discrete distribution with values $-1$, $0$, $+1$ and 
probabilities $\frac14$, $\frac12$, $\frac14$. 
\end{definition}
The name of $\Prob_{\lcf}$ may stem from the fact that this is the 
distribution obtained when someone sets out to generate the entries of some 
$B\in \{0,\pm\}^I$ by performing $\lvert I \rvert$ independent fair coin flips, 
but there is a probability of $\frac12$ at every single trial that out of a 
fleeting laziness the person decides to simply write $0$ instead of flipping 
the coin. 

The lazy coin flip distribution $\Prob_{\lcf}$ plays a role in the recent article 
\cite{MR2557947} of J. Bourgain, V. H. Vu and P. M. Wood in which the authors 
set the current record in a chain of successive improvements of upper bounds for 
$\Prob[\Ra_{<n}(\{\pm\}^{[n]^2})]$ (see Koml{\'o}s \cite{MR0221962}, 
Kahn--Koml{\'o}s-Szemer{\'e}di \cite{MR1260107} 
and Tao--Vu \cite{MR2187480} \cite{MR2291914}):

\begin{theorem}[Bourgain--Vu--Wood \cite{MR2557947}]
\label{thm:bourgainvuwood}
For $n\rightarrow \infty$ it is true that 
\begin{align}
\Prob[\{ A\in \{\pm\}^{[n]^2}\colon \det(A) = 0 \}] & 
\leq (\frac{1}{\sqrt{2}} + o(1))^n \quad , 
\label{thm:bourgainvuwood:uniformdistribution}\\
\Prob_{\lcf}[\{ B\in \{0,\pm\}^{[n]^2}\colon \det(B) = 0\}] & 
\sim (\frac12 + o(1))^n \label{thm:bourgainvuwood:lazycoinflip} \quad .
\end{align}
\end{theorem}
\begin{proof}[Comments]
The upper bound within $\sim$ in \eqref{thm:bourgainvuwood:lazycoinflip} 
is the special case $\mu = \frac12$ of \cite[Corollary 3.1, p. 567]{MR2557947}. 
The lower bound within the $\sim$ is obvious: consider the event that the first 
column has only zero entries (the lower bound is also explitly stated 
in \cite[formula (7), p. 561]{MR2557947}). The upper bound in 
\eqref{thm:bourgainvuwood:uniformdistribution} is the special 
case $S = \{\pm\}$ and $p=\frac12$ in \cite[Corollary 4.3, p. 576]{MR2557947}. 
\end{proof}

So Bourgain--Vu--Wood proved that the correct order of decay of 
$\Prob_{\lcf}[\{ B\in \{0,\pm\}^{[n]^2}\colon \det(B) = 0\}]$ 
is $(\frac12 + o(1))^n$---which is also the conjectured one 
for $\Prob[\{ A\in \{\pm\}^{[n]^2}\colon \det(A) = 0\}]$. It is this latter 
achievement, combined with an observation made by the present author, 
which spurred the present paper. Note that using the uniform distribution 
on $\{\pm\}^{[n]^2}$ is equivalent to considering the $n^2$ entries 
as i.i.d. Bernoulli variables with probability $\frac12$. The observation is this: 
When we apply one step of Chio condensation (see Definition \ref{def:chiocondensation}) 
to this Bernoulli matrix, the result is a matrix whose entries are $3$-wise 
(and `almost' $6$-wise, see Theorem \ref{thm:comparativecountingtheorem} below) 
stochastically independent with $\{-2,0,+2\}$-values which are distributed as if by 
the lazy coin flip distribution. Since Bourgain--Vu--Wood demonstrated that for 
$\Prob_{\lcf}$-distributed entries an asymptotically correct order of decay can be 
proved, the observation feels like a hint at deeper connections and makes it seem 
imperative to investigate Chio condensation of sign matrices. A first step is taken 
in the present paper. 

\section{Definitions}

Let $[n]:=\{1,\dotsc,n\}$. For any $A = (a_{i,j})_{(i,j)\in [n]^2}\in \{\pm\}^{[n]^2}$, 
any $(i,j)\in [n-1]^2$ and any $\emptyset\subseteq I \subseteq [n-1]^2$ let 
$A[i,j] := a_{i,j}$ and $A[I] := (a_{i,j})_{(i,j)\in I}$, hence in particular 
$A[\emptyset]$ is the empty function and $A\bigl[[n]^2\bigr]  = A$. Let 
$\mathfrak{P}(X)$ denote the power set of a set $X$. 

For a cartesian product $M\times N$ of two sets $M$ and $N$ let 
$\upp_1\colon M\times N \rightarrow M$ be the projection onto the first, 
and $\upp_2\colon M\times N \rightarrow N$ the projection onto the second factor. 
If $M$ and $N$ are finite and $\emptyset \subseteq I\subseteq M\times N$ is 
some subset, then $I$ is called \emph{rectangular} if and only if 
$\lvert I\rvert =\lvert \upp_1(I) \rvert \cdot \lvert \upp_2(I) \rvert$. 

Let us view functions as sets and  matrices as functions. If $D$ is a set and 
$f\colon D\rightarrow \Z$ is a function let us write  
$D=:\Dom(f)\supseteq \Supp(f):=\{ d\in D\colon f(d)\neq 0\}$ for its domain and 
support, and let us employ the abbreviations 
$\lvert \Dom(f) \rvert=:\dom(f)\geq\supp(f) := \lvert \Supp(f) \rvert$. We have 
$\Dom(\emptyset)=\Supp(\emptyset)=\emptyset$ and 
therefore $\dom(\emptyset)=\supp(\emptyset)=0$.  If $U\subseteq \Z$, 
$\emptyset\subseteq I \subseteq [s-1]\times [t-1]$ and $B\in U^I$, then we have 
$[s-1]\times [t-1]\supseteq I = \Dom(B) \supseteq \Supp(B) = 
\{(i,j)\in [s-1]\times [t-1]\colon B[(i,j)]\neq 0\}$. 
For a matrix $M = (m_{i,j})_{(i,j)\in I} \in \Q^I$ and a $q\in \Q$ we define, as usual, 
$q\cdot M := (q\cdot m_{i,j})_{(i,j)\in I}$. The symbol $\sqcup$ denotes a set union 
$\cup$ and at the same time makes the claim that the union is disjoint. The term 
\emph{rank} of matrix has its usual meaning (and we will only use it in the context 
of integral domains, so that row-rank, column-rank and determinantal rank are all 
the same). For a set $S$, the group of all permutations of $S$ is denoted 
by $\Sym(S)$. 

The word `graph' without any further qualifications means `finite simple graph' 
(i.e. `finite $1$-dimensional simplicial complex'). We will use $\upV(X)$ 
(resp. $\upE(X)$) to denote the vertex set (resp. edge set) of a graph $X$, and we 
follow \cite{MR1373656} in reserving the more specific term \emph{circuit} for what 
is called a \emph{cycle} in \cite{MR2159259} (i. e. closed walk 
without self-intersections). Moreover, we will use $f_1(X)$ for the number 
of edges of a graph $X$ and $f_0(X)$ for the 
number of its of vertices. The \emph{cycle space of $X$} (i.e. $1$-dimensional 
cycle group with $\Z/2$-coefficients in the sense of simplicial homology theory) 
will be denoted by $\upZ_1(X;\Z/2)$ and the \emph{coboundary space of $X$} 
by $\upB^1(X;\;\Z/2)$ (this is the $1$-dimensional coboundary group 
with $\Z/2$-coefficients in the sense of simplicial cohomology theory; a synonym 
is `cut space of $X$'). Let $\beta_0(X)$ denote the number of connected components 
of a graph $X$ and $\beta_1(X) := \dim_{\Z/2}\ Z_1(X;\Z/2)$ the first Betti 
number (a synonym in the graph-theoretical literature is `cyclomatic number' 
\cite{MR1847424}). We will (without further notification) use the $1$-dimensional 
case of the alternating sum relation between the ranks of the chain groups and 
the ranks of the homology groups of a free chain complex, 
i.e.  $\beta_1(X) - \beta_0(X) = f_1(X) - f_0(X)$ for every graph $X$. 
For any two disjoint graphs $X_1$ and $X_2$, the graph obtained by identifying 
exactly one vertex of $X_1$ with exactly one vertex of $X_2$ is called 
the \emph{(one-point) wedge of $X_1$ and $X_2$} and denoted by $X_1 \vee X_2$. 
This is the standard wedge product of pointed topological spaces (but only 
vertices of a graph are allowed as basepoints); a synonym within the 
graph-theoretical literature is `coalescence' \cite[p. 140]{MR1054137}. 

We will use the language of \emph{signed graphs}  (see \cite{MR1744869} for a 
comprehensive overview). It is customary in signed graph theory to work with 
multigraphs (i.e. finite $1$-dimensional CW-complexes) for reasons of higher 
flexibility in proofs and applications. However, in the present paper, all we will 
need are signed simple graphs, i.e. for us a signed graph $(X,\sigma)$ will 
simply consist of a  graph $X = (V,E)$ together with an 
arbitrary \emph{sign function} $\sigma\colon E\rightarrow \{\pm\}$. 
We call \emph{$(+)$-edge} (resp. $(-)$-edge) every $e\in \upE(X)$ with 
$\sigma(e) = +$ (resp. $\sigma(e) = - $). Define $(+)$-paths (resp. $(-)$-paths) 
as paths all of whose edges are $(+)$-edges (resp. $(-)$-edges). For emphasizing the 
sign function we employ the notation `($\sigma$, $+$)-edge'.  
If $(X,\sigma)$ is a signed graph let 
$f_1^{(-)}(X,\sigma):=\lvert \{ e\in \upE(X)\colon \sigma(e) = - \}\rvert$ denote
the number of ($\sigma$, $-$)-edges in it. A signed graph $(X,\sigma)$ is 
called \emph{balanced}\footnote{The use of this term seems to have been initiated 
in \cite{MR0067468}. The notion itself was already studied over seventy years ago by 
D. K{\H{o}}nig \cite[p. 149, Paragraph 3]{MR0036989} under the 
name `$p$-Teilgraph'.} if and only if $f_1^{(-)}(C,\sigma)$ is even for every 
circuit $C$ of $X$. We will denote the set of all balanced signings of $X$ by 
$S_{\bal}(X) := 
\{ \sigma\in\{\pm\}^{\upE(X)}\colon \text{$(X,\sigma)$ balanced} \}$.

\begin{definition}[{Chio\footnote{In earlier versions I wrote `Chi{\`o}' but this 
now seems wrong. All three spellings Chi{\`o}, `Chio' and `Chi{\'o}' are to be 
found in the literature. My sole reason for using `{\`o}' was that 
in \cite{MR1500275} the authors consistently use the 
spelling `Chi{\`o}' and it is said \cite[p. 790]{MR1500275} that a copy of 
Chio's original paper had been at the authors' disposal. However, 
an original 1853 copy of \cite{chio} which I recently bought from an 
antiquarian bookstore in Italy gives strong circumstantial evidence in favour of 
the spelling `Chio': on the title page and the inside-cover his given name 
`F{\'e}lix' is written with an accent whereas `Chio' does not carry any accent. 
Moreover, the title page bears a hand-written dedication to a colleague, 
signed `L'autore'. Therefore, to all appearances, Chio signed this title page 
himself, 158 years ago. Possibly extant autographs aside, putting this on record 
might come as close to a personal statement by Chio as one will ever get nowadays.
Moreover, the spelling is further corroborated by the usage adopted by 
Cauchy in \cite{cauchyoeuvrescompletespremieresserietomex}. Cauchy on several 
occasions consistently writes `M. F{\'e}lix Chio' \cite[p. 110, pp. 112--113]{cauchyoeuvrescompletespremieresserietomex}.} set}] 
Let $(s,t)\in \Z_{\geq 2}^2$ and $I\subseteq [s]\times [t]$. Then $I$ is 
called a \emph{Chio set} if and only if $(s,t)\in I$ and 
for every $(i,j)\in I$ we have $(i,t)\in I$ and $(s,j)\in I$. 
\end{definition}

\begin{definition}[Chio extension\footnote{Due to the change of spelling explained 
in the previous footnote I now use a breve instead of a grave accent to denote 
Chio extension.}]\label{def:chioextensionofaset} 
For every $(s,t)\in\Z_{\geq 2}^2$ and every 
$\emptyset \subseteq I \subseteq [s-1]\times [t-1]$, 
\begin{equation}\label{eq:definitionofchiodom}
\breve{I}:= \{ (s,t)\} \sqcup \bigcup_{i\in\upp_1(I)}\{ (i,t) \} 
\sqcup\bigcup_{j\in\upp_2(I)} \{ (s,j)\} \sqcup I \quad.
\end{equation}
\end{definition}

Note that $\breve{I} \subseteq [s]\times [t]$ for every 
$\emptyset \subseteq I \subseteq [s-1]\times [t-1]$, in particular 
$\breve{\emptyset} = \{(s,t)\}$ and $([s-1]\times [t-1])^{\breve{}} = [s]\times [t]$.
Moreover, a set $I'\subseteq [s]\times [t]$ is a Chio set if and only if there 
exist an $I\subseteq[s-1]\times [t-1]$ with $I' = \breve{I}$. 

Now we come to Chio condensation. In the special (and very common) 
case of $s=t$ (hence $[s]\times [t] = [n]^2$) the following definition differs 
from the standard convention (as is to be found in \cite{MR1500275} and 
\cite{MR649067}) in that the entry $a_{n,n}$ instead of $a_{1,1,}$ is taken to be 
the pivot. This seems to be more convenient for handling the indices of a Chio 
condensate. There does not appear to  be  any logical necessity  for multiplying 
by $\tfrac12$, but the author decided to keep the discussion within the 
realm of $\{0,\pm\}$-matrices (instead of $\{-2,0,+2\}$-matrices). 

\begin{definition}[Chio condensation, $\tfrac12\upC_{(s,t)}^{\breve{I}}$]
\label{def:chiocondensation}
For  every $(s,t)\in \Z_{\geq 2}^2$, and every $I\subseteq [s-1]\times [t-1]$ define 
the \emph{Chio map with pivot $a_{s,t}$} as 
\begin{equation}
\tfrac12\upC_{(s,t)}^{\breve{I}}\colon \{\pm\}^{\breve{I}} 
\longrightarrow \{0,\pm\}^I \quad , \qquad 
A \longmapsto \tfrac12\cdot  \upC_{(s,t)}(A) \quad ,
\end{equation}
where $\upC_{(s,t)}(A) :=\bigr (\det(
\begin{smallmatrix} 
a_{i,j} & a_{i,t} \\
a_{s,j} & a_{s,t}
\end{smallmatrix})\bigr)_{(i,j)\in I} \in \{-2,0,+2\}$. An image $\upC_{(s,t)}(A)$ of 
some $A\in \{\pm\}^{\breve{I}}$ is referred to as the 
\emph{Chio condensate of $A$}. 
\end{definition}

\begin{definition}[the measure $\Prob_{\chio}$]\label{def:measurePchio}
For every $(s,t)\in\Z_{\geq 2}^2$ and every $I\subseteq [s-1]\times [t-1]$ define
\begin{equation}\label{eq:definitionmeasurePchio}
\Prob_{\chio} \colon \mathfrak{P}\bigl ( \{0,\pm\}^I  \bigr ) 
\longrightarrow [0,1] \quad , \qquad 
\mathcal{B} \longmapsto \frac{1}{2^{\lvert \breve{I} \rvert}}\ \sum_{B\in\mathcal{B}}\; 
\lvert(\tfrac12 \upC_{(s,t)}^{\breve{I}})^{-1}(B)\rvert \quad .
\end{equation}
\end{definition}

Note that in the special case of $s:=t:=n$ and $I:=[n-1]^2$, 
the measure $\Prob_{\chio}$ maps a single 
$B\in \{0,\pm\}^{[n-1]^2}$ to $\Prob_{\chio}[B] := \Prob_{\chio}[\{B\}] = 2^{-n^2}\cdot 
\bigl \lvert \{ A\in \{\pm\}^{[n]^2}\colon 
B = \tfrac12\cdot \upC_{(n,n)}(A) \} \bigr \rvert$.

We now define two additional measures. Later we will recognize both of 
them as familiar ones: 

\begin{definition}[averaged Chio measure]
For every $\emptyset\subseteq I \subseteq [n-1]^2$ define
\begin{equation}
\overline{\Prob}_{\chio} \colon  \mathfrak{P}(\{0,\pm\}^I) 
\longrightarrow [0,1]\quad , \qquad \mathcal{B} \longmapsto 
\sum_{B\in\mathcal{B}} \frac{1}{2^{\supp (B)}} 
\sum_{\tilde{B} \in \{0,\pm\}^I\colon \Supp(\tilde{B}) = \Supp(B)} 
\Prob_{\chio}[\tilde{B}]\quad .
\end{equation}
\end{definition}

\begin{definition}[$\lvert \frac12\upC_{(s,t)}^{\breve{I}}(\cdot)\rvert$ and 
$\Prob_{\chio}^{\lvert\cdot\rvert, I}$]
For every $(s,t)\in\Z_{\geq 2}^2$ and every $I\subseteq [s-1]\times [t-1]$ define 
a map 
\begin{equation}
\lvert\tfrac12\upC_{(s,t)}^{\breve{I}}\rvert\colon \{\pm\}^{\breve{I}} 
\longrightarrow \{0,1\}^I \quad , \qquad 
A  \longmapsto \tfrac12\cdot \lvert \upC_{(s,t)}(A) \rvert \quad ,
\end{equation}
where $\lvert\upC_{(s,t)} (A)\rvert := \bigl (\lvert \det
\left(\begin{smallmatrix} 
a_{i,j} & a_{i,t} \\
a_{s,j} & a_{s,t}
\end{smallmatrix}\right)\rvert\bigr)_{(i,j)\in I} \in \{0,2\}^I$. 
Furthermore, define
\begin{equation}
 \Prob_{\chio}^{\lvert\cdot\rvert, I}\colon\mathfrak{P}(\{0,1\}^I) 
\longrightarrow [0,1] \cap \Q \quad , \qquad \mathcal{B} 
\longmapsto \frac{1}{2^{\lvert\breve{I}\rvert}}\sum_{B\in\mathcal{B}} 
\bigl \lvert (\lvert\tfrac12\upC_{(s,t)}^{\breve{I}}\rvert)^{-1}(B) \bigr \rvert \quad .
\end{equation}
\end{definition}

\begin{definition}[the entry-specification-events $\mathcal{E}_B^J$]
\label{def:entryspecificationevents}
For $(s,t)\in\Z_{\geq 2}^2$, 
$\emptyset \subseteq I \subseteq J \subseteq [s-1]\times[t-1]$, 
and $B\in \{0,\pm\}^I$ let $\mathcal{E}_B^J := 
\bigl \{ \tilde{B} \in \{0,\pm\}^J\colon \tilde{B}\mid_{\Dom(B)} = B \bigr \}$.
\end{definition}

Note that $\Dom(B) = I \subseteq J =\Dom(\tilde{B})$, hence 
$\tilde{B}\mid_{\Dom(B)}$ is defined. If $\Dom(B)=\emptyset$, i.e. $B=\emptyset$,  then 
$\mathcal{E}_{\emptyset}^J = \{0,\pm\}^J$, and if $\Dom(B)=J$, 
then $\mathcal{E}_B^J = \{B\}$. Furthermore, 
$\lvert \mathcal{E}_B^J \rvert = 3^{\lvert J \rvert - \lvert I \rvert}$ for 
arbitrary $\emptyset\subseteq I \subseteq J \subseteq [s-1]\times [t-1]$ 
and $B\in\{0,\pm\}^I$.

In this paper we intend to use graph-theoretical language. For the sake of 
specificity and ease of reference, we will explicitly name the set of auxiliary 
labelled bipartite graphs that we will talk about (and give it a vertex set which 
blends well with the matrix setting).

\begin{definition}
For every $(s,t)\in\Z_{\geq 2}^2$ denote by $\BG_{s,t}$ the $2^{(s-1)\cdot (t-1)}$-element 
set of all bipartite graphs $X=(V_1\sqcup V_2,E)$  with 
$V_1 =\{(i,t)\colon 1\leq i \leq s-1\}$ and $V_2 =\{(s,j)\colon 1\leq j \leq t-1\}$. 
\end{definition}
There is an obvious bijection $\BG_{s,t} \longleftrightarrow \{0,1\}^{[s-1]\times [t-1]}$. 
Associating with a (partially specified) $\{0,\pm\}$-matrix the following bipartite 
signed graph will be helpful in our study of $\Prob_{\chio}$. The definition can be 
summarized by saying that $\upX$ interprets a $B\in\{0,\pm\}^I$ as a bipartite 
adjacency matrix in the natural way (while ignoring the signs), and that $\sigma$ takes 
the signs in $B$ as a definition of a sign function.

\begin{definition}[$\upX$ and $\sigma$]\label{def:XBandecXB}
For every $(s,t)\in\Z_{\geq 2}^2$ and every $0\leq k \leq (s-1)(t-1)$ define 
\begin{equation}\label{eq:def:XBandecXB}
\upX^{k,s,t}\colon \bigsqcup_{I \in \binom{[s-1]\times [t-1]}{k}} \{ 0, \pm \}^I 
\longrightarrow \BG_{s,t},\quad B \longmapsto \upX_B^{k,s,t}
\end{equation}
by letting vertex-set and edge-set be defined as  
\begin{align}
\upV(\upX_B^{k,s,t}) & := 
(\Dom(B))^{\breve{}}\setminus \Dom(B) \setminus \{(s,t)\} \quad , 
\label{def:XBandecXB:definitionofvertexset} \\
\upE(\upX_B^{k,s,t}) & := 
\bigsqcup_{(i,j) \in \Supp(B)} \bigl \{ \{(i,t),(s,j)\} \bigr \}\quad .
\label{def:XBandecXB:definitionofedgeset}
\end{align}
Define  $\sigma_B\colon \upE(\upX_B) \rightarrow \{\pm\}$ by 
$\sigma_B (\{(i,t),(s,j)\}) := b_{i,j} \in \{\pm\}$ 
for every $\{(i,t),(s,j)\}\in \upE(\upX_B)$. \\
\end{definition}
If $k=0$, hence $I=\emptyset$, hence $B=\emptyset$ is the empty matrix, 
then $\upX_B^{k,s,t}$ is the empty graph $(\emptyset,\emptyset)$ and 
$\sigma_B = \emptyset$ is the empty function. Note that while for 
a $B\in\{0,\pm\}^I$ the set $\upV(\upX_B)$ depends only on $I=\Dom(B)$, 
the set $\upE(\upX_B)$ depends on $\Supp(B)$ and $\sigma_B$ even depends 
on $B$ itself. 

When we take the \emph{image} of a matrix $B\in \{0,\pm\}^I$ under $\upX^{k,s,t}$, 
then usually we will know what $I\in\binom{[s-1]\times [t-1]}{k}$ we are talking 
about and then the superscripts $k,s,t$ give redundant information. 
Whenever possible we will suppress the superscripts in such a situation 
and only write $\upX_B$. When we take \emph{preimages} of a graph $X\in\BG_{s,t}$ 
under $\upX^{k,s,t}$, however, the full notation has to be used since in general 
it is not possible to tell $k$ from the labelled graph $X$ (let alone from its 
isomorphism type). As an example for this, consider the graph $X\in\BG_{4,4}$ 
with vertex set $\{(1,4),(2,4),(3,4)\}\sqcup\{(4,1),(4,2),(4,3)\}$ 
and edge set $\{ \{(1,4),(4,1)\},\{(2,4),(4,1)\},\{(2,4),(4,2)\},\{(1,4),(4,2)\}\}$, 
which is isomorphic to a $4$-circuit with two additional isolated vertices. Then we 
have $\upX^{5,4,4}_{B_1} = \upX_{B_2}^{6,4,4} = X$ for $B_1 :=\left(\begin{smallmatrix} 
1 & 1 &  \\ 1 & 1 &  \\  &  & 0 
\end{smallmatrix}\right)\neq B_2 :=\left(\begin{smallmatrix} 
1 & 1 &  \\ 1 & 1 & 0 \\  & 0 &  
\end{smallmatrix}\right)$. Here, $\dom(B_1) = 5 \neq 6 = \dom(B_2)$.

In the following, we deliberately do not define `isomorphism type of a graph' more 
precisely. We would not have much use for any of the existing formalizations of an 
unlabelled graph.

\begin{definition}[$\ul$, $\beta_1^{\ul}$]
Let $\ul$ be the map which assigns a labelled graph to its isomorphism type. 
Let $\beta_1^{\ul}\colon \ul(\BG_{n,n}) \to \Z_{\geq 0}$ be the map which assigns an 
unlabelled bipartite graph to its $1$-dimensional Betti number.
\end{definition}

\begin{definition}[$\upXul^{k,s,t}$]\label{def:upXul}
For   $(s,t) \in \Z_{\geq 2}^2$ and $0\leq k \leq (s-1)(t-1)$ let 
$\upXul^{k,s,t} := \ul\circ\upX^{k,s,t}$. 
\end{definition}

If $\mathfrak{X}$ is some (verbal, pictorial, ...) description of an isomorphism 
type of graphs, we can now take its preimage 
$(\upXul^{k,s,t})^{-1}(\mathfrak{X}) \subseteq \bigsqcup_{I\in \binom{[s-1]\times [t-1]}{k}} 
\{0,\pm\}^I$. To analyse how $\Prob_{\chio}$ and $\Prob_{\lcf}$ relate to one another, 
it is useful to have the following notations. 

\begin{definition}[failure sets]\label{def:parametrizedfailuresets}
For every $k\geq 0$, $n\geq 2$, $\ell\in \Q_{\geq 0}$ and $p \in [0,1]\cap \Q$ let 
\begin{enumerate}[label={\rm(\arabic{*})}]
\item $\mathcal{F}^{\mathrm{M}}(k,n)$ $:=$ 
$\{$ $B\in\{0,\pm\}^I\colon I\in\binom{[n-1]^2}{k},\; 
\Prob_{\chio}[\mathcal{E}_B^{[n-1]^2}] \neq \Prob_{\lcf}[\mathcal{E}_B^{[n-1]^2}]$ $\}$ \quad ,
\item $\mathcal{F}_{\cdot\ell}^{\mathrm{M}}(k,n)$ $:=$ 
$\{$ $B\in\mathcal{F}^{\mathrm{M}}(k,n)\colon 
\Prob_{\chio}[\mathcal{E}_B^{[n-1]^2}] = \ell\cdot \Prob_{\lcf}[\mathcal{E}_B^{[n-1]^2}]$ $\}$ 
 $\subseteq$ $\mathcal{F}^{\mathrm{M}}(k,n)$ \quad ,
\item $\mathcal{F}_{=p}^{\mathrm{M}}(k,n) := $ 
$\{$ $B\in \mathcal{F}^{\mathrm{M}}(k,n)\colon 
\Prob_{\chio}[\mathcal{E}_B^{[n-1]^2}] = p$ $\}$ 
$\subseteq$ $\mathcal{F}^{\mathrm{M}}(k,n)$ \quad ,
\item\label{def:graphversionsofthefailuresets} 
$\mathcal{F}^{\mathrm{G}}(k,n) := 
\upXul^{k,n,n}(\mathcal{F}^{\mathrm{M}}(k,n))$, 
$\mathcal{F}_{\cdot \ell}^{\mathrm{G}}(k,n) := 
\upXul^{k,n,n}(\mathcal{F}_{\cdot\ell}^{\mathrm{M}}(k,n))$, \\ 
and $\mathcal{F}_{=p}^{\mathrm{G}}(k,n) := 
\upXul^{k,n,n}(\mathcal{F}_{=p}^{\mathrm{M}}(k,n))$.
\end{enumerate}
We abbreviate $\mathcal{F}^{\mathrm{M}}(k,n) := \mathcal{F}^{\mathrm{M}}(k,n,n)$, and 
analogously for all the other sets just defined.  
\end{definition}

Obviously, $\mathcal{F}_{\cdot 1}^{\mathrm{M}}(k,n) = \emptyset$ 
and $\mathcal{F}_{\cdot 0}^{\mathrm{M}}(k,n) = \mathcal{F}_{= 0}^{\mathrm{M}}(k,n)$ for 
all $k$ and $n$. Item \ref{relationbeweenchiomeasureandlazycoinflipmeasuregovernedbyfirstbettinumber} in 
Theorem \ref{thm:graphtheoreticalcharacterizationofthechiomeasure} will teach us 
that $\mathcal{F}_{\cdot \ell}^{\mathrm{M}}(k,n) = \emptyset$ for every 
$\ell\notin \{ 0 \} \sqcup \{2^i\colon i\in\Z_{\geq 0}\}$ (hence in particular 
$\Prob_{\chio}[\mathcal{E}_B^{[n-1]^2}] \geq \Prob_{\lcf}[\mathcal{E}_B^{[n-1]^2}]$ for 
every $B\in\mathcal{F}^{\mathrm{M}}(k,n)$ with 
$\Prob_{\chio}[\mathcal{E}_B^{[n-1]^2}] > 0$).

\begin{definition}[matrix-circuit, $\Cir(s,n)$]
\label{def:circuitofmatrixentrypositions}
For every $(s,t)\in\Z_{\geq 2}^2$ and every $L \subseteq [s-1]\times [t-1]$ 
with even $l := \lvert L \rvert$, the set $L$ is called a \emph{matrix-$l$-circuit} 
if and only if $\upX_{\{1\}^L}$ is a graph-theoretical $l$-circuit. 
Moreover, $\Cir(l,s,t):=\{ L\subseteq [s-1]\times [t-1]\colon \lvert L \rvert = l,\, 
\text{$L$ is a matrix-$l$-circuit}\}$ and $\Cir(l,n):=\Cir(l,n,n)$. 
\end{definition}

\begin{definition}[$(-)$-constant, $(+)$-proper vertex $2$-colouring of a 
signed graph]
For a graph $X=(V,E)$ and a $\sigma\in \{\pm\}^E$, a function 
$\upc\in \{\pm\}^V$ is called 
\emph{($\sigma$, $-$)-constant, ($\sigma$, $+$)-proper} if and only if 
$\upc(u) = \upc(v)$ for every $e = uv\in \upE(X)$ with $\sigma(e) = -$ 
and 
$\upc(u) \neq \upc(v)$ for every $e = uv\in \upE(X)$ with $\sigma(e) = +$. 
\end{definition}

\begin{definition}[$\Col(X,\sigma)$]\label{def:ColecX}
For a graph $X=(V,E)$ and a $\sigma\in \{\pm\}^E$ let $\Col(X,\sigma)$ be the set of 
all ($\sigma$, $-$)-constant, ($\sigma$, $+$)-proper vertex-$2$-colourings 
$\upc\in \{\pm\}^V$.
\end{definition}

\begin{definition}[rank-level-sets of matrices]\label{def:ranksets}
For $(s,t)\in\Z_{\geq 2}^2$, $0\leq r \leq \min(s,t)$, $R$ an integral domain, 
$U\subseteq R$ and $\mathcal{R}\in\mathfrak{P}(\{0,1,\dotsc,\min(s,t)\})$  let 
$\Ra_r(U^{[s]\times[t]} )$ $:=$ $\{ A\in U^{[s]\times[t]}\colon \rk(A) = r\}$, 
$\Ra_{\mathcal{R}}(\{\pm\}^{[s]\times [t]}) := \bigsqcup_{r\in\mathcal{R}} 
\Ra_r(\{\pm\}^{[s]\times [t]})$ and 
$\Ra_{<r}(U^{[s]\times[t]} )$ $:=$ $\Ra_{\{0,1,\dotsc,r-1\}}(\{\pm\}^{[s]\times [t]})$. 
\end{definition}

\section{Lemmas}

We will use the following elementary fact: 

\begin{lemma}\label{lem:elementarypreimagestatements:finvffinv}
$f^{-1}(f(f^{-1}(U))) = f^{-1}(U)$ for any map $f\colon A\to B$ and any 
subset $U\subseteq B$. \hfill $\Box$
\end{lemma}

The following simple statement is essential for the approach developed in the 
present paper. More information on this identity can be found in 
\cite[last paragraph of Section 9]{MR1500275} and 
\cite[Ch. 4, p. 282, Exerc. 43]{MR1181420}. The formulation given here differs from 
those in \cite[p. 11]{chio} and \cite{MR1500275} in that $a_{n,n}$ instead of $a_{1,1}$ 
is taken to be the pivot. This seems more convenient for handling the indices of a 
Chio condensate. 

\begin{lemma}[Chio's identity]\label{lem:chioidentity}
For $n\geq 2$, $R$ an integral domain and $(a_{i,j}) = A\in R^{[n]^2}$, 
\begin{equation}\label{eq:chiocondensation}
\det \bigl ( \upC_{(n,n)}(A) \bigr ) = a_{n,n}^{n-2} \cdot \det(A) \quad .
\end{equation}
\end{lemma}
\begin{proof}
This is stated by Chio in \cite[p. 11, Th{\'e}or{\`e}me 4, equation (20)]{chio} and 
he proves it on pp. 6--11 (the notation `$\pm a_0b_1$' employed in 
\cite[equation (13'')]{chio} is defined at the beginning of p. 6 of \cite{chio}). 
To contemporary eyes, this is an easy consequence of the behavior of determinants 
under linear transformations, cf. \cite{MR649067} for a direct proof of the version 
with pivot $a_{1,1}$. Moreover, this is a special case of 
\emph{Sylvester's determinant identity}. To see this, set $k=1$ in formula $(8)$ of 
\cite{MR1500275} to get a version of \eqref{eq:chiocondensation} with pivot 
$a_{1,1}$. Obvious modifications in the proof in \cite{MR1500275} yield the version 
with pivot $a_{n,n}$. 
\end{proof}

The following three assertions are obviously true: 

\begin{corollary}\label{lem:equivalenceofvanishingofdeterminants}
For every $A\in \{\pm\}^{[n]^2}$, $\det(A) = 0$ if and only if 
$\det(\tfrac12\upC_{(n,n)}(A)) = 0$. \hfill $\Box$
\end{corollary}

\begin{lemma}[value of lazy coin flip distribution on single matrix]
\label{lem:valueoflazycoinflipdistribution}
For every $\emptyset\subseteq I \subseteq [n-1]^2$ and every 
$B \in \{0,\pm\}^{I}$, $\Prob_{\lcf}[\mathcal{E}_B] = (\frac12)^{ \dom(B) + \supp(B) }$. 
\hfill $\Box$
\end{lemma}

\begin{lemma}\label{lem:onepointwedgepreservesbalancedness}
For any two disjoint graphs $X_1$ and $X_2$ and any two sign functions
$\sigma_{X_1}\in \{\pm\}^{\upE(X_1)}$ and  $\sigma_{X_2}\in\{\pm\}^{\upE(X_2)}$, and for 
every graph $X$ obtained by a one-point wedge of $X_1$ and $X_2$ at two arbitrary 
vertices, the sign function  $\sigma_X\in \{\pm\}^{\upE(X)}$ obtained by uniting the 
maps $\sigma_{X_1}$ and $\sigma_{X_2}$ is balanced  if and only if both 
$(X_1,\sigma_{X_1})$ and $(X_2,\sigma_{X_2})$ are balanced. \hfill $\Box$
\end{lemma}

The following will be needed for counting failures of equality 
of $\Prob_{\chio}$ and $\Prob_{\lcf}$. 

\begin{lemma}
\label{lem:numberofmatrixcircuitsofgivenlengthwithingivencartesianproduct} 
$\lvert \Cir(2j,s,t) \rvert = \binom{s-1}{j}\cdot\binom{t-1}{j}\cdot 
\frac{j!(j-1)!}{2}$ for every $n\geq 2$ and every 
$1\leq j \leq \min(\lfloor\frac{s}{2}\rfloor, \lfloor\frac{t}{2}\rfloor)$. 
\end{lemma}
\begin{proof}
For every $N'\in \binom{[s-1]}{j}$ and $N''\in \binom{[t-1]}{j}$ let 
$\Cir(2j,s,t,N',N'') := \{ S\in \Cir(2j,s,t)\colon \upp_1(S) = N',\; 
\upp_2(S) = N'' \}$. Obviously, $\lvert\Cir(2j,s,t)\rvert = 
\sum_{N'\in \binom{[s-1]}{j}} \sum_{N''\in \binom{[t-1]}{j}} \lvert\Cir(2j,s,t,N',N'')\rvert$.
To count $\Cir(2j,s,t,N',N'')$, define $\Perm(N',N'')$ $:=$ $\{$ $P$ $\colon$ $P$ is 
a permutation matrix, $\Supp(P)$ $\subseteq$ $N'\times N''$ $\}$. Define a binary 
relation $\mathfrak{R} \subseteq \Cir(2j,s,t,N',N'')\times \Perm(N',N'')$ by letting 
$(S,P) \in \mathfrak{R}$ if and only if $\Supp(S)\supseteq \Supp(P)$. For every 
$S\in\Cir(2j,s,t,N',N'')$ we have 
$\lvert \{ P\in\Perm(N',N'')\colon (S,P)\in\mathfrak{R} \} \rvert = 2$. On the other 
hand, for every $P\in\Perm(N',N'')$ we have 
$\lvert \{ S\in\Cir(2j,s,t,N',N'')\colon (S,P)\in\mathfrak{R} \} \rvert = (j-1)!$. 
Therefore, $(j-1)!$ $\cdot$ $j!$ $=$ $(j-1)!$ $\cdot$ $\lvert \Perm(N',N'')\rvert$ 
$=$ $\sum_{P\in\Perm(N',N'')}$ $\lvert \{$ $S\in\Cir(2j,s,t,N',N'')\colon$ 
$(S,P)$ $\in$ $\mathfrak{R}$ $\} \rvert$ $=$ $\sum_{S\in\Cir(2j,s,t,N',N'')}$ 
$\lvert \{$ $P$ $\in$ $\Perm(N',N'')\colon$ $(S,P)$ $\in$ $\mathfrak{R}$ 
$\} \rvert$ $=$ $2$ $\cdot$ $\lvert \Cir(2j,s,t,N',N'') \rvert$. 
\end{proof}

The following is contained in K{\H{o}}nig's 1936 classic \cite{MR0036989}. 

\begin{lemma}[D. K{\H{o}}nig]\label{lem:equivalenceofexistenceofbconstantrpropervertex2coloringandcyclicallyrevenness}
Let $\mathfrak{X}$ be a labelled or an unlabelled  graph. Then:
\begin{enumerate}[label={\rm(K{\H{o}}\arabic{*})}]
\item\label{item:equivalenceofexistenceofbconstantrpropervertex2coloringandcyclicallyrevenness}  $S_{\bal}(\mathfrak{X})\neq\emptyset$  if and only if 
$\Col(\mathfrak{X},\sigma)\neq\emptyset$. 
\item\label{item:cardinalityofsetofbconstantrproper2colorings}  
Let $\sigma\colon \{\pm\}^E$ be arbitrary. Then 
$\Col(\mathfrak{X},\sigma)\neq \emptyset$ if and only if 
$\lvert\Col(\mathfrak{X},\sigma)\rvert = 2^{\beta_0(\mathfrak{X})}$.
\item\label{item:numberofbalancedsignfunctions} 
$\lvert \{\sigma\in\{\pm\}^{\upE(\mathfrak{X})}\colon 
(\mathfrak{X},\sigma)\;\mathrm{balanced}\}\rvert = 
2^{f_0(\mathfrak{X}) - \beta_0(\mathfrak{X})}$.
\end{enumerate}
\end{lemma}
\begin{proof}
Modulo terminology a proof for \ref{item:equivalenceofexistenceofbconstantrpropervertex2coloringandcyclicallyrevenness} can be found in \cite[p. 152, Satz 11]{MR0036989} 
(for the definition of `$p$-Teilgraph' cf. \cite[p. 149, Paragraph 3]{MR0036989}). 
Statement \ref{item:equivalenceofexistenceofbconstantrpropervertex2coloringandcyclicallyrevenness} is also proved in \cite[Theorem 3]{MR0067468}. 
Statement \ref{item:cardinalityofsetofbconstantrproper2colorings} is implicit 
in the proof of \cite[p. 152, Satz 14]{MR0036989} and can easily be proved 
directly by induction on $\lvert E \rvert$. For a proof of 
\ref{item:numberofbalancedsignfunctions} cf. e.g. \cite[p. 152, Satz 14]{MR0036989}. 
\end{proof}

While for a given graph $X = (V,E)$ and a given sign function 
$\sigma\colon E\rightarrow \{\pm\}$, the decision problem of whether $(X,\sigma)$ 
balanced is trivially in co-NP, the less obvious fact that it is also in NP 
follows from \ref{item:equivalenceofexistenceofbconstantrpropervertex2coloringandcyclicallyrevenness}: any ($\sigma$, $-$)-constant, ($\sigma$, $+$)-proper 
vertex-$2$-colouring $\upc\colon V \rightarrow \{\pm\}$ is a polynomially-sized 
certificate for $(X,\sigma)$ being balanced. However, the problem is not only in the 
intersection of these two classes but easily seen to be in P:

\begin{corollary}\label{cor:decidingwhetherecXiscyclicallyreven}
For every graph $X = (V,E)$ and every sign function $\sigma\colon E \rightarrow \{\pm\}$, the decision problem whether $(X,\sigma)$ is balanced can be solved in 
time $O( f_0(X) + f_1(X))$.
\end{corollary}
\begin{proof}
By \ref{item:equivalenceofexistenceofbconstantrpropervertex2coloringandcyclicallyrevenness}, the question is equivalent to whether there exists a a $(-)$-constant, 
$(+)$-proper vertex-$2$-colouring $\upc\colon V\rightarrow \{\pm\}$. It is easy to 
see that an obvious greedy algorithm via a (e.g.) depth-first search on $X$ succeeds 
in finding such a colouring if and only if such a colouring exists. Moreover, the 
algorithm requires time $O(f_0(X) + f_1( X ))$. 
\end{proof}

The following simple lemma encapsulates a basic mechanism linking  Chio 
condensation with the auxiliary graph-theoretical viewpoint. For want of topologies 
on source or target, `$k$-fold cover' is nothing but shorthand for `surjective map 
each of whose fibres has cardinality $k$'. 

\begin{lemma}\label{lem:theparametrizedcover}
For every $(s,t)\in\Z_{\geq 2}^2$, arbitrary  
$\emptyset \subseteq I \subseteq J \subseteq [s-1]\times [t-1]$, every 
$B\in \{0,\pm\}^I$,  and with $h(I,J) := \lvert \breve{J} \rvert - \lvert  I \rvert 
- \lvert \upp_1(I)\rvert - \lvert \upp_2(I)\rvert 
\in \Z_{\geq 1}$, there exists an $2^{h(I,J)}$-fold cover 
\begin{equation}
\Phi\colon (\tfrac12\upC_{(s,t)}^{\breve{J}})^{-1}(\mathcal{E}_{B}^J) \longrightarrow 
\Col(\upX_B,\sigma_B)\quad .
\end{equation}
\end{lemma}
\begin{proof}
Let $s$, $t$, $I$, $J$ and $B = (b_{i,j})_{(i,j)\in I}$ be given as stated. The claim 
$h(I,J)\in \Z_{\geq 1}$ is true since Definition \ref{def:chioextensionofaset} 
implies $\lvert \breve{J} \rvert = 
1 + \lvert \upp_1(J)\rvert + \lvert \upp_2(J)\rvert + \lvert J \rvert$
and because $J\supseteq I$ implies $\lvert J \rvert \geq \lvert I \rvert$, 
$\lvert \upp_1(J) \rvert\geq\lvert\upp_1(I)\rvert$ 
and $\lvert\upp_2(J)\rvert\geq \lvert\upp_1(I)\rvert$. 

If $\Col(\upX_B,\sigma_B) = \emptyset$, the statement of the lemma is vacuously 
true (there not being any point of the target for which the definition of a 
covering would have to hold). We therefore can assume 
that $\Col(\upX_B,\sigma_B)\neq \emptyset$. We now first 
show that this implies that $(\tfrac12\upC_{(s,t)}^{\breve{J}})^{-1}(\mathcal{E}_B^J) 
\neq \emptyset$. Then we will construct a cover of the stated kind. 

To prove $(\tfrac12\upC_{(s,t)}^{\breve{J}})^{-1}(\mathcal{E}_B^J) \neq \emptyset$, 
choose an arbitrary  $\upc\in \Col(\upX_B,\sigma_B)$ and define 
$A = (a_{i,j})\in\{\pm\}^{\breve{J}}$ by $a_{s,t} := +$, 
$a_{i,t} := \upc((i,t))$ for every $i\in\upp_1(J)$, 
$a_{s,j} := \upc((s,j))$ for every $j\in\upp_2(J)$. Moreover, 
for every $(i,j)\in J$ let $a_{i,j} := b_{i,j}$ if $b_{i,j}\neq 0$ and 
$a_{i,j}:=\upc((i,t))\cdot\upc((s,j))$ if $b_{i,j} = 0$. We now show that 
$\tfrac12\upC_{(s,t)}^{\breve{J}}(A)\mid_{\Dom(B)} = B$. Let $(i,j)\in I = \Dom(B)$ be 
arbitrary. If $b_{i,j}=0$, then $\tfrac12\upC_{(s,t)}^{\breve{J}}(A)[i,j] = 
\tfrac12( a_{i,j} a_{s,t} - a_{i,t} a_{s,j} ) 
= \tfrac12\bigl ( a_{i,j} - \upc((i,t)) \upc((s,j)) \bigr ) = 
\tfrac12\bigl(\upc((i,t)) \upc((s,j)) - \upc((i,t)) \upc((s,j)) \bigr) 
= 0 = b_{i,j}$. If $b_{i,j}\neq 0$, then 
$\tfrac12\upC_{(s,t)}^{\breve{J}}(A)[i,j] = \tfrac12( a_{i,j} a_{s,t} - a_{i,t} a_{s,j} ) 
= \tfrac12\bigl (b_{i,j} - \upc((i,t)) \upc((s,j)) \bigr )$. 
Now if $b_{i,j} = -$, then $\upc\in\Col(\upX_B,\sigma_B)$ implies 
$\upc((i,t)) = \upc((s,j))$, hence 
$\tfrac12\upC_{(s,t)}^{\breve{J}}(A)[i,j] = \tfrac12((-) - (+)) = - = b_{i,j}$, 
and if $b_{i,j} = +$, then $\upc\in\Col(\upX_B,\sigma_B)$ implies 
$\upc((i,t)) \neq \upc((s,j))$, hence 
$\tfrac12\upC_{(s,t)}^{\breve{J}}(A)[i,j] = \tfrac12((+) - (-)) = + = b_{i,j}$. 

We have proved that $A\in(\tfrac12\upC_{(s,t)}^{\breve{J}})^{-1}(\mathcal{E}_B^J)$, and 
therefore $(\tfrac12\upC_{(s,t)}^{\breve{J}})^{-1}(\mathcal{E}_B^J) \neq \emptyset$. 
We  can therefore define a nonempty map 
$\Phi\colon (\tfrac12\upC_{(s,t)}^{\breve{J}})^{-1}(\mathcal{E}_B^J) 
\rightarrow \Col(\upX_B,\sigma_B)$ as follows: for every 
$A$ $=$ $(a_{i,j})_{(i,j)\in\breve{J}}$ $\in$ 
$(\tfrac12\upC_{(s,t)}^{\breve{J}})^{-1}(\mathcal{E}_B^J) 
\subseteq \{\pm\}^{\breve{J}}$ we let $\Phi(A)$ be the function 
$\upV(X_B) \rightarrow \{\pm\}$ defined by $\Phi(A)((i,t)) := a_{i,t}$ 
and $\Phi(A)((s,j)) := a_{s,j}$ for all $i\in\upp_1(I)$ and $j\in \upp_2(I)$.

Claim 1. $\Phi$ is indeed a map of the stated kind, i.e. 
$\Phi(A)\in \Col(\upX_B,\sigma_B)$. Proof: Let $\{ (i,t),(s,j) \} \in \upE(X_B)$ be 
arbitrary. There are two cases. If $\sigma_B(\{ (i,t),(s,j) \}) = -$ then 
$b_{i,j} = -$ by Definition \ref{def:XBandecXB}. Moreover, since 
$(\tfrac12\upC_{(s,t)}^{\breve{J}}(A))\rvert_{\Dom(B)}= B$ by the choice of $A$, it 
follows that for every $(i,j)\in I\subseteq J$ we have 
$- = b_{i,j} = (\tfrac12\upC_{(s,t)}^{\breve{J}}(A))[i,j] 
= \tfrac12\cdot (a_{i,j}a_{s,t} - a_{i,t}a_{s,j})$. In view of 
$a_{i,j},\ a_{s,t},\ a_{i,t},\ a_{s,j}\in \{\pm\}$, this equation 
implies $\Phi(A)((i,t)) = a_{i,t} = a_{s,j} = \Phi(A)((s,j))$. This 
proves that $\Phi(A)$ is ($\sigma_B$, $-$)-constant. If  
$\sigma_B(\{(i,t),(s,j)\}) = +$ then $b_{i,j} = +$ by Definition 
\ref{def:XBandecXB}. Again by the choice of $A$, it is true 
that $+ = b_{i,j} = (\tfrac12\upC_{(s,t)}^{\breve{J}}(A))[i,j] = 
\tfrac12\cdot (a_{i,j}a_{s,t} - a_{i,t}a_{s,j})$. Again in view of 
$a_{i,j},\ a_{s,t},\ a_{i,t},\ a_{s,j}\in \{\pm\}$, this equation 
implies $\Phi(A)((i,t)) = a_{i,t} \neq a_{s,j} = \Phi(A)((s,j))$. This proves 
that $\Phi(A)$ is ($\sigma_B$, $+$)-proper and hence Claim 1.

Claim 2. $\Phi$ is surjective and every fibre under $\Phi$ has cardinality 
$2^{h(I,J)}$. Proof: Let an arbitrary $\upc \in \Col(\upX_B,\sigma_B)$ be 
given. We are now looking for those 
$A\in (\tfrac12\upC_{(s,t)}^{\breve{J}})^{-1}(\mathcal{E}_B^J)$ 
with $\Phi(A) = \upc$. Since the definition of $\Phi$ demands  
$\Phi(A)((i,t)) = a_{i,t}$ and $\Phi(A)((s,j)) =  a_{s,j}$ 
for all $(i,t)$ and $(s,j)$ $\in \upV(X_B)$, it follows that with regard to the 
$\lvert \upp_1(I) \rvert + \lvert \upp_2(I) \rvert$ different entries $a_{i,j}$
with $(i,j)\in \upV(\upX_{B})$ we know from the outset that we have no choice 
but to define $a_{i,t} := \upc((i,t))$ and $a_{s,j} := \upc((s,j))$. 

Furthermore, since $A$ must be in 
$(\tfrac12\upC_{(s,t)}^{\breve{J}})^{-1}(\mathcal{E}_B^J)$, there is, 
for every $(i,j)\in I\subseteq J$, the condition that 
$b_{i,j} = (\tfrac12\upC_{(s,t)}^{\breve{J}}(A))[i,j] = 
\tfrac12\cdot (\upC_{(s,t)}^{\breve{J}}(A)[i,j]) = 
\tfrac12\cdot( a_{i,j}a_{s,t} - a_{i,t}a_{s,j} ) = 
\tfrac12( a_{i,j}a_{s,t} - \upc((i,t))\ \upc((s,j)) )$, 
where in the last step we have used the information about $A$ that we 
already have. Now there are three cases that can occur. 

\textsl{Case 1.} $b_{i,j} = -$. Then by Definition \ref{def:XBandecXB} we have 
$\{(i,t),(s,j)\} \in \upE(X_B)$ and $\sigma_B(\{(i,t),(s,j)\}) = -$. Therefore, 
due to the fact that $\upc \in \Col(\upX_B,\sigma_B)$ is 
($\sigma_B$, $-$)-constant, $\upc((i,t)) = \upc((s,j))$. 
Thus, in this case, $- = \tfrac12( a_{i,j}a_{s,t} - 1 )$, 
equivalently, $a_{i,j} = - a_{s,t}$.

\textsl{Case 2.} $b_{i,j} = 0$. Then by Definition \ref{def:XBandecXB}, 
$\{(i,t),(s,j)\} \notin \upE(X_B)$, hence $\sigma_B(\{(i,t),(s,j)\})$ is 
not defined  and therefore the product $\upc((i,t))\cdot \upc((s,j))$ in the 
equation $0 = \tfrac12\cdot( a_{i,j}a_{s,t} - \upc((i,t))\cdot \upc((s,j)) )$
cannot be simplified further, but the equation itself can: it is equivalent 
to $a_{i,j} = \upc((i,t))\cdot \upc((s,j))\cdot a_{s,t} \in \{\pm\}$ 
(where we used that $a_{s,t}^{-1} = a_{s,t}$). 

\textsl{Case 3.} $b_{i,j} = +$. Then an entirely analogous argument as in 
Case 1, but this time using the ($\sigma_B$, $+$)-properness of $\upc$, shows 
that then there is the equation $a_{i,j} = a_{s,t}$.

We now know what it means to require 
$A\in (\tfrac12\upC_{(s,t)}^{\breve{J}})^{-1}(\mathcal{E}_B^J)$ in the present situation: 
among the $\lvert J \rvert$ entries of $A = (a_{i,j}) \in \{\pm\}^J$, there are the 
$\lvert \upp_1(I) \rvert + \lvert \upp_2(I) \rvert$ `immediately determined' entries 
$a_{i,j}$ which have ($i\in \upp_1(I)$ and $j=t$) or ($i=s$ and $j\in \upp_2(I)$), 
and moreover the $\lvert I \rvert$ different entries $a_{i,j}$ with $(i,j)\in I$ 
which are determined by a system $ \{ a_{i,j} = h_{i,j} \colon (i,j)\in I\}$ 
of $\lvert I\rvert$ equations where the right-hand sides $h_{i,j}$ are defined by 
the Cases 1-3 above. For the remaining $h(I,J)$ different 
entries $a_{i,j}\in\{\pm\}$ (note that the pivot $a_{s,t}$ is among them since it 
is on the right-hand side in Case 2, hence not determined by the system), the choice 
of their value is free; any of the $2^{h(I,J)}$ possible choices gives an 
$A\in (\tfrac12\upC_{(s,t)}^{\breve{J}})^{-1}(\mathcal{E}_B^J)$. This proves that the 
cardinality of the fibre $\Phi^{-1}(\upc)$ is indeed $2^{h(I,J)}$, and in particular 
that $\Phi$ is surjective. Now Claim 2 and Lemma \ref{lem:theparametrizedcover} are 
proved.  
\end{proof}
Note that in the special case $I=J$, i.e. when all entries are specified, then
$h(I,J)=1$  and the statement says that there is a 
double cover 
$\Phi\colon (\tfrac12\upC_{(s,t)}^{\breve{J}})^{-1}(\{B\}) \rightarrow 
\Col(\upX_B,\sigma_B)$. This corresponds to the freedom of choosing the sign 
of the pivot. Now we can relate Chio condensation to balancedness: 

\begin{lemma}\label{lem:graphtheoreticalcharacterizationofchiorealizability}
For every $(s,t)\in \Z_{\geq 2}^2$, every 
$\emptyset \subseteq I\subseteq J\subseteq [s-1]\times [t-1]$  
and every $B\in \{0,\pm\}^I$, the following statements are equivalent:

\begin{minipage}[b]{0.35\linewidth}
\begin{enumerate}[label={\rm(\arabic{*})}]
\item\label{lem:prop:associatededge2coloringofBiscyclicallyreven} 
$(\upX_B,\sigma_B)$ is balanced\quad , 
\item\label{lem:prop:thereexistsbconstantrpropervertex2coloring} 
$\Col(\upX_B,\sigma_B)\neq \emptyset$ \quad ,
\end{enumerate}
\end{minipage}
\begin{minipage}[b]{0.65\linewidth}
\begin{enumerate}[label={\rm(\arabic{*})},start=3]
\item\label{lem:prop:BisChiorealizable} 
$(\tfrac12\upC_{(s,t)}^{\breve{J}})^{-1}(\mathcal{E}_B^J) \neq \emptyset$ \quad ,
\item\label{lem:prop:cardinalityofinverseimage} 
$\bigl \lvert (\tfrac12\upC_{(s,t)}^{\breve{J}})^{-1}(\mathcal{E}_B^J) 
\bigr \rvert = 2^{ \lvert \breve{J} \rvert - \dom(B) - f_0(\upX_B) + \beta_0(\upX_B) }$ \quad .
\end{enumerate}
\end{minipage}

\end{lemma}
\begin{proof}
Equivalence \ref{lem:prop:associatededge2coloringofBiscyclicallyreven} 
$\Leftrightarrow$ \ref{lem:prop:thereexistsbconstantrpropervertex2coloring} 
is true by \ref{item:equivalenceofexistenceofbconstantrpropervertex2coloringandcyclicallyrevenness} with $X:=\upX_B$. Equivalence 
\ref{lem:prop:thereexistsbconstantrpropervertex2coloring} 
$\Leftrightarrow$ \ref{lem:prop:BisChiorealizable} is an immediate consequence of 
Lemma \ref{lem:theparametrizedcover} (non-emptiness of the target of a surjective 
map implies non-emptiness of its source; non-emptiness of the source 
of \emph{any} map implies non-emptiness of its target). 
As to \ref{lem:prop:BisChiorealizable} 
$\Leftrightarrow$ \ref{lem:prop:cardinalityofinverseimage}, note that 
by Lemma \ref{lem:theparametrizedcover}, there is the equation 
$\lvert (\tfrac12\upC_{(s,t)}^{\breve{J}})^{-1}(\mathcal{E}_B^J)\rvert
= 2^{\lvert \breve{J} \rvert - \dom(B) - f_0(\upX_B)}\cdot 
\lvert \Col(\upX_B,\sigma_B) \rvert$, which may have the form $0=0$. 
Now if \ref{lem:prop:BisChiorealizable}, then 
$\Col(\upX_B,\sigma_B) \neq \emptyset$ by the already proved equivalence 
\ref{lem:prop:thereexistsbconstantrpropervertex2coloring} 
$\Leftrightarrow$ \ref{lem:prop:BisChiorealizable}, therefore  
Lemma \ref{item:cardinalityofsetofbconstantrproper2colorings} implies 
$\lvert \Col(\upX_B,\sigma_B) \rvert = 2^{\beta_0(\upX_B)}$ and hence 
\ref{lem:prop:cardinalityofinverseimage} is true. Conversely, if 
\ref{lem:prop:cardinalityofinverseimage} is true, then this formula alone implies 
\ref{lem:prop:BisChiorealizable}. This completes the proof of 
\ref{lem:prop:BisChiorealizable} $\Leftrightarrow$ 
\ref{lem:prop:cardinalityofinverseimage} and also the proof of Lemma 
\ref{lem:graphtheoreticalcharacterizationofchiorealizability}. 
\end{proof}

As an example, consider the special case $s:=t:=n$, 
$\{(1,1)\} =: I \subseteq J := [n-1]^2$, $B[(1,1)]:=0$, 
i.e. $\mathcal{E}_B^J$ is the event that a 
$\tilde{B} = (\tilde{b}_{i,j})\in \{0,\pm\}^{[n-1]^2}$ has $\tilde{b}_{1,1} = 0$. 
For these data \ref{lem:prop:cardinalityofinverseimage} 
in Lemma \ref{lem:graphtheoreticalcharacterizationofchiorealizability} 
yields $2^{n^2-1}$. And indeed, it is easy to convince oneself directly that 
there are $2^{n^2-1}$ possibilities to realize this event by 
Chio condensates of sign matrices $A\in \{\pm\}^{[n]^2}$.

\section{Understanding the Chio measure}

\begin{theorem}[graph-theoretical characterization of the Chio measure 
of entry-specification events]
\label{thm:graphtheoreticalcharacterizationofthechiomeasure}
For every $(s,t)\in\Z_{\geq 2}^2$, arbitrary 
$\emptyset \subseteq I \subseteq J \subseteq [s-1]\times [t-1]$ 
and every $B\in \{0,\pm\}^I$:
\begin{enumerate}[label={\rm(C\arabic{*})}]
\item\label{characterizationofwhenchiomeasureispositive} 
\textbf{positivity is determined by balancedness:} \\ 
$\Prob_{\chio}[\mathcal{E}_B^J] > 0$ if and only if $(\upX_B,\sigma_B)$ is 
balanced\quad ,
\item\label{explicitformulaforchiomeasureinthecaseofrealizability} 
\textbf{absolute value is determined by the coboundary space:}  \\
$\Prob_{\chio}[\mathcal{E}_B^J]>0$ if and only if  
\begin{equation}
\Prob_{\chio}[\mathcal{E}_B^J]  = (\tfrac12)^{\dom(B) + f_0(\upX_B) - \beta_0(\upX_B)} 
= \frac{(\tfrac12)^{\dom(B)}}{\lvert \upB^1(\upX_B;\; \Z/2)\rvert} \quad ,
\end{equation}
\item\label{relationbeweenchiomeasureandlazycoinflipmeasuregovernedbyfirstbettinumber}
\textbf{relative value is determined by the cycle space:} \\
$\Prob_{\chio}[\mathcal{E}_B^J]>0$ if and only if 
\begin{equation}
  \Prob_{\chio}[\mathcal{E}_B^J] = 2^{\beta_1(\upX_B)}\cdot \Prob_{\lcf}[\mathcal{E}_B^J] 
= \lvert \upZ_1(\upX_B;\; \Z/2) \rvert \cdot \Prob_{\lcf}[\mathcal{E}_B^J] \quad .
\end{equation}
\end{enumerate}
\end{theorem}
\begin{proof}
As to \ref{characterizationofwhenchiomeasureispositive}, Definition 
\ref{def:measurePchio} implies that $\Prob_{\chio}[\mathcal{E}_B^J] > 0$ 
if and only if $(\tfrac12\upC_{(s,t)}^{\breve{J}})^{-1}(\mathcal{E}_B^J) \neq \emptyset$, 
hence item \ref{characterizationofwhenchiomeasureispositive} follows from 
the equivalence \ref{lem:prop:associatededge2coloringofBiscyclicallyreven} 
$\Leftrightarrow$ \ref{lem:prop:BisChiorealizable}  in 
Lemma \ref{lem:graphtheoreticalcharacterizationofchiorealizability}. 

As to \ref{explicitformulaforchiomeasureinthecaseofrealizability}, by the just 
proved item \ref{characterizationofwhenchiomeasureispositive}  we have
$\Prob_{\chio}[\mathcal{E}_B^J]>0$ if and only if $(\upX_B,\sigma_B)$ is balanced, 
and by equivalence \ref{lem:prop:associatededge2coloringofBiscyclicallyreven} 
$\Leftrightarrow$ \ref{lem:prop:cardinalityofinverseimage} 
in Lemma \ref{lem:graphtheoreticalcharacterizationofchiorealizability} 
this is equivalent to 
$\bigl \lvert (\tfrac12\upC_{(s,t)})^{-1}(\mathcal{E}_B^J) 
\bigr \rvert = 2^{\lvert \breve{J} \rvert - \dom(B) - f_0( \upX_B  ) + \beta_0(\upX_B)}$. Dividing 
by $2^{\lvert \breve{J} \rvert}$ in accordance with Definition \ref{def:measurePchio} 
gives the first equality claimed 
in \ref{explicitformulaforchiomeasureinthecaseofrealizability}. As to the 
second equality, this is a reformulation not necessary for the equivalence and 
is true by the known formula (e.g. \cite[Theorem 14.1.1]{MR1829620}) for the 
dimension of the coboundary space of a graph, together with the obvious formula for 
the number of elements of a finite-dimensional vector space over a finite field. 

As to \ref{relationbeweenchiomeasureandlazycoinflipmeasuregovernedbyfirstbettinumber}, this follows by a simple calculation from \ref{explicitformulaforchiomeasureinthecaseofrealizability}, Lemma \ref{lem:valueoflazycoinflipdistribution} 
and $\supp(X_B) = f_1(X_B)$. The second equality in \ref{relationbeweenchiomeasureandlazycoinflipmeasuregovernedbyfirstbettinumber} is true by definition 
of $\beta_1(\cdot)$ (and therefore again a reformulation not necessary for the 
equivalence). The proof of Theorem \ref{thm:graphtheoreticalcharacterizationofthechiomeasure} is now complete. 
\end{proof}

We will now derive several consequences of Theorem \ref{thm:graphtheoreticalcharacterizationofthechiomeasure}.  Let us start with:
 
\begin{corollary}\label{cor:quickconsequencesofthecharacterizations}
Let $(s,t)\in\Z_{\geq 2}^2$, $B\in \{0,\pm\}^{[s-1]\times [t-1]}$, 
$\emptyset \subseteq I_1 \subseteq J_1 \subseteq [s-1]\times [t-1]$,  
$\emptyset \subseteq I_2 \subseteq J_2 \subseteq [s-1]\times [t-1]$ with 
$\lvert I_1 \rvert = \lvert I_2 \rvert$, $B_1\in \{0,\pm\}^{I_1}$ and 
$B_2\in\{0,\pm\}^{I_2}$ be arbitrary. Then 
\begin{enumerate}[label={\rm(\arabic{*})}]
\item\label{cor:graphfailuresetaspreimage}
$\mathcal{F}^{\mathrm{M}}(k,n) = (\beta_1\circ \upX^{k,n,n})^{-1}(\Z_{\geq 1}) = 
(\beta_1^{\ul}\circ \upXul^{k,n,n})^{-1}(\Z_{\geq 1})$ \quad ,
\item\label{cor:chiomeasureofasinglematrix} 
$B\in \im (\tfrac12 \upC_{(s,t)}\colon 
\{\pm\}^{[s]\times [t]} \rightarrow \{0,\pm\}^{[s-1]\times [t-1]})$ if and only if 
$\Prob_{\chio}[B] = \frac{2\cdot 2^{\beta_0(\upX_B)}}{2^{s\cdot t}}$ \quad , 
\item\label{cor:equalityofChiomeasuresofeventswithcoalescedassociatedgraphs} 
$\Prob_{\chio}[\mathcal{E}_{B_1}^{J_1}] = \Prob_{\chio}[\mathcal{E}_{B_2}^{J_2}]$ if 
$\upX_{B_1}$ is a one-point wedge product of two components of $\upX_{B_2}$ \quad . 
\end{enumerate}
\end{corollary}
\begin{proof}
As to \ref{cor:graphfailuresetaspreimage}, this is immediate from 
\ref{relationbeweenchiomeasureandlazycoinflipmeasuregovernedbyfirstbettinumber} in 
Theorem \ref{thm:graphtheoreticalcharacterizationofthechiomeasure}. 
As to \ref{cor:chiomeasureofasinglematrix}, this follows by setting 
$I:=J:=[s-1]\times [t-1]$ and combining the equivalence 
\ref{lem:prop:associatededge2coloringofBiscyclicallyreven} $\Leftrightarrow$ 
\ref{lem:prop:BisChiorealizable} in Lemma 
\ref{lem:graphtheoreticalcharacterizationofchiorealizability} with 
\ref{characterizationofwhenchiomeasureispositive} $\Leftrightarrow$ 
\ref{explicitformulaforchiomeasureinthecaseofrealizability} in Theorem 
\ref{thm:graphtheoreticalcharacterizationofthechiomeasure}.

As to \ref{cor:equalityofChiomeasuresofeventswithcoalescedassociatedgraphs}, let us 
first note that Lemma \ref{lem:onepointwedgepreservesbalancedness} implies that 
either $(\upX_{B_1},\sigma_{B_1})$ and $(\upX_{B_2},\sigma_{B_2})$ are both not balanced, 
or both are. If both are not balanced, then by item 
\ref{characterizationofwhenchiomeasureispositive} in Theorem 
\ref{thm:graphtheoreticalcharacterizationofthechiomeasure}, the claim is 
true in the form of $0=0$. If both are, then by item 
\ref{explicitformulaforchiomeasureinthecaseofrealizability} in Theorem 
\ref{thm:graphtheoreticalcharacterizationofthechiomeasure}, and using 
$\lvert I_1 \rvert  = \lvert I_2 \rvert$, the equation  
$\Prob_{\chio}[\mathcal{E}_{B_1}^{J_1}] = \Prob_{\chio}[\mathcal{E}_{B_2}^{J_2}]$ is 
equivalent to $f_0(\upX_{B_1})-\beta_0(\upX_{B_1}) = 
f_0(\upX_{B_2}) - \beta_0(\upX_{B_2})$. Since the one-point wedge product of two 
graphs keeps $f_0(\cdot) - \beta_0(\cdot)$ invariant, the equation is true also in 
this case and the proof is complete. 
\end{proof}

\begin{corollary}[the lazy coin flip measure is an averaged Chio measure]\label{cor:lazycoinflipmeasureisanaveragedchiomeasure}
$\Prob_{\lcf}\bigl[B\bigr] = \overline{\Prob}_{\chio}\bigl[B\bigr]$ for every 
$\emptyset\subseteq I \subseteq [n-1]^2$ and every $B\in\{0,\pm\}^I$.
\end{corollary}
\begin{proof}
It follows from Definition \ref{def:XBandecXB} that $\supp(B) = f_1(\upX_B)$ 
and that $\{ \tilde{B} \in \{0,\pm\}^I\colon \Supp(\tilde{B}) = \Supp(B) \} = 
\{ \tilde{B} \in \{0,\pm\}^I\colon \upX_{\tilde{B}} = \upX_B \}$. Moreover, 
by \ref{characterizationofwhenchiomeasureispositive} in 
Theorem \ref{thm:graphtheoreticalcharacterizationofthechiomeasure}, every 
summand with the property that $(\upX_{\tilde{B}},\sigma_{\tilde{B}})$ is not balanced 
vanishes. Thus, for every $B\in\{0,\pm\}^I$, 
\begin{align}
2^{f_1(\upX_{B})}\cdot \overline{\Prob}_{\chio}\bigl[B \bigr] & = 
 \sum_{\substack{\mathrm{all}\; \tilde{B}\in \{0,\pm\}^I\; 
\mathrm{with} \\ \upX_{\tilde{B}} = \upX_{B}\; \mathrm{and}\; 
(\upX_{\tilde{B}},\sigma_{\tilde{B}})\; \mathrm{balanced} } } \Prob_{\chio}[\tilde{B}] 
\notag \\
& \By{\ref{relationbeweenchiomeasureandlazycoinflipmeasuregovernedbyfirstbettinumber}}{=} 2^{\beta_1(\upX_{\tilde{B}})}\cdot \Prob_{\lcf}[B] \cdot 
\bigl\lvert \bigl\{ \tilde{B}\in \{0,\pm\}^I\colon 
\text{$\upX_{\tilde{B}} = \upX_{B}$ and $(\upX_{\tilde{B}},\sigma_{\tilde{B}})$ balanced} 
\bigr \} \bigr \rvert\notag \\
& \By{\ref{item:numberofbalancedsignfunctions}}{=} 
2^{\beta_1(\upX_{\tilde{B}})}\cdot \Prob_{\lcf}[B] \cdot 
2^{f_0(\upX_{B}) - \beta_0(\upX_{B})} = 2^{f_1(\upX_{B})}\cdot \Prob_{\lcf}[B]\quad . \qedhere
\end{align}
\end{proof}

\begin{corollary}[$\Prob_{\chio}^{\lvert\cdot\rvert, I}$ is just the uniform 
distribution on $\{0,1\}^I$]\label{cor:signforgettingchiomeasure}
For every $(s,t)\in\Z_{\geq 2}^2$ and every 
$\emptyset\subseteq I \subseteq [s-1]\times [t-1]$ let $\Prob_{0,1}^I$ denote 
the uniform distribution on $\{0,1\}^I$. 
Then $\Prob_{\chio}^{\lvert\cdot\rvert,I} = \Prob_{0,1}^I$.  
\end{corollary}
\begin{proof}
This is true since for 
$(s,t)\in \Z_{\geq 2}^2$, $\emptyset\subseteq I \subseteq [s-1]\times [t-1]$ 
and $B\in\{0,1\}^I$ we have 
\begin{align}
 2^{\lvert\breve{I}\rvert} \cdot \Prob_{\chio}^{\lvert\cdot\rvert, I} [B] & = 
\bigl\lvert \{ A\in\{\pm\}^{\breve{I}}\colon 
\lvert \tfrac12\cdot\upC_{(s,t)}(A)\rvert = B \} \bigr\rvert \notag \\
& = \bigl\lvert \{ A\in\{\pm\}^{\breve{I}}\colon 
 \Supp(\tfrac12 \upC_{(s,t)}(A)) = \Supp(B) \} \bigr\rvert \notag \\
\parbox{0.2\linewidth}{\tiny 
(using \ref{lem:prop:associatededge2coloringofBiscyclicallyreven} 
$\Leftrightarrow$ \ref{lem:prop:BisChiorealizable} in \\
Lemma \ref{lem:graphtheoreticalcharacterizationofchiorealizability})} 
& = \sum_{\tilde{B}\in\{0,\pm\}^{I}\colon\Supp(\tilde{B}) = \Supp(B),\ 
\text{$(\upX_{\tilde{B}},\sigma_{\tilde{B}})$ balanced}} 
\lvert (\tfrac12\upC_{(s,t)}^{\breve{I}})^{-1}
(\mathcal{E}_{\tilde{B}}^{I}) \rvert \notag \\
\parbox{0.2\linewidth}{\tiny (by \ref{lem:prop:cardinalityofinverseimage} in 
Lemma \ref{lem:graphtheoreticalcharacterizationofchiorealizability})}
& = \lvert \{ \tilde{B}\in\{0,\pm\}^{I}\colon\text{\scriptsize
$\Supp(\tilde{B}) = \Supp(B)$,\ $\tilde{B}$ balanced} \} \rvert \cdot 
2^{\lvert\breve{I}\rvert - \dom(B) - f_0(\upX_{\tilde{B}}) + \beta_0(\upX_{\tilde{B}})} \notag \\
\parbox{0.2\linewidth}{\tiny (by \ref{item:numberofbalancedsignfunctions} 
in Lemma \ref{lem:equivalenceofexistenceofbconstantrpropervertex2coloringandcyclicallyrevenness} )} 
& = 2^{f_0(\upX_{\tilde{B}}) - \beta_0(\upX_{\tilde{B}})} \cdot 
2^{\lvert\breve{I}\rvert - \dom(B)    - f_0(\upX_{\tilde{B}}) + \beta_0(\upX_{\tilde{B}})}  
= 2^{\lvert\breve{I}\rvert - \dom(B)} \notag \qedhere
\end{align}
\end{proof}
Let us state the special case $s:=t:=n$ and $I:=[n-1]^2$ in 
graph-theoretical language:

\begin{corollary}\label{cor:whatrandombipartitegraphisX12CnnAforrandomA}
For random $A\in\{\pm\}^{[n]^2}$, the graph $\upX_{\frac12\upC_{(n,n)}(A)}$ is a 
random bipartite graph with $n-1$ vertices in each class and each edge 
chosen i.i.d. with probability $\tfrac12$. \hfill $\Box$
\end{corollary}

Theorem \ref{thm:graphtheoreticalcharacterizationofthechiomeasure} also 
teaches us how fast $\Prob_{\chio}$ can be computed. In both 
Corollary \ref{cor:upperboundontimecomplexityofcomputingPchio} and 
\ref{cor:complexityofdecidingwhetherchiomeasureequalslazycoinflipmeasure} the 
asymptotic statements are referring to $n\rightarrow\infty$ and to sequences 
$I = I(n)$ of index sets with the property that 
$\lvert I(n) \rvert \rightarrow \infty$ (and therefore also 
$\lvert \upp_1(I(n)) \rvert \cdot \lvert \upp_2(I(n)) \rvert \rightarrow\infty$) 
as $n\rightarrow\infty$.

\begin{corollary}[complexity of computing $\Prob_{\chio}$]
\label{cor:upperboundontimecomplexityofcomputingPchio}
For every $\emptyset \subseteq I \subseteq J \subseteq [n-1]^2$ and 
every $B\in \{0,\pm\}^I$, the value of $\Prob_{\chio}[\mathcal{E}_B^J] \in \Q$ can  be 
computed exactly in time 
$O(\lvert \upp_1(I) \rvert + \lvert \upp_2(I) \rvert + \lvert I\rvert) 
\subseteq O(\lvert \upp_1(I) \rvert \cdot \lvert \upp_2(I) \rvert) 
\subseteq O(n^2)$. However, there does not exist a fixed algorithm 
computing $\Prob_{\chio}[\mathcal{E}_B^J] \in \Q$ exactly on arbitrary instances 
$B\in \{0,\pm\}^{[n-1]^2}$ and $\emptyset \subseteq I \subseteq [n-1]^2$ and taking 
time $o(\lvert \upp_1(I) \rvert \cdot \lvert \upp_2(I) \rvert)$.
\end{corollary}
\begin{proof}
By items \ref{characterizationofwhenchiomeasureispositive} 
and \ref{explicitformulaforchiomeasureinthecaseofrealizability} in 
Theorem \ref{thm:graphtheoreticalcharacterizationofthechiomeasure}, to compute 
$\Prob_{\chio}[\mathcal{E}_{B}^J]$ it suffices to first decide whether $\sigma_{B}$ 
(which in view of Definition \ref{def:XBandecXB} evidently can be read in time 
$O(\lvert \upp_1(I) \rvert \cdot \lvert \upp_2(I) \rvert) \subseteq O(n^2)$) is 
balanced, and, if so, to compute $f_0(\upX_{B})$ and $\beta_0(\upX_{B})$.  
By Corollary \ref{cor:decidingwhetherecXiscyclicallyreven}, and since the 
depth-first seach mentioned there also computes the numbers $f_0( \upX_{B} )$ 
and $\beta_0(\upX_{B})$, both tasks can be accomplished by one depth-first 
search in time $O(f_0( \upX_{B} ) + f_1( \upX_{B} )) \subseteq
O(\lvert \upp_1(I) \rvert + \lvert \upp_2(I) \rvert + \lvert I\rvert) \subseteq 
O(n^2)$. If $(\upX_B,\sigma_B)$ is found to be not balanced, then 
$\Prob_{\chio}[\mathcal{E}_B^J]=0$. Otherwise, the answer is 
$(\frac12)^{\lvert I\rvert + f_0( \upX_{B} ) - \beta_0(\upX_{B})}$. 
Since the bitlength  of this dyadic fraction is 
$\lvert I\rvert + f_0( \upX_{B} ) - \beta_0(\upX_{B}) = 
\lvert I\rvert + f_1( \upX_{B} ) - \beta_1(\upX_{B}) \leq 
\lvert I\rvert + f_1( \upX_{B} ) \leq 2\lvert I \rvert \in 
O(\lvert \upp_1(I)\rvert + \lvert\upp_2(I) \rvert + \lvert I \rvert)$ it is 
possible to write the output in the time claimed. This proves the first 
statement in Corollary \ref{cor:upperboundontimecomplexityofcomputingPchio}.

As to the second statement, notice that any such fixed algorithm could 
in particular compute $\Prob_{\chio}[\mathcal{E}_{B}^J] \in \Q$ exactly on those 
instances $B\in \{0,\pm\}^{[n-1]^2}$ and $\emptyset \subseteq I \subseteq [n-1]^2$ 
for which $I$ is rectangular. But if $I$ is rectangular, then 
$\lvert I\rvert + f_0( \upX_{B} ) - \beta_0(\upX_{B}) \geq \lvert I\rvert = 
\lvert \upp_1(I) \rvert \cdot \lvert \upp_2(I) \rvert$. Therefore, for these inputs, 
the bitlength of the dyadic fraction 
$(\frac12)^{\lvert I\rvert + f_0( \upX_{B} ) - \beta_0(\upX_{B})}$ is at least 
$\lvert \upp_1(I) \rvert \cdot \lvert \upp_2(I) \rvert$. Hence for such inputs the 
very task of writing the output takes time 
$\Omega(\lvert \upp_1(I) \rvert \cdot \lvert \upp_2(I) \rvert)$, 
which precludes a running time of 
$o(\lvert \upp_1(I) \rvert \cdot \lvert \upp_2(I) \rvert) \subseteq o(n^2)$.
The proof of Corollary \ref{cor:upperboundontimecomplexityofcomputingPchio} is 
now complete.
\end{proof}

A priori one might suspect that the task of merely \emph{deciding} whether 
$\Prob_{\chio}[B] = \Prob_{\lcf}[B]$ could be accomplished much faster than the 
task of computing the value of $\Prob_{\chio}[B]$. Theorem \ref{thm:graphtheoreticalcharacterizationofthechiomeasure} can also tell us that this is not the case.

\begin{corollary}[{complexity of deciding whether $\Prob_{\chio}$ and 
$\Prob_{\lcf}$ agree}]
\label{cor:complexityofdecidingwhetherchiomeasureequalslazycoinflipmeasure} 
For every $\emptyset \subseteq I \subseteq J \subseteq [n-1]^2$ 
and every $B\in \{0,\pm\}^I$, the answer to the decision problem of whether 
$\Prob_{\chio}[\mathcal{E}_{B}^J] = \Prob_{\lcf}[\mathcal{E}_{B}^J]$ can be computed in 
time $O(\lvert \upp_1(I) \rvert \cdot \lvert \upp_2(I) \rvert) \subseteq O(n^2)$. 
However, there does not exist a fixed algorithm (having only entry-wise 
access to $B$ and no further a priori information) which decides that question on 
arbitrary instances $B\in \{0,\pm\}^I$ with 
$\emptyset \subseteq I \subseteq [n-1]^2$ in time 
$o(\lvert \upp_1(I)\rvert \cdot \lvert \upp_2(I)\rvert)$. 
\end{corollary}
\begin{proof}
Given $\emptyset \subseteq I \subseteq [n-1]^2$ and $B \in\{0,\pm\}^I$, it follows 
from item 
\ref{relationbeweenchiomeasureandlazycoinflipmeasuregovernedbyfirstbettinumber} in 
Theorem \ref{thm:graphtheoreticalcharacterizationofthechiomeasure} that the question 
of whether $\Prob_{\chio}[\mathcal{E}_{B}^J] = \Prob_{\lcf}[\mathcal{E}_{B}^J]$ is 
equivalent to asking whether $\upX_{B}$ is a forest. The graph $\upX_{B}$ can 
obviously be computed from $B$ in time 
$O(\lvert \upp_1(I) \rvert \cdot \lvert \upp_2(I) \rvert) \subseteq O(n^2)$, and 
deciding whether $\upX_{B}$ is a forest, i.e. whether $\upX_{B}$ contains a circuit, 
can be done by a depth-first search in time $O(f_0(\upX_{B}) + f_1( \upX_{B} )) 
\subseteq O(\lvert \upp_1(I) \rvert + \lvert \upp_2(I) \rvert + 
\lvert I\rvert) \subseteq O(\lvert \upp_1(I) \rvert \cdot \lvert \upp_2(I) \rvert )$, so the first claim in Corollary \ref{cor:complexityofdecidingwhetherchiomeasureequalslazycoinflipmeasure} is proved.

As to the additional claim, suppose there were a fixed algorithm $A$ with the stated 
properties. Let $\mathcal{I}$ be the set of all rectangular 
$\emptyset \subseteq I \subseteq [n-1]^2$. By 
assumption, the algorithm $A$ is in particular capable of deciding whether 
$\Prob_{\chio}[\mathcal{E}_{B}^J] = \Prob_{\lcf}[\mathcal{E}_{B}^J]$ for each input 
$I\in \mathcal{I}$ and for each of them taking time 
$o(\lvert \upp_1(I)\rvert\cdot\lvert\upp_2(I)\rvert)$. However, every bipartite 
graph with bipartition sizes of $\lvert \upp_1(I) \rvert$ and 
$\lvert \upp_2(I) \rvert$ can be realized as a $\upX_{B}$ with $I\in\mathcal{I}$. 
By item 
\ref{relationbeweenchiomeasureandlazycoinflipmeasuregovernedbyfirstbettinumber} in 
Theorem \ref{thm:graphtheoreticalcharacterizationofthechiomeasure} the property 
$\Prob_{\chio}[\mathcal{E}_{B}^J] = \Prob_{\lcf}[\mathcal{E}_{B}^J]$ is equivalent to 
$\upX_{B}$ being a forest. Therefore $A$ decides set membership for the set of all 
bipartite graphs which have the fixed (that is, fixed for every fixed value of $n$) 
bipartition classes $ \upp_1(I) $ and $ \upp_2(I) $ and do not contain a circuit. 
This set is a decreasing (i.e. closed w.r.t. deleting edges) graph property 
consisting of bipartite graphs only. Since all graphs in the property have the 
same bipartition classes $\upp_1(I)$ and $\upp_2(I)$ we may appeal to a 
theorem of A. C.-C. Yao \cite[p. 518, Theorem 1]{MR941942} which says that every 
such property is evasive.\footnote{Due to the fact that the bipartition classes 
are the same for all the graphs in the property, it is not necessary to appeal 
to the more general theorem of E. Triesch \cite[p. 266, Theorem 4]{MR1401898} in 
which the assumption of \emph{fixed} bipartition classes is no longer made. An 
earlier version of the present paper stated that one would need this more general 
theorem. I was wrong regarding this particular point. Moreover I confused the 
adjectives `balanced' and `fixed'. The theorem of Yao suffices.} Hence there exists 
at least one $I\in \mathcal{I}$ with the property that $A$ examines every entry of 
$B$. This takes time $\Omega(\lvert I\rvert) = \Omega(\lvert \upp_1(I)\rvert \cdot 
\lvert \upp_2(I)\rvert)$, the equality being true  because of 
$\lvert I\rvert = \lvert \upp_1(I) \rvert \cdot \lvert \upp_2(I) \rvert$. 
This is a contradiction to the assumption about the running time of $A$. 
The proof of Corollary \ref{cor:complexityofdecidingwhetherchiomeasureequalslazycoinflipmeasure} is now complete.
\end{proof}

We now take a more quantitative look at the relationship between $\Prob_{\chio}$ and 
$\Prob_{\lcf}$. We start with an enumeration of bipartite nonforests. The fact that 
we stop the enumeration at the $f$-vector $(f_0,f_1) = (8,6)$, even though there 
are bipartite nonforests with $(f_0,f_1) = (8,7)$,  is explained by the application 
we have in mind; we will only be concerned with bipartite nonforests having up to 
six edges. 

\begin{lemma}[bipartite nonforests ordered by their $f$-vectors]\label{lem:bipartitenonforestsorderedbytheirfvectors}
The isomorphism types of bipartite nonforests, ordered lexicographically by 
their $f$-vectors up to $(f_0,f_1) = (8,6)$, are:

{\scriptsize
\begin{minipage}[b]{0.45\linewidth}
\begin{enumerate}[label={\rm(t\arabic{*})}]
\item\label{item:comparisonproofcaseLhassize4:4circuit} 
$=$ $C^4$
\item\label{item:comparisonproofcaseLhassize6:oneisolatedvertex} 
$=$ disjoint union of $C^4$ \\ and one isolated vertex
\item\label{item:comparisonproofcaseLhassize6:oneadditionaledgeintersectingC4} 
$=$ $C^4$ intersecting one edge
\item\label{item:comparisonproofcaseLhassize6:XBLisomorphictoK23} 
$=$ $K^{2,3}$ 
\item\label{item:comparisonproofcaseLhassize6:twoisolatedvertices} 
$=$ disjoint union of $C^4$ \\ and two isolated vertices
\item\label{item:comparisonproofcaseLhassize6:oneadditionaledgeintersectingC4andoneisolatedvertex} $=$ 
$C^4$ intersecting one edge, \\ and one extra isolated vertex
\item\label{item:comparisonproofcaseLhassize6:C4withoneadditionaldisjointedge}
$=$ disjoint union of $C^4$ \\ and an isolated edge
\item\label{item:comparisonproofcaseLhassize6:C4intersectingtwoedgesinseparatenonadjacentvertices} 
$=$ $C^4$ intersecting two disjoint edges, \\ the intersection set no edge of $C^4$
\item\label{item:comparisonproofcaseLhassize6:C4intersectingtwoedgesinseparateadjacentvertices} 
$=$ $C^4$ intersecting two disjoint edges, \\ the intersection set an edge of $C^4$
\item\label{item:comparisonproofcaseLhassize6:C4intersectingatwopathinanendvertex} 
$=$ $C^4$ intersecting a $2$-path in an endvertex
\end{enumerate}
\end{minipage}
\begin{minipage}[b]{0.48\linewidth}
\begin{enumerate}[label={\rm(t\arabic{*})},start=11]
\item\label{item:comparisonproofcaseLhassize6:C4intersectingatwopathinitsinnervertex} 
$=$ $C^4$ intersecting a $2$-path in its inner vertex
\item\label{item:comparisonproofcaseLhassize6:C6} 
$=$ $C^6$
\item\label{item:comparisonproofcaseLhassize6:threeisolatedvertices} 
$=$ disjoint union of $C^4$ \\ and three isolated vertices
\item\label{item:comparisonproofcaseLhassize6:oneadditionaledgeintersectingC4andtwoisolatedvertices} $=$ $C^4$ intersecting one edge, \\ and two extra isolated vertices 
\item\label{item:comparisonproofcaseLhassize6:oneadditionaldisjointedgeandoneisolatedvertex} $=$ disjoint union of $C^4$ and an edge, \\ and one extra isolated vertex 
\item\label{item:comparisonproofcaseLhassize6:twoadditionaledgesonlyoneofthemdisjoint} $=$ $C^4$ intersecting one edge, \\ and one extra isolated edge
\item\label{item:comparisonproofcaseLhassize6:twoadditionalnondisjointedgesdisjointfromC4} $=$ disjoint union of $C^4$ and a $2$-path
\item\label{item:comparisonproofcaseLhassize6:fourisolatedvertices} 
$=$ disjoint union of $C^4$ \\ and four isolated vertices 
\item\label{item:comparisonproofcaseLhassize6:oneadditionaldisjointedgeandtwoisolatedvertices} $=$ disjoint union of $C^4$ \\ and an edge and two extra isolated vertices 
\item\label{item:comparisonproofcaseLhassize6:twoadditionaldisjointedgesdisjointfromC4} $=$ disjoint union of $C^4$ and two disjoint edges 
\end{enumerate}
\end{minipage}
}
\end{lemma}
\begin{proof}
Easy to check since the graphs are required to be bipartite and have $f_1\leq 6$.
\end{proof}

\begin{corollary}[isomorphism types for which equality of measures of entry 
specification events fails]\label{cor:isomorphismtypesforwhichequalityofmeasuresofentryspecificationeventsfails}
\begin{minipage}[b]{0.4\linewidth}
\begin{enumerate}[label={\rm(Fa\arabic{*})},start=3]
\item\label{item:failuresetofisomorphismtypesk3} 
 $\mathcal{F}^{\mathrm{G}}(k,n) = \emptyset$ for $0\leq k \leq 3$,
\item\label{item:failuresetofisomorphismtypesk4} 
$\mathcal{F}^{\mathrm{G}}(4,n)=\{\ref{item:comparisonproofcaseLhassize4:4circuit}\}$, 
\end{enumerate}
\end{minipage}
\begin{minipage}[b]{0.65\linewidth}
\begin{enumerate}[label={\rm(Fa\arabic{*})},start=5]
\item\label{item:failuresetofisomorphismtypesk5} 
$\mathcal{F}^{\mathrm{G}}(5,n) = \{ 
\ref{item:comparisonproofcaseLhassize6:oneisolatedvertex}, 
\ref{item:comparisonproofcaseLhassize6:oneadditionaledgeintersectingC4}, 
\ref{item:comparisonproofcaseLhassize6:twoisolatedvertices}, 
\ref{item:comparisonproofcaseLhassize6:C4withoneadditionaldisjointedge}
\}$, 
\item\label{item:failuresetofisomorphismtypesk6} 
$\mathcal{F}^{\mathrm{G}}(6,n) = \mathcal{F}^{\mathrm{G}}(5,n)
\sqcup\{ \ref{item:comparisonproofcaseLhassize6:XBLisomorphictoK23}\}
\sqcup \{
\ref{item:comparisonproofcaseLhassize6:oneadditionaledgeintersectingC4andoneisolatedvertex},\dotsc, \ref{item:comparisonproofcaseLhassize6:twoadditionaldisjointedgesdisjointfromC4}\}$. 
\end{enumerate}
\end{minipage}
\end{corollary}
\begin{proof}
By \ref{relationbeweenchiomeasureandlazycoinflipmeasuregovernedbyfirstbettinumber} 
in Theorem \ref{thm:graphtheoreticalcharacterizationofthechiomeasure} we have 
$\Prob_{\chio}[\mathcal{E}_B^{[n-1]^2}] \neq \Prob_{\lcf}[\mathcal{E}_B^{[n-1]^2}]$  if and 
only if $\beta_1(\upX_{B}) > 0$. Moreover, directly from 
Definition \ref{def:XBandecXB} we have the bound 
$f_1(\upX_{B}) \leq \lvert I \rvert$. Therefore, for every $k$, we can get a set of 
candidates  for membership in $\mathcal{F}^{\mathrm{G}}(k,n)$ by collecting all 
isomorphism types in Corollary \ref{cor:isomorphismtypesforwhichequalityofmeasuresofentryspecificationeventsfails} having $f_1\leq k$. We then have to decide for each of 
these types whether it is possible to realize it as a $\upX_{B}$  with 
$B\in\{0,\pm\}^I$ and $I\in\binom{[n-1]^2}{k}$. 

As to \ref{item:failuresetofisomorphismtypesk3}, this is true since there do not 
exist bipartite nonforests with three edges or less.

As to \ref{item:failuresetofisomorphismtypesk4}, i.e. $k=4$, note that the only 
isomorphism types in Corollary \ref{cor:isomorphismtypesforwhichequalityofmeasuresofentryspecificationeventsfails} with $f_1 \leq 4$ are \ref{item:comparisonproofcaseLhassize4:4circuit} and \ref{item:comparisonproofcaseLhassize6:oneisolatedvertex}. Because 
of $\beta_1(\upX_B)\geq 1$ for every $B\in \mathcal{F}^{\mathrm{M}}(4,n)$ the set $I$ 
must be a matrix-$4$-circuit. This implies $f_0(\upX_{B}) = 4$. 
Since $f_0\ref{item:comparisonproofcaseLhassize6:oneisolatedvertex} = 5$, it 
follows that $\ref{item:comparisonproofcaseLhassize6:oneisolatedvertex}\notin\mathcal{F}^{\mathrm{G}}(4,n)$. Since type \ref{item:comparisonproofcaseLhassize4:4circuit}  
obviously can be realized, \ref{item:failuresetofisomorphismtypesk4} is true.

As to \ref{item:failuresetofisomorphismtypesk5}, i.e. $k=5$, note that the only 
isomorphism types with $f_1\leq 5$ in Corollary \ref{cor:isomorphismtypesforwhichequalityofmeasuresofentryspecificationeventsfails} are \ref{item:comparisonproofcaseLhassize4:4circuit}, \ref{item:comparisonproofcaseLhassize6:oneisolatedvertex}, \ref{item:comparisonproofcaseLhassize6:oneadditionaledgeintersectingC4}, \ref{item:comparisonproofcaseLhassize6:twoisolatedvertices}, \ref{item:comparisonproofcaseLhassize6:oneadditionaledgeintersectingC4andoneisolatedvertex} and \ref{item:comparisonproofcaseLhassize6:C4withoneadditionaldisjointedge}. 
Since $C^4$ is a subgraph of each of these types, it is necessary that there be a 
matrix-$4$-circuit $S\subseteq I$. Since the sole non-matrix-circuit entry 
must create at least one addition vertex of $\upX_B$, 
type \ref{item:comparisonproofcaseLhassize4:4circuit} cannot occur. The type 
\ref{item:comparisonproofcaseLhassize6:oneadditionaledgeintersectingC4andoneisolatedvertex} cannot occur either since there is only one position 
$u\in I\setminus S\in \binom{[n-1]^2}{1}$ left for us to choose freely 
and by the choice of $u$ and $B[u]$ we can either create an isolated vertex 
in $\upX_{B}$ or an edge intersecting the $C^4\hookrightarrow \upX_{B}$ but not both. 
The remaining types \ref{item:comparisonproofcaseLhassize6:oneisolatedvertex}, \ref{item:comparisonproofcaseLhassize6:oneadditionaledgeintersectingC4}, \ref{item:comparisonproofcaseLhassize6:twoisolatedvertices} and 
\ref{item:comparisonproofcaseLhassize6:C4withoneadditionaldisjointedge} can be 
indeed be realized, as is proved by the following examples. For all the examples let 
$S:=\{(1,1),(1,2),(2,1),(2,2)\}$, $B\mid_{S}:=\{-\}^S$ and $\{u\}:=I\setminus S$. 
For \ref{item:comparisonproofcaseLhassize6:oneisolatedvertex} take e.g. 
$n:=4$, $u:=(2,3)$ and $B[u] := 0$. 
For \ref{item:comparisonproofcaseLhassize6:oneadditionaledgeintersectingC4} take 
e.g. $n:=4$, $u:=(2,3)$ and $B[u]:=-$.  
For \ref{item:comparisonproofcaseLhassize6:twoisolatedvertices} take 
e.g. $n:=4$, $u:=(3,3)$ and $B[u] := 0$. 
For \ref{item:comparisonproofcaseLhassize6:C4withoneadditionaldisjointedge} take 
e.g. $n:=4$, $u:=(3,3)$ and $B[u]:=-$. 
This proves \ref{item:failuresetofisomorphismtypesk5}. 

As to \ref{item:failuresetofisomorphismtypesk6}, i.e. $k=6$, as far as only the 
necessary condition $f_1(\upX_{B}) \leq \lvert I \rvert = k$ is concerned, all 
types in Lemma \ref{lem:bipartitenonforestsorderedbytheirfvectors} are candidates. 
Type \ref{item:comparisonproofcaseLhassize4:4circuit} cannot be realized 
since $f_0\ref{item:comparisonproofcaseLhassize4:4circuit}=4$ but 
$f_0(\upX_{B}) \geq 5$ for every $I\in \binom{[n-1]^2}{6}$. All others can, as will 
now be proved by giving one example for each of them. In all examples again 
let $S:=\{(1,1),(1,2),(2,1),(2,2)\}$ and $B\mid_{S}:=\{-\}^S$. 
Here, $\{u,v\}:=I\setminus S$.
For \ref{item:comparisonproofcaseLhassize6:oneisolatedvertex} take e.g. $n:=4$, 
$u:=(1,3)$, $v:=(2,3)$  and $B[u]:=B[v]:=0$. 
For \ref{item:comparisonproofcaseLhassize6:oneadditionaledgeintersectingC4} 
take e.g. $n:=4$, $u:=(1,3)$, $v:=(2,3)$, $B[u]:=-$ and $B[v]:=0$. 
For \ref{item:comparisonproofcaseLhassize6:XBLisomorphictoK23}
take e.g. $n:=4$, $u:=(1,3)$, $v:=(2,3)$ and $B[u]:=B[v]:=-$.
For \ref{item:comparisonproofcaseLhassize6:twoisolatedvertices} 
take e.g. $n:=4$, $u:=(1,3)$, $v:=(3,3)$ and $B[u]:=B[v]:=0$. 
For \ref{item:comparisonproofcaseLhassize6:oneadditionaledgeintersectingC4andoneisolatedvertex} take e.g. $n := 4$, $u:=(1,3)$, $v:=(3,3)$ and $B[u]:=-$ and $B[v]:=0$. 
For \ref{item:comparisonproofcaseLhassize6:C4withoneadditionaldisjointedge} 
take e.g. $n := 4$, $u:=(1,3)$, $v:=(3,3)$, $B[u]:=0$ and $B[v]:=-$. 
For \ref{item:comparisonproofcaseLhassize6:C4intersectingtwoedgesinseparatenonadjacentvertices} take e.g. $n := 5$, $u:=(1,3)$, $v:=(2,4)$ and $B[u]:=B[v]:=-$. 
For \ref{item:comparisonproofcaseLhassize6:C4intersectingtwoedgesinseparateadjacentvertices} take e.g. $n:=4$, $u:=(1,3)$, $v:=(3,2)$ and $B[u]:=B[v]:=-$. 
For \ref{item:comparisonproofcaseLhassize6:C4intersectingatwopathinanendvertex} 
take e.g. $n := 4$, $u:=(1,3)$, $v:=(3,3)$ and $B[u]:=B[v]:=-$.  
For \ref{item:comparisonproofcaseLhassize6:C4intersectingatwopathinitsinnervertex}
take e.g. $n:=5$, $u:=(1,3)$, $v:=(1,4)$ and $B[u]:=B[v]:=-$.    
For \ref{item:comparisonproofcaseLhassize6:C6} we have to make an exception to our 
convention that $\{u,v\} = I\setminus S$ with $S$ defined as above, and have to 
define the set $I$ in its entirety. We can take e.g. 
$n:=4$, $I:=\{(1,1),(1,2),(2,2),(2,3),(3,1),(3,3)\}\in\binom{[n-1]^2}{k}$ 
and $B=\{-\}^I$.
For \ref{item:comparisonproofcaseLhassize6:threeisolatedvertices}  take e.g. $n:=5$, 
$u:=(3,3)$, $v:=(3,4)$ and $B[u]:=B[v]:=0$. 
For \ref{item:comparisonproofcaseLhassize6:oneadditionaledgeintersectingC4andtwoisolatedvertices}, take e.g. $n:=5$, $u:=(1,3)$, $v:=(3,4)$, $B[u]:=-$ and $B[v]:=0$. 
For \ref{item:comparisonproofcaseLhassize6:oneadditionaldisjointedgeandoneisolatedvertex}, take e.g. $n:=5$, $u:=(3,3)$, $v:=(3,4)$, $B[u]:=-$ and $B[v]:=0$. 
For \ref{item:comparisonproofcaseLhassize6:twoadditionaledgesonlyoneofthemdisjoint}, 
take e.g. $n:=5$, $u:=(1,3)$, $v:=(3,4)$ and $B[u]:=B[v]:=-$. 
For \ref{item:comparisonproofcaseLhassize6:twoadditionalnondisjointedgesdisjointfromC4} 
take e.g. $n:=5$, $u:=(3,3)$, $v:=(3,4)$ and $B[u]:=B[v]:=-$. 
For \ref{item:comparisonproofcaseLhassize6:fourisolatedvertices} take e.g. $n:=5$, 
$u:=(3,3)$, $v:=(4,4)$ and $B[u]:=B[v]:=0$. 
For \ref{item:comparisonproofcaseLhassize6:oneadditionaldisjointedgeandtwoisolatedvertices} take e.g. $n:=5$, $u:=(3,3)$, $v:=(4,4)$,  $B[u]:= 0$ and $B[v]:=1$. 
For \ref{item:comparisonproofcaseLhassize6:twoadditionaldisjointedgesdisjointfromC4} 
take e.g. $n:=5$, $u:=(3,3)$, $v:=(4,4)$ and $B[u]:=B[v]:=-$. 
This proves \ref{item:failuresetofisomorphismtypesk6}. 
The proof of Corollary \ref{cor:isomorphismtypesforwhichequalityofmeasuresofentryspecificationeventsfails} is now complete.
\end{proof}

\begin{corollary}[ratios and absolute values of $\Prob_{\Chio}$ for up to six entry 
specifications]
\label{cor:ratiosandvaluesofpchioandplcffortheisomorphismtypesforwhichequalityfails} 
{\quad }
\begin{enumerate}[label={\rm(R\arabic{*})}] 
\item\label{it:decompositionsofgraphfailuresetsingeneral} 
$\mathcal{F}^{\mathrm{G}}(k,n)  = \mathcal{F}_{\cdot 0}^{\mathrm{G}}(k,n) = 
\bigsqcup_{\beta\in\Z_{\geq 1}}\mathcal{F}_{\cdot 2^\beta}^{\mathrm{G}}(k,n)$\quad ,
\item\label{it:decompositionsofmatrixfailuresetsingeneral} 
$\mathcal{F}^{\mathrm{M}}(k,n) = \mathcal{F}_{\cdot 0}^{\mathrm{M}}(k,n) \sqcup 
\bigsqcup_{\beta\in\Z_{\geq 1}}\mathcal{F}_{\cdot 2^\beta}^{\mathrm{M}}(k,n)$\quad ,
\item\label{it:decompositionsofgraphfailuresets} 
$\mathcal{F}^{\mathrm{G}}(4,n) = \mathcal{F}_{\cdot 2}^{\mathrm{G}}(4,n)$, 
$\mathcal{F}^{\mathrm{G}}(5,n) = \mathcal{F}_{\cdot 2}^{\mathrm{G}}(5,n)$ and 
$\mathcal{F}^{\mathrm{G}}(6,n) = 
\mathcal{F}_{\cdot 2}^{\mathrm{G}}(6,n) \sqcup \mathcal{F}_{\cdot 4}^{\mathrm{G}}(6,n)$\quad ,
\item\label{it:decompositionsofmatrixfailuresets} 
$\mathcal{F}^{\mathrm{M}}(4,n) = \mathcal{F}_{\cdot 0}^{\mathrm{M}}(4,n) \sqcup 
\mathcal{F}_{\cdot 2}^{\mathrm{M}}(4,n)$,  $\mathcal{F}^{\mathrm{M}}(5,n) = 
\mathcal{F}_{\cdot 0}^{\mathrm{M}}(5,n) \sqcup \mathcal{F}_{\cdot 2}^{\mathrm{M}}(5,n)$ \\ 
and $\mathcal{F}^{\mathrm{M}}(6,n) = \mathcal{F}_{\cdot 0}^{\mathrm{M}}(6,n) \sqcup
\mathcal{F}_{\cdot 2}^{\mathrm{M}}(6,n) \sqcup \mathcal{F}_{\cdot 4}^{\mathrm{M}}(6,n)$\quad .
\end{enumerate}
Moreover, 
\begin{enumerate}[label={\rm(A\arabic{*})}, start=4] 
\item\label{item:absolutevaluesofchiomeasureonfailureentryspecificationevent:4entriesspecified} 
$\mathcal{F}^{\mathrm{G}}(4,n) = \mathcal{F}_{=(\frac12)^{7}}^{\mathrm{G}}(4,n)$, \; 
with $\mathcal{F}_{=(\frac12)^{7}}^{\mathrm{G}}(4,n) = 
\{\ref{item:comparisonproofcaseLhassize4:4circuit}\}$\quad , 
\item\label{item:absolutevaluesofchiomeasureonfailureentryspecificationevent:5entriesspecified} 
$\mathcal{F}^{\mathrm{G}}(5,n) = \mathcal{F}_{=(\frac12)^8}^{\mathrm{G}}(5,n) \sqcup 
\mathcal{F}_{=(\frac12)^9}^{\mathrm{G}}(5,n)$\quad , \\
with $\mathcal{F}_{=(\frac12)^{8}}^{\mathrm{G}}(5,n) = \{\ref{item:comparisonproofcaseLhassize6:oneisolatedvertex}, \ref{item:comparisonproofcaseLhassize6:twoisolatedvertices}\}$\quad and\quad  $\mathcal{F}_{=(\frac12)^9}^{\mathrm{G}}(5,n) = \{\ref{item:comparisonproofcaseLhassize6:oneadditionaledgeintersectingC4}, \ref{item:comparisonproofcaseLhassize6:C4withoneadditionaldisjointedge}\}$\quad , 
\item\label{item:absolutevaluesofchiomeasureonfailureentryspecificationevent:6entriesspecified} 
$\mathcal{F}^{\mathrm{G}}(6,n) = 
\mathcal{F}_{=(\frac12)^9}^{\mathrm{G}}(6,n)\sqcup
\mathcal{F}_{=(\frac12)^{10}}^{\mathrm{G}}(6,n)\sqcup
\mathcal{F}_{=(\frac12)^{11}}^{\mathrm{G}}(6,n)$, where
\begin{enumerate}[label={\rm(\roman{*})}] 
\item\label{item:absolutevaluesofchiomeasureonfailureentryspecificationevent:6entriesspecified:typeswithvalueoneovertwotothenine} 
$\mathcal{F}_{=(\frac12)^9}^{\mathrm{G}}(6,n)$
$ = \{ 
\ref{item:comparisonproofcaseLhassize6:oneisolatedvertex},
\ref{item:comparisonproofcaseLhassize6:twoisolatedvertices}, 
\ref{item:comparisonproofcaseLhassize6:threeisolatedvertices}, 
\ref{item:comparisonproofcaseLhassize6:fourisolatedvertices} \}$ \quad ,
\item\label{item:absolutevaluesofchiomeasureonfailureentryspecificationevent:6entriesspecified:typeswithvalueoneovertwototheten}  
$\mathcal{F}_{=(\frac12)^{10}}^{\mathrm{G}}(6,n)$
$=\{ 
\ref{item:comparisonproofcaseLhassize6:oneadditionaledgeintersectingC4}, 
\ref{item:comparisonproofcaseLhassize6:XBLisomorphictoK23}, 
\ref{item:comparisonproofcaseLhassize6:oneadditionaledgeintersectingC4andoneisolatedvertex}, 
\ref{item:comparisonproofcaseLhassize6:C4withoneadditionaldisjointedge}, 
\ref{item:comparisonproofcaseLhassize6:oneadditionaledgeintersectingC4andtwoisolatedvertices}, 
\ref{item:comparisonproofcaseLhassize6:oneadditionaldisjointedgeandoneisolatedvertex}, 
\ref{item:comparisonproofcaseLhassize6:oneadditionaldisjointedgeandtwoisolatedvertices} \}$ \quad ,
\item\label{item:absolutevaluesofchiomeasureonfailureentryspecificationevent:6entriesspecified:typeswithvalueoneovertwototheeleven}  
$\mathcal{F}_{=(\frac12)^{11}}^{\mathrm{G}}(6,n)$
$ = \{ \ref{item:comparisonproofcaseLhassize6:C4intersectingtwoedgesinseparatenonadjacentvertices}, 
\ref{item:comparisonproofcaseLhassize6:C4intersectingtwoedgesinseparateadjacentvertices}, 
\ref{item:comparisonproofcaseLhassize6:C4intersectingatwopathinanendvertex}, 
\ref{item:comparisonproofcaseLhassize6:C4intersectingatwopathinitsinnervertex},
\ref{item:comparisonproofcaseLhassize6:C6},  
\ref{item:comparisonproofcaseLhassize6:twoadditionaledgesonlyoneofthemdisjoint}, 
\ref{item:comparisonproofcaseLhassize6:twoadditionalnondisjointedgesdisjointfromC4}, 
\ref{item:comparisonproofcaseLhassize6:twoadditionaldisjointedgesdisjointfromC4} \}$ 
\quad .
\end{enumerate}
\end{enumerate}
\end{corollary}
\begin{proof}
As to \ref{it:decompositionsofgraphfailuresetsingeneral}, let us start with 
$\mathcal{F}^{\mathrm{G}}(k,n) = \mathcal{F}_{\cdot 0}^{\mathrm{G}}(k,n)$. The inclusion 
$\supseteq$ is true directly by \ref{def:graphversionsofthefailuresets} in 
Definition \ref{def:parametrizedfailuresets}. Conversely, let 
$\mathfrak{X}\in\mathcal{F}^{\mathrm{G}}(k,n)$. Then $\beta_1(\mathfrak{X})\geq 1$ by 
Corollary \ref{cor:quickconsequencesofthecharacterizations}.\ref{cor:graphfailuresetaspreimage}, hence 
$\lvert \{\pm\}^{\upE(\mathfrak{X})}\setminus S_{\bal}(\mathfrak{X})\rvert = 
2^{f_1(\mathfrak{X})} - 2^{f_0(\mathfrak{X}) - \beta_0(\mathfrak{X})} = 
2^{f_1(\mathfrak{X})} - 2^{f_1(\mathfrak{X}) - \beta_1(\mathfrak{X})} > 0$, hence there exists 
$B\in\mathcal{F}_{\cdot 0}^{\mathrm{M}}(k,n)$ with $\mathfrak{X} = \upXul^{k,n,n}(B)$, 
hence $\mathfrak{X}\in\upXul^{k,n,n}(\mathcal{F}_{\cdot 0}^{\mathrm{M}}(k,n)) 
\By{Definition 
\ref{def:parametrizedfailuresets}.\ref{def:graphversionsofthefailuresets}}{=} 
\mathcal{F}_{\cdot 0}^{\mathrm{G}}(k,n)$, proving $\subseteq$. 

As to the partition claimed in \ref{it:decompositionsofgraphfailuresetsingeneral}, 
both claims follow immediately from 
\ref{relationbeweenchiomeasureandlazycoinflipmeasuregovernedbyfirstbettinumber} 
in Theorem \ref{thm:graphtheoreticalcharacterizationofthechiomeasure} (and the 
equality $\mathcal{F}^{\mathrm{G}}(k,n) = \mathcal{F}^{\mathrm{G}}_{\cdot 0}(k,n)$ is the 
reason why $\mathcal{F}_{\cdot 0}^{\mathrm{G}}(k,n)$ is missing in the disjoint union 
in \ref{it:decompositionsofgraphfailuresetsingeneral}).
As to \ref{it:decompositionsofgraphfailuresets} (respectively 
\ref{it:decompositionsofmatrixfailuresets}), this follows by combining 
\ref{it:decompositionsofgraphfailuresetsingeneral} 
(respectively \ref{it:decompositionsofmatrixfailuresetsingeneral}) with \ref{item:failuresetofisomorphismtypesk4}--\ref{item:failuresetofisomorphismtypesk6} in 
Corollary \ref{cor:isomorphismtypesforwhichequalityofmeasuresofentryspecificationeventsfails}. The claims \ref{item:absolutevaluesofchiomeasureonfailureentryspecificationevent:4entriesspecified}--\ref{item:absolutevaluesofchiomeasureonfailureentryspecificationevent:6entriesspecified} will be proved in reverse order. 

As to \ref{item:absolutevaluesofchiomeasureonfailureentryspecificationevent:6entriesspecified}, this seems to require some calculations. However, 
Corollary \ref{cor:quickconsequencesofthecharacterizations}.\ref{cor:equalityofChiomeasuresofeventswithcoalescedassociatedgraphs} can be used to reduce the amount of 
work to be done: if $(a)$ and $(b)$ are isomorphism types of graphs, let us write 
$(a) \twoheadrightarrow  (b)$ if and only if $(b)$ can be obtained from $(a)$ by a 
single one-point wedge of two connected components of $(a)$. Moreover, if $(a)$ is 
any of the isomorphism types in $\mathcal{F}^{\mathrm{M}}(6,n)$, let us employ the 
abbreviation $\mathcal{E}_B := \mathcal{E}_B^{[n-1]^2}$ and let us write 
$\Prob_{\chio}[(a)]$ for the number $\Prob_{\chio}[\mathcal{E}_B^{[n-1]^2}]$ 
with $B$ an arbitrary $B\in\{0,\pm\}^I$,  $I\in\binom{[n-1]^2}{6}$, $\upX_B\in (a)$ 
and $(\upX_B,\sigma_B)$ balanced. 
By \ref{explicitformulaforchiomeasureinthecaseofrealizability} in  Theorem 
\ref{thm:graphtheoreticalcharacterizationofthechiomeasure} we know that 
$\Prob_{\chio}[(a)]$ then does indeed  only depend on $(a)$, not on the choice 
of such a $B$.

Since evidently 
\ref{item:comparisonproofcaseLhassize6:fourisolatedvertices}
$\twoheadrightarrow$ 
\ref{item:comparisonproofcaseLhassize6:threeisolatedvertices}
$\twoheadrightarrow$ 
\ref{item:comparisonproofcaseLhassize6:twoisolatedvertices}
$\twoheadrightarrow$ 
\ref{item:comparisonproofcaseLhassize6:oneisolatedvertex},
Corollary \ref{cor:quickconsequencesofthecharacterizations}.\ref{cor:equalityofChiomeasuresofeventswithcoalescedassociatedgraphs} 
implies 
$\Prob_{\chio}[\ref{item:comparisonproofcaseLhassize6:oneisolatedvertex}]$
$=$
$\Prob_{\chio}[\ref{item:comparisonproofcaseLhassize6:twoisolatedvertices}]$
$=$ 
$\Prob_{\chio}[\ref{item:comparisonproofcaseLhassize6:threeisolatedvertices}]$
$=$
$\Prob_{\chio}[\ref{item:comparisonproofcaseLhassize6:fourisolatedvertices}]$.
Since evidently 
\ref{item:comparisonproofcaseLhassize6:oneadditionaldisjointedgeandtwoisolatedvertices} $\twoheadrightarrow$ \ref{item:comparisonproofcaseLhassize6:oneadditionaldisjointedgeandoneisolatedvertex} $\twoheadrightarrow$ \ref{item:comparisonproofcaseLhassize6:oneadditionaledgeintersectingC4andoneisolatedvertex} $\twoheadrightarrow$
\ref{item:comparisonproofcaseLhassize6:oneadditionaledgeintersectingC4}, Corollary 
\ref{cor:quickconsequencesofthecharacterizations}.\ref{cor:equalityofChiomeasuresofeventswithcoalescedassociatedgraphs} implies 
$\Prob_{\chio}[\ref{item:comparisonproofcaseLhassize6:oneadditionaldisjointedgeandtwoisolatedvertices}]$ 
$=$ 
$\Prob_{\chio}[\ref{item:comparisonproofcaseLhassize6:oneadditionaldisjointedgeandoneisolatedvertex}]$ 
$=$
$\Prob_{\chio}[\ref{item:comparisonproofcaseLhassize6:oneadditionaledgeintersectingC4andoneisolatedvertex}]$
$=$
$\Prob_{\chio}[\ref{item:comparisonproofcaseLhassize6:oneadditionaledgeintersectingC4}]$. Since also 
\ref{item:comparisonproofcaseLhassize6:oneadditionaledgeintersectingC4andtwoisolatedvertices} $\twoheadrightarrow$ \ref{item:comparisonproofcaseLhassize6:oneadditionaledgeintersectingC4andoneisolatedvertex}, Corollary 
\ref{cor:quickconsequencesofthecharacterizations}.\ref{cor:equalityofChiomeasuresofeventswithcoalescedassociatedgraphs} implies 
$\Prob_{\chio}[\ref{item:comparisonproofcaseLhassize6:oneadditionaledgeintersectingC4andtwoisolatedvertices}]$ 
$=$ 
$\Prob_{\chio}[\ref{item:comparisonproofcaseLhassize6:oneadditionaledgeintersectingC4andoneisolatedvertex}]$. Since moreover \ref{item:comparisonproofcaseLhassize6:C4withoneadditionaldisjointedge} $\twoheadrightarrow$ \ref{item:comparisonproofcaseLhassize6:oneadditionaledgeintersectingC4}, Corollary 
\ref{cor:quickconsequencesofthecharacterizations}.\ref{cor:equalityofChiomeasuresofeventswithcoalescedassociatedgraphs} implies 
$\Prob_{\chio}[\ref{item:comparisonproofcaseLhassize6:C4withoneadditionaldisjointedge}]$ $=$ $\Prob_{\chio}[\ref{item:comparisonproofcaseLhassize6:oneadditionaledgeintersectingC4}]$.
These equations together imply 
$\Prob_{\chio}[\ref{item:comparisonproofcaseLhassize6:oneadditionaledgeintersectingC4}]$
$=$
$\Prob_{\chio}[\ref{item:comparisonproofcaseLhassize6:oneadditionaledgeintersectingC4andoneisolatedvertex}]$
$=$
$\Prob_{\chio}[\ref{item:comparisonproofcaseLhassize6:C4withoneadditionaldisjointedge}]$ 
$=$ 
$\Prob_{\chio}[\ref{item:comparisonproofcaseLhassize6:oneadditionaledgeintersectingC4andtwoisolatedvertices}]$ 
$=$
$\Prob_{\chio}[\ref{item:comparisonproofcaseLhassize6:oneadditionaldisjointedgeandoneisolatedvertex}]$
$=$
$\Prob_{\chio}[\ref{item:comparisonproofcaseLhassize6:oneadditionaldisjointedgeandtwoisolatedvertices}]$.
Since evidently 
\ref{item:comparisonproofcaseLhassize6:twoadditionaldisjointedgesdisjointfromC4} 
$\twoheadrightarrow$ \ref{item:comparisonproofcaseLhassize6:twoadditionalnondisjointedgesdisjointfromC4} $\twoheadrightarrow$ \ref{item:comparisonproofcaseLhassize6:C4intersectingatwopathinanendvertex}, Corollary 
\ref{cor:quickconsequencesofthecharacterizations}.\ref{cor:equalityofChiomeasuresofeventswithcoalescedassociatedgraphs} implies 
$\Prob_{\chio}[\ref{item:comparisonproofcaseLhassize6:twoadditionaldisjointedgesdisjointfromC4}]$ 
$=$ 
$\Prob_{\chio}[\ref{item:comparisonproofcaseLhassize6:twoadditionalnondisjointedgesdisjointfromC4}]$ 
$=$
$\Prob_{\chio}[\ref{item:comparisonproofcaseLhassize6:C4intersectingatwopathinanendvertex}]$.
Since evidently 
\ref{item:comparisonproofcaseLhassize6:twoadditionaldisjointedgesdisjointfromC4} 
$\twoheadrightarrow$ \ref{item:comparisonproofcaseLhassize6:twoadditionaledgesonlyoneofthemdisjoint} $\twoheadrightarrow$ \ref{item:comparisonproofcaseLhassize6:C4intersectingatwopathinanendvertex}, Corollary 
\ref{cor:quickconsequencesofthecharacterizations}.\ref{cor:equalityofChiomeasuresofeventswithcoalescedassociatedgraphs} implies 
$\Prob_{\chio}[\ref{item:comparisonproofcaseLhassize6:twoadditionaldisjointedgesdisjointfromC4}]$ 
$=$ 
$\Prob_{\chio}[\ref{item:comparisonproofcaseLhassize6:twoadditionaledgesonlyoneofthemdisjoint}]$ 
$=$
$\Prob_{\chio}[\ref{item:comparisonproofcaseLhassize6:C4intersectingatwopathinanendvertex}]$.
Since evidently 
\ref{item:comparisonproofcaseLhassize6:twoadditionalnondisjointedgesdisjointfromC4} $\twoheadrightarrow$ \ref{item:comparisonproofcaseLhassize6:C4intersectingatwopathinitsinnervertex}, Corollary 
\ref{cor:quickconsequencesofthecharacterizations}.\ref{cor:equalityofChiomeasuresofeventswithcoalescedassociatedgraphs} implies 
$\Prob_{\chio}[\ref{item:comparisonproofcaseLhassize6:twoadditionalnondisjointedgesdisjointfromC4}]$ 
$=$
$\Prob_{\chio}[\ref{item:comparisonproofcaseLhassize6:C4intersectingatwopathinitsinnervertex}]$.
Since evidently 
\ref{item:comparisonproofcaseLhassize6:twoadditionaledgesonlyoneofthemdisjoint} $\twoheadrightarrow$ \ref{item:comparisonproofcaseLhassize6:C4intersectingtwoedgesinseparateadjacentvertices}, Corollary 
\ref{cor:quickconsequencesofthecharacterizations}.\ref{cor:equalityofChiomeasuresofeventswithcoalescedassociatedgraphs} implies 
$\Prob_{\chio}[\ref{item:comparisonproofcaseLhassize6:twoadditionaledgesonlyoneofthemdisjoint}]$ 
$=$
$\Prob_{\chio}[\ref{item:comparisonproofcaseLhassize6:C4intersectingtwoedgesinseparateadjacentvertices}]$.
Since evidently 
\ref{item:comparisonproofcaseLhassize6:twoadditionaledgesonlyoneofthemdisjoint} $\twoheadrightarrow$ \ref{item:comparisonproofcaseLhassize6:C4intersectingtwoedgesinseparatenonadjacentvertices}, Corollary 
\ref{cor:quickconsequencesofthecharacterizations}.\ref{cor:equalityofChiomeasuresofeventswithcoalescedassociatedgraphs} implies 
$\Prob_{\chio}[\ref{item:comparisonproofcaseLhassize6:twoadditionaledgesonlyoneofthemdisjoint}]$ 
$=$
$\Prob_{\chio}[\ref{item:comparisonproofcaseLhassize6:C4intersectingtwoedgesinseparatenonadjacentvertices}]$. These equations together imply that 
$\Prob_{\chio}[\ref{item:comparisonproofcaseLhassize6:C4intersectingtwoedgesinseparatenonadjacentvertices}]$
$=$
$\Prob_{\chio}[\ref{item:comparisonproofcaseLhassize6:C4intersectingtwoedgesinseparateadjacentvertices}]$
$=$
$\Prob_{\chio}[\ref{item:comparisonproofcaseLhassize6:C4intersectingatwopathinanendvertex}]$
$=$
$\Prob_{\chio}[\ref{item:comparisonproofcaseLhassize6:C4intersectingatwopathinitsinnervertex}]$
$=$
$\Prob_{\chio}[\ref{item:comparisonproofcaseLhassize6:twoadditionaledgesonlyoneofthemdisjoint}]$ 
$=$ 
$\Prob_{\chio}[\ref{item:comparisonproofcaseLhassize6:twoadditionalnondisjointedgesdisjointfromC4}]$ 
$=$
$\Prob_{\chio}[\ref{item:comparisonproofcaseLhassize6:twoadditionaldisjointedgesdisjointfromC4}]$.

This proves that it suffices (note that of the nineteen elements of 
$\mathcal{F}^{\mathrm{M}}(6,n)$  exactly 
$\ref{item:comparisonproofcaseLhassize6:XBLisomorphictoK23}$ 
and $\ref{item:comparisonproofcaseLhassize6:C6}$ 
have not been part of one of the equality chains) to calculate only 
$\Prob_{\chio}[\ref{item:comparisonproofcaseLhassize6:oneisolatedvertex}]$, 
$\Prob_{\chio}[\ref{item:comparisonproofcaseLhassize6:oneadditionaledgeintersectingC4}]$, $\Prob_{\chio}[\ref{item:comparisonproofcaseLhassize6:XBLisomorphictoK23}]$, 
$\Prob_{\chio}[\ref{item:comparisonproofcaseLhassize6:C4intersectingtwoedgesinseparatenonadjacentvertices}]$ and $\Prob_{\chio}[\ref{item:comparisonproofcaseLhassize6:C6}]$. With the formulas in \ref{explicitformulaforchiomeasureinthecaseofrealizability} 
of Theorem \ref{thm:graphtheoreticalcharacterizationofthechiomeasure} and 
in Lemma \ref{lem:valueoflazycoinflipdistribution}, this can be done as 
follows (keep in mind that, being within item \ref{item:absolutevaluesofchiomeasureonfailureentryspecificationevent:6entriesspecified}, $\dom(B) = \lvert I \rvert = 6$ 
in each calculation):
If $\mathfrak{X}=\ref{item:comparisonproofcaseLhassize6:oneisolatedvertex}$, then 
$f_0( \mathfrak{X} ) = 5$, $\beta_0(\mathfrak{X}) = 2$, hence 
$\Prob_{\chio}[\mathcal{E}_B]$ $=$ $(\frac12)^{\lvert I \rvert+5-2}$ $=$ $(\frac12)^9$. \\
If $\mathfrak{X}\in\{\ref{item:comparisonproofcaseLhassize6:oneadditionaledgeintersectingC4}, \ref{item:comparisonproofcaseLhassize6:XBLisomorphictoK23}\}$, then 
$f_0( \mathfrak{X} ) = 5$, $\beta_0(\mathfrak{X}) = 1$, hence 
$\Prob_{\chio}[\mathcal{E}_B]$ $=$ $(\frac12)^{\lvert I \rvert+5-1}$ $=$ $(\frac12)^{10}$.\\ 
If $\mathfrak{X}\in \{\ref{item:comparisonproofcaseLhassize6:C4intersectingtwoedgesinseparatenonadjacentvertices}, \ref{item:comparisonproofcaseLhassize6:C6}\}$, then 
$f_0( \mathfrak{X} ) = 6$, $\beta_0(\mathfrak{X}) = 1$, hence 
$\Prob_{\chio}[\mathcal{E}_B]$ $=$ $(\frac12)^{\lvert I \rvert+6-1}$ $=$ 
$(\frac12)^{11}$. 

As to \ref{item:absolutevaluesofchiomeasureonfailureentryspecificationevent:5entriesspecified}, it follows by an entirely analogous (but much shorter) argument as the 
one given for \ref{item:absolutevaluesofchiomeasureonfailureentryspecificationevent:5entriesspecified} that it suffices to calculate only $\Prob_{\chio}[\ref{item:comparisonproofcaseLhassize6:oneisolatedvertex}]$ and $\Prob_{\chio}[\ref{item:comparisonproofcaseLhassize6:oneadditionaledgeintersectingC4}]$, and these calculations are 
identical to the ones made for $\Prob_{\chio}[\ref{item:comparisonproofcaseLhassize6:oneisolatedvertex}]$ and $\Prob_{\chio}[\ref{item:comparisonproofcaseLhassize6:oneadditionaledgeintersectingC4}]$ in the preceding paragraph, 
except that now $\lvert I \rvert = 5$.

As to \ref{item:absolutevaluesofchiomeasureonfailureentryspecificationevent:4entriesspecified}, in view of \ref{item:failuresetofisomorphismtypesk4} in 
Corollary \ref{cor:isomorphismtypesforwhichequalityofmeasuresofentryspecificationeventsfails}, we only have to deal with the single type \ref{item:comparisonproofcaseLhassize4:4circuit} where $f_0\ref{item:comparisonproofcaseLhassize4:4circuit} = 4$, 
$\beta_0\ref{item:comparisonproofcaseLhassize4:4circuit} = 1$ and therefore 
$\Prob_{\chio}[\ref{item:comparisonproofcaseLhassize4:4circuit}] = 
(\tfrac12)^{\lvert I \rvert + 4 - 1} = (\tfrac12)^{4+4-1} = (\tfrac12)^7$. 
\end{proof}

The results obtained so far can be turned into a set of instructions of how to 
tell the measure of $\Prob_{\chio}[\mathcal{E}_B^{[n-1]^2}]$ from a 
given $B\in \{0,\pm\}^I$ provided that  $\lvert I \rvert \leq 6$. We formulate the 
instructions exclusively in terms of those data, avoiding any mention 
of the associated signed graph ($\upX_B$,$\sigma_B$). 

\begin{corollary}[how to find the measure under $\Prob_{\chio}$ of large 
entry-specification events]
\label{cor:recipe} 
For every $\emptyset \subseteq I \subseteq [n-1]^2$ with $\lvert I \rvert \leq 6$ 
and every $B\in \{0,\pm\}^I $, the following instructions lead to the 
correct Chio measure of $\mathcal{E}_B := \mathcal{E}_B^{[n-1]^2}$:
\begin{enumerate}[label={\rm(H\arabic{*})}, start=3]
\item\label{comparisonthm:item:uptto3entries} If 
$0\leq \lvert I\rvert \leq 3$, then 
$\Prob_{\chio}[\mathcal{E}_B] = \Prob_{\lcf} [\mathcal{E}_B] 
= (\frac12)^{\dom(B)+\supp(B)}$ .
\item\label{thm:it:fourentriesspecified} If $\lvert I\rvert = 4$, 
then check whether $I$ is a matrix-$4$-circuit \emph{such that} 
$B\in \{\pm\}^I$. If not, then 
$\Prob_{\chio}[\mathcal{E}_B] = \Prob_{\lcf} [\mathcal{E}_B] 
= (\frac12)^{\dom(B)+\supp(B)}$. If $B$ has this property, then check whether 
an odd number of the four nonzero values $B[i,j]$ with $(i,j)\in I$ are $+$. 
If so, $\Prob_{\chio}[\mathcal{E}_B]=0$. If not, then 
$\Prob_{\chio}[\mathcal{E}_B] = (\frac12)^7 = 2\cdot \Prob_{\lcf}[\mathcal{E}_B]$.
\item\label{thm:it:fiveentriesspecified} 
If $\lvert I\rvert = 5$, then check whether there exists within $I$ a 
matrix-$4$-circuit $S\subseteq I$ \emph{such that} $B\mid_{S}\in \{\pm\}^S$. 
If not, then 
$\Prob_{\chio}[\mathcal{E}_B] = \Prob_{\lcf} [\mathcal{E}_B] 
= (\frac12)^{\dom(B)+\supp(B)}$. If so, then check whether an odd number of 
the four nonzero values $B[i,j]$ with $(i,j)\in S$ are $+$. 
If so, $\Prob_{\chio}[\mathcal{E}_B]=0$. If not, then check whether the 
entry $B[i,j]$ with $(i,j)\in I\setminus S$ is zero. If it is, then 
$\Prob_{\chio}[\mathcal{E}_B] = (\frac12)^8 = 
2 \cdot \Prob_{\lcf}[\mathcal{E}_B]$. If it is nonzero, then 
$\Prob_{\chio}[\mathcal{E}_B] = (\frac12)^9 = 
2\cdot \Prob_{\lcf}[\mathcal{E}_B]$.
\item\label{thm:it:sixentriesspecified} 
If $\lvert I\rvert = 6$, then check whether at least four entries in $B$ 
are nonzero. If not, then $\Prob_{\chio}[\mathcal{E}_B] = \Prob_{\lcf} [\mathcal{E}_B] 
= (\frac12)^{\dom(B)+\supp(B)}$. If so, then check whether $I$ contains a 
matrix-$4$-circuit $S\subseteq I$ with $B\mid_{S} \in \{\pm\}^S$. 
\begin{enumerate}[label={\rm(\roman{*})}]
\item\label{item:thmonfindingPchio:sizeLis6no4circuit} If $I$ does not 
contain such a matrix-$4$-circuit $S$, then check whether $I$ is a 
matrix-$6$-circuit such that $B \in \{\pm\}^I$. If not, then  
$\Prob_{\chio}[\mathcal{E}_B] = \Prob_{\lcf} [\mathcal{E}_B] 
= (\frac12)^{\dom(B)+\supp(B)}$. If so, then check 
whether an odd number of these six entries are $+$. If so, then 
$\Prob_{\chio}[\mathcal{E}_B] = 0$. If not, then 
$\Prob_{\chio}[\mathcal{E}_B] = (\frac12)^{11} = 
2\cdot \Prob_{\lcf}[\mathcal{E}_B]$.
\item\label{item:thmonfindingPchio:sizeLis6thereis4circuit} If $I$ does indeed 
contain such a matrix-$4$-circuit, then check whether an odd number of the four 
nonzero values $B[i,j]$ with $(i,j)\in S$ are $+$. 
If so, then $\Prob_{\chio}[\mathcal{E}_B] = 0$. Else, there are three 
further cases: 
\begin{enumerate}[label={\rm(\alph{*})}]
\item
\label{item:thmonfindingPchio:sizeLis6thereis4circuit:bothnoncircuitentrieszero} 
If both entries indexed by $I\setminus S$ are zero, then  
$\Prob_{\chio}[\mathcal{E}_B] = (\frac12)^9 = 2\cdot 
\Prob_{\lcf} [\mathcal{E}_B]$.
\item
\label{item:thmonfindingPchio:sizeLis6thereis4circuit:onenoncircuitentrieszero} 
If exactly one of the two entries indexed by $I\setminus S$ is zero, then  
$\Prob_{\chio}[\mathcal{E}_B] = (\frac12)^{10} = 2\cdot 
\Prob_{\lcf} [\mathcal{E}_B]$.
\item
\label{item:thmonfindingPchio:sizeLis6thereis4circuit:bothnoncircuitentriesnonzero} 
If both entries indexed by $I\setminus S$ are nonzero, then the positions of 
these two nonzero entries must be taken into account: if there do \emph{neither} 
exist $1\leq i<i'<i'' \leq n-1$ and $1 \leq j < j' \leq n-1$ such 
that $I$ $=$ $\{$ $(i,j),$ $(i',j),$ $(i'',j),$ $(i,j'),$ $(i',j'),$ $(i'',j')$ 
$\}$ \emph{nor} $1\leq i < i' \leq n-1$ and $1\leq j < j' < j'' \leq n-1$ such 
that $I$ $=$ $\{$ $(i,j),$ $(i,j'),$ $(i,j''),$ $(i',j),$ $(i',j'),$ 
$(i',j'')$ $\}$, then---whatever the $B$-values indexed by the two elements in 
$I\setminus S$ may be---you know that 
$\Prob_{\chio}[\mathcal{E}_B] = (\frac12)^{11} = 2\cdot 
\Prob_{\lcf} [\mathcal{E}_B]$. Else, check whether in any one of the then 
existing two additional matrix-$4$-circuits in $I$ the number of $+$ entries 
is odd. If so, then $\Prob_{\chio}[\mathcal{E}_B]=0$. If not, then 
$\Prob_{\chio}[\mathcal{E}_B] = (\frac12)^{10} 
= 4\cdot \Prob_{\lcf} [\mathcal{E}_B]$. \hfill $\Box$
\end{enumerate} 
\end{enumerate}
\end{enumerate}
\end{corollary}

If $U\subseteq \Dom(\upX^{k,n,n}) = \bigsqcup_{I\in\binom{[s-1]\times [t-1]}{k}}\{0,\pm\}^I$ 
is an arbitrary subset, then on the abstract set-theoretical level all we know is 
$(\upX^{k,n,n})^{-1}(\upX^{k,n,n}(U))\supseteq U$. For the specific subsets 
$U = \mathcal{F}^{\mathrm{M}}(k,n)$, however, the inclusion is an equality: 

\begin{corollary}\label{cor:graphmapcanbeinvertedsetwiseonthefailuresets}
$(\upX^{k,n,n})^{-1}(\upX^{k,n,n}(\mathcal{F}^{\mathrm{M}}(k,n))) = 
\mathcal{F}^{\mathrm{M}}(k,n)$ 
\end{corollary}
\begin{proof}
Since $(\upX^{k,n,n})^{-1}(\upX^{k,n,n}(\mathcal{F}^{\mathrm{M}}(k,n)))$ 
$\By{Corollary \ref{cor:quickconsequencesofthecharacterizations}.\ref{cor:graphfailuresetaspreimage}}{=}$ 
$(\upX^{k,n,n})^{-1}(\upX^{k,n,n}( (\upX^{k,n,n})^{-1}(\beta_1^{-1}(\Z_{\geq 1}))))$ \\
$\By{Lemma \ref{lem:elementarypreimagestatements:finvffinv}}{=}$ 
$(\upX^{k,n,n})^{-1} (\beta_1^{-1}(\Z_{\geq 1}))$ 
$\By{again Corollary \ref{cor:quickconsequencesofthecharacterizations}.\ref{cor:graphfailuresetaspreimage}}{=}$ 
$\mathcal{F}^{\mathrm{M}}(k,n)$.
\end{proof}

Corollaries \ref{cor:isomorphismtypesforwhichequalityofmeasuresofentryspecificationeventsfails} and \ref{cor:graphmapcanbeinvertedsetwiseonthefailuresets} allow us to 
express the failure sets $\mathcal{F}^{\mathrm{M}}(k,n)$ as partitions indexed by 
isomorphism types of bipartite graphs: 

\begin{corollary}\label{cor:partitionsoffailuresets}
For every $n$, 
\begin{enumerate}[label={\rm(M\arabic{*})}, start=4] 
\item\label{it:partitionoffailuresets:kequals4} $\mathcal{F}^{\mathrm{M}}(4,n) = 
(\upXul^{4,n,n})^{-1}\ref{item:comparisonproofcaseLhassize4:4circuit}$\quad ,
\item\label{it:partitionoffailuresets:kequals5} $\mathcal{F}^{\mathrm{M}}(5,n) = 
(\upXul^{5,n,n})^{-1}\ref{item:comparisonproofcaseLhassize6:oneisolatedvertex} \sqcup 
(\upXul^{5,n,n})^{-1}\ref{item:comparisonproofcaseLhassize6:oneadditionaledgeintersectingC4} \sqcup (\upXul^{5,n,n})^{-1}\ref{item:comparisonproofcaseLhassize6:twoisolatedvertices}\sqcup (\upXul^{5,n,n})^{-1}\ref{item:comparisonproofcaseLhassize6:C4withoneadditionaldisjointedge}$\quad ,
\item\label{it:partitionoffailuresets:kequals6} $\mathcal{F}^{\mathrm{M}}(6,n) = (\upXul^{6,n,n})^{-1}\ref{item:comparisonproofcaseLhassize6:oneisolatedvertex} \sqcup (\upXul^{6,n,n})^{-1}\ref{item:comparisonproofcaseLhassize6:oneadditionaledgeintersectingC4} \sqcup (\upXul^{6,n,n})^{-1}\ref{item:comparisonproofcaseLhassize6:twoisolatedvertices} \sqcup (\upXul^{6,n,n})^{-1}\ref{item:comparisonproofcaseLhassize6:C4withoneadditionaldisjointedge} \\ 
 \sqcup (\upXul^{6,n,n})^{-1} 
\ref{item:comparisonproofcaseLhassize6:XBLisomorphictoK23} 
\sqcup \bigsqcup_{6\leq k \leq 20} (\upXul^{6,n,n})^{-1}(\mathrm{t}k)$\quad .
\end{enumerate}
\end{corollary}
\begin{proof}
In general we have
$
\mathcal{F}^{\mathrm{M}}(k,n) 
\By{\text{\tiny Corollary \ref{cor:graphmapcanbeinvertedsetwiseonthefailuresets}}}{=} 
(\upXul^{k,n,n})^{-1}(\upXul^{k,n,n}(\mathcal{F}^{\mathrm{M}}(k,n))) 
\By{\text{\tiny (by Definition \ref{def:parametrizedfailuresets}.\ref{def:graphversionsofthefailuresets})}}{=} (\upXul^{k,n,n})^{-1}( \mathcal{F}^{\mathrm{G}}(k,n) ) 
\By{\text{\tiny (for every map)}}{=} 
\bigsqcup_{\mathfrak{X}\in\mathcal{F}^{\mathrm{G}}(k,n)} (\upXul^{k,n,n})^{-1}(\mathfrak{X})$, 
and for the specific values $4\leq k \leq 6$ we can use Corollary 
\ref{cor:isomorphismtypesforwhichequalityofmeasuresofentryspecificationeventsfails} 
to obtain the claimed partitions. 
\end{proof}

While having the aim of explicitly determining the 
numbers $\lvert(\upXul^{k,n,n})^{-1}(\mathfrak{X})\rvert$ for 
certain $k$ and $\mathfrak{X}$ which interest us, we will start slowly by first 
formulating some linear relations among $\lvert(\upXul^{5,n,n})^{-1}\ref{item:comparisonproofcaseLhassize6:oneisolatedvertex}\rvert$, $\dotsc$, $\lvert(\upXul^{6,n,n})^{-1}\ref{item:comparisonproofcaseLhassize6:twoadditionaldisjointedgesdisjointfromC4}\rvert$ which will later serve as a check for the formulas given in 
Theorem \ref{thm:countingmatrixrealizations}. 

\begin{lemma}[linear relations among 
$\lvert (\upXul^{k,n,n})^{-1}(\mathfrak{X}) \rvert$ for $5\leq k \leq 6$]
\label{lem:relationsamongthenumbersofmatrixrealizations}
$\quad$
\begin{enumerate}[label={\rm(l\arabic{*})}] 
\item\label{linearrelationbetweent2andt3:kequals5}
$(3^1-1)\cdot \lvert (\upXul^{5,n,n})^{-1}\ref{item:comparisonproofcaseLhassize6:oneisolatedvertex} \rvert
= \lvert (\upXul^{5,n,n})^{-1}\ref{item:comparisonproofcaseLhassize6:oneadditionaledgeintersectingC4}\rvert$\quad ,
\item\label{linearrelationbetweent5andt7:kequals5}
$(3^1-1)\cdot \lvert (\upXul^{5,n,n})^{-1}\ref{item:comparisonproofcaseLhassize6:twoisolatedvertices} \rvert
=  \lvert (\upXul^{5,n,n})^{-1}\ref{item:comparisonproofcaseLhassize6:C4withoneadditionaldisjointedge}\rvert$\quad ,
\item\label{linearrelationbetweent4tot10:kequals6}
$
(3^2 -1) \cdot \lvert(\upXul^{6,n,n})^{-1}\ref{item:comparisonproofcaseLhassize6:twoisolatedvertices}\rvert
=
\lvert(\upXul^{6,n,n})^{-1}\ref{item:comparisonproofcaseLhassize6:oneadditionaledgeintersectingC4andoneisolatedvertex}\rvert
+ \dotsm +
\lvert(\upXul^{6,n,n})^{-1}\ref{item:comparisonproofcaseLhassize6:C4intersectingatwopathinitsinnervertex}\rvert$\quad , 
\item\label{linearrelationbetweent11tot15:kequals6} 
$
(3^2 - 1) \cdot \lvert(\upXul^{6,n,n})^{-1}\ref{item:comparisonproofcaseLhassize6:threeisolatedvertices}\rvert
=
\lvert(\upXul^{6,n,n})^{-1}\ref{item:comparisonproofcaseLhassize6:oneadditionaledgeintersectingC4andtwoisolatedvertices}\rvert
+
\dotsm 
+
\lvert(\upXul^{6,n,n})^{-1}\ref{item:comparisonproofcaseLhassize6:twoadditionalnondisjointedgesdisjointfromC4}\rvert$\quad ,
\item\label{linearrelationbetweent16tot18:kequals6} 
$
(3^2 - 1) \cdot \lvert(\upXul^{6,n,n})^{-1}\ref{item:comparisonproofcaseLhassize6:fourisolatedvertices}\rvert
=
\lvert(\upXul^{6,n,n})^{-1}\ref{item:comparisonproofcaseLhassize6:oneadditionaldisjointedgeandtwoisolatedvertices}\rvert
+
\lvert(\upXul^{6,n,n})^{-1}\ref{item:comparisonproofcaseLhassize6:twoadditionaldisjointedgesdisjointfromC4}\rvert$\quad .
\end{enumerate}
\end{lemma}
\begin{proof}
It follows from Definition \ref{def:XBandecXB} that 
$\upX_{B}\in\ref{item:comparisonproofcaseLhassize6:oneadditionaledgeintersectingC4}$ if and only if equation \eqref{eq:conditioninproofofnumberofrealizationsoftype:oneisolatedvertex:kequals5} is true and $B[u]\in\{\pm\}$. This implies 
$\lvert(\upXul^{5,n,n})^{-1}\ref{item:comparisonproofcaseLhassize6:oneadditionaledgeintersectingC4}\rvert = 
2\cdot \lvert(\upXul^{5,n,n})^{-1}\ref{item:comparisonproofcaseLhassize6:oneisolatedvertex}\rvert$, proving \ref{linearrelationbetweent2andt3:kequals5}. 
It also follows from Definition \ref{def:XBandecXB} that $\upX_{B}\in\ref{item:comparisonproofcaseLhassize6:C4withoneadditionaldisjointedge}$ if and only if equation \eqref{eq:conditioninproofofnumberofrealizationsoftype:twoisolatedvertices:kequals5} is true and $B[u]\in\{\pm\}$. This implies $\lvert(\upXul^{5,n,n})^{-1}\ref{item:comparisonproofcaseLhassize6:C4withoneadditionaldisjointedge}\rvert= 2\cdot \lvert(\upXul^{5,n,n})^{-1}\ref{item:comparisonproofcaseLhassize6:twoisolatedvertices}\rvert$, proving \ref{linearrelationbetweent5andt7:kequals5}.

The isomorphism types \ref{item:comparisonproofcaseLhassize6:twoisolatedvertices}--\ref{item:comparisonproofcaseLhassize6:C4intersectingatwopathinitsinnervertex} are all the isomorphism types of bipartite nonforests with six vertices and exactly one 
copy of $C^4$. Therefore 
$\lvert(\upXul^{6,n,n})^{-1}\ref{item:comparisonproofcaseLhassize6:twoisolatedvertices}\rvert + \dotsm + \lvert(\upXul^{6,n,n})^{-1}\ref{item:comparisonproofcaseLhassize6:C4intersectingatwopathinitsinnervertex}\rvert
$ is the number of all $B\in\{0,\pm\}^I$ with $I\in\binom{[n-1]^2}{6}$ such that 
$\upX_{B}$ contains exactly one $C^4$ and $f_0(\upX_{B}) = 6$. Imagine counting these 
$B$ by partitioning the set of all such $B$ according to the copy of $C^4$, and 
for each such copy, further partitioning the $B$ according to the mandatory 
$\pm$-values on the edges of the $C^4$, and then further partitioning according to 
the \emph{positions} of the two elements of $I$ which are not responsible for the 
copy of $C^4$. When partitioning in that way, the number of blocks of the partition 
obtained so far equals $\lvert(\upXul^{6,n,n})^{-1}\ref{item:comparisonproofcaseLhassize6:twoisolatedvertices}\rvert$. The reason for this is that to realize the 
type $\ref{item:comparisonproofcaseLhassize6:twoisolatedvertices}$ there is no 
choice for the values indexed by the positions which are not responsible for 
the $C^4$, both must be zero. In the enumeration we are currently carrying out, 
however, there \emph{is} still complete freedom left on how to choose any one of the 
$\lvert \{0,\pm\}^{[2]} \rvert = 3^2$ values which can be indexed by these two 
positions, in other words, each of the blocks has size $3^2$. Therefore 
$
\lvert(\upXul^{6,n,n})^{-1}\ref{item:comparisonproofcaseLhassize6:twoisolatedvertices}\rvert
+ \dotsm +
\lvert(\upXul^{6,n,n})^{-1}\ref{item:comparisonproofcaseLhassize6:C4intersectingatwopathinitsinnervertex}\rvert
=
3^2\cdot \lvert(\upXul^{6,n,n})^{-1}\ref{item:comparisonproofcaseLhassize6:twoisolatedvertices}\rvert$, which proves \ref{linearrelationbetweent4tot10:kequals6}. Equations  \ref{linearrelationbetweent11tot15:kequals6} and \ref{linearrelationbetweent16tot18:kequals6} are true for an entirely analogous reason. 
\end{proof}

We will now quantify the claims in Corollary \ref{cor:isomorphismtypesforwhichequalityofmeasuresofentryspecificationeventsfails} by determining 
$\lvert(\upXul^{k,n,n})^{-1}(\mathfrak{X})\rvert$ for each $k$ and each isomorphism 
type $\mathfrak{X}$ mentioned there. A few preparatory comments seem in order. The 
behaviour of $\lvert(\upXul^{k,n,n})^{-1}(\mathfrak{X})\rvert$ as a function of $k$ 
for a given isomorphism type $\mathfrak{X}$ is a little subtle. For example, note that 
Theorem \ref{thm:countingmatrixrealizations} tells us that 
\begin{equation}\label{eq:exampleforcounterintuitivepointinenumerationofrealizations}
\lvert(\upXul^{5,n,n})^{-1}\ref{item:comparisonproofcaseLhassize6:oneisolatedvertex}\rvert > \lvert(\upXul^{6,n,n})^{-1}\ref{item:comparisonproofcaseLhassize6:oneisolatedvertex}\rvert
\end{equation}
in spite of the fact that in the case of 
$\lvert(\upXul^{6,n,n})^{-1}\ref{item:comparisonproofcaseLhassize6:oneisolatedvertex}\rvert$ we have one matrix entry more at our disposal to realize \ref{item:comparisonproofcaseLhassize6:oneisolatedvertex}. The reason for this could be summarized thus: 
when wanting to keep the number of isolated vertices in $\upXul_{B}$ at one, the additional matrix entry curtails our freedom more than it 
adds to it---after having chosen a position for one of the 
non-matrix-circuit-entries which `hides' one of its two `shadows' in one 
of the four shadows of the matrix-circuit-entries, we then have to position  the 
second non-matrix-circuit-entry so as to hide \emph{both} of its two 
shadows in already existing shadows, and this determines it position completely. 
Moreover, since \ref{item:comparisonproofcaseLhassize6:oneisolatedvertex} is an 
isomorphism type in which there do not exist edges outside the $4$-circuit, the 
non-matrix-circuit positions must index the value $0$. The net result of these 
rigid requirements are (since in effect for $\lvert(\upXul^{6,n,n})^{-1}\ref{item:comparisonproofcaseLhassize6:oneisolatedvertex}\rvert$ we are counting the 
possible \emph{$2$-sets} of non-circuit positions while for $\lvert(\upXul^{5,n,n})^{-1}\ref{item:comparisonproofcaseLhassize6:oneisolatedvertex}\rvert$ we counted the 
possible $1$-sets of such positions) \emph{less} possibilities.  For other types it 
can happen that the mechanism just described is counterbalanced by the additional 
possibilities of indexing different values. This is the essential reason why 
$\lvert(\upXul^{5,n,n})^{-1}\ref{item:comparisonproofcaseLhassize6:oneadditionaledgeintersectingC4}\rvert
=
\lvert(\upXul^{6,n,n})^{-1}\ref{item:comparisonproofcaseLhassize6:oneadditionaledgeintersectingC4}\rvert$, despite \eqref{eq:exampleforcounterintuitivepointinenumerationofrealizations} and despite the fact that the set of all \emph{domains} in the 
preimages in question are the same as in 
\eqref{eq:exampleforcounterintuitivepointinenumerationofrealizations}, i.e. 

{
\scriptsize
\begin{equation}
\Dom((\upXul^{5,n,n})^{-1}\ref{item:comparisonproofcaseLhassize6:oneisolatedvertex}) 
 = \Dom((\upXul^{5,n,n})^{-1}\ref{item:comparisonproofcaseLhassize6:oneadditionaledgeintersectingC4}) \quad ,  \qquad 
\Dom((\upXul^{6,n,n})^{-1}\ref{item:comparisonproofcaseLhassize6:oneisolatedvertex}) 
 = \Dom((\upXul^{6,n,n})^{-1}\ref{item:comparisonproofcaseLhassize6:oneadditionaledgeintersectingC4}) \quad . 
\end{equation}
}

Since biadjacency matrices are quite a fundamental topic, it would be of interest to 
treat these phenomena in more generality. It seems advisable to do this with a 
view towards the theory of $\{0,1\}$-matrices with given row and column sums (for a 
start, cf. e.g. \cite{MR878703}, \cite{arXiv:1010.5706v1} and \cite{MR2600999}). 
However, so far the author could not harness the literature on this topic in any way 
which would lessen the burden of proving the following theorem: 

\begin{theorem}[cardinality of preimages of $\upXul^{k,n,n}$ on bipartite 
nonforests for $4\leq k \leq 6$]\label{thm:countingmatrixrealizations}
The claims \ref{item:failuresetofisomorphismtypesk4}---\ref{item:failuresetofisomorphismtypesk6} can be quantified as follows 
(with $\xi_n:=2^4\cdot \lvert \Cir(4,n)\rvert = 2^4\cdot\binom{n-1}{2}^2$), 
\begin{enumerate}[label={\rm(QFa\arabic{*})}, start=4]
\item\label{item:numberofrealizationsofnonforestswhenkequals4} For every $n\geq 3$, 
$\lvert(\upXul^{4,n,n})^{-1}\ref{item:comparisonproofcaseLhassize4:4circuit}\rvert = \xi_n$\quad .
\item\label{item:numberofrealizationsofnonforestswhenkequals5} For every $n\geq 3$, 

{\scriptsize
\begin{minipage}[b]{0.45\linewidth}
\begin{enumerate}[label={\rm(m5.t\arabic{*})}, start=2]
\item\label{numberofmatrixrealizationsoftype:oneisolatedvertex:kequals5} 
$\lvert(\upXul^{5,n,n})^{-1}\ref{item:comparisonproofcaseLhassize6:oneisolatedvertex}
\rvert = 4\cdot(n-3) \cdot \xi_n$
\item\label{numberofmatrixrealizationsoftype:oneadditionaledgeintersectingC4:kequals5} 
$\lvert(\upXul^{5,n,n})^{-1}\ref{item:comparisonproofcaseLhassize6:oneadditionaledgeintersectingC4}\rvert = 8\cdot(n-3) \cdot \xi_n$
\end{enumerate}
\end{minipage}
\begin{minipage}[b]{0.45\linewidth}
\begin{enumerate}[label={\rm(m5.t\arabic{*})}, start=5]
\item\label{numberofmatrixrealizationsoftype:twoisolatedvertices:kequals5}
$\lvert(\upXul^{5,n,n})^{-1}\ref{item:comparisonproofcaseLhassize6:twoisolatedvertices}\rvert  = 1\cdot (n-3)^2 \cdot \xi_n$
\end{enumerate}
\begin{enumerate}[label={\rm(m5.t\arabic{*})}, start=7]
\item\label{numberofmatrixrealizationsoftype:C4withoneadditionaldisjointedge:kequals5} 
$\lvert(\upXul^{5,n,n})^{-1}\ref{item:comparisonproofcaseLhassize6:C4withoneadditionaldisjointedge}\rvert = 2\cdot (n-3)^2 \cdot \xi_n$\quad .
\end{enumerate}
\end{minipage}
}
\item\label{item:numberofrealizationsofnonforestswhenkequals6} 
For every $n\geq 3$, 

{\scriptsize
\begin{minipage}[b]{0.5\linewidth}
\begin{enumerate}[label={\rm(m6.t\arabic{*})}, start=2]
\item\label{numberofmatrixrealizationsoftype:oneisolatedvertex} 
$\lvert(\upXul^{6,n,n})^{-1}\ref{item:comparisonproofcaseLhassize6:oneisolatedvertex}\rvert = 2 (n-3) \xi_n$
\item\label{numberofmatrixrealizationsoftype:oneadditionaledgeintersectingC4} 
$\lvert(\upXul^{6,n,n})^{-1}\ref{item:comparisonproofcaseLhassize6:oneadditionaledgeintersectingC4}\rvert =  8 (n-3)  \xi_n$
\item\label{numberofmatrixrealizationsoftype:XBLisomorphictoK23} 
$\lvert(\upXul^{6,n,n})^{-1}\ref{item:comparisonproofcaseLhassize6:XBLisomorphictoK23}\rvert =  2^7 \binom{n-1}{2}  \binom{n-1}{3}$
\item\label{numberofmatrixrealizationsoftype:twoisolatedvertices} 
$\lvert(\upXul^{6,n,n})^{-1}\ref{item:comparisonproofcaseLhassize6:twoisolatedvertices}\rvert =  (8 (n-3)^2 + 8 \binom{n-3}{2})  \xi_n$
\item\label{numberofmatrixrealizationsoftype:oneadditionaledgeintersectingC4andoneisolatedvertex} 
$\lvert(\upXul^{6,n,n})^{-1}\ref{item:comparisonproofcaseLhassize6:oneadditionaledgeintersectingC4andoneisolatedvertex}\rvert = (24 (n-3)^2 + 32 \binom{n-3}{2})  \xi_n$
\item\label{numberofmatrixrealizationsoftype:C4withoneadditionaldisjointedge} 
$\lvert(\upXul^{6,n,n})^{-1}\ref{item:comparisonproofcaseLhassize6:C4withoneadditionaldisjointedge}\rvert = 8 (n-3)^2  \xi_n$
\item\label{numberofmatrixrealizationsoftype:C4intersectingtwoedgesinseparatenonadjacentvertices} 
$\lvert(\upXul^{6,n,n})^{-1}\ref{item:comparisonproofcaseLhassize6:C4intersectingtwoedgesinseparatenonadjacentvertices}\rvert = 16 \binom{n-3}{2} \xi_n$
\item\label{numberofmatrixrealizationsoftype:C4intersectingtwoedgesinseparateadjacentvertices} 
$\lvert(\upXul^{6,n,n})^{-1}\ref{item:comparisonproofcaseLhassize6:C4intersectingtwoedgesinseparateadjacentvertices}\rvert = 16 (n-3)^2  \xi_n$
\item\label{numberofmatrixrealizationsoftype:C4intersectingatwopathinanendvertex} 
$\lvert(\upXul^{6,n,n})^{-1}\ref{item:comparisonproofcaseLhassize6:C4intersectingatwopathinanendvertex}\rvert = 16  (n-3)^2  \xi_n$
\item\label{numberofmatrixrealizationsoftype:C4intersectingatwopathinitsinnervertex} 
$\lvert(\upXul^{6,n,n})^{-1}\ref{item:comparisonproofcaseLhassize6:C4intersectingatwopathinitsinnervertex}\rvert = 16 \binom{n-3}{2}   \xi_n$
\end{enumerate}
\end{minipage}
\begin{minipage}[b]{0.5\linewidth}
\begin{enumerate}[label={\rm(m6.t\arabic{*})}, start=12]
\item\label{numberofmatrixrealizationsoftype:C6} 
$\lvert(\upXul^{6,n,n})^{-1}\ref{item:comparisonproofcaseLhassize6:C6}\rvert = 
2^6 \lvert \Cir(6,n) \rvert$ 
\item\label{numberofmatrixrealizationsoftype:threeisolatedvertices} 
$\lvert(\upXul^{6,n,n})^{-1}\ref{item:comparisonproofcaseLhassize6:threeisolatedvertices}\rvert = 10 (n-3)  \binom{n-3}{2}  \xi_n$
\item\label{numberofmatrixrealizationsoftype:oneadditionaledgeintersectingC4andtwoisolatedvertices} 
$\lvert(\upXul^{6,n,n})^{-1}\ref{item:comparisonproofcaseLhassize6:oneadditionaledgeintersectingC4andtwoisolatedvertices}\rvert = 16 (n-3) \binom{n-3}{2}  \xi_n$
\item\label{numberofmatrixrealizationsoftype:oneadditionaldisjointedgeandoneisolatedvertex} 
$\lvert(\upXul^{6,n,n})^{-1}\ref{item:comparisonproofcaseLhassize6:oneadditionaldisjointedgeandoneisolatedvertex}\rvert = 24 (n-3) \binom{n-3}{2}  \xi_n$
\item\label{numberofmatrixrealizationsoftype:twoadditionaledgesonlyoneofthemdisjoint} 
$\lvert(\upXul^{6,n,n})^{-1}\ref{item:comparisonproofcaseLhassize6:twoadditionaledgesonlyoneofthemdisjoint}\rvert = 32 (n-3)\binom{n-3}{2}  \xi_n$
\item\label{numberofmatrixrealizationsoftype:twoadditionalnondisjointedgesdisjointfromC4} 
$\lvert(\upXul^{6,n,n})^{-1}\ref{item:comparisonproofcaseLhassize6:twoadditionalnondisjointedgesdisjointfromC4}\rvert = 8 (n-3)\binom{n-3}{2}  \xi_n$
\item\label{numberofmatrixrealizationsoftype:fourisolatedvertices} 
$\lvert(\upXul^{6,n,n})^{-1}\ref{item:comparisonproofcaseLhassize6:fourisolatedvertices}\rvert = 2 \binom{n-3}{2} \binom{n-3}{2}  \xi_n$
\item\label{numberofmatrixrealizationsoftype:oneadditionaldisjointedgeandtwoisolatedvertices} 
 $\lvert(\upXul^{6,n,n})^{-1}\ref{item:comparisonproofcaseLhassize6:oneadditionaldisjointedgeandtwoisolatedvertices}\rvert = 8  \binom{n-3}{2} \binom{n-3}{2}  \xi_n$
\item\label{numberofmatrixrealizationsoftype:twoadditionaldisjointedgesdisjointfromC4} 
$\lvert(\upXul^{6,n,n})^{-1}\ref{item:comparisonproofcaseLhassize6:twoadditionaldisjointedgesdisjointfromC4}\rvert = 8  \binom{n-3}{2} \binom{n-3}{2}  \xi_n$\quad .
\end{enumerate}
\end{minipage}
}
\end{enumerate}
\end{theorem}
\begin{proof}[Proof of \ref{item:numberofrealizationsofnonforestswhenkequals4}]
We have $\upX_B \cong C^4$ if and only if $I$ is a matrix-$4$-circuit and 
$\Supp(B) = I$. By Lemma \ref{lem:numberofmatrixcircuitsofgivenlengthwithingivencartesianproduct} there exist $\binom{n-1}{2}^2$ possible matrix-$4$-circuits $I$ and 
for each of them there are $2^4$ possibilities for a $B\in \{0,\pm\}^I$ 
with $\Supp(B)=I$. 
\end{proof}

Let us now prepare for the rest of the proof of 
Theorem \ref{thm:countingmatrixrealizations} with some observations and definitions. 
Inspecting the isomorphism types in $\mathcal{F}^{\mathrm{G}}(6,n) 
\setminus \{ \ref{item:comparisonproofcaseLhassize6:XBLisomorphictoK23}, 
\ref{item:comparisonproofcaseLhassize6:C6} \}$ (the types \ref{item:comparisonproofcaseLhassize6:XBLisomorphictoK23} and \ref{item:comparisonproofcaseLhassize6:C6} are 
exceptions whose preimages are also exceptionally easy to count) we see that  
in each of them the graph contains exactly one $C^4$. We therefore know that for 
every $\mathfrak{X}\in \mathcal{F}^{\mathrm{G}}(6,n)$ (hence in particular for every 
$\mathfrak{X}\in \mathcal{F}^{\mathrm{G}}(5,n)$ since 
$\mathcal{F}^{\mathrm{G}}(5,n)\subseteq \mathcal{F}^{\mathrm{G}}(6,n)$ 
by \ref{item:failuresetofisomorphismtypesk6} in Corollary \ref{cor:isomorphismtypesforwhichequalityofmeasuresofentryspecificationeventsfails}), and for every 
$I\in \binom{[n-1]^2}{6}$ it is necessary that there exist a matrix-$4$-circuit 
$S\subseteq I$ with $B\mid_{S}\in \{\pm\}^S$. For this there are 
$2^4\cdot \lvert \Cir(4,n) \rvert$ possibilities. A priori it could be that the 
number of possibilities to realize an isomorphism type depends on the choice of 
this necessary $S\subseteq I$. However, since we will take this $S$ to be arbitrary 
in the proofs to follow, and since we 
will get results which do not depend on $S$, it follows as a byproduct that they 
are not, more precisely that for each $\mathfrak{X}\in  \mathcal{F}^{\mathrm{G}}(6,n) 
\setminus \{ \ref{item:comparisonproofcaseLhassize6:C6}, 
\ref{item:comparisonproofcaseLhassize6:XBLisomorphictoK23} \}$ 
the values of $\lvert(\upXul^{k,n,n})^{-1}(\mathfrak{X})\rvert$ are equal to the 
product of $2^4\cdot \lvert \Cir(4,n)\rvert$ and the number of possibilities to choose 
$B\mid_{I\setminus S}\in \{0,\pm\}^{I\setminus S}$ in such a way that 
$\upX_{B} \in \mathfrak{X}$. By determining the latter number for each of the 
isomorphism types, we will prove all of the formulas 
\ref{numberofmatrixrealizationsoftype:oneisolatedvertex:kequals5}---\ref{numberofmatrixrealizationsoftype:twoadditionaldisjointedgesdisjointfromC4}, except, as 
already mentioned, \ref{numberofmatrixrealizationsoftype:XBLisomorphictoK23} 
and \ref{numberofmatrixrealizationsoftype:C6}, which do not fit into the overall 
plan of the proof (in the case of 
\ref{numberofmatrixrealizationsoftype:XBLisomorphictoK23} we would be overcounting 
the number of realizations since $K^{2,3}$ contains three copies of $C^4$) but which 
are easy to count directly.

Let $\prec$ denote the lexicographic ordering on $[n-1]^2$. Throughout the proof, 
we use the following conventions: we consider $I\supseteq S\in \binom{[n-1]^2}{4}$ 
and $B\mid_{S}\in\{\pm\}^S$ to be arbitrary. 
We set  $\{a,b,c,d\} := S$, $a_1 := \upp_1(a)$, $a_2 := \upp_2(a)$ and analogously 
for $b_1$, $b_2$, $c_1$, $c_2$, $d_1$ and $d_2$. Since $\prec$ is a total 
order, we may assume $a\prec b \prec c \prec d$ which combined with the fact 
that $S$ is a matrix-$4$-circuit implies $a_1 = b_1$, $c_1=d_1$, $a_2 = c_2$ 
and $b_2 = d_2$. The cardinality of $I\setminus S$ depends on whether we are proving 
formulas from \ref{item:numberofrealizationsofnonforestswhenkequals5} or \ref{item:numberofrealizationsofnonforestswhenkequals6}.  In the former case we set 
$\{u\}:=I\setminus S$, in the latter $\{u,v\}:=I\setminus S$ with the 
assumption that $u\prec v$. Moreover, $u_1:=\upp_1(u)$, $u_2:=\upp_2(u)$, 
$v_1:=\upp_1(v)$ and $v_2:=\upp_2(v)$. Finally, let us use the abbreviation 
$\upp(S) := \upp_1(S)\sqcup\upp_2(S)$.

\begin{proof}[Proof of \ref{item:numberofrealizationsofnonforestswhenkequals5}]
As to \ref{numberofmatrixrealizationsoftype:oneisolatedvertex:kequals5}, we start 
by noting that it follows directly from Definition \ref{def:XBandecXB} that 
$\upX_{B} \in \ref{item:comparisonproofcaseLhassize6:oneisolatedvertex}$ 
if and only if $B[u]=0$ and 
\begin{equation}\label{eq:conditioninproofofnumberofrealizationsoftype:oneisolatedvertex:kequals5}
\lvert \{u_1\}\setminus \upp_1(S)\rvert + \lvert \{u_2\}\setminus \upp_2(S)\rvert = 1
\quad .
\end{equation}
We distinguish cases according to how \eqref{eq:conditioninproofofnumberofrealizationsoftype:oneisolatedvertex:kequals5} is satisfied. 
\begin{enumerate}[label={\rm(C.\ref{numberofmatrixrealizationsoftype:oneisolatedvertex:kequals5}.\arabic{*})}]
\item\label{numberofmatrixrealizationsoftype:oneisolatedvertex:kequals5:case:lvertu1rvertbackslashp1Sequals0} $\lvert \{u_1\}\setminus \upp_1(S)\rvert = 0$, 
i.e. $u_1\in\upp_1(S)$. Then \eqref{eq:conditioninproofofnumberofrealizationsoftype:oneisolatedvertex:kequals5} implies that $u_2\notin \upp_2(S)$. 
Since there are $2$ different $u_1$ with $u_1\in\upp_1(S)$ and for each of them there 
are  $((n-1)-2)=(n-3)$ different $u_2$ with $u_2\notin\upp_2(S)$ it follows that 
if \ref{numberofmatrixrealizationsoftype:oneisolatedvertex:kequals5:case:lvertu1rvertbackslashp1Sequals0}, then there are  $2(n-3)$ realizations of type \ref{item:comparisonproofcaseLhassize6:oneisolatedvertex} by $B[u]$. 
\item\label{numberofmatrixrealizationsoftype:oneisolatedvertex:kequals5:case:lvertu1rvertbackslashp1Sequals1} $\lvert \{u_1\}\setminus \upp_1(S)\rvert = 1$. This case is 
easily seen to be symmetric to \ref{numberofmatrixrealizationsoftype:oneisolatedvertex:kequals5:case:lvertu1rvertbackslashp1Sequals0} w.r.t. swapping the subscripts $1$ and 
$2$. Therefore, if \ref{numberofmatrixrealizationsoftype:oneisolatedvertex:kequals5:case:lvertu1rvertbackslashp1Sequals1}, then there are also  $2(n-3)$ realizations of 
type \ref{item:comparisonproofcaseLhassize6:oneisolatedvertex} by $B[u]$. 
\end{enumerate}
It follows that there are  $2(n-3)+2(n-3)=4(n-3)$ realizations of type 
\ref{item:comparisonproofcaseLhassize6:oneisolatedvertex} by $B[u]$, proving 
\ref{numberofmatrixrealizationsoftype:oneisolatedvertex:kequals5}. 
As to \ref{numberofmatrixrealizationsoftype:oneadditionaledgeintersectingC4:kequals5}, 
this follows from \ref{numberofmatrixrealizationsoftype:oneisolatedvertex:kequals5} 
and Lemma \ref{lem:relationsamongthenumbersofmatrixrealizations}.\ref{linearrelationbetweent2andt3:kequals5}. 

As to \ref{numberofmatrixrealizationsoftype:twoisolatedvertices:kequals5}, it follows 
from Definition \ref{def:XBandecXB} that $\upX_{B} \in \ref{item:comparisonproofcaseLhassize6:twoisolatedvertices}$ if and only if $B[u]=0$ and 
\begin{equation}\label{eq:conditioninproofofnumberofrealizationsoftype:twoisolatedvertices:kequals5}
\lvert \{u_1\}\setminus \upp_1(S)\rvert + \lvert\{u_2\}\setminus \upp_2(S)\rvert = 2
\quad .
\end{equation}
Property \eqref{eq:conditioninproofofnumberofrealizationsoftype:twoisolatedvertices:kequals5} is equivalent to $u_1\notin\upp_1(S)$ and $u_2\notin\upp_2(S)$, and there are 
obviously $((n-1)-2)^2=(n-3)^2$ different $u\in [n-1]^2$ satisfying this. Therefore, 
\ref{numberofmatrixrealizationsoftype:twoisolatedvertices:kequals5} is correct. 
As to \ref{numberofmatrixrealizationsoftype:C4withoneadditionaldisjointedge:kequals5}, this follows from \ref{numberofmatrixrealizationsoftype:twoisolatedvertices:kequals5}
and Lemma \ref{lem:relationsamongthenumbersofmatrixrealizations}.\ref{linearrelationbetweent5andt7:kequals5}. This completes the proof of \ref{item:numberofrealizationsofnonforestswhenkequals5}. 
\end{proof}

We now take on the task of proving \ref{item:numberofrealizationsofnonforestswhenkequals6}, which will take some effort. We prepare by proving four lemmas characterizing 
the realizations of the types \ref{item:comparisonproofcaseLhassize6:C4intersectingtwoedgesinseparatenonadjacentvertices}--\ref{item:comparisonproofcaseLhassize6:C4intersectingatwopathinitsinnervertex}. 

\begin{lemma}\label{lem:characterizationofisomorphismtype:C4intersectingtwoedgesinseparatenonadjacentvertices}
For every $B\in \{0,\pm\}^I$ with $I\in \binom{[n-1]^2}{6}$, 
$I = S \sqcup \{u,v\}$ and $\upX_{B\mid_{S}}\cong C^4$ we have $\upX_{B} \cong \ref{item:comparisonproofcaseLhassize6:C4intersectingtwoedgesinseparatenonadjacentvertices}$ if and only if 

{\scriptsize
\begin{minipage}[b]{0.35\linewidth}
\begin{enumerate}[label={\rm(P.\ref{item:comparisonproofcaseLhassize6:C4intersectingtwoedgesinseparatenonadjacentvertices}.\arabic{*})}]
\item\label{characterizationoftype:C4intersectingtwoedgesinseparatenonadjacentvertices:property1} 
$\upX_{\{0\}^{\{u,v\}} \sqcup B\mid_{S}} \cong 
\ref{item:comparisonproofcaseLhassize6:twoisolatedvertices}$, 
\item\label{characterizationoftype:C4intersectingtwoedgesinseparatenonadjacentvertices:property2} 
$B[u]\in\{\pm\}$ and $B[v]\in \{\pm\}$,
\end{enumerate}
\end{minipage}
\begin{minipage}[b]{0.65\linewidth}
\begin{enumerate}[label={\rm(P.\ref{item:comparisonproofcaseLhassize6:C4intersectingtwoedgesinseparatenonadjacentvertices}.\arabic{*})},start=3]
\item\label{characterizationoftype:C4intersectingtwoedgesinseparatenonadjacentvertices:property3} 
$\{u_1,u_2\}\cap\{v_1,v_2\} = \emptyset$, 
\item\label{characterizationoftype:C4intersectingtwoedgesinseparatenonadjacentvertices:property4} 
($u_1\in\upp_1(S)$ and $v_1\in\upp_1(S)$) 
or 
($u_2\in\upp_2(S)$ and $v_2\in\upp_2(S)$). 
\end{enumerate}
\end{minipage}
}
\end{lemma}
\begin{proof}
First suppose that $\upX_{B} \cong \ref{item:comparisonproofcaseLhassize6:C4intersectingtwoedgesinseparatenonadjacentvertices}$. Then Definition \ref{def:XBandecXB} 
implies that both \ref{characterizationoftype:C4intersectingtwoedgesinseparatenonadjacentvertices:property1} and \ref{characterizationoftype:C4intersectingtwoedgesinseparatenonadjacentvertices:property2} are true. To prove \ref{characterizationoftype:C4intersectingtwoedgesinseparatenonadjacentvertices:property3} and 
\ref{characterizationoftype:C4intersectingtwoedgesinseparatenonadjacentvertices:property4}, let $e\neq f\in \upE(\upX_{B})$ denote the two edges in $\upX_{B\mid_{\{u,v\}}}$, 
where $\{e\}$ $:=$ $\upE(\upX_{B[u]})$ $=$ $\bigl\{\{(u_1,n), (n,u_2)\}\bigr\}$ and 
$\{f\}$ $:=$ $\upE(\upX_{B[v]})$ $=$ $\bigl\{ \{(v_1,n), (n,v_2)\}\bigr\}$. By 
hypothesis, $e\cap f = \emptyset$ and this implies that \ref{characterizationoftype:C4intersectingtwoedgesinseparatenonadjacentvertices:property3} is true. Moreover, again 
by hypothesis, both $e$ and $f$ intersect $\upX_{B\mid_S}\cong C^4$ and the 
intersection set is \emph{not} an edge of it. 

If $u_1\in\upp_1(S)$, then there are still two possibilities for the intersection set 
$f\cap \upV(\upX_{B\mid_S})$, namely $f\cap \upV(\upX_{B\mid_S})=\{(v_1,n)\}$ 
(equivalently, $v_1\in\upp_1(S)$) or 
$f\cap \upV(\upX_{B\mid_S}) = \{(n,v_2)\}$ (equivalently, $v_2\in\upp_2(S)$). It follows from 
Definition \ref{def:XBandecXB} that the vertex in the intersection set 
$e\cap \upV(\upX_{B\mid_S}) = \{ (u_1,n)\}$ is \emph{not} adjacent to the vertex in $f\cap \upV(\upX_{B\mid_S})$ 
if and only if the first possibility is true, i.e. $f\cap \upV(\upX_{B\mid_S}) = \{(v_1,n)\}$, i.e. 
$v_1\in\upp_1(S)$. This proves that the first 
clause of \ref{characterizationoftype:C4intersectingtwoedgesinseparatenonadjacentvertices:property4}, 
and hence \ref{characterizationoftype:C4intersectingtwoedgesinseparatenonadjacentvertices:property4} itself, is true. 

If $u_2\in\upp_2(S)$, then an entirely analogous argument as the one in the preceding 
paragraph shows that the second clause of \ref{characterizationoftype:C4intersectingtwoedgesinseparatenonadjacentvertices:property4}, 
hence again \ref{characterizationoftype:C4intersectingtwoedgesinseparatenonadjacentvertices:property4} itself, is true. This proves 
that $\upX_{B} \cong \ref{item:comparisonproofcaseLhassize6:C4intersectingtwoedgesinseparatenonadjacentvertices}$ implies that 
\ref{characterizationoftype:C4intersectingtwoedgesinseparatenonadjacentvertices:property1}--\ref{characterizationoftype:C4intersectingtwoedgesinseparatenonadjacentvertices:property4} 
are true.  

Conversely, assume \ref{characterizationoftype:C4intersectingtwoedgesinseparatenonadjacentvertices:property1}--\ref{characterizationoftype:C4intersectingtwoedgesinseparatenonadjacentvertices:property4}. Then \ref{characterizationoftype:C4intersectingtwoedgesinseparatenonadjacentvertices:property2} implies 
that $f_1( \upX_{B}) = 6$ and \ref{characterizationoftype:C4intersectingtwoedgesinseparatenonadjacentvertices:property3} implies that the two edges in $\upE(\upX_{B\mid_{\{u,v\}}})$ do 
not intersect. Let $e$ and $f$ be defined as in the preceding proof of the other 
implication. It remains to show that 
$(e\cap \upV(\upX_{B\mid_{S}}))\cup (f\cap \upV(\upX_{B\mid_{S}})) \notin \upE(\upX_{B\mid_{S}})$. By 
definition of $e$, either 
$e\cap \upV(\upX_{B\mid_{S}}) = \{ (u_1,n)\}$ or $e\cap \upV(\upX_{B\mid_{S}}) = \{(n,u_2)\}$. 

In the former case we have $u_1\in\upp_1(S)$, hence the first clause 
of \ref{characterizationoftype:C4intersectingtwoedgesinseparatenonadjacentvertices:property4} implies $v_1\in\upp_1(S)$, hence $(v_1,n)\in f\cap \upV(\upX_{B})$ by 
definition of $f$, hence $f\cap \upV( \upX_{B}) = \{(v_1,n)\}$ since 
$f\cap \upV( \upX_{B})$ is a singleton by construction. In view of 
Definition \ref{def:XBandecXB} this implies that indeed $(e\cap \upV(\upX_{B\mid_{S}}))\cup (f\cap \upV(\upX_{B\mid_{S}})) = \{ (u_1,n) , (v_1,n)\}\notin \upE(\upX_{B\mid_{S}})$.

In the latter case we have $u_2\in\upp_2(S)$, hence the second clause 
of \ref{characterizationoftype:C4intersectingtwoedgesinseparatenonadjacentvertices:property4} implies $v_2\in\upp_2(S)$, hence $(n,v_2)\in f\cap \upV(\upX_{B})$ by 
definition of $f$, hence $f\cap \upV( \upX_{B}) = \{(n,v_2)\}$ since 
$f\cap \upV( \upX_{B})$ is a singleton by construction. In view of 
Definition \ref{def:XBandecXB} this implies that indeed $(e\cap \upV(\upX_{B\mid_{S}}))\cup (f\cap \upV(\upX_{B\mid_{S}})) = \{ (n,u_2) , (n,v_2)\}\notin \upE(\upX_{B\mid_{S}})$.
This completes the proof that \ref{characterizationoftype:C4intersectingtwoedgesinseparatenonadjacentvertices:property1}--\ref{characterizationoftype:C4intersectingtwoedgesinseparatenonadjacentvertices:property4} imply $\upX_{B} \cong \ref{item:comparisonproofcaseLhassize6:C4intersectingtwoedgesinseparatenonadjacentvertices}$.
\end{proof}

\begin{lemma}\label{lem:characterizationofisomorphismtype:C4intersectingtwoedgesinseparateadjacentvertices}
For every $B\in \{0,\pm\}^I$ with $I\in\binom{[n-1]^2}{6}$, 
$I = S \sqcup \{u,v\}$ and $\upX_{B\mid_{S}}\cong C^4$ we have $\upX_{B} \cong \ref{item:comparisonproofcaseLhassize6:C4intersectingtwoedgesinseparateadjacentvertices}$ 
if and only if 

{\scriptsize
\begin{minipage}[b]{0.35\linewidth}
\begin{enumerate}[label={\rm(P.\ref{item:comparisonproofcaseLhassize6:C4intersectingtwoedgesinseparateadjacentvertices}.\arabic{*})}]
\item\label{characterizationoftype:C4intersectingtwoedgesinseparateadjacentvertices:property1} 
$\upX_{\{0\}^{\{u,v\}} \sqcup B\mid_{S}} \cong 
\ref{item:comparisonproofcaseLhassize6:twoisolatedvertices}$, 
\item\label{characterizationoftype:C4intersectingtwoedgesinseparateadjacentvertices:property2} 
$B[u]\in\{\pm\}$ and $B[v]\in \{\pm\}$,
\end{enumerate}
\end{minipage}
\begin{minipage}[b]{0.65\linewidth}
\begin{enumerate}[label={\rm(P.\ref{item:comparisonproofcaseLhassize6:C4intersectingtwoedgesinseparateadjacentvertices}.\arabic{*})},start=3]
\item\label{characterizationoftype:C4intersectingtwoedgesinseparateadjacentvertices:property3} 
$\{u_1,u_2\}\cap\{v_1,v_2\} = \emptyset$, 
\item\label{characterizationoftype:C4intersectingtwoedgesinseparateadjacentvertices:property4} 
($u_1\in\upp_1(S)$ and $v_2\in\upp_2(S)$) 
or 
($u_2\in\upp_2(S)$ and $v_1\in\upp_1(S)$). 
\end{enumerate}
\end{minipage}
}
\end{lemma}
\begin{proof}
First suppose that $\upX_{B}\cong \ref{item:comparisonproofcaseLhassize6:C4intersectingtwoedgesinseparateadjacentvertices}$. Then Definition \ref{def:XBandecXB} implies 
that both \ref{characterizationoftype:C4intersectingtwoedgesinseparateadjacentvertices:property1} and \ref{characterizationoftype:C4intersectingtwoedgesinseparateadjacentvertices:property2} are true. 
To prove \ref{characterizationoftype:C4intersectingtwoedgesinseparateadjacentvertices:property3} and 
\ref{characterizationoftype:C4intersectingtwoedgesinseparateadjacentvertices:property4}, let $e\neq f\in \upE(\upX_{B})$ denote the two edges in $\upX_{B\mid_{\{u,v\}}}$, where 
$\{e\}$ $:=$ $\upE(\upX_{B[u]})$ $=$ $\bigl\{\{(u_1,n), (n,u_2)\}\bigr\}$ and 
$\{f\}$ $:=$ $\upE(\upX_{B[v]})$ $=$ $\bigl\{ \{(v_1,n), (n,v_2)\}\bigr\}$. 
By hypothesis $e\cap f = \emptyset$ and this implies that \ref{characterizationoftype:C4intersectingtwoedgesinseparateadjacentvertices:property3} is true. Moreover, again 
by hypothesis, both $e$ and $f$ intersect $\upX_{B\mid_S}$ $\cong$ $C^4$ and the 
intersection set is an edge of it. 

If $u_1\in \upp_1(S)$, then there are still two possibilities for the intersection 
set $f\cap \upV(\upX_{B\mid_S})$, namely $f\cap \upV(\upX_{B\mid_S}) = \{(v_1,n)\}$ 
(equivalently, $v_1\in \upp_1(S)$) or $f\cap \upV(\upX_{B\mid_S}) = \{(n,v_2)\}$ 
(equivalently, $v_2\in \upp_2(S)$). It is evident from Definition \ref{def:XBandecXB} 
that the vertex in the intersection set $e\cap \upV(\upX_{B\mid_S}) = \{(u_1,n)\}$ is 
adjacent to the vertex in $f\cap \upV(\upX_{B\mid_S})$ if and only if the second 
possibility is true, i.e. $f\cap \upV(\upX_{B\mid_S}) = \{(n,v_2)\}$, 
i.e. $v_2\in\upp_2(S)$. This proves that the first clause of \ref{characterizationoftype:C4intersectingtwoedgesinseparateadjacentvertices:property4}, and hence 
\ref{characterizationoftype:C4intersectingtwoedgesinseparateadjacentvertices:property4} itself, is true. 

If $u_2\in \upp_2(S)$, an entirely analogous argument as the one in the 
preceding paragraph shows that then the second clause 
of \ref{characterizationoftype:C4intersectingtwoedgesinseparateadjacentvertices:property4}, hence again \ref{characterizationoftype:C4intersectingtwoedgesinseparateadjacentvertices:property4} itself is true. This completes the proof that $\upX_{B} \cong \ref{item:comparisonproofcaseLhassize6:C4intersectingtwoedgesinseparateadjacentvertices}$ 
implies properties \ref{characterizationoftype:C4intersectingtwoedgesinseparateadjacentvertices:property1}--\ref{characterizationoftype:C4intersectingtwoedgesinseparateadjacentvertices:property4}. 

Conversely, suppose that \ref{characterizationoftype:C4intersectingtwoedgesinseparateadjacentvertices:property1}--\ref{characterizationoftype:C4intersectingtwoedgesinseparateadjacentvertices:property4} are true. Then \ref{characterizationoftype:C4intersectingtwoedgesinseparateadjacentvertices:property2} implies 
that $f_1( \upX_{B} ) = 6$ and \ref{characterizationoftype:C4intersectingtwoedgesinseparateadjacentvertices:property3} implies that the two edges in 
$\upE(\upX_{B\mid_{\{u,v\}}})$ do not intersect. Let $e$ and $f$ be defined as in the 
preceding proof of the other implication. It remains to show that 
$(e\cap \upV(\upX_{B\mid_{S}})) \cup (f\cap \upV(\upX_{B\mid_{S}})) \in \upE(\upX_{B\mid_{S}})$. By definition of $e$, either $e\cap \upV(\upX_{B\mid_{S}}) = \{(u_1,n)\}$ or 
$e\cap \upV(\upX_{B\mid_{S}}) = \{(n,u_2)\}$. 

In the former case we have $u_1\in \upp_1(S)$, hence 
the first clause of \ref{characterizationoftype:C4intersectingtwoedgesinseparateadjacentvertices:property4} implies that $v_2\in \upp_2(S)$, hence  
$(n,v_2) \in  f\cap \upV(\upX_{B\mid_{S}})$ by definition of $f$, hence 
$f\cap \upV(\upX_{B\mid_{S}}) = \{ (n,v_2) \}$ since $f\cap \upV(\upX_{B\mid_{S}})$ is a 
singleton by construction. In view of Definition \ref{def:XBandecXB} this implies 
that indeed $(e\cap \upV(\upX_{B\mid_{S}})) \cup (f\cap \upV(\upX_{B\mid_{S}})) = 
\{ (u_1,n), (n,v_2)\} \in \upE(\upX_{B\mid_{S}})$. 

In the latter case we have $u_2\in \upp_2(S)$, hence the second clause of 
\ref{characterizationoftype:C4intersectingtwoedgesinseparateadjacentvertices:property4} 
implies that $v_1\in \upp_1(S)$, hence  $(v_1,n) \in  f\cap \upV(\upX_{B\mid_{S}})$ by 
definition of $f$, hence $f\cap \upV(\upX_{B\mid_{S}}) = \{ (v_1,n) \}$ since 
$f\cap \upV(\upX_{B\mid_{S}})$ is a singleton by construction. In view of 
Definition \ref{def:XBandecXB} this implies that indeed 
$(e\cap \upV(\upX_{B\mid_{S}})) \cup (f\cap \upV(\upX_{B\mid_{S}})) = \{ (n,u_2), (v_1,n)\} \in \upE(\upX_{B\mid_{S}})$. This completes the proof that \ref{characterizationoftype:C4intersectingtwoedgesinseparateadjacentvertices:property1}--\ref{characterizationoftype:C4intersectingtwoedgesinseparateadjacentvertices:property4} imply $\upX_{B} \cong \ref{item:comparisonproofcaseLhassize6:C4intersectingtwoedgesinseparateadjacentvertices}$.
\end{proof}

\begin{lemma}\label{lem:characterizationofisomorphismtype:C4intersectingatwopathinanendvertex}
For every $B\in \{0,\pm\}^I$ with $I\in\binom{[n-1]^2}{6}$, 
$I = S \sqcup \{u,v\}$ and $\upX_{B\mid_{S}}\cong C^4$ we have $\upX_{B} \cong \ref{item:comparisonproofcaseLhassize6:C4intersectingatwopathinanendvertex}$ if and only if

{\scriptsize
\begin{minipage}[b]{0.35\linewidth}
\begin{enumerate}[label={\rm(P.\ref{item:comparisonproofcaseLhassize6:C4intersectingatwopathinanendvertex}.\arabic{*})}]
\item\label{characterizationoftype:C4intersectingatwopathinanendvertex:property1}  
$\upX_{\{0\}^{\{u,v\}} \sqcup B\mid_{S}} \cong 
\ref{item:comparisonproofcaseLhassize6:twoisolatedvertices}$, 
\item\label{characterizationoftype:C4intersectingatwopathinanendvertex:property2}  
$B[u]\in\{\pm\}$ and $B[v]\in \{\pm\}$,
\end{enumerate}
\end{minipage}
\begin{minipage}[b]{0.65\linewidth}
\begin{enumerate}[label={\rm(P.\ref{item:comparisonproofcaseLhassize6:C4intersectingatwopathinanendvertex}.\arabic{*})},start=3]
\item\label{characterizationoftype:C4intersectingatwopathinanendvertex:property3}  
$\{u_1,u_2\}\cap\{v_1,v_2\} \neq \emptyset$, 
\item\label{characterizationoftype:C4intersectingatwopathinanendvertex:property4}  
$\{u_1,u_2\}\cap \upp(S) = \emptyset$ 
or $\{v_1,v_2\}\cap \upp(S) = \emptyset$. 
\end{enumerate}
\end{minipage}
}
\end{lemma}
\begin{proof} 
First suppose that $\upX_{B}\cong \ref{item:comparisonproofcaseLhassize6:C4intersectingatwopathinanendvertex}$. Then Definition \ref{def:XBandecXB} implies that both \ref{characterizationoftype:C4intersectingatwopathinanendvertex:property1} and \ref{characterizationoftype:C4intersectingatwopathinanendvertex:property2} are true. To prove 
\ref{characterizationoftype:C4intersectingatwopathinanendvertex:property3} and \ref{characterizationoftype:C4intersectingatwopathinanendvertex:property4}, 
let $e\in \upE(\upX_{B})$ denote the unique edge which does not intersect 
$\upX_{B\mid_S}\cong C^4$, and let $f\in \upE(\upX_{B})$ denote the unique edge 
which intersects $\upX_{B\mid_S}\cong C^4$. 
In view of Definition \ref{def:XBandecXB}, 
{\scriptsize
\begin{equation}\label{eq:propertyinthepreparationofanalysisoftypet14}
\text{($e = \{ (u_1,n), (n,u_2) \} $ and $f=\{(v_1,n),(n,v_2)\}$) 
or 
($e = \{(v_1,n), (n,v_2)\}$ and $f=\{(u_1,n),(n,u_2)\}$).
}
\end{equation}
}
By definition of $e$ and $f$ we have $e\cap f\neq \emptyset$, hence whatever of the 
two clauses of \eqref{eq:propertyinthepreparationofanalysisoftypet14} is true, either 
$u_1=v_1$ or $u_2=v_2$. Therefore property \ref{characterizationoftype:C4intersectingatwopathinanendvertex:property3} 
is true. 
By definition of $e$, if $e = \{ (u_1,n), (n,u_2)\}$, then $u_1\notin\upp_1(S)$ 
and $u_2\notin\upp_2(S)$, hence the first clause 
of \ref{characterizationoftype:C4intersectingatwopathinanendvertex:property4} is true, and if $e=\{(v_1,n),(n,v_2)\}$, 
then $v_1\notin \upp_1(S)$ and $v_2\notin\upp_2(S)$, hence then the second clause of 
\ref{characterizationoftype:C4intersectingatwopathinanendvertex:property4} is true. 

Conversely, suppose that properties \ref{characterizationoftype:C4intersectingatwopathinanendvertex:property1}--\ref{characterizationoftype:C4intersectingatwopathinanendvertex:property4} are true. 
Then \ref{characterizationoftype:C4intersectingatwopathinanendvertex:property2}
implies $f_1( \upX_{B}) = 6$ and \ref{characterizationoftype:C4intersectingatwopathinanendvertex:property3} 
implies that the two edges corresponding to $u\neq v$ intersect. 
It remains to prove that the $2$-path consisting of these edges intersects 
$\upX_{B\mid_S} \cong C^4$ with one of its endvertices. To see this, note first 
that \ref{characterizationoftype:C4intersectingatwopathinanendvertex:property1} implies \eqref{eq:proofofnumberofeventsinducing:item:comparisonproofcaseLhassize6:twoisolatedvertices} which combined with $u\neq v$ implies that  
\begin{equation}\label{eq:preparationofanalysisoftypet14:the2pathintersectstheC4}
\{u_1,u_2,v_1,v_2\}\cap \upp(S) \neq \emptyset\quad . 
\end{equation}
Moreover, we know from \eqref{eq:propertyinthepreparationofanalysisoftypet14} together 
with $e\cap f\neq \emptyset$ that $u_1=v_1$ or $u_2=v_2$. 

If $u_1=v_1$, then \eqref{eq:preparationofanalysisoftypet14:the2pathintersectstheC4} 
cannot be true by virtue of $u_1=v_1\in \upp_1(S)$ since then \ref{characterizationoftype:C4intersectingatwopathinanendvertex:property4} would become false. Therefore, if $u_1=v_1$, then 
$u_1=v_1\notin\upp_1(S)$, and \eqref{eq:preparationofanalysisoftypet14:the2pathintersectstheC4} implies  $\{u_2,v_2\}\cap\upp_2(S) \neq \emptyset$. It is impossible 
that $\{u_2,v_2\}\subseteq \upp_2(S)$ for this combined with $u_1=v_1$ would imply 
$\lvert \{u_1,v_1\}\setminus\upp_1(S)\rvert 
+ \lvert \{u_2,v_2\}\setminus\upp_2(S)\rvert  = 1$, hence 
contradict \eqref{eq:proofofnumberofeventsinducing:item:comparisonproofcaseLhassize6:twoisolatedvertices}. Therefore, $\lvert \{u_2,v_2\}\cap\upp_2(S)\rvert = 1$, and 
since $u_1=v_1$ implies $u_2\neq v_2$, this is what we wanted to prove: exactly one 
of the two endvertices $(n,u_2), (n,v_2)\in \upV(\upX_{B})$ intersects 
$\upX_{B\mid_S}\cong C^4$. 

If $u_2=v_2$, then an entirely analogous argument as in the preceding paragraph 
shows that exactly one of the two endvertices $(u_1,n), (v_1,n)\in \upV(\upX_{B})$ 
intersects the $\upX_{B\mid_S}\cong C^4$.

The proof that properties \ref{characterizationoftype:C4intersectingatwopathinanendvertex:property1}--\ref{characterizationoftype:C4intersectingatwopathinanendvertex:property4} imply $\upX_{B} \cong \ref{item:comparisonproofcaseLhassize6:C4intersectingatwopathinanendvertex}$ is now complete. 
\end{proof}

\begin{lemma}\label{lem:characterizationofisomorphismtype:C4intersectingatwopathinitsinnervertex}
For every $B\in \{0,\pm\}^I$ with $I\in \binom{[n-1]^2}{6}$, 
$I = S \sqcup \{u,v\}$ and $\upX_{B\mid_{S}}\cong C^4$ we have 
$\upX_{B} \cong \ref{item:comparisonproofcaseLhassize6:C4intersectingatwopathinitsinnervertex}$ if and only if 

{\scriptsize
\begin{minipage}[b]{0.35\linewidth}
\begin{enumerate}[label={\rm(P.\ref{item:comparisonproofcaseLhassize6:C4intersectingatwopathinitsinnervertex}.\arabic{*})}]
\item\label{characterizationoftype:C4intersectingatwopathinitsinnervertex:property1}  
$\upX_{\{0\}^{\{u,v\}} \sqcup B\mid_{S}} \cong 
\ref{item:comparisonproofcaseLhassize6:twoisolatedvertices}$,  
\item\label{characterizationoftype:C4intersectingatwopathinitsinnervertex:property2} 
$B[u]\in\{\pm\}$ and $B[v]\in \{\pm\}$,
\end{enumerate}
\end{minipage}
\begin{minipage}[b]{0.65\linewidth}
\begin{enumerate}[label={\rm(P.\ref{item:comparisonproofcaseLhassize6:C4intersectingatwopathinitsinnervertex}.\arabic{*})}, start=3]
\item\label{characterizationoftype:C4intersectingatwopathinitsinnervertex:property3} 
$\{u_1,u_2\}\cap\{v_1,v_2\} \neq \emptyset$, 
\item\label{characterizationoftype:C4intersectingatwopathinitsinnervertex:property4} 
$\{u_1,u_2\}\cap \upp(S)\neq\emptyset$ and $\{v_1,v_2\}\cap \upp(S) \neq \emptyset$.
\end{enumerate}
\end{minipage}
}
\end{lemma}
\begin{proof}
First suppose that $\upX_{B}\cong \ref{item:comparisonproofcaseLhassize6:C4intersectingatwopathinitsinnervertex}$. Then Definition \ref{def:XBandecXB} implies that 
both \ref{characterizationoftype:C4intersectingatwopathinitsinnervertex:property1} 
and \ref{characterizationoftype:C4intersectingatwopathinitsinnervertex:property2} 
are true. To prove 
\ref{characterizationoftype:C4intersectingatwopathinitsinnervertex:property3} and 
\ref{characterizationoftype:C4intersectingatwopathinitsinnervertex:property4}, let 
$e\neq f\in \upE(\upX_{B})$ denote the two edges in $\upE(\upX_{B})$ forming 
the $2$-path which intersects $\upX_{B\mid_S}\cong C^4$ with its inner vertex. 
As in the proof of Lemma \ref{lem:characterizationofisomorphismtype:C4intersectingatwopathinanendvertex}, we know that \eqref{eq:propertyinthepreparationofanalysisoftypet14} is true. By definition of $e$ and $f$ we have $e\cap f\neq \emptyset$, hence 
whatever of the two clauses of \eqref{eq:propertyinthepreparationofanalysisoftypet14} is true, either $u_1=v_1$ or $u_2=v_2$. Therefore property 
\ref{characterizationoftype:C4intersectingatwopathinitsinnervertex:property3} is 
true. By definition of $e$ and $f$, both $e$ and $f$ intersect 
$\upX_{B\mid_S}$ $\cong$ $C^4$. If the first clause 
in \eqref{eq:propertyinthepreparationofanalysisoftypet14} is true then 
$e$ intersecting $\upX_{B\mid_S}\cong C^4$ is equivalent 
to ($u_1\in \upp_1(S)$ or $u_2\in \upp_2(S)$) 
and $f$ intersecting $\upX_{B\mid_S}\cong C^4$ is equivalent 
to ($v_1\in \upp_1(S)$ or $v_2\in \upp_2(S)$). Then \ref{characterizationoftype:C4intersectingatwopathinitsinnervertex:property4} is indeed true. If the second clause 
in \eqref{eq:propertyinthepreparationofanalysisoftypet14} is true, 
interchanging `$u$' and `$v$' in the preceding sentence shows that then 
\ref{characterizationoftype:C4intersectingatwopathinitsinnervertex:property4} is true as well. This completes the proof that $\upX_{B}\cong \ref{item:comparisonproofcaseLhassize6:C4intersectingatwopathinitsinnervertex}$ implies 
\ref{characterizationoftype:C4intersectingatwopathinitsinnervertex:property1}--\ref{characterizationoftype:C4intersectingatwopathinitsinnervertex:property4}.

Conversely, suppose that properties \ref{characterizationoftype:C4intersectingatwopathinitsinnervertex:property1}--\ref{characterizationoftype:C4intersectingatwopathinitsinnervertex:property4} are true. 
Then \ref{characterizationoftype:C4intersectingatwopathinitsinnervertex:property2}
implies $f_1( \upX_{B}) = 6$ and \ref{characterizationoftype:C4intersectingatwopathinitsinnervertex:property3} implies that the two edges corresponding to $u\neq v$ 
intersect. It remains to prove that the $2$-path consisting of these edges 
intersects $\upX_{B\mid_S}\cong C^4$ with its inner vertex. Similar to the proof of 
Lemma \ref{lem:characterizationofisomorphismtype:C4intersectingatwopathinanendvertex} 
we know that \eqref{eq:preparationofanalysisoftypet14:the2pathintersectstheC4} and 
that $u_1=v_1$ or $u_2=v_2$. 

If $u_1=v_1$, then $u_2\in \upp_2(S)$ is impossible since this together 
with $u\neq v$ would imply $u_1=v_1\notin \upp_1(S)$ which due to the second 
clause of \ref{characterizationoftype:C4intersectingatwopathinitsinnervertex:property4} would imply $v_2\in\upp_2(S)$; but 
$u_2\in\upp_2(S)$, $u_1=v_1$ and $v_2\in \upp_2(S)$ combined imply 
$\lvert \{u_1,v_1\}\setminus\upp_1(S)\rvert + \lvert \{u_2,v_2\}\setminus \upp_2(S)\rvert =1$, a contradiction to \eqref{eq:proofofnumberofeventsinducing:item:comparisonproofcaseLhassize6:twoisolatedvertices}. For an entirely analogous reason 
$v_2\in \upp_2(S)$ is impossible, too. Since both $u_2\notin\upp_2(S)$ 
and $v_2\notin\upp_2(S)$, it follows from 
\eqref{eq:preparationofanalysisoftypet14:the2pathintersectstheC4} that 
$u_1=v_1\in \upp_1(S)$. This is what we wanted to prove: the common 
vertex $(u_1,n)=(v_1,n)$ of $e$ and $f$ (i.e. the inner vertex 
of the $2$-path formed by $e$ and $f$) is also the unique vertex of intersection 
with $\upX_{B\mid_S}\cong C^4$.

If $u_2=v_2$, then  an entirely analogous argumentation as in the preceding 
paragraph shows that the common vertex $(n,u_2)=(n,v_2)$ of $e$ and $f$ (i.e. the 
inner vertex of the $2$-path which is formed by $e$ and $f$) is also the unique 
vertex of intersection with $\upX_{B\mid_S}\cong C^4$.

The proof that properties \ref{characterizationoftype:C4intersectingatwopathinitsinnervertex:property1}--\ref{characterizationoftype:C4intersectingatwopathinitsinnervertex:property4} imply $\upX_{B} \cong \ref{item:comparisonproofcaseLhassize6:C4intersectingatwopathinitsinnervertex}$ is now complete. 
\end{proof}

\begin{proof}[Proof of \ref{item:numberofrealizationsofnonforestswhenkequals6}]
As to \ref{numberofmatrixrealizationsoftype:oneisolatedvertex}, it follows 
from Definition \ref{def:XBandecXB} that 
$\upX_{B} \cong \ref{item:comparisonproofcaseLhassize6:oneisolatedvertex}$ 
if and only if 
\begin{equation}\label{eq:proofofnumberofeventsinducing:item:comparisonproofcaseLhassize6:oneisolatedvertex}
\lvert \{u_1,v_1\}\setminus\upp_1(S) \rvert + 
\lvert\{ u_2,v_2 \} \setminus \upp_2(S) \rvert = 1 \quad .
\end{equation}

\begin{enumerate}[label={\rm(C.\ref{numberofmatrixrealizationsoftype:oneisolatedvertex}.\arabic{*})}]
\item\label{numberofmatrixrealizationsoftype:oneisolatedvertex:kequals6:case:lvertu1v1rvertbackslashp1Sequals0}  $\lvert\{u_1,v_1\}\setminus\upp_1(S)\rvert = 0$. 
Then \eqref{eq:proofofnumberofeventsinducing:item:comparisonproofcaseLhassize6:oneisolatedvertex} implies that $\lvert\{u_2,v_2\}\setminus\upp_2(S)\rvert = 1$ which is 
equivalent to \eqref{eq:proofofcomparativecountingtheoremcase3secondeq}. The 
property defining Case 1 is equivalent to $\{u_1,v_1\}\subseteq\upp_1(S)$. Property 
\eqref{eq:proofofcomparativecountingtheoremcase3secondeq} implies two cases: 
\begin{enumerate}[label={\rm(\arabic{*})}]
\item\label{numberofmatrixrealizationsoftype:twoisolatedvertices:kequals6:case:lvertu1v1rvertbackslashp1Sequals0:case:u2equalsv2} $u_2=v_2$ and 
$\{u_2,v_2\}\cap\upp_2(S)=\emptyset$. Then $u_2=v_2$ and $u\prec v$ imply that 
$u_1 < v_1$. This together with $\{u_1,v_1\} \subseteq\upp_1(S)$ implies 
$u_1=a_1=c_1$ and $v_1=b_1=d_1$. Therefore, it is $u_2=v_2$ alone which determines 
the two pairs $u$ and $v$. The property $u_2=v_2$ and 
$\{u_2,v_2\}\cap\upp_2(S)=\emptyset$ is equivalent to $u_2=v_2\notin \upp_2(S)$. It 
follows that if \ref{numberofmatrixrealizationsoftype:oneisolatedvertex:kequals6:case:lvertu1v1rvertbackslashp1Sequals0}.\ref{numberofmatrixrealizationsoftype:twoisolatedvertices:kequals6:case:lvertu1v1rvertbackslashp1Sequals0:case:u2equalsv2}, then there 
are $(n-1)-2$ realizations of type \ref{item:comparisonproofcaseLhassize6:oneisolatedvertex} by $u$ and $v$. 
\item\label{numberofmatrixrealizationsoftype:twoisolatedvertices:kequals6:case:lvertu1v1rvertbackslashp1Sequals0:case:u2notequaltov2} $u_2\neq v_2$ and 
$\lvert \{u_2,v_2\}\cap\upp_2(S) \rvert = 1$. Then either $u_2\in \upp_2(S)$ or 
$v_2\in \upp_2(S)$. If $u_2\in \upp_2(S)$, then because of 
$\{u_1,v_1\}\subseteq\upp_1(S)$ it follows that $u\in \{a,b,c,d\}$, a contradiction 
to $I\setminus S = \{ u, v\}$. Similarly, if $v_2\in \upp_2(S)$, then the same 
contradiction arises with regard to $v$. Therefore, the case \ref{numberofmatrixrealizationsoftype:oneisolatedvertex:kequals6:case:lvertu1v1rvertbackslashp1Sequals0}.\ref{numberofmatrixrealizationsoftype:twoisolatedvertices:kequals6:case:lvertu1v1rvertbackslashp1Sequals0:case:u2notequaltov2} cannot occur. 
\end{enumerate}
It follows that if \ref{numberofmatrixrealizationsoftype:oneisolatedvertex:kequals6:case:lvertu1v1rvertbackslashp1Sequals0}, then there are exactly $(n-1)-2$ realizations 
of type \ref{item:comparisonproofcaseLhassize6:oneisolatedvertex} by $B\mid_{\{u,v\}}$.

\item\label{numberofmatrixrealizationsoftype:oneisolatedvertex:kequals6:case:lvertu1v1rvertbackslashp1Sequals1} $\lvert\{u_1,v_1\}\setminus\upp_1(S)\rvert = 1$. Then 
\eqref{eq:proofofnumberofeventsinducing:item:comparisonproofcaseLhassize6:oneisolatedvertex} implies $\lvert\{u_2,v_2\}\setminus\upp_2(S)\rvert = 0$ which is 
equivalent to $\{u_2,v_2\} \subseteq\upp_2(S)$. The property defining 
\ref{numberofmatrixrealizationsoftype:oneisolatedvertex:kequals6:case:lvertu1v1rvertbackslashp1Sequals1} is equivalent 
to \eqref{eq:proofofcomparativecountingtheoremcase2secondeq}. Now an argument 
entirely analogous to the one given for \ref{numberofmatrixrealizationsoftype:oneisolatedvertex:kequals6:case:lvertu1v1rvertbackslashp1Sequals0} shows that if 
\ref{numberofmatrixrealizationsoftype:oneisolatedvertex:kequals6:case:lvertu1v1rvertbackslashp1Sequals1}, then there are exactly $(n-1)-2$ realizations of type 
\ref{item:comparisonproofcaseLhassize6:oneisolatedvertex} by $B\mid_{\{u,v\}}$.
\end{enumerate}
It follows that there are exactly $2\cdot ((n-1)-2)$ different 
$I\setminus S = \{u,v\}$ with $\upX_{B} \cong \ref{item:comparisonproofcaseLhassize6:oneisolatedvertex}$. This completes the proof of \ref{numberofmatrixrealizationsoftype:oneisolatedvertex}. 

As to  \ref{numberofmatrixrealizationsoftype:oneadditionaledgeintersectingC4}, 
notice that a necessary condition for type \ref{item:comparisonproofcaseLhassize6:oneadditionaledgeintersectingC4} is that 
$\lvert \upV(\upX_{B})\setminus \upV(\upX_{B\mid_{S}}) \rvert = 1$. Therefore 
the set of all suitable $I\in\binom{[n-1]^2}{6}$ is a subset (possibly nonproper) of 
those which are suitable for 
type \ref{item:comparisonproofcaseLhassize6:oneisolatedvertex}. We may therefore 
reexamine the analysis carried out for 
\ref{numberofmatrixrealizationsoftype:oneisolatedvertex} and in each of the 
cases count the number of $B\in\{0,\pm\}^I$ with $\upX_{B}\cong \ref{item:comparisonproofcaseLhassize6:oneadditionaledgeintersectingC4}$. 

If \ref{numberofmatrixrealizationsoftype:oneisolatedvertex:kequals6:case:lvertu1v1rvertbackslashp1Sequals0}.\ref{numberofmatrixrealizationsoftype:twoisolatedvertices:kequals6:case:lvertu1v1rvertbackslashp1Sequals0:case:u2equalsv2}, then properties 
$u_1=a_1=c_1$ and $v_1=b_1=d_1$ show that both $u$ and $v$ have the property that if 
one of them indexes a nonzero value of $B$, then there is an edge intersecting 
$\upX_{B\mid_S}\cong C^4$. Since otherwise we would have $K^{2,3}$, \emph{exactly} one 
of them must be nonzero. This implies exactly $4$ possibilities to realize type 
\ref{item:comparisonproofcaseLhassize6:oneadditionaledgeintersectingC4} for 
each of the $(n-1)-2$ realizations of type 
\ref{item:comparisonproofcaseLhassize6:oneisolatedvertex} which were offered 
in \ref{numberofmatrixrealizationsoftype:oneisolatedvertex:kequals6:case:lvertu1v1rvertbackslashp1Sequals0}.\ref{numberofmatrixrealizationsoftype:twoisolatedvertices:kequals6:case:lvertu1v1rvertbackslashp1Sequals0:case:u2equalsv2}. Therefore, if 
\ref{numberofmatrixrealizationsoftype:oneisolatedvertex:kequals6:case:lvertu1v1rvertbackslashp1Sequals0}.\ref{numberofmatrixrealizationsoftype:twoisolatedvertices:kequals6:case:lvertu1v1rvertbackslashp1Sequals0:case:u2equalsv2}, then there are exactly 
$4\cdot ((n-1)-2)$ realizations of type \ref{item:comparisonproofcaseLhassize6:oneadditionaledgeintersectingC4} by $B\mid_{I\setminus S}$. Since the case \ref{numberofmatrixrealizationsoftype:oneisolatedvertex:kequals6:case:lvertu1v1rvertbackslashp1Sequals0}.\ref{numberofmatrixrealizationsoftype:twoisolatedvertices:kequals6:case:lvertu1v1rvertbackslashp1Sequals0:case:u2notequaltov2} is as impossible now as it was back then, it 
follows that this is also the number of realizations for the entire \ref{numberofmatrixrealizationsoftype:oneisolatedvertex:kequals6:case:lvertu1v1rvertbackslashp1Sequals0}. 

If \ref{numberofmatrixrealizationsoftype:oneisolatedvertex:kequals6:case:lvertu1v1rvertbackslashp1Sequals1}, then by interchanging the subscripts $1$ 
and $2$ we may use the same analysis as for the case \ref{numberofmatrixrealizationsoftype:oneisolatedvertex:kequals6:case:lvertu1v1rvertbackslashp1Sequals0} to reach the 
conclusion that there are exactly $4\cdot ((n-1)-2)$ realizations of type 
\ref{item:comparisonproofcaseLhassize6:oneadditionaledgeintersectingC4} by 
$B\mid_{\{u,v\}} = B\mid_{I\setminus S}$.  

It follows that there are are exactly $4\cdot ((n-1)-2) + 4\cdot ((n-1)-2) = 
8\cdot (n-3)$ realizations of type \ref{item:comparisonproofcaseLhassize6:oneadditionaledgeintersectingC4} by $B\mid_{I\setminus S}$. This proves \ref{numberofmatrixrealizationsoftype:oneadditionaledgeintersectingC4}. 

As to \ref{numberofmatrixrealizationsoftype:XBLisomorphictoK23}, it is 
evident that the number of possibilities to realize a $K^{2,3}$ is 
$2\cdot 2^6 \cdot \binom{n-1}{2}\cdot \binom{n-1}{3}$, the first factor accouting 
for the two possibilities of either choosing two of the first indices and three of 
the last, or vice versa. 

As to \ref{numberofmatrixrealizationsoftype:twoisolatedvertices}, note that 
$f_1\ref{item:comparisonproofcaseLhassize6:twoisolatedvertices} = 4$, hence it is 
necessary that $B[u]=B[v]=0$. Therefore, the number of 
$B\mid_{\{u,v\}}\in\{0,\pm\}^{I\setminus S}$ with 
$\upX_{B} \cong \ref{item:comparisonproofcaseLhassize6:twoisolatedvertices}$ equals 
the number of $\{u,v\}\in \binom{[n-1]^2\setminus S}{2}$ such that 
\begin{equation}\label{eq:proofofnumberofeventsinducing:item:comparisonproofcaseLhassize6:twoisolatedvertices}
\lvert \{ u_1, v_1\} \setminus \upp_1(S) \rvert +
\lvert \{ u_2, v_2\} \setminus \upp_2(S) \rvert = 2\quad .
\end{equation} 
We now distinguish cases according to how \eqref{eq:proofofnumberofeventsinducing:item:comparisonproofcaseLhassize6:twoisolatedvertices} is satisfied.

\begin{enumerate}[label={\rm(C.\ref{numberofmatrixrealizationsoftype:twoisolatedvertices}.\arabic{*})}]
\item\label{numberofmatrixrealizationsoftype:twoisolatedvertices:case:lvertu1v1rvertbackslashp1Sequals0} $\lvert \{ u_1, v_1\} \setminus \upp_1(S) \rvert = 0$. Then \eqref{eq:proofofnumberofeventsinducing:item:comparisonproofcaseLhassize6:twoisolatedvertices} implies $\lvert \{u_2, v_2  \} \setminus \upp_2(S) \rvert = 2$, which is equivalent to \eqref{eq:proofofcomparativecountingtheoremcase2firsteq}. The property 
defining Case \ref{numberofmatrixrealizationsoftype:twoisolatedvertices:case:lvertu1v1rvertbackslashp1Sequals0} is equivalent to $\{u_1,v_1\}\subseteq \upp_1(S)$. 
There are now two further cases: 
\begin{enumerate}[label={\rm(\arabic{*})}]
\item\label{numberofmatrixrealizationsoftype:twoisolatedvertices:case:lvertu1v1rvertbackslashp1Sequals0:case:u1equalsv1} $u_1=v_1$. Then there are 
exactly $2$ possible set inclusions $\{u_1,v_1\} = \{u_1\} \subseteq \upp_1(S) 
= \{a_1,c_1\}$. For each of them, there are exactly $\binom{(n-1)-2}{2}$ different 
sets $\{u_2,v_2\}$ with property 
\eqref{eq:proofofcomparativecountingtheoremcase2firsteq}. Since 
$u_1=v_1$ and $u\prec v$ imply $u_2<v_2$, each of these sets determines the two 
pairs $u$ and $v$. 
Therefore, if \ref{numberofmatrixrealizationsoftype:twoisolatedvertices:case:lvertu1v1rvertbackslashp1Sequals0}.\ref{numberofmatrixrealizationsoftype:twoisolatedvertices:case:lvertu1v1rvertbackslashp1Sequals0:case:u1equalsv1}, 
then there are exactly $2\cdot \binom{(n-1)-2}{2}$ 
realizations of type \ref{item:comparisonproofcaseLhassize6:twoisolatedvertices} 
by $B\mid_{\{u,v\}}$. 
\item\label{numberofmatrixrealizationsoftype:twoisolatedvertices:case:lvertu1v1rvertbackslashp1Sequals0:case:u1doesnotequalv1} $u_1 \neq v_1$. Then $u\prec u$ implies 
$u_1 < u_1$. Now there is exactly one possible set inclusion 
$\{u_1,v_1\}\subseteq \upp_1(S)=\{a_1,c_1\}$. When this inclusion holds, there are 
exactly $\binom{(n-1)-2}{2}$ different sets $\{u_2,v_2\}$ with 
property \eqref{eq:proofofcomparativecountingtheoremcase2firsteq}. Each of them can 
be realized in exactly $2$ ways by $u$ and $v$, either by $u_2<v_2$ or by $v_2<u_2$. 

Therefore, if \ref{numberofmatrixrealizationsoftype:twoisolatedvertices:case:lvertu1v1rvertbackslashp1Sequals0}.\ref{numberofmatrixrealizationsoftype:twoisolatedvertices:case:lvertu1v1rvertbackslashp1Sequals0:case:u1doesnotequalv1}, then there are again 
(with a qualitatively different reason for the factor $2$) 
exactly $2\cdot \binom{(n-1)-2}{2}$ realizations of 
type \ref{item:comparisonproofcaseLhassize6:twoisolatedvertices} by $B\mid_{\{u,v\}}$. 
\end{enumerate}
It follows that if \ref{numberofmatrixrealizationsoftype:twoisolatedvertices:case:lvertu1v1rvertbackslashp1Sequals0}, then there are exactly $4\cdot \binom{(n-1)-2}{2}$ 
different realizations of type \ref{item:comparisonproofcaseLhassize6:twoisolatedvertices} by $B\mid_{\{u,v\}}$. 
\item\label{numberofmatrixrealizationsoftype:twoisolatedvertices:case:lvertu1v1rvertbackslashp1Sequals1} $\lvert \{ u_1, v_1\} \setminus \upp_1(S) \rvert = 1$. Then \eqref{eq:proofofnumberofeventsinducing:item:comparisonproofcaseLhassize6:twoisolatedvertices} 
implies $\lvert \{u_2, v_2  \} \setminus \upp_2(S) \rvert = 1$. Hence, in the present 
situation, the equations  \eqref{eq:proofofcomparativecountingtheoremcase2secondeq} 
and \eqref{eq:proofofcomparativecountingtheoremcase3secondeq} are simultaneously 
true. There are now two further cases depending on the manner in 
which \eqref{eq:proofofcomparativecountingtheoremcase2secondeq} is true: 
\begin{enumerate}[label={\rm(\arabic{*})}]
\item\label{numberofmatrixrealizationsoftype:twoisolatedvertices:case:lvertu1v1rvertbackslashp1Sequals1:u1equalsv1andbracesu1v1bracesintersectp1Sisempty} 
$u_1=v_1$ and $\{u_1,v_1\}\cap\upp_1(S)=\emptyset$. There are $((n-1)-2)$ 
possibilities for this. For each of them the clause $u_2=v_2$ and 
$\{u_2,v_2\}\cap\upp_1(S)=\emptyset$ in 
\eqref{eq:proofofcomparativecountingtheoremcase3secondeq} cannot be true because it 
would imply $u=v$ and therefore for each of them $u_2\neq v_2$ and 
$\lvert \{u_2,v_2\}\cap\upp_2(S)\rvert = 1$ must be true. In the following, keep in 
mind that $u_1=v_1$ and $u\prec v$ implies $u_2<v_2$. If $u_2 = a_2 = c_2$, then 
$v_2\neq b_2=d_2$, hence there are exactly $n-1 - a_2 - 1$ different $v_2$, and 
therefore as many different realizations of type \ref{item:comparisonproofcaseLhassize6:twoisolatedvertices} by $u$ and $v$. 
If $u_2 = b_2 = d_2$, then there are exactly $n-1 - b_2$ different $v_2$, and 
therefore as many different realizations of type \ref{item:comparisonproofcaseLhassize6:twoisolatedvertices} by $u$ and $v$. 
If $v_2 = a_2 = c_2$, then there are exactly $a_2 - 1$ different 
$u_2$, and therefore as many different realizations of type \ref{item:comparisonproofcaseLhassize6:twoisolatedvertices} by $u$ and $v$. 
If $v_2 = b_2 = d_2$, then $u_2\neq a_2=c_2$, hence there  are 
exactly $b_2 - 1 - 1$ different $u_2$ and therefore as many different 
realizations of type \ref{item:comparisonproofcaseLhassize6:twoisolatedvertices} by $u$ and $v$. 

Therefore, if \ref{numberofmatrixrealizationsoftype:twoisolatedvertices:case:lvertu1v1rvertbackslashp1Sequals1}.\ref{numberofmatrixrealizationsoftype:twoisolatedvertices:case:lvertu1v1rvertbackslashp1Sequals1:u1equalsv1andbracesu1v1bracesintersectp1Sisempty}, then there are 
$
\bigl( (n-1-a_2-1)+(n-1-b_2)+(a_2-1)+(b_2-1-1)\bigr) \cdot ((n-1)-2) 
= 2\cdot ((n-1)-2)^2 
$
realizations of type \ref{item:comparisonproofcaseLhassize6:twoisolatedvertices} 
by $B\mid_{I\setminus S} = B\mid_{\{u,v\}}$.
\item\label{numberofmatrixrealizationsoftype:twoisolatedvertices:case:lvertu1v1rvertbackslashp1Sequals1:u1doesnotequalv1andbracesu1v1bracesintersectp1Sisasingleton} 
$u_1 \neq v_1$ and $\lvert \{u_1,v_1\} \cap \upp_1(S) \rvert = 1$. 
Then $u\prec v$ implies $u_1 < v_1$ and there are two further cases: 
\begin{enumerate}[label={\rm(\arabic{*})}]
\item\label{numberofmatrixrealizationsoftype:twoisolatedvertices:case:lvertu1v1rvertbackslashp1Sequals1:u1doesnotequalv1andbracesu1v1bracesintersectp1Sisasingletoncontainingu1} $\lvert \{u_1,v_1\} \cap \upp_1(S) \rvert = 1$ is true due 
to $u_1 \in \upp_1(S) = \{a_1,c_1\}$. Then $v_1\notin \{a_1,c_1\}=\upp_1(S)$ and there 
are two further cases:
\begin{enumerate}[label={\rm(\arabic{*})}]
\item\label{numberofmatrixrealizationsoftype:twoisolatedvertices:case:lvertu1v1rvertbackslashp1Sequals1:u1doesnotequalv1andbracesu1v1bracesintersectp1Sisasingletoncontainingu1:u1equaltoa1}  $u_1 = a_1 = b_1$. Since $u_1 < v_1$ and $v_1\notin \{a_1,c_1\}$, 
there are then exactly $n-1 - a_1 - 1$ different $v_1$. For each of them, there are 
two cases. If the first clause of 
\eqref{eq:proofofcomparativecountingtheoremcase3secondeq} is true, then there are 
exactly $(n-1)-2$ different such $u_2=v_2$, and each such 
$u_2=v_2$ \emph{determines} the two pairs $u$ and $v$. Therefore in this case there 
are exactly $(n-1-a_1-1)\cdot ((n-1)-2)$ realizations of type 
\ref{item:comparisonproofcaseLhassize6:twoisolatedvertices} by $B\mid_{\{u,v\}}$. 
If the second clause of \eqref{eq:proofofcomparativecountingtheoremcase3secondeq} is 
true, then since $u_1=a_1=b_1$ combined with $u\neq a$ and $u\neq b$ implies 
$u_2\notin \{a_2=c_2,b_2=d_2\} = \upp_2(S)$, it follows that 
$\lvert \{u_2, v_2\} \cap \upp_2(S)\rvert =1$ is true as $v_2\in \upp_2(S)$. 
Therefore in  this case there are exactly $(n-1)-2$ different $u_2$ and 
exactly $2$ different $v_2$ for each of them and hence exactly 
$(n-1-a_1-1)\cdot 2\cdot ((n-1)-2)$ realizations of type 
\ref{item:comparisonproofcaseLhassize6:twoisolatedvertices} by $B\mid_{\{u,v\}}$. 

Therefore, if \ref{numberofmatrixrealizationsoftype:twoisolatedvertices:case:lvertu1v1rvertbackslashp1Sequals1}.\ref{numberofmatrixrealizationsoftype:twoisolatedvertices:case:lvertu1v1rvertbackslashp1Sequals1:u1doesnotequalv1andbracesu1v1bracesintersectp1Sisasingleton}.\ref{numberofmatrixrealizationsoftype:twoisolatedvertices:case:lvertu1v1rvertbackslashp1Sequals1:u1doesnotequalv1andbracesu1v1bracesintersectp1Sisasingletoncontainingu1}.\ref{numberofmatrixrealizationsoftype:twoisolatedvertices:case:lvertu1v1rvertbackslashp1Sequals1:u1doesnotequalv1andbracesu1v1bracesintersectp1Sisasingletoncontainingu1:u1equaltoa1}, then there are exactly $3\cdot (n-1-a_1-1)\cdot ((n-1)-2)$ 
different realizations of type \ref{item:comparisonproofcaseLhassize6:twoisolatedvertices} by $B\mid_{\{u,v\}}$.
\item\label{numberofmatrixrealizationsoftype:twoisolatedvertices:case:lvertu1v1rvertbackslashp1Sequals1:u1doesnotequalv1andbracesu1v1bracesintersectp1Sisasingletoncontainingu1:u1equaltoc1} $u_1 = c_1 = d_1$. Since $u_1< v_1$, there are then exactly 
$n-1 - c_1$ different $v_1$. For each of them, there are two cases. If the first 
clause of \eqref{eq:proofofcomparativecountingtheoremcase3secondeq} is true, then 
there are exactly $(n-1 - c_1)\cdot ((n-1)-2)$ realizations of type 
\ref{item:comparisonproofcaseLhassize6:twoisolatedvertices} by $B\mid_{\{u,v\}}$. If 
the second clause of \eqref{eq:proofofcomparativecountingtheoremcase3secondeq} is 
true, then since $u_1=c_1=d_1$  combined with $u\neq c$ and $u\neq d$ implies 
$u_2\notin \{a_2=c_2,b_2=d_2\}= \upp_2(S)$, it follows that $\lvert \{u_2, v_2\} \cap \upp_2(S)\rvert =1$ is true as $v_2\in \upp_2(S)$. Therefore in this case there are 
exactly $(n-1)-2$ different $u_2$ and for each of them exactly $2$ different $v_2$, 
thus exactly $(n-1-c_1) \cdot 2\cdot ((n-1)-2)$ realizations of type \ref{item:comparisonproofcaseLhassize6:twoisolatedvertices} by $B\mid_{\{u,v\}}$.
  
Therefore, if \ref{numberofmatrixrealizationsoftype:twoisolatedvertices:case:lvertu1v1rvertbackslashp1Sequals1}.\ref{numberofmatrixrealizationsoftype:twoisolatedvertices:case:lvertu1v1rvertbackslashp1Sequals1:u1doesnotequalv1andbracesu1v1bracesintersectp1Sisasingleton}.\ref{numberofmatrixrealizationsoftype:twoisolatedvertices:case:lvertu1v1rvertbackslashp1Sequals1:u1doesnotequalv1andbracesu1v1bracesintersectp1Sisasingletoncontainingu1}.\ref{numberofmatrixrealizationsoftype:twoisolatedvertices:case:lvertu1v1rvertbackslashp1Sequals1:u1doesnotequalv1andbracesu1v1bracesintersectp1Sisasingletoncontainingu1:u1equaltoc1}, then there are exactly $3\cdot (n-1-c_1)\cdot ((n-1)-2)$ realizations of 
type \ref{item:comparisonproofcaseLhassize6:twoisolatedvertices} by $B\mid_{\{u,v\}}$.
\end{enumerate}
It follows that if \ref{numberofmatrixrealizationsoftype:twoisolatedvertices:case:lvertu1v1rvertbackslashp1Sequals1}.\ref{numberofmatrixrealizationsoftype:twoisolatedvertices:case:lvertu1v1rvertbackslashp1Sequals1:u1doesnotequalv1andbracesu1v1bracesintersectp1Sisasingleton}.\ref{numberofmatrixrealizationsoftype:twoisolatedvertices:case:lvertu1v1rvertbackslashp1Sequals1:u1doesnotequalv1andbracesu1v1bracesintersectp1Sisasingletoncontainingu1}, then there are exactly 
$3\cdot (n-1-a_1-1)\cdot ((n-1)-2) + 3\cdot (n-1-c_1)\cdot ((n-1)-2) 
= 3\cdot (2n-a_1-c_1-3)\cdot ((n-1)-2)$ different realizations of 
type \ref{item:comparisonproofcaseLhassize6:twoisolatedvertices} by $B\mid_{\{u,v\}}$. 
\item\label{numberofmatrixrealizationsoftype:twoisolatedvertices:case:lvertu1v1rvertbackslashp1Sequals1:u1doesnotequalv1andbracesu1v1bracesintersectp1Sisasingletoncontainingv1} $\lvert \{u_1,v_1\} \cap \upp_1(S) \rvert = 1$ is true due to 
$v_1 \in \upp_1(S) = \{a_1, c_1\}$. Then $u_1\notin \{a_1,c_1\}$ and there are two 
further cases:
\begin{enumerate}[label={\rm(\arabic{*})}]
\item\label{numberofmatrixrealizationsoftype:twoisolatedvertices:case:lvertu1v1rvertbackslashp1Sequals1:u1doesnotequalv1andbracesu1v1bracesintersectp1Sisasingletoncontainingu1:v1equaltoa1}  $v_1 = a_1 = b_1$. Since $u_1<v_1$, there are then exactly $a_1-1$ 
different $u_1$. For each of them, there are two cases. If the first clause of 
\eqref{eq:proofofcomparativecountingtheoremcase3secondeq} is true, then there are 
exactly $(n-1)-2$ different $u_2=v_2$ and each such $u_2=v_2$ determines the 
pair $(u,v)$. Hence in this case there are exactly $(a_1-1)\cdot ((n-1)-2)$ 
realizations of type \ref{item:comparisonproofcaseLhassize6:twoisolatedvertices} by 
$B\mid_{\{u,v\}}$. If the second clause of 
\eqref{eq:proofofcomparativecountingtheoremcase3secondeq} is true, then since 
$v_1=a_1=b_1$ combined with $v\neq a$ and $v\neq b$ implies 
$v_2\notin\{a_2=c_2,b_2=d_2\} = \upp_2(S)$, we know that 
$\lvert \{u_2, v_2\} \cap \upp_2(S)\rvert =1$ must be true as $u_2\in \upp_2(S)$, 
hence there are $2$ different $u_2$ and for each of them exactly $(n-1)-2$ different 
$v_2$, hence in this case there are exactly $(a_1-1)\cdot 2\cdot ((n-1)-2)$ 
realizations of type \ref{item:comparisonproofcaseLhassize6:twoisolatedvertices} 
by $B\mid_{\{u,v\}}$.

Therefore, if \ref{numberofmatrixrealizationsoftype:twoisolatedvertices:case:lvertu1v1rvertbackslashp1Sequals1}.\ref{numberofmatrixrealizationsoftype:twoisolatedvertices:case:lvertu1v1rvertbackslashp1Sequals1:u1doesnotequalv1andbracesu1v1bracesintersectp1Sisasingleton}.\ref{numberofmatrixrealizationsoftype:twoisolatedvertices:case:lvertu1v1rvertbackslashp1Sequals1:u1doesnotequalv1andbracesu1v1bracesintersectp1Sisasingletoncontainingv1}.\ref{numberofmatrixrealizationsoftype:twoisolatedvertices:case:lvertu1v1rvertbackslashp1Sequals1:u1doesnotequalv1andbracesu1v1bracesintersectp1Sisasingletoncontainingu1:v1equaltoa1}, then there are exactly $3\cdot (a_1-1)\cdot ((n-1)-2)$ different 
realizations of type \ref{item:comparisonproofcaseLhassize6:twoisolatedvertices} 
by $B\mid_{\{u,v\}}$. 
\item\label{numberofmatrixrealizationsoftype:twoisolatedvertices:case:lvertu1v1rvertbackslashp1Sequals1:u1doesnotequalv1andbracesu1v1bracesintersectp1Sisasingletoncontainingu1:v1equaltoc1}  $v_1 = c_1 = d_1$. Since $u_1 < v_1$ and $u_1\notin \{a_1,c_1\}$, 
there are then exactly $c_1-1-1$ different $u_1$. For each of them, there are two 
cases. If the first clause of 
\eqref{eq:proofofcomparativecountingtheoremcase3secondeq} is true, then there are 
exactly $(c_1-1-1)\cdot ((n-1)-2)$ realizations of type 
\ref{item:comparisonproofcaseLhassize6:twoisolatedvertices} by $B\mid_{\{u,v\}}$. If 
the second clause of \eqref{eq:proofofcomparativecountingtheoremcase3secondeq} is 
true, then since $v_1 = c_1 = d_1$ combined with $v\neq c$ and $v\neq d$ implies 
$v_2\notin\{a_2=c_2,b_2=d_2\}=\upp_2(S)$, it follows that 
$\lvert\{u_2,v_2 \}\cap\upp_2(S)\rvert=1$ is true as $u_2\in \upp_2(S)$. Therefore 
in this case there are exactly $(n-1)-2$ different $v_2$ and for each of them 
exactly $2$ different $u_2$, hence exactly $(c_1-1-1)\cdot 2\cdot ((n-1) - 2)$ 
realizations of type \ref{item:comparisonproofcaseLhassize6:twoisolatedvertices} by 
$B\mid_{\{u,v\}}$.  

Therefore, if \ref{numberofmatrixrealizationsoftype:twoisolatedvertices:case:lvertu1v1rvertbackslashp1Sequals1}.\ref{numberofmatrixrealizationsoftype:twoisolatedvertices:case:lvertu1v1rvertbackslashp1Sequals1:u1doesnotequalv1andbracesu1v1bracesintersectp1Sisasingleton}.\ref{numberofmatrixrealizationsoftype:twoisolatedvertices:case:lvertu1v1rvertbackslashp1Sequals1:u1doesnotequalv1andbracesu1v1bracesintersectp1Sisasingletoncontainingv1}.\ref{numberofmatrixrealizationsoftype:twoisolatedvertices:case:lvertu1v1rvertbackslashp1Sequals1:u1doesnotequalv1andbracesu1v1bracesintersectp1Sisasingletoncontainingu1:v1equaltoc1}, then there are exactly $3\cdot (c_1-1-1)\cdot ((n-1)-2)$ realizations 
of type \ref{item:comparisonproofcaseLhassize6:twoisolatedvertices}  by $B\mid_{\{u,v\}}$.
\end{enumerate}
It follows that if \ref{numberofmatrixrealizationsoftype:twoisolatedvertices:case:lvertu1v1rvertbackslashp1Sequals1}.\ref{numberofmatrixrealizationsoftype:twoisolatedvertices:case:lvertu1v1rvertbackslashp1Sequals1:u1doesnotequalv1andbracesu1v1bracesintersectp1Sisasingleton}.\ref{numberofmatrixrealizationsoftype:twoisolatedvertices:case:lvertu1v1rvertbackslashp1Sequals1:u1doesnotequalv1andbracesu1v1bracesintersectp1Sisasingletoncontainingv1}, then there are exactly 
$3\cdot (a_1-1)\cdot ((n-1)-2) + 3\cdot (c_1-1-1)\cdot ((n-1)-2) 
= 3\cdot (a_1+c_1-3)\cdot ((n-1)-2)$ realizations of  
type \ref{item:comparisonproofcaseLhassize6:twoisolatedvertices}  by $B\mid_{\{u,v\}}$.
\end{enumerate} 
It follows that if \ref{numberofmatrixrealizationsoftype:twoisolatedvertices:case:lvertu1v1rvertbackslashp1Sequals1}.\ref{numberofmatrixrealizationsoftype:twoisolatedvertices:case:lvertu1v1rvertbackslashp1Sequals1:u1doesnotequalv1andbracesu1v1bracesintersectp1Sisasingleton}, then there are exactly 
$3\cdot (2n-a_1-c_1-3)\cdot ((n-1)-2) + 3\cdot (a_1+c_1-3)\cdot ((n-1)-2) 
= 6\cdot ((n-1)-2)^2$ realizations of type 
\ref{item:comparisonproofcaseLhassize6:twoisolatedvertices}  by $B\mid_{\{u,v\}}$.
\end{enumerate}
It follows that if \ref{numberofmatrixrealizationsoftype:twoisolatedvertices:case:lvertu1v1rvertbackslashp1Sequals1}, then there are exactly  
$2\cdot ((n-1)-2)^2 + 6\cdot ((n-1)-2)^2 = 8\cdot ((n-1)-2)^2$ 
realizations of type \ref{item:comparisonproofcaseLhassize6:twoisolatedvertices} 
by $B\mid_{\{u,v\}}$. 
\item\label{numberofmatrixrealizationsoftype:twoisolatedvertices:case:lvertu1v1rvertbackslashp1Sequals2} $\lvert \{u_1,v_1\}\setminus\upp_1(S)\rvert = 2$. This is 
equivalent to \eqref{eq:proofofcomparativecountingtheoremcase3firsteq}. Moreover 
\eqref{eq:proofofnumberofeventsinducing:item:comparisonproofcaseLhassize6:twoisolatedvertices} implies $\lvert \{u_2,v_2\}\setminus\upp_2(S)\rvert = 0$, which is 
equivalent to $\{u_2,v_2\}\subseteq\upp_2(S)$. Swapping the subscripts $1$ and 
$2$ in the analysis of \ref{numberofmatrixrealizationsoftype:twoisolatedvertices:case:lvertu1v1rvertbackslashp1Sequals0} shows that if \ref{numberofmatrixrealizationsoftype:twoisolatedvertices:case:lvertu1v1rvertbackslashp1Sequals2}, then there are exactly 
$4\cdot\binom{(n-1)-2}{2}$ realizations of type 
\ref{item:comparisonproofcaseLhassize6:twoisolatedvertices} by $B\mid_{\{u,v\}}$. 
\end{enumerate} 

It follows that for the fixed $S$, there are exactly 
$4\cdot \binom{(n-1)-2}{2} + 8\cdot((n-1)-2)^2 + 4\cdot \binom{(n-1)-2}{2} 
= 8\cdot(n-3)^2 + 8\cdot \binom{n-3}{2}$ different $I\setminus S = \{u,v\}$ with 
the property $\upX_{B} \cong \ref{item:comparisonproofcaseLhassize6:twoisolatedvertices}$. This completes the proof of 
\ref{numberofmatrixrealizationsoftype:twoisolatedvertices}. 

As to \ref{numberofmatrixrealizationsoftype:oneadditionaledgeintersectingC4andoneisolatedvertex} and \ref{numberofmatrixrealizationsoftype:C4withoneadditionaldisjointedge}, 
let us first note that both for \ref{item:comparisonproofcaseLhassize6:oneadditionaledgeintersectingC4andoneisolatedvertex} and for \ref{item:comparisonproofcaseLhassize6:C4withoneadditionaldisjointedge} a necessary condition is that $\upX_{B}$ contain 
exactly two vertices not in $\upX_{B\mid_S}$ $\cong$ $C^4$. 
Therefore, the set of all $I\setminus S = \{u,v\}$ with 
$\upX_B \cong \ref{item:comparisonproofcaseLhassize6:oneadditionaledgeintersectingC4andoneisolatedvertex}$ is a subset of the set of those $\{u,v\}$ with  
$\upX_{\{0\}^{\{u,v\}}} \cong \ref{item:comparisonproofcaseLhassize6:twoisolatedvertices}$, 
and likewise for \ref{item:comparisonproofcaseLhassize6:C4withoneadditionaldisjointedge}. We may therefore determine both \ref{numberofmatrixrealizationsoftype:oneadditionaledgeintersectingC4andoneisolatedvertex} and \ref{numberofmatrixrealizationsoftype:C4withoneadditionaldisjointedge} by a single reexamination of the analysis given for type 
\ref{item:comparisonproofcaseLhassize6:twoisolatedvertices}. In each of the cases 
which we distinguished there we now have to count the number of those 
$B\mid_{\{u,v\}}\in\{0,\pm\}^{\{u,v\}}$ with $\upX_B \cong \ref{item:comparisonproofcaseLhassize6:oneadditionaledgeintersectingC4andoneisolatedvertex}$ and also of 
those $B\mid_{\{u,v\}}\in\{0,\pm\}^{\{u,v\}}$ with $\upX_{B} \cong \ref{item:comparisonproofcaseLhassize6:C4withoneadditionaldisjointedge}$.

We can prepare for this as follows. Consider the properties
{\small
\begin{enumerate}[label={\rm(i\arabic{*})}]
\item
\label{item:casesfortype8casethatexactlyoneofuandvhasaprojectionmeetingaprojectionofS} 
($\{u_1,u_2\}\cap \upp(S) \neq \emptyset$ and 
$\{v_1,v_2\}\cap \upp(S) = \emptyset$) or 
($\{u_1,u_2\}\cap \upp(S) = \emptyset$ and 
$\{v_1,v_2\}\cap \upp(S) \neq \emptyset$)\quad , 
\item
\label{item:casesfortype8casethatbothuandvhaveaprojectionmeetingaprojectionofS} 
($\{u_1,u_2\}\cap \upp(S) \neq \emptyset$ and 
$\{v_1,v_2\}\cap \upp(S) \neq \emptyset$)\quad .
\end{enumerate}
}
In each of the cases to be reexamined, these properties alone  determine  how many 
$B\mid_{\{u,v\}}\in \{0,\pm\}^{\{u,v\}}$ realize \ref{item:comparisonproofcaseLhassize6:oneadditionaledgeintersectingC4andoneisolatedvertex} or \ref{item:comparisonproofcaseLhassize6:C4withoneadditionaldisjointedge}. 

Let us first focus on \ref{item:comparisonproofcaseLhassize6:oneadditionaledgeintersectingC4andoneisolatedvertex}. In \ref{item:casesfortype8casethatexactlyoneofuandvhasaprojectionmeetingaprojectionofS}, each of the two clauses of that disjunction has the 
property that if it is true, then there are exactly $2$ possibilities for a 
$B\mid_{\{u,v\}}$ with $\upX_B \cong \ref{item:comparisonproofcaseLhassize6:oneadditionaledgeintersectingC4andoneisolatedvertex}$.  For the first clause these 
are ($B[u]\in \{\pm\}$ and $B[v]=0$), for the second 
clause ($B[u] = 0$ and $B[v]\in\{\pm\}$). Moreover, the disjunction is evidently 
exclusive. Therefore, if property \ref{item:casesfortype8casethatexactlyoneofuandvhasaprojectionmeetingaprojectionofS} is true, then the number 
of $B\mid_{\{u,v\}}\in \{0,\pm\}^{\{u,v\}}$ with $\upX_B \cong \ref{item:comparisonproofcaseLhassize6:oneadditionaledgeintersectingC4andoneisolatedvertex}$ 
is exactly $2$-times as large as the number of 
$\{u,v\}\in\binom{[n-1]^2}{2}$ with $\upX_{\{0\}^{\{u,v\}}} \cong \ref{item:comparisonproofcaseLhassize6:twoisolatedvertices}$ which was determined in the proof 
of \ref{numberofmatrixrealizationsoftype:twoisolatedvertices}.
If property \ref{item:casesfortype8casethatbothuandvhaveaprojectionmeetingaprojectionofS} is true, then there are exactly $4$ possibilities for a $B\mid_{\{u,v\}}$ with 
$\upX_B \cong \ref{item:comparisonproofcaseLhassize6:oneadditionaledgeintersectingC4andoneisolatedvertex}$: 
($B[u]\in \{\pm\}$ and $B[v]=0$) or ($B[u] = 0$ and $B[v] \in \{\pm\}$). Therefore, 
if property \ref{item:casesfortype8casethatbothuandvhaveaprojectionmeetingaprojectionofS} is true, then there are exactly $4$-times as many realizations of isomorphism type 
\ref{item:comparisonproofcaseLhassize6:oneadditionaledgeintersectingC4andoneisolatedvertex} by $B\mid_{I\setminus S} = B\mid_{\{u,v\}}$ as there had been for type 
\ref{item:comparisonproofcaseLhassize6:twoisolatedvertices}.

Let us now turn to \ref{item:comparisonproofcaseLhassize6:C4withoneadditionaldisjointedge}. In case \ref{item:casesfortype8casethatexactlyoneofuandvhasaprojectionmeetingaprojectionofS} there are (just as for type \ref{item:comparisonproofcaseLhassize6:oneadditionaledgeintersectingC4andoneisolatedvertex}) exactly $2$ possibilities for 
a $B\mid_{\{u,v\}}$ with $\upX_{B} \cong \ref{item:comparisonproofcaseLhassize6:C4withoneadditionaldisjointedge}$. This time, these are ($B[u]=0$ and $B[v]\in\{\pm\}$) for 
the first clause of \ref{item:casesfortype8casethatexactlyoneofuandvhasaprojectionmeetingaprojectionofS}, and ($B[u]\in\{\pm\}$ and $B[v]=0$) for the second clause. 
Again, due to the mutual exclusiveness of the clauses, it follows that whenever 
case \ref{item:casesfortype8casethatexactlyoneofuandvhasaprojectionmeetingaprojectionofS} is true (no matter by way of which clause), there are exactly $2$-times as many 
$B\mid_{\{u,v\}}\in \{0,\pm\}^{\{u,v\}}$ with $\upX_B \cong \ref{item:comparisonproofcaseLhassize6:C4withoneadditionaldisjointedge}$ as there are 
$\{u,v\}\in \binom{[n-1]^2}{2}$ with $\upX_{\{0\}^{\{u,v\}}} \cong \ref{item:comparisonproofcaseLhassize6:twoisolatedvertices}$. Concerning property \ref{item:casesfortype8casethatbothuandvhaveaprojectionmeetingaprojectionofS}, however, there is a genuine 
difference: when this property is true, there is \emph{no} possibility to 
choose $B\mid_{\{u,v\}}$ so as to create exactly one edge disjoint from 
the $\upX_{B\mid_S}\cong C^4$. We can now begin inspecting the cases. 

If \ref{numberofmatrixrealizationsoftype:twoisolatedvertices:case:lvertu1v1rvertbackslashp1Sequals0}, then the inclusion $\{u_1,v_1\}\subseteq \upp_1(S)$ alone, no matter 
whether $u_1=v_1$ or not, implies that property \ref{item:casesfortype8casethatbothuandvhaveaprojectionmeetingaprojectionofS} is true and without going any deeper  we know 
that there are exactly 
$4\cdot 4 \cdot \binom{(n-1)-2}{2} = 16 \cdot \binom{(n-1)-2}{2}$ 
realizations of type \ref{item:comparisonproofcaseLhassize6:oneadditionaledgeintersectingC4andoneisolatedvertex} and $0$ realizations of type \ref{item:comparisonproofcaseLhassize6:C4withoneadditionaldisjointedge} by $B\mid_{\{u,v\}}$. 

If \ref{numberofmatrixrealizationsoftype:threeisolatedvertices:case:lvertu1v1rvertbackslashp1Sequals1}, we have to descend one level deeper. If \ref{numberofmatrixrealizationsoftype:threeisolatedvertices:case:lvertu1v1rvertbackslashp1Sequals1}.\ref{numberofmatrixrealizationsoftype:twoisolatedvertices:case:lvertu1v1rvertbackslashp1Sequals1:u1equalsv1andbracesu1v1bracesintersectp1Sisempty} then it is known that 
$u_1=v_1$, $\{u_1,v_1\}\cap\upp_1(S) = \emptyset$, 
$u_2<v_2$ and $\lvert \{ u_2,v_2\}\cap\upp_2(S) \rvert = 1$ and obviously this implies 
that property \ref{item:casesfortype8casethatexactlyoneofuandvhasaprojectionmeetingaprojectionofS} is true. Therefore without having to reexamine further subcases 
we then know that if \ref{numberofmatrixrealizationsoftype:threeisolatedvertices:case:lvertu1v1rvertbackslashp1Sequals1}.\ref{numberofmatrixrealizationsoftype:twoisolatedvertices:case:lvertu1v1rvertbackslashp1Sequals1:u1equalsv1andbracesu1v1bracesintersectp1Sisempty}, then there are exactly $2\cdot 2\cdot ((n-1)-2)^2 = 4\cdot ((n-1)-2)^2$ 
realizations of type \ref{item:comparisonproofcaseLhassize6:oneadditionaledgeintersectingC4andoneisolatedvertex} and also exactly $2 \cdot 2 \cdot ((n-1)-2)^2 
= 4\cdot((n-1)-2)^2$ realizations of type \ref{item:comparisonproofcaseLhassize6:C4withoneadditionaldisjointedge} by $B\mid_{\{u,v\}}$.

If \ref{numberofmatrixrealizationsoftype:threeisolatedvertices:case:lvertu1v1rvertbackslashp1Sequals1}.\ref{numberofmatrixrealizationsoftype:twoisolatedvertices:case:lvertu1v1rvertbackslashp1Sequals1:u1doesnotequalv1andbracesu1v1bracesintersectp1Sisasingleton}, however, then we have to go deeper still. Although we then already know 
that $u_1<v_1$ and $\lvert \{u_1,v_1\} \cap \upp_1(S) \rvert = 1$, and therefore 
know that 
\begin{equation}\label{partialknowledgeintheanalysisoftype8}
\text{$\{u_1,u_2\} \cap \upp(S)\neq\emptyset$ or 
$\{v_1,v_2\} \cap \upp(S)\neq\emptyset$ \quad , }
\end{equation}
at the present stage of our knowledge this latter property is compatible with both 
\ref{item:casesfortype8casethatexactlyoneofuandvhasaprojectionmeetingaprojectionofS} 
and \ref{item:casesfortype8casethatbothuandvhaveaprojectionmeetingaprojectionofS} 
(i.e., we do not know yet whether the `or' in 
\eqref{partialknowledgeintheanalysisoftype8} is true as an `and'). The 
reason is that we do not yet have any knowledge about $u_2$ and $v_2$. Therefore, 
neither descending down to \ref{numberofmatrixrealizationsoftype:threeisolatedvertices:case:lvertu1v1rvertbackslashp1Sequals1}.\ref{numberofmatrixrealizationsoftype:twoisolatedvertices:case:lvertu1v1rvertbackslashp1Sequals1:u1doesnotequalv1andbracesu1v1bracesintersectp1Sisasingleton}.\ref{numberofmatrixrealizationsoftype:twoisolatedvertices:case:lvertu1v1rvertbackslashp1Sequals1:u1doesnotequalv1andbracesu1v1bracesintersectp1Sisasingletoncontainingu1} nor to \ref{numberofmatrixrealizationsoftype:threeisolatedvertices:case:lvertu1v1rvertbackslashp1Sequals1}.\ref{numberofmatrixrealizationsoftype:twoisolatedvertices:case:lvertu1v1rvertbackslashp1Sequals1:u1doesnotequalv1andbracesu1v1bracesintersectp1Sisasingleton}.\ref{numberofmatrixrealizationsoftype:twoisolatedvertices:case:lvertu1v1rvertbackslashp1Sequals1:u1doesnotequalv1andbracesu1v1bracesintersectp1Sisasingletoncontainingu1}.\ref{numberofmatrixrealizationsoftype:twoisolatedvertices:case:lvertu1v1rvertbackslashp1Sequals1:u1doesnotequalv1andbracesu1v1bracesintersectp1Sisasingletoncontainingu1:u1equaltoa1} is sufficient for us to know whether \ref{item:casesfortype8casethatexactlyoneofuandvhasaprojectionmeetingaprojectionofS} or \ref{item:casesfortype8casethatbothuandvhaveaprojectionmeetingaprojectionofS} is true. We therefore have to go all the way down to the two (anonymous) subcases of maximal depth within the case \ref{numberofmatrixrealizationsoftype:threeisolatedvertices:case:lvertu1v1rvertbackslashp1Sequals1}.\ref{numberofmatrixrealizationsoftype:twoisolatedvertices:case:lvertu1v1rvertbackslashp1Sequals1:u1doesnotequalv1andbracesu1v1bracesintersectp1Sisasingleton}.\ref{numberofmatrixrealizationsoftype:twoisolatedvertices:case:lvertu1v1rvertbackslashp1Sequals1:u1doesnotequalv1andbracesu1v1bracesintersectp1Sisasingletoncontainingu1}.\ref{numberofmatrixrealizationsoftype:twoisolatedvertices:case:lvertu1v1rvertbackslashp1Sequals1:u1doesnotequalv1andbracesu1v1bracesintersectp1Sisasingletoncontainingu1:u1equaltoa1}. In the first of the two subcases we know that 
$\{u_2,v_2\}\cap \upp_2(S) = \emptyset$ and combining this with our knowledge 
of $\lvert \{u_1,v_1\} \cap \upp_1(S) \rvert = 1$ we may conclude that exactly one 
of the two clauses in \eqref{partialknowledgeintheanalysisoftype8}, and hence 
property \ref{item:casesfortype8casethatexactlyoneofuandvhasaprojectionmeetingaprojectionofS} is true. Therefore, in the present subcase there are exactly 
$2\cdot (n-1-a_1-1)\cdot ((n-1)-2)$ realizations of type \ref{item:comparisonproofcaseLhassize6:oneadditionaledgeintersectingC4andoneisolatedvertex} and also $2\cdot (n-1-a_1-1)\cdot ((n-1)-2)$ realizations of type \ref{item:comparisonproofcaseLhassize6:C4withoneadditionaldisjointedge} by $B\mid_{\{u,v\}}$. 

In the second of the two subcases we know that $u_2\neq v_2$ and 
$\lvert \{u_2,v_2\}\cap\upp_2(S) \rvert = 1$. Recall that at present we also know 
that  $u_1<v_1$ and $\lvert \{u_1,v_1\} \cap \upp_1(S) \rvert = 1$. 
Keeping in mind the fact that because of $u\notin S$ at most one projection of $u$ 
can be contained in $\upp(S)$ (and the analogous fact about $v$), 
we may argue that if 
$\lvert \{u_1,v_1\} \cap \upp_1(S) \rvert = 1$ is true as 
($u_1\in \upp_1(S)$ and $v_1\notin\upp_1(S)$), then  
$\lvert \{u_2,v_2\}\cap\upp_2(S) \rvert = 1$ must be true 
as ($u_2\notin\upp_2(S)$ and $v_2\in\upp_2(S)$), and if 
$\lvert \{u_1,v_1\} \cap \upp_1(S) \rvert = 1$ is true as 
($u_1\notin\upp_1(S)$ and $v_1\in \upp_1(S)$ ), 
then  $\lvert \{u_2,v_2\}\cap\upp_2(S) \rvert = 1$ must be true as 
($u_2\in\upp_2(S)$ and $v_2\notin\upp_2(S)$). Since in both 
cases both clauses of \eqref{partialknowledgeintheanalysisoftype8} are true, 
it follows that \ref{item:casesfortype8casethatbothuandvhaveaprojectionmeetingaprojectionofS} is true. Therefore in the present subcase there are exactly 
$4\cdot (n-1-a_1-1)\cdot 2\cdot ((n-1)-2) = 8\cdot (n-1-a_1-1)\cdot ((n-1)-2)$ 
realizations of type \ref{item:comparisonproofcaseLhassize6:oneadditionaledgeintersectingC4andoneisolatedvertex} and $0$ realizations of type \ref{item:comparisonproofcaseLhassize6:C4withoneadditionaldisjointedge} by $B\mid_{\{u,v\}}$. 
Adding up our findings, it follows that if \ref{numberofmatrixrealizationsoftype:threeisolatedvertices:case:lvertu1v1rvertbackslashp1Sequals1}.\ref{numberofmatrixrealizationsoftype:twoisolatedvertices:case:lvertu1v1rvertbackslashp1Sequals1:u1doesnotequalv1andbracesu1v1bracesintersectp1Sisasingleton}.\ref{numberofmatrixrealizationsoftype:twoisolatedvertices:case:lvertu1v1rvertbackslashp1Sequals1:u1doesnotequalv1andbracesu1v1bracesintersectp1Sisasingletoncontainingu1}.\ref{numberofmatrixrealizationsoftype:twoisolatedvertices:case:lvertu1v1rvertbackslashp1Sequals1:u1doesnotequalv1andbracesu1v1bracesintersectp1Sisasingletoncontainingu1:u1equaltoa1}, then there are exactly 
$2\cdot (n-1-a_1-1)\cdot ((n-1)-2) + 8\cdot (n-1-a_1-1)\cdot ((n-1)-2)
=
10\cdot (n-1-a_1-1)\cdot ((n-1)-2)$ realizations of type \ref{item:comparisonproofcaseLhassize6:oneadditionaledgeintersectingC4andoneisolatedvertex} but merely 
$2\cdot (n-1-a_1-1)\cdot ((n-1)-2)$ realizations of type \ref{item:comparisonproofcaseLhassize6:C4withoneadditionaldisjointedge} by $B\mid_{\{u,v\}}$. 

The case \ref{numberofmatrixrealizationsoftype:threeisolatedvertices:case:lvertu1v1rvertbackslashp1Sequals1}.\ref{numberofmatrixrealizationsoftype:twoisolatedvertices:case:lvertu1v1rvertbackslashp1Sequals1:u1doesnotequalv1andbracesu1v1bracesintersectp1Sisasingleton}.\ref{numberofmatrixrealizationsoftype:twoisolatedvertices:case:lvertu1v1rvertbackslashp1Sequals1:u1doesnotequalv1andbracesu1v1bracesintersectp1Sisasingletoncontainingu1}.\ref{numberofmatrixrealizationsoftype:twoisolatedvertices:case:lvertu1v1rvertbackslashp1Sequals1:u1doesnotequalv1andbracesu1v1bracesintersectp1Sisasingletoncontainingu1:u1equaltoc1} is again not sufficient for us to know whether \ref{item:casesfortype8casethatexactlyoneofuandvhasaprojectionmeetingaprojectionofS} or \ref{item:casesfortype8casethatbothuandvhaveaprojectionmeetingaprojectionofS} is true and we again have to consider its anonymous subcases. In the first of them, an argument entirely analogous to 
the one given for the first subcase of \ref{numberofmatrixrealizationsoftype:threeisolatedvertices:case:lvertu1v1rvertbackslashp1Sequals1}.\ref{numberofmatrixrealizationsoftype:twoisolatedvertices:case:lvertu1v1rvertbackslashp1Sequals1:u1doesnotequalv1andbracesu1v1bracesintersectp1Sisasingleton}.\ref{numberofmatrixrealizationsoftype:twoisolatedvertices:case:lvertu1v1rvertbackslashp1Sequals1:u1doesnotequalv1andbracesu1v1bracesintersectp1Sisasingletoncontainingu1}.\ref{numberofmatrixrealizationsoftype:twoisolatedvertices:case:lvertu1v1rvertbackslashp1Sequals1:u1doesnotequalv1andbracesu1v1bracesintersectp1Sisasingletoncontainingu1:u1equaltoa1} proves that then property \ref{item:casesfortype8casethatexactlyoneofuandvhasaprojectionmeetingaprojectionofS} is true and therefore we know that there are exactly $2\cdot (n-1-c_1)\cdot ((n-1)-2)$ realizations of 
type \ref{item:comparisonproofcaseLhassize6:oneadditionaledgeintersectingC4andoneisolatedvertex} and also exactly $2\cdot (n-1-c_1)\cdot ((n-1)-2)$ realizations of 
type \ref{item:comparisonproofcaseLhassize6:C4withoneadditionaldisjointedge} 
by $B\mid_{\{u,v\}}$. In the second of them, analogously to the second subcase 
of  \ref{numberofmatrixrealizationsoftype:threeisolatedvertices:case:lvertu1v1rvertbackslashp1Sequals1}.\ref{numberofmatrixrealizationsoftype:twoisolatedvertices:case:lvertu1v1rvertbackslashp1Sequals1:u1doesnotequalv1andbracesu1v1bracesintersectp1Sisasingleton}.\ref{numberofmatrixrealizationsoftype:twoisolatedvertices:case:lvertu1v1rvertbackslashp1Sequals1:u1doesnotequalv1andbracesu1v1bracesintersectp1Sisasingletoncontainingu1}.\ref{numberofmatrixrealizationsoftype:twoisolatedvertices:case:lvertu1v1rvertbackslashp1Sequals1:u1doesnotequalv1andbracesu1v1bracesintersectp1Sisasingletoncontainingu1:u1equaltoa1}  proves that then property \ref{item:casesfortype8casethatbothuandvhaveaprojectionmeetingaprojectionofS} is true and therefore there are exactly  
$4\cdot (n-1-c_1)\cdot 2\cdot ((n-1)-2) = 8\cdot (n-1-c_1)\cdot ((n-1)-2)$ 
realizations of type \ref{item:comparisonproofcaseLhassize6:oneadditionaledgeintersectingC4andoneisolatedvertex} and $0$ realizations of type \ref{item:comparisonproofcaseLhassize6:C4withoneadditionaldisjointedge} by $B\mid_{\{u,v\}}$. 
It follows that if \ref{numberofmatrixrealizationsoftype:threeisolatedvertices:case:lvertu1v1rvertbackslashp1Sequals1}.\ref{numberofmatrixrealizationsoftype:twoisolatedvertices:case:lvertu1v1rvertbackslashp1Sequals1:u1doesnotequalv1andbracesu1v1bracesintersectp1Sisasingleton}.\ref{numberofmatrixrealizationsoftype:twoisolatedvertices:case:lvertu1v1rvertbackslashp1Sequals1:u1doesnotequalv1andbracesu1v1bracesintersectp1Sisasingletoncontainingu1}.\ref{numberofmatrixrealizationsoftype:twoisolatedvertices:case:lvertu1v1rvertbackslashp1Sequals1:u1doesnotequalv1andbracesu1v1bracesintersectp1Sisasingletoncontainingu1:u1equaltoc1}, then  there are exactly 
$2\cdot (n-1-c_1)\cdot ((n-1)-2) + 8\cdot (n-1-c_1)\cdot ((n-1)-2)
=
10\cdot (n-1-c_1)\cdot ((n-1)-2)$
realizations of type \ref{item:comparisonproofcaseLhassize6:oneadditionaledgeintersectingC4andoneisolatedvertex} but only $2\cdot (n-1-c_1)\cdot ((n-1)-2)$ realizations of 
type \ref{item:comparisonproofcaseLhassize6:C4withoneadditionaldisjointedge} by $B\mid_{\{u,v\}}$. 

It now follows that if \ref{numberofmatrixrealizationsoftype:threeisolatedvertices:case:lvertu1v1rvertbackslashp1Sequals1}.\ref{numberofmatrixrealizationsoftype:twoisolatedvertices:case:lvertu1v1rvertbackslashp1Sequals1:u1doesnotequalv1andbracesu1v1bracesintersectp1Sisasingleton}.\ref{numberofmatrixrealizationsoftype:twoisolatedvertices:case:lvertu1v1rvertbackslashp1Sequals1:u1doesnotequalv1andbracesu1v1bracesintersectp1Sisasingletoncontainingu1}, then there are exactly 
$10\cdot (n-1-a_1-1)\cdot ((n-1)-2) + 10\cdot (n-1-c_1)\cdot ((n-1)-2)
= 
10\cdot (2n-a_1-c_1-3) \cdot ((n-1)-2)$
realizations of type \ref{item:comparisonproofcaseLhassize6:oneadditionaledgeintersectingC4andoneisolatedvertex} but only $2\cdot( 2n-a_1-c_1-3 )\cdot( (n-1)-2 )$ of type 
\ref{item:comparisonproofcaseLhassize6:C4withoneadditionaldisjointedge} by 
$B\mid_{\{u,v\}}$.

The case \ref{numberofmatrixrealizationsoftype:threeisolatedvertices:case:lvertu1v1rvertbackslashp1Sequals1}.\ref{numberofmatrixrealizationsoftype:twoisolatedvertices:case:lvertu1v1rvertbackslashp1Sequals1:u1doesnotequalv1andbracesu1v1bracesintersectp1Sisasingleton}.\ref{numberofmatrixrealizationsoftype:twoisolatedvertices:case:lvertu1v1rvertbackslashp1Sequals1:u1doesnotequalv1andbracesu1v1bracesintersectp1Sisasingletoncontainingv1} will now be treated analogously to \ref{numberofmatrixrealizationsoftype:threeisolatedvertices:case:lvertu1v1rvertbackslashp1Sequals1}.\ref{numberofmatrixrealizationsoftype:twoisolatedvertices:case:lvertu1v1rvertbackslashp1Sequals1:u1doesnotequalv1andbracesu1v1bracesintersectp1Sisasingleton}.\ref{numberofmatrixrealizationsoftype:twoisolatedvertices:case:lvertu1v1rvertbackslashp1Sequals1:u1doesnotequalv1andbracesu1v1bracesintersectp1Sisasingletoncontainingu1}.

In the first subcase of \ref{numberofmatrixrealizationsoftype:threeisolatedvertices:case:lvertu1v1rvertbackslashp1Sequals1}.\ref{numberofmatrixrealizationsoftype:twoisolatedvertices:case:lvertu1v1rvertbackslashp1Sequals1:u1doesnotequalv1andbracesu1v1bracesintersectp1Sisasingleton}.\ref{numberofmatrixrealizationsoftype:twoisolatedvertices:case:lvertu1v1rvertbackslashp1Sequals1:u1doesnotequalv1andbracesu1v1bracesintersectp1Sisasingletoncontainingv1}.\ref{numberofmatrixrealizationsoftype:twoisolatedvertices:case:lvertu1v1rvertbackslashp1Sequals1:u1doesnotequalv1andbracesu1v1bracesintersectp1Sisasingletoncontainingu1:v1equaltoa1} we find that  property \ref{item:casesfortype8casethatexactlyoneofuandvhasaprojectionmeetingaprojectionofS} is true and therefore there are exactly 
$2\cdot (a_1-1)\cdot((n-1)-2)$ realizations of type \ref{item:comparisonproofcaseLhassize6:oneadditionaledgeintersectingC4andoneisolatedvertex} and also 
$2\cdot (a_1-1)\cdot((n-1)-2)$ realizations of type \ref{item:comparisonproofcaseLhassize6:C4withoneadditionaldisjointedge} by $B\mid_{\{u,v\}}$. In the 
second subcase of \ref{numberofmatrixrealizationsoftype:threeisolatedvertices:case:lvertu1v1rvertbackslashp1Sequals1}.\ref{numberofmatrixrealizationsoftype:twoisolatedvertices:case:lvertu1v1rvertbackslashp1Sequals1:u1doesnotequalv1andbracesu1v1bracesintersectp1Sisasingleton}.\ref{numberofmatrixrealizationsoftype:twoisolatedvertices:case:lvertu1v1rvertbackslashp1Sequals1:u1doesnotequalv1andbracesu1v1bracesintersectp1Sisasingletoncontainingv1}.\ref{numberofmatrixrealizationsoftype:twoisolatedvertices:case:lvertu1v1rvertbackslashp1Sequals1:u1doesnotequalv1andbracesu1v1bracesintersectp1Sisasingletoncontainingu1:v1equaltoa1} we find that property \ref{item:casesfortype8casethatbothuandvhaveaprojectionmeetingaprojectionofS} is true and therefore there are  exactly 
$4\cdot (a_1-1)\cdot 2\cdot ((n-1)-2) = 8\cdot (a_1-1)\cdot ((n-1)-2)$  
realizations of type \ref{item:comparisonproofcaseLhassize6:oneadditionaledgeintersectingC4andoneisolatedvertex} but $0$ realizations of type \ref{item:comparisonproofcaseLhassize6:C4withoneadditionaldisjointedge} by $B\mid_{\{u,v\}}$. Therefore, if \ref{numberofmatrixrealizationsoftype:threeisolatedvertices:case:lvertu1v1rvertbackslashp1Sequals1}.\ref{numberofmatrixrealizationsoftype:twoisolatedvertices:case:lvertu1v1rvertbackslashp1Sequals1:u1doesnotequalv1andbracesu1v1bracesintersectp1Sisasingleton}.\ref{numberofmatrixrealizationsoftype:twoisolatedvertices:case:lvertu1v1rvertbackslashp1Sequals1:u1doesnotequalv1andbracesu1v1bracesintersectp1Sisasingletoncontainingv1}.\ref{numberofmatrixrealizationsoftype:twoisolatedvertices:case:lvertu1v1rvertbackslashp1Sequals1:u1doesnotequalv1andbracesu1v1bracesintersectp1Sisasingletoncontainingu1:v1equaltoa1}, then there 
are exactly $2\cdot (a_1-1)\cdot((n-1)-2) + 8\cdot (a_1-1)\cdot ((n-1)-2) 
= 10\cdot (a_1-1)\cdot((n-1)-2)$ realizations of type \ref{item:comparisonproofcaseLhassize6:oneadditionaledgeintersectingC4andoneisolatedvertex} but only 
$2\cdot (a_1-1)\cdot ( (n-1)-2)$ realizations of type \ref{item:comparisonproofcaseLhassize6:C4withoneadditionaldisjointedge} by $B\mid_{\{u,v\}}$. 

In the first subcase of \ref{numberofmatrixrealizationsoftype:threeisolatedvertices:case:lvertu1v1rvertbackslashp1Sequals1}.\ref{numberofmatrixrealizationsoftype:twoisolatedvertices:case:lvertu1v1rvertbackslashp1Sequals1:u1doesnotequalv1andbracesu1v1bracesintersectp1Sisasingleton}.\ref{numberofmatrixrealizationsoftype:twoisolatedvertices:case:lvertu1v1rvertbackslashp1Sequals1:u1doesnotequalv1andbracesu1v1bracesintersectp1Sisasingletoncontainingv1}.\ref{numberofmatrixrealizationsoftype:twoisolatedvertices:case:lvertu1v1rvertbackslashp1Sequals1:u1doesnotequalv1andbracesu1v1bracesintersectp1Sisasingletoncontainingu1:v1equaltoc1} we conclude that property \ref{item:casesfortype8casethatexactlyoneofuandvhasaprojectionmeetingaprojectionofS} is true and therefore there exist  
exactly $2\cdot (c_1-1-1)\cdot ((n-1)-2)$ realizations of type \ref{item:comparisonproofcaseLhassize6:oneadditionaledgeintersectingC4andoneisolatedvertex} and also 
$2\cdot (c_1-1-1)\cdot ((n-1)-2)$ realizations of type \ref{item:comparisonproofcaseLhassize6:C4withoneadditionaldisjointedge} by $B\mid_{\{u,v\}}$. 
In the  second subcase of \ref{numberofmatrixrealizationsoftype:threeisolatedvertices:case:lvertu1v1rvertbackslashp1Sequals1}.\ref{numberofmatrixrealizationsoftype:twoisolatedvertices:case:lvertu1v1rvertbackslashp1Sequals1:u1doesnotequalv1andbracesu1v1bracesintersectp1Sisasingleton}.\ref{numberofmatrixrealizationsoftype:twoisolatedvertices:case:lvertu1v1rvertbackslashp1Sequals1:u1doesnotequalv1andbracesu1v1bracesintersectp1Sisasingletoncontainingv1}.\ref{numberofmatrixrealizationsoftype:twoisolatedvertices:case:lvertu1v1rvertbackslashp1Sequals1:u1doesnotequalv1andbracesu1v1bracesintersectp1Sisasingletoncontainingu1:v1equaltoc1} we conclude that property \ref{item:casesfortype8casethatbothuandvhaveaprojectionmeetingaprojectionofS} is true and therefore there are exactly 
$4\cdot (c_1-1-1)\cdot 2 \cdot ((n-1)-2) = 8\cdot (c_1-1-1)\cdot ((n-1)-2)$  
realizations of type \ref{item:comparisonproofcaseLhassize6:oneadditionaledgeintersectingC4andoneisolatedvertex} and $0$ realizations of type \ref{item:comparisonproofcaseLhassize6:C4withoneadditionaldisjointedge} by $B\mid_{\{u,v\}}$. Therefore, if \ref{numberofmatrixrealizationsoftype:threeisolatedvertices:case:lvertu1v1rvertbackslashp1Sequals1}.\ref{numberofmatrixrealizationsoftype:twoisolatedvertices:case:lvertu1v1rvertbackslashp1Sequals1:u1doesnotequalv1andbracesu1v1bracesintersectp1Sisasingleton}.\ref{numberofmatrixrealizationsoftype:twoisolatedvertices:case:lvertu1v1rvertbackslashp1Sequals1:u1doesnotequalv1andbracesu1v1bracesintersectp1Sisasingletoncontainingv1}.\ref{numberofmatrixrealizationsoftype:twoisolatedvertices:case:lvertu1v1rvertbackslashp1Sequals1:u1doesnotequalv1andbracesu1v1bracesintersectp1Sisasingletoncontainingu1:v1equaltoc1}, then there are exactly 
$2\cdot (c_1-1-1)\cdot ((n-1)-2) + 8\cdot (c_1-1-1)\cdot ((n-1)-2)
= 10\cdot (c_1-1-1)\cdot ((n-1)-2)$ realizations of 
type \ref{item:comparisonproofcaseLhassize6:oneadditionaledgeintersectingC4andoneisolatedvertex} but only $2\cdot (c_1-1-1)\cdot ((n-1)-2)$ realizations of type 
\ref{item:comparisonproofcaseLhassize6:C4withoneadditionaldisjointedge} by 
$B\mid_{\{u,v\}}$. 

It now follows that if \ref{numberofmatrixrealizationsoftype:threeisolatedvertices:case:lvertu1v1rvertbackslashp1Sequals1}.\ref{numberofmatrixrealizationsoftype:twoisolatedvertices:case:lvertu1v1rvertbackslashp1Sequals1:u1doesnotequalv1andbracesu1v1bracesintersectp1Sisasingleton}.\ref{numberofmatrixrealizationsoftype:twoisolatedvertices:case:lvertu1v1rvertbackslashp1Sequals1:u1doesnotequalv1andbracesu1v1bracesintersectp1Sisasingletoncontainingv1}, then there are exactly 
$10\cdot (a_1-1)\cdot((n-1)-2) + 10\cdot (c_1-1-1)\cdot ((n-1)-2)
= 10\cdot (a_1+c_1-3) \cdot ((n-1)-2)$
realizations of type \ref{item:comparisonproofcaseLhassize6:oneadditionaledgeintersectingC4andoneisolatedvertex} but merely 
$2 \cdot (a_1-1)\cdot((n-1)-2) + 2 \cdot (c_1-1-1)\cdot ((n-1)-2)
= 2 \cdot (a_1+c_1-3) \cdot ((n-1)-2)$ realizations of type \ref{item:comparisonproofcaseLhassize6:C4withoneadditionaldisjointedge} by $B\mid_{\{u,v\}}$. Moreover we may now 
conclude that if \ref{numberofmatrixrealizationsoftype:threeisolatedvertices:case:lvertu1v1rvertbackslashp1Sequals1}.\ref{numberofmatrixrealizationsoftype:twoisolatedvertices:case:lvertu1v1rvertbackslashp1Sequals1:u1doesnotequalv1andbracesu1v1bracesintersectp1Sisasingleton}, then there are exactly 
$10\cdot (2n-a_1-c_1-3) \cdot ((n-1)-2) + 10\cdot (a_1+c_1-3) \cdot ((n-1)-2)
=
20\cdot ((n-1)-2)^2$ realizations of type \ref{item:comparisonproofcaseLhassize6:oneadditionaledgeintersectingC4andoneisolatedvertex} but only 
$2 \cdot (2n-a_1-c_1-3) \cdot ((n-1)-2) + 2 \cdot (a_1+c_1-3) \cdot ((n-1)-2)
=
4 \cdot ((n-1)-2)^2$ realizations of type \ref{item:comparisonproofcaseLhassize6:C4withoneadditionaldisjointedge} by $B\mid_{\{u,v\}}$. Finally we can conclude that 
if \ref{numberofmatrixrealizationsoftype:threeisolatedvertices:case:lvertu1v1rvertbackslashp1Sequals1}, then there are exactly 
$4\cdot ((n-1)-2)^2 + 20\cdot ((n-1)-2)^2 = 24\cdot ((n-1)-2)^2$ realizations 
of type \ref{item:comparisonproofcaseLhassize6:oneadditionaledgeintersectingC4andoneisolatedvertex} but only $4\cdot ((n-1)-2)^2 + 4\cdot((n-1)-2)^2 = 8\cdot ((n-1)-2)^2$
realizations of type \ref{item:comparisonproofcaseLhassize6:C4withoneadditionaldisjointedge} by $B\mid_{\{u,v\}}$.

If \ref{numberofmatrixrealizationsoftype:twoisolatedvertices:case:lvertu1v1rvertbackslashp1Sequals2}, then the inclusion $\{u_2,v_2\}\subseteq\upp_2(S)$  alone implies that 
property \ref{item:casesfortype8casethatbothuandvhaveaprojectionmeetingaprojectionofS}
is true and therefore there are exactly 
$4\cdot 4\cdot \binom{(n-1)-2}{2} = 16\cdot \binom{(n-1)-2}{2}$ 
realizations of type \ref{item:comparisonproofcaseLhassize6:oneadditionaledgeintersectingC4andoneisolatedvertex} but $0$ realizations of type \ref{item:comparisonproofcaseLhassize6:C4withoneadditionaldisjointedge} by $B\mid_{\{u,v\}}$. 

Summing up, it follows that for each fixed $S$ there are exactly 
$2\cdot 16\cdot \binom{(n-1)-2}{2} + 24\cdot ((n-1)-2)^2 = 24\cdot(n-3)^2 + 32\cdot\binom{n-3}{2}$ realizations of type \ref{item:comparisonproofcaseLhassize6:oneadditionaledgeintersectingC4andoneisolatedvertex} but only $8\cdot ((n-1)-2)^2$ realizations of type \ref{item:comparisonproofcaseLhassize6:C4withoneadditionaldisjointedge} by $B\mid_{\{u,v\}}$. This completes the proof of both \ref{numberofmatrixrealizationsoftype:oneadditionaledgeintersectingC4andoneisolatedvertex} and \ref{numberofmatrixrealizationsoftype:C4withoneadditionaldisjointedge}.

We can now turn to counting the realizations of \ref{item:comparisonproofcaseLhassize6:C4intersectingtwoedgesinseparatenonadjacentvertices}--\ref{item:comparisonproofcaseLhassize6:C4intersectingatwopathinitsinnervertex}, i.e. to proving \ref{numberofmatrixrealizationsoftype:C4intersectingtwoedgesinseparatenonadjacentvertices}--\ref{numberofmatrixrealizationsoftype:C4intersectingatwopathinitsinnervertex}

By \ref{characterizationoftype:C4intersectingtwoedgesinseparatenonadjacentvertices:property1}, \ref{characterizationoftype:C4intersectingtwoedgesinseparateadjacentvertices:property1}, \ref{characterizationoftype:C4intersectingatwopathinanendvertex:property1} and \ref{characterizationoftype:C4intersectingatwopathinitsinnervertex:property1}, for each of the four types \ref{item:comparisonproofcaseLhassize6:C4intersectingtwoedgesinseparatenonadjacentvertices}, \ref{item:comparisonproofcaseLhassize6:C4intersectingtwoedgesinseparateadjacentvertices}, \ref{item:comparisonproofcaseLhassize6:C4intersectingatwopathinanendvertex} and \ref{item:comparisonproofcaseLhassize6:C4intersectingatwopathinitsinnervertex} it is necessary that $\upX_{\{0\}^{u,v}\sqcup B\mid_{S}} \cong \ref{item:comparisonproofcaseLhassize6:twoisolatedvertices}$. 
We may therefore determine 
each of the four functions \ref{numberofmatrixrealizationsoftype:C4intersectingtwoedgesinseparatenonadjacentvertices}, \ref{numberofmatrixrealizationsoftype:C4intersectingtwoedgesinseparateadjacentvertices}, \ref{numberofmatrixrealizationsoftype:C4intersectingatwopathinanendvertex}, \ref{numberofmatrixrealizationsoftype:C4intersectingatwopathinitsinnervertex} in the course of one reexamination of the proof of 
\ref{numberofmatrixrealizationsoftype:twoisolatedvertices}. 
We consider each of the cases in turn, each time descending down just deep enough 
until we are able to decide which of the four isomorphism types \ref{item:comparisonproofcaseLhassize6:C4intersectingtwoedgesinseparatenonadjacentvertices}, \ref{item:comparisonproofcaseLhassize6:C4intersectingtwoedgesinseparateadjacentvertices}, \ref{item:comparisonproofcaseLhassize6:C4intersectingatwopathinanendvertex} can be realized in 
that case.

Since by \ref{characterizationoftype:C4intersectingtwoedgesinseparatenonadjacentvertices:property2}, \ref{characterizationoftype:C4intersectingtwoedgesinseparateadjacentvertices:property2}, \ref{characterizationoftype:C4intersectingatwopathinanendvertex:property2} and \ref{characterizationoftype:C4intersectingatwopathinitsinnervertex:property2} the property  $B[u]\in\{\pm\}$ and $B[v]\in \{\pm\}$ is necessary for each of 
the four types, the positions $u$ and $v$ alone, not $B\mid_{\{u,v\}}$ itself, decide 
about which type can be realized. Therefore, if a decision is reached about which 
of the four types can be realized in a case, then we obtain the number of 
realizations by multiplying the number of realizations of 
type \ref{item:comparisonproofcaseLhassize6:twoisolatedvertices} in that particular 
case by $4$. 

We now reexamine \ref{numberofmatrixrealizationsoftype:twoisolatedvertices:case:lvertu1v1rvertbackslashp1Sequals0}. 
If \ref{numberofmatrixrealizationsoftype:twoisolatedvertices:case:lvertu1v1rvertbackslashp1Sequals0}, then the property $\lvert \{u_1,v_1\}\setminus \upp_1(S)\rvert = 0$ 
makes \ref{characterizationoftype:C4intersectingatwopathinanendvertex:property4} 
impossible, hence in the entire case \ref{numberofmatrixrealizationsoftype:twoisolatedvertices:case:lvertu1v1rvertbackslashp1Sequals0} the type \ref{item:comparisonproofcaseLhassize6:C4intersectingatwopathinanendvertex} is impossible. 

If \ref{numberofmatrixrealizationsoftype:twoisolatedvertices:case:lvertu1v1rvertbackslashp1Sequals0}.\ref{numberofmatrixrealizationsoftype:twoisolatedvertices:case:lvertu1v1rvertbackslashp1Sequals0:case:u1equalsv1}, then we have $u_1=v_1$, which 
makes both \ref{characterizationoftype:C4intersectingtwoedgesinseparateadjacentvertices:property3} and \ref{characterizationoftype:C4intersectingtwoedgesinseparatenonadjacentvertices:property3} impossible. The only type remaining is \ref{item:comparisonproofcaseLhassize6:C4intersectingatwopathinitsinnervertex} (and all the properties \ref{characterizationoftype:C4intersectingatwopathinitsinnervertex:property1}--\ref{characterizationoftype:C4intersectingatwopathinitsinnervertex:property4} are indeed satisfied). It 
follows that if \ref{numberofmatrixrealizationsoftype:twoisolatedvertices:case:lvertu1v1rvertbackslashp1Sequals0}.\ref{numberofmatrixrealizationsoftype:twoisolatedvertices:case:lvertu1v1rvertbackslashp1Sequals0:case:u1equalsv1}, then  there are exactly 
$8\cdot\binom{(n-1)-2}{2}$ realizations of type \ref{item:comparisonproofcaseLhassize6:C4intersectingatwopathinitsinnervertex} by $B\mid_{\{u,v\}}$.

If \ref{numberofmatrixrealizationsoftype:twoisolatedvertices:case:lvertu1v1rvertbackslashp1Sequals0}.\ref{numberofmatrixrealizationsoftype:twoisolatedvertices:case:lvertu1v1rvertbackslashp1Sequals0:case:u1doesnotequalv1}, then we have $u_1\neq v_1$. 
Since we also have $u_2\neq v_2$ throughout \ref{numberofmatrixrealizationsoftype:twoisolatedvertices:case:lvertu1v1rvertbackslashp1Sequals0}, this 
makes \ref{characterizationoftype:C4intersectingatwopathinitsinnervertex:property3}
impossible. Moreover, since throughout \ref{numberofmatrixrealizationsoftype:twoisolatedvertices:case:lvertu1v1rvertbackslashp1Sequals0} we also have \eqref{eq:proofofcomparativecountingtheoremcase2firsteq}, in particular 
$\{ u_2, v_2 \} \cap \upp_2(S) = \emptyset$, it follows that \ref{characterizationoftype:C4intersectingtwoedgesinseparateadjacentvertices:property4} is impossible. The only 
type remaining is \ref{item:comparisonproofcaseLhassize6:C4intersectingtwoedgesinseparatenonadjacentvertices} (and all the properties \ref{characterizationoftype:C4intersectingtwoedgesinseparatenonadjacentvertices:property1}--\ref{characterizationoftype:C4intersectingtwoedgesinseparatenonadjacentvertices:property4} are indeed satisfied; 
notice in particular that in \ref{characterizationoftype:C4intersectingtwoedgesinseparatenonadjacentvertices:property4} the first of the two mutually exclusive clauses 
is true). If follows that if \ref{numberofmatrixrealizationsoftype:twoisolatedvertices:case:lvertu1v1rvertbackslashp1Sequals0}.\ref{numberofmatrixrealizationsoftype:twoisolatedvertices:case:lvertu1v1rvertbackslashp1Sequals0:case:u1doesnotequalv1}, then 
there are exactly $8\cdot\binom{(n-1)-2}{2}$ realizations of 
type \ref{item:comparisonproofcaseLhassize6:C4intersectingtwoedgesinseparatenonadjacentvertices} by $B\mid_{\{u,v\}}$. This completes our reexamination of \ref{numberofmatrixrealizationsoftype:twoisolatedvertices:case:lvertu1v1rvertbackslashp1Sequals0}.

We now reexamine \ref{numberofmatrixrealizationsoftype:twoisolatedvertices:case:lvertu1v1rvertbackslashp1Sequals1}. The information defining \ref{numberofmatrixrealizationsoftype:twoisolatedvertices:case:lvertu1v1rvertbackslashp1Sequals1} is by itself 
not yet sufficient to rule out any of the four types \ref{item:comparisonproofcaseLhassize6:C4intersectingtwoedgesinseparatenonadjacentvertices}--\ref{item:comparisonproofcaseLhassize6:C4intersectingatwopathinitsinnervertex}.  The information defining \ref{numberofmatrixrealizationsoftype:twoisolatedvertices:case:lvertu1v1rvertbackslashp1Sequals1}.\ref{numberofmatrixrealizationsoftype:twoisolatedvertices:case:lvertu1v1rvertbackslashp1Sequals1:u1equalsv1andbracesu1v1bracesintersectp1Sisempty}, more 
specificly $u_1=v_1$, makes both \ref{characterizationoftype:C4intersectingtwoedgesinseparatenonadjacentvertices:property3} and \ref{characterizationoftype:C4intersectingtwoedgesinseparateadjacentvertices:property3} impossible, still leaving two types. 
We argued in \ref{numberofmatrixrealizationsoftype:twoisolatedvertices:case:lvertu1v1rvertbackslashp1Sequals1}.\ref{numberofmatrixrealizationsoftype:twoisolatedvertices:case:lvertu1v1rvertbackslashp1Sequals1:u1equalsv1andbracesu1v1bracesintersectp1Sisempty} that we have $u_2\neq v_2$ and $\lvert\{u_2,v_2\}\cap\upp_2(S)\rvert = 1$ in this 
case. If the latter is true as $u_2\in\upp_2(S)$, then $v_2\notin\upp_2(S)$, and 
combining this information with  $v_1\notin\upp_1(S)$ (which we know since we are 
in \ref{numberofmatrixrealizationsoftype:twoisolatedvertices:case:lvertu1v1rvertbackslashp1Sequals1}.\ref{numberofmatrixrealizationsoftype:twoisolatedvertices:case:lvertu1v1rvertbackslashp1Sequals1:u1equalsv1andbracesu1v1bracesintersectp1Sisempty}) makes the 
second clause of the conjunction \ref{characterizationoftype:C4intersectingatwopathinitsinnervertex:property4} impossible. If on the other hand it is true 
as $v_2\in\upp_2(S)$, then $u_2\notin\upp_2(S)$, and  combining this 
with $u_1\notin\upp_1(S)$ (which again we know since we are in \ref{numberofmatrixrealizationsoftype:twoisolatedvertices:case:lvertu1v1rvertbackslashp1Sequals1}.\ref{numberofmatrixrealizationsoftype:twoisolatedvertices:case:lvertu1v1rvertbackslashp1Sequals1:u1equalsv1andbracesu1v1bracesintersectp1Sisempty}) makes \ref{characterizationoftype:C4intersectingatwopathinitsinnervertex:property4} impossible (this time, the first clause). This rules out type \ref{item:comparisonproofcaseLhassize6:C4intersectingatwopathinitsinnervertex}. 
The only type remaining is \ref{item:comparisonproofcaseLhassize6:C4intersectingatwopathinanendvertex} (and all the properties \ref{characterizationoftype:C4intersectingatwopathinanendvertex:property1}--\ref{characterizationoftype:C4intersectingatwopathinanendvertex:property4} are indeed satisfied; note that due 
to $\lvert \{u_2,v_2\}\cap\upp_2(S)\rvert = 1$ the two clauses of \ref{characterizationoftype:C4intersectingatwopathinanendvertex:property4} are mutually exclusive, the 
second being true if $u_2\in\upp_2(S)$ and the first if $v_2\in\upp_2(S)$).
It follows that in case \ref{numberofmatrixrealizationsoftype:twoisolatedvertices:case:lvertu1v1rvertbackslashp1Sequals1:u1equalsv1andbracesu1v1bracesintersectp1Sisempty} of \ref{numberofmatrixrealizationsoftype:twoisolatedvertices:case:lvertu1v1rvertbackslashp1Sequals1} there are exactly $8\cdot ((n-1)-2)^2$ realizations of type \ref{item:comparisonproofcaseLhassize6:C4intersectingatwopathinanendvertex} by $B\mid_{\{u,v\}}$. 

The information defining \ref{numberofmatrixrealizationsoftype:twoisolatedvertices:case:lvertu1v1rvertbackslashp1Sequals1}.\ref{numberofmatrixrealizationsoftype:twoisolatedvertices:case:lvertu1v1rvertbackslashp1Sequals1:u1doesnotequalv1andbracesu1v1bracesintersectp1Sisasingleton} is not enough to rule out any of the four types 
\ref{item:comparisonproofcaseLhassize6:C4intersectingtwoedgesinseparatenonadjacentvertices}--\ref{item:comparisonproofcaseLhassize6:C4intersectingatwopathinitsinnervertex}, 
and descending one level deeper to \ref{numberofmatrixrealizationsoftype:twoisolatedvertices:case:lvertu1v1rvertbackslashp1Sequals1}.\ref{numberofmatrixrealizationsoftype:twoisolatedvertices:case:lvertu1v1rvertbackslashp1Sequals1:u1doesnotequalv1andbracesu1v1bracesintersectp1Sisasingleton}.\ref{numberofmatrixrealizationsoftype:twoisolatedvertices:case:lvertu1v1rvertbackslashp1Sequals1:u1doesnotequalv1andbracesu1v1bracesintersectp1Sisasingletoncontainingu1} does not change this.  

If \ref{numberofmatrixrealizationsoftype:twoisolatedvertices:case:lvertu1v1rvertbackslashp1Sequals1}.\ref{numberofmatrixrealizationsoftype:twoisolatedvertices:case:lvertu1v1rvertbackslashp1Sequals1:u1doesnotequalv1andbracesu1v1bracesintersectp1Sisasingleton}.\ref{numberofmatrixrealizationsoftype:twoisolatedvertices:case:lvertu1v1rvertbackslashp1Sequals1:u1doesnotequalv1andbracesu1v1bracesintersectp1Sisasingletoncontainingu1}.\ref{numberofmatrixrealizationsoftype:twoisolatedvertices:case:lvertu1v1rvertbackslashp1Sequals1:u1doesnotequalv1andbracesu1v1bracesintersectp1Sisasingletoncontainingu1:u1equaltoa1}, then the decision still cannot be made and depends on the (anonymous) subcases 
which we distinguished in that case, namely whether the first or the second clause of 
\eqref{eq:proofofcomparativecountingtheoremcase3secondeq} is true: 

If \ref{numberofmatrixrealizationsoftype:twoisolatedvertices:case:lvertu1v1rvertbackslashp1Sequals1}.\ref{numberofmatrixrealizationsoftype:twoisolatedvertices:case:lvertu1v1rvertbackslashp1Sequals1:u1doesnotequalv1andbracesu1v1bracesintersectp1Sisasingleton}.\ref{numberofmatrixrealizationsoftype:twoisolatedvertices:case:lvertu1v1rvertbackslashp1Sequals1:u1doesnotequalv1andbracesu1v1bracesintersectp1Sisasingletoncontainingu1}.\ref{numberofmatrixrealizationsoftype:twoisolatedvertices:case:lvertu1v1rvertbackslashp1Sequals1:u1doesnotequalv1andbracesu1v1bracesintersectp1Sisasingletoncontainingu1:u1equaltoa1}, and the first clause of \eqref{eq:proofofcomparativecountingtheoremcase3secondeq} 
is true, then in particular we know that  
$v_1\notin\upp_1(S)$ and $u_2=v_2\notin\upp_2(S)$. The latter contradicts \ref{characterizationoftype:C4intersectingtwoedgesinseparatenonadjacentvertices:property3} and \ref{characterizationoftype:C4intersectingtwoedgesinseparateadjacentvertices:property3}. Moreover, 
$v_1\notin\upp_1(S)$ and $v_2\notin\upp_2(S)$ combined render the second clause of the 
conjunction \ref{characterizationoftype:C4intersectingatwopathinitsinnervertex:property4} false. Note that for each of three discarded types we used in whole or in part 
the information $u_2=v_2\notin\upp_2(S)$, which defines the present subcase. 
Hence deferring any decision about the types for so so long was necessary. 
We are now left with only the type 
\ref{item:comparisonproofcaseLhassize6:C4intersectingatwopathinanendvertex} (and indeed the properties \ref{characterizationoftype:C4intersectingatwopathinanendvertex:property1}--\ref{characterizationoftype:C4intersectingatwopathinanendvertex:property4} are all satisfied). Since within \ref{numberofmatrixrealizationsoftype:twoisolatedvertices:case:lvertu1v1rvertbackslashp1Sequals1}.\ref{numberofmatrixrealizationsoftype:twoisolatedvertices:case:lvertu1v1rvertbackslashp1Sequals1:u1doesnotequalv1andbracesu1v1bracesintersectp1Sisasingleton}.\ref{numberofmatrixrealizationsoftype:twoisolatedvertices:case:lvertu1v1rvertbackslashp1Sequals1:u1doesnotequalv1andbracesu1v1bracesintersectp1Sisasingletoncontainingu1}.\ref{numberofmatrixrealizationsoftype:twoisolatedvertices:case:lvertu1v1rvertbackslashp1Sequals1:u1doesnotequalv1andbracesu1v1bracesintersectp1Sisasingletoncontainingu1:u1equaltoa1} we found that in this 
situation there are exactly $(n-1-a_1-1)\cdot ((n-1)-2)$ realizations of type 
\ref{item:comparisonproofcaseLhassize6:twoisolatedvertices} by $B\mid_{\{u,v\}}$, it follows 
that here there are exactly $4\cdot (n-1-a_1-1)\cdot ((n-1)-2)$ realizations of type 
\ref{item:comparisonproofcaseLhassize6:C4intersectingatwopathinanendvertex}. 

If \ref{numberofmatrixrealizationsoftype:twoisolatedvertices:case:lvertu1v1rvertbackslashp1Sequals1}.\ref{numberofmatrixrealizationsoftype:twoisolatedvertices:case:lvertu1v1rvertbackslashp1Sequals1:u1doesnotequalv1andbracesu1v1bracesintersectp1Sisasingleton}.\ref{numberofmatrixrealizationsoftype:twoisolatedvertices:case:lvertu1v1rvertbackslashp1Sequals1:u1doesnotequalv1andbracesu1v1bracesintersectp1Sisasingletoncontainingu1}.\ref{numberofmatrixrealizationsoftype:twoisolatedvertices:case:lvertu1v1rvertbackslashp1Sequals1:u1doesnotequalv1andbracesu1v1bracesintersectp1Sisasingletoncontainingu1:u1equaltoa1} and the second clause of \eqref{eq:proofofcomparativecountingtheoremcase3secondeq} 
is true, then we know that $u_1\neq v_1$, $u_1 \in\upp_1(S)$,
$v_1\notin\upp_1(S)$, $u_2\neq v_2$, $u_2\notin\upp_2(S)$ and $v_2\in\upp_2(S)$. 
We can now rule out three types: properties $v_1\notin\upp_1(S)$ and $u_2\notin\upp_2(S)$ combined 
render both clauses of the disjunction \ref{characterizationoftype:C4intersectingtwoedgesinseparatenonadjacentvertices:property4} false. 
Properties $u_1\in\upp_1(S)$ and $v_2\in\upp_2(S)$ combined render both clauses of 
the disjunction \ref{characterizationoftype:C4intersectingatwopathinanendvertex:property4} false. 
Properties $u_1\neq v_1$ and $u_2\neq v_2$ proves  \ref{characterizationoftype:C4intersectingatwopathinitsinnervertex:property3} to be false. 
Again note that in all three decisions we used the information defining the present 
subcase. The only type remaining now is \ref{item:comparisonproofcaseLhassize6:C4intersectingtwoedgesinseparateadjacentvertices}, and indeed all properties \ref{characterizationoftype:C4intersectingtwoedgesinseparateadjacentvertices:property1}--\ref{characterizationoftype:C4intersectingtwoedgesinseparateadjacentvertices:property4} are 
satisfied (in \ref{characterizationoftype:C4intersectingtwoedgesinseparateadjacentvertices:property4} only the first clause of the disjunction). Since in \ref{numberofmatrixrealizationsoftype:twoisolatedvertices:case:lvertu1v1rvertbackslashp1Sequals1}.\ref{numberofmatrixrealizationsoftype:twoisolatedvertices:case:lvertu1v1rvertbackslashp1Sequals1:u1doesnotequalv1andbracesu1v1bracesintersectp1Sisasingleton}.\ref{numberofmatrixrealizationsoftype:twoisolatedvertices:case:lvertu1v1rvertbackslashp1Sequals1:u1doesnotequalv1andbracesu1v1bracesintersectp1Sisasingletoncontainingu1}.\ref{numberofmatrixrealizationsoftype:twoisolatedvertices:case:lvertu1v1rvertbackslashp1Sequals1:u1doesnotequalv1andbracesu1v1bracesintersectp1Sisasingletoncontainingu1:u1equaltoa1} we found that in this situation 
there are exactly $(n-1-a_1-1)\cdot 2\cdot ((n-1)-2)$ realizations of 
type \ref{item:comparisonproofcaseLhassize6:twoisolatedvertices} by $B\mid_{\{u,v\}}$, 
it follows that here there are exactly $8\cdot (n-1-a_1-1)\cdot ((n-1)-2)$ 
realizations of type \ref{item:comparisonproofcaseLhassize6:C4intersectingtwoedgesinseparateadjacentvertices}. 

If \ref{numberofmatrixrealizationsoftype:twoisolatedvertices:case:lvertu1v1rvertbackslashp1Sequals1}.\ref{numberofmatrixrealizationsoftype:twoisolatedvertices:case:lvertu1v1rvertbackslashp1Sequals1:u1doesnotequalv1andbracesu1v1bracesintersectp1Sisasingleton}.\ref{numberofmatrixrealizationsoftype:twoisolatedvertices:case:lvertu1v1rvertbackslashp1Sequals1:u1doesnotequalv1andbracesu1v1bracesintersectp1Sisasingletoncontainingu1}.\ref{numberofmatrixrealizationsoftype:twoisolatedvertices:case:lvertu1v1rvertbackslashp1Sequals1:u1doesnotequalv1andbracesu1v1bracesintersectp1Sisasingletoncontainingu1:u1equaltoc1}, then again we have to distinguish whether the first or the second clause of 
\eqref{eq:proofofcomparativecountingtheoremcase3secondeq} is true to reach a 
conclusion:

If \ref{numberofmatrixrealizationsoftype:twoisolatedvertices:case:lvertu1v1rvertbackslashp1Sequals1}.\ref{numberofmatrixrealizationsoftype:twoisolatedvertices:case:lvertu1v1rvertbackslashp1Sequals1:u1doesnotequalv1andbracesu1v1bracesintersectp1Sisasingleton}.\ref{numberofmatrixrealizationsoftype:twoisolatedvertices:case:lvertu1v1rvertbackslashp1Sequals1:u1doesnotequalv1andbracesu1v1bracesintersectp1Sisasingletoncontainingu1}.\ref{numberofmatrixrealizationsoftype:twoisolatedvertices:case:lvertu1v1rvertbackslashp1Sequals1:u1doesnotequalv1andbracesu1v1bracesintersectp1Sisasingletoncontainingu1:u1equaltoc1} and the first clause of \eqref{eq:proofofcomparativecountingtheoremcase3secondeq} 
is true, then again we in particular know that 
$v_1\notin\upp_1(S)$ and $u_2=v_2\notin\upp_2(S)$, and therefore an argument 
analogous to the one given for the first subcase of \ref{numberofmatrixrealizationsoftype:twoisolatedvertices:case:lvertu1v1rvertbackslashp1Sequals1}.\ref{numberofmatrixrealizationsoftype:twoisolatedvertices:case:lvertu1v1rvertbackslashp1Sequals1:u1doesnotequalv1andbracesu1v1bracesintersectp1Sisasingleton}.\ref{numberofmatrixrealizationsoftype:twoisolatedvertices:case:lvertu1v1rvertbackslashp1Sequals1:u1doesnotequalv1andbracesu1v1bracesintersectp1Sisasingletoncontainingu1}.\ref{numberofmatrixrealizationsoftype:twoisolatedvertices:case:lvertu1v1rvertbackslashp1Sequals1:u1doesnotequalv1andbracesu1v1bracesintersectp1Sisasingletoncontainingu1:u1equaltoa1} shows that in the present 
situation there are exactly $4\cdot (n-1-c_1)\cdot((n-1)-2)$ realizations of 
type \ref{item:comparisonproofcaseLhassize6:C4intersectingatwopathinanendvertex}
 by $B\mid_{\{u,v\}}$ and no other type possible. 

If \ref{numberofmatrixrealizationsoftype:twoisolatedvertices:case:lvertu1v1rvertbackslashp1Sequals1}.\ref{numberofmatrixrealizationsoftype:twoisolatedvertices:case:lvertu1v1rvertbackslashp1Sequals1:u1doesnotequalv1andbracesu1v1bracesintersectp1Sisasingleton}.\ref{numberofmatrixrealizationsoftype:twoisolatedvertices:case:lvertu1v1rvertbackslashp1Sequals1:u1doesnotequalv1andbracesu1v1bracesintersectp1Sisasingletoncontainingu1}.\ref{numberofmatrixrealizationsoftype:twoisolatedvertices:case:lvertu1v1rvertbackslashp1Sequals1:u1doesnotequalv1andbracesu1v1bracesintersectp1Sisasingletoncontainingu1:u1equaltoc1} and the second clause of \eqref{eq:proofofcomparativecountingtheoremcase3secondeq} 
is true, then again we know that $u_1\neq v_1$, $u_1\in\upp_1(S)$, 
$v_1\notin\upp_1(S)$ and $u_2\neq v_2$, 
but this time we have $u_2\notin \upp_2(S)$ and $v_2\in\upp_2(S)$. We can now 
rule out three types: properties 
$v_1\notin\upp_1(S)$ and $u_2\notin\upp_2(S)$ combined 
render both clauses of the disjunction \ref{characterizationoftype:C4intersectingtwoedgesinseparatenonadjacentvertices:property4} false.
Properties $u_1\in\upp_1(S)$  and $v_2\in\upp_2(S)$ combined render both clauses of 
the disjunction \ref{characterizationoftype:C4intersectingatwopathinanendvertex:property4} false. Properties $u_1\neq v_1$ and $u_2\neq v_2$ contradict \ref{characterizationoftype:C4intersectingatwopathinitsinnervertex:property3}. What 
remains is type \ref{item:comparisonproofcaseLhassize6:C4intersectingtwoedgesinseparateadjacentvertices}, and all properties \ref{characterizationoftype:C4intersectingtwoedgesinseparateadjacentvertices:property1}--\ref{characterizationoftype:C4intersectingtwoedgesinseparateadjacentvertices:property4} are satisfied (in \ref{characterizationoftype:C4intersectingtwoedgesinseparateadjacentvertices:property4} only the first clause of the disjunction). Since in \ref{numberofmatrixrealizationsoftype:twoisolatedvertices:case:lvertu1v1rvertbackslashp1Sequals1}.\ref{numberofmatrixrealizationsoftype:twoisolatedvertices:case:lvertu1v1rvertbackslashp1Sequals1:u1doesnotequalv1andbracesu1v1bracesintersectp1Sisasingleton}.\ref{numberofmatrixrealizationsoftype:twoisolatedvertices:case:lvertu1v1rvertbackslashp1Sequals1:u1doesnotequalv1andbracesu1v1bracesintersectp1Sisasingletoncontainingu1}.\ref{numberofmatrixrealizationsoftype:twoisolatedvertices:case:lvertu1v1rvertbackslashp1Sequals1:u1doesnotequalv1andbracesu1v1bracesintersectp1Sisasingletoncontainingu1:u1equaltoc1} we found that in this 
situation there are exactly $2\cdot (n-1-c_1)\cdot ((n-1)-2)$ realizations of 
type \ref{item:comparisonproofcaseLhassize6:twoisolatedvertices} by $B\mid_{\{u,v\}}$, 
it follows that here there are exactly $8\cdot (n-1-c_1)\cdot ((n-1)-2)$ 
realizations of type \ref{item:comparisonproofcaseLhassize6:C4intersectingtwoedgesinseparateadjacentvertices}.  

The next case to reexamine is \ref{numberofmatrixrealizationsoftype:twoisolatedvertices:case:lvertu1v1rvertbackslashp1Sequals1}.\ref{numberofmatrixrealizationsoftype:twoisolatedvertices:case:lvertu1v1rvertbackslashp1Sequals1:u1doesnotequalv1andbracesu1v1bracesintersectp1Sisasingleton}.\ref{numberofmatrixrealizationsoftype:twoisolatedvertices:case:lvertu1v1rvertbackslashp1Sequals1:u1doesnotequalv1andbracesu1v1bracesintersectp1Sisasingletoncontainingv1} which again does not give enough information to decide about 
the types. 

If \ref{numberofmatrixrealizationsoftype:twoisolatedvertices:case:lvertu1v1rvertbackslashp1Sequals1}.\ref{numberofmatrixrealizationsoftype:twoisolatedvertices:case:lvertu1v1rvertbackslashp1Sequals1:u1doesnotequalv1andbracesu1v1bracesintersectp1Sisasingleton}.\ref{numberofmatrixrealizationsoftype:twoisolatedvertices:case:lvertu1v1rvertbackslashp1Sequals1:u1doesnotequalv1andbracesu1v1bracesintersectp1Sisasingletoncontainingv1}.\ref{numberofmatrixrealizationsoftype:twoisolatedvertices:case:lvertu1v1rvertbackslashp1Sequals1:u1doesnotequalv1andbracesu1v1bracesintersectp1Sisasingletoncontainingu1:v1equaltoa1}, then the decision still cannot be made and once more depends on the nameless 
subcases that were distinguished in that case, namely whether the first or the 
second clause of \eqref{eq:proofofcomparativecountingtheoremcase3secondeq} is true:

If \ref{numberofmatrixrealizationsoftype:twoisolatedvertices:case:lvertu1v1rvertbackslashp1Sequals1}.\ref{numberofmatrixrealizationsoftype:twoisolatedvertices:case:lvertu1v1rvertbackslashp1Sequals1:u1doesnotequalv1andbracesu1v1bracesintersectp1Sisasingleton}.\ref{numberofmatrixrealizationsoftype:twoisolatedvertices:case:lvertu1v1rvertbackslashp1Sequals1:u1doesnotequalv1andbracesu1v1bracesintersectp1Sisasingletoncontainingv1}.\ref{numberofmatrixrealizationsoftype:twoisolatedvertices:case:lvertu1v1rvertbackslashp1Sequals1:u1doesnotequalv1andbracesu1v1bracesintersectp1Sisasingletoncontainingu1:v1equaltoa1}, and the first clause of \ref{eq:proofofcomparativecountingtheoremcase3secondeq} is 
true, then we know that $u_1\neq v_1$, $u_1\notin\upp_1(S)$, $v_1\in\upp_1(S)$ and 
$u_2 = v_2 \notin \upp_2(S)$. The latter, more specificly $u_2 = v_2$, contradicts 
both \ref{characterizationoftype:C4intersectingtwoedgesinseparatenonadjacentvertices:property3} and \ref{characterizationoftype:C4intersectingtwoedgesinseparateadjacentvertices:property3}. Combining $u_1\notin\upp_1(S)$ and  $u_2\notin \upp_2(S)$ proves the 
first clause of the conjunction \ref{characterizationoftype:C4intersectingatwopathinitsinnervertex:property4} to be false. The only type remaining is \ref{item:comparisonproofcaseLhassize6:C4intersectingatwopathinanendvertex}, and all properties 
\ref{characterizationoftype:C4intersectingatwopathinanendvertex:property1}--\ref{characterizationoftype:C4intersectingatwopathinanendvertex:property4} are indeed satisfied (with only the first clause of the disjunction \ref{characterizationoftype:C4intersectingatwopathinanendvertex:property4} being true). Since in \ref{numberofmatrixrealizationsoftype:twoisolatedvertices:case:lvertu1v1rvertbackslashp1Sequals1}.\ref{numberofmatrixrealizationsoftype:twoisolatedvertices:case:lvertu1v1rvertbackslashp1Sequals1:u1doesnotequalv1andbracesu1v1bracesintersectp1Sisasingleton}.\ref{numberofmatrixrealizationsoftype:twoisolatedvertices:case:lvertu1v1rvertbackslashp1Sequals1:u1doesnotequalv1andbracesu1v1bracesintersectp1Sisasingletoncontainingv1}.\ref{numberofmatrixrealizationsoftype:twoisolatedvertices:case:lvertu1v1rvertbackslashp1Sequals1:u1doesnotequalv1andbracesu1v1bracesintersectp1Sisasingletoncontainingu1:v1equaltoa1} we found that there are exactly 
$(a_1-1)\cdot ((n-1)-2)$ realizations of type \ref{item:comparisonproofcaseLhassize6:twoisolatedvertices} by $B\mid_{\{u,v\}}$, it follows that in the present situation there are 
exactly $4\cdot (a_1-1)\cdot ((n-1)-2)$ realizations of type \ref{item:comparisonproofcaseLhassize6:C4intersectingatwopathinanendvertex}. 

If \ref{numberofmatrixrealizationsoftype:twoisolatedvertices:case:lvertu1v1rvertbackslashp1Sequals1}.\ref{numberofmatrixrealizationsoftype:twoisolatedvertices:case:lvertu1v1rvertbackslashp1Sequals1:u1doesnotequalv1andbracesu1v1bracesintersectp1Sisasingleton}.\ref{numberofmatrixrealizationsoftype:twoisolatedvertices:case:lvertu1v1rvertbackslashp1Sequals1:u1doesnotequalv1andbracesu1v1bracesintersectp1Sisasingletoncontainingv1}.\ref{numberofmatrixrealizationsoftype:twoisolatedvertices:case:lvertu1v1rvertbackslashp1Sequals1:u1doesnotequalv1andbracesu1v1bracesintersectp1Sisasingletoncontainingu1:v1equaltoa1}, and the second clause of \eqref{eq:proofofcomparativecountingtheoremcase3secondeq} 
is true, then we know that $u_1\neq v_1$, $u_1\notin\upp_1(S)$, $v_1\in\upp_1(S)$,  
$u_2\in\upp_2(S)$  and  $v_2\notin\upp_2(S)$. Since then $u_2\neq v_2$ 
and $u_1\neq v_1$, both \ref{characterizationoftype:C4intersectingatwopathinanendvertex:property3} and \ref{characterizationoftype:C4intersectingatwopathinitsinnervertex:property3} are impossible. Moreover, combining 
$u_1\notin\upp_1(S)$ and $v_2\notin\upp_2(S)$ shows that both clauses of the 
disjunction \ref{characterizationoftype:C4intersectingtwoedgesinseparatenonadjacentvertices:property4}. The remaining type is \ref{item:comparisonproofcaseLhassize6:C4intersectingtwoedgesinseparateadjacentvertices} and all the properties 
\ref{characterizationoftype:C4intersectingtwoedgesinseparateadjacentvertices:property1}--\ref{characterizationoftype:C4intersectingtwoedgesinseparateadjacentvertices:property4} are indeed satisfied (as to the disjunction \ref{characterizationoftype:C4intersectingtwoedgesinseparateadjacentvertices:property4}, only its second clause is true).
Since in \ref{numberofmatrixrealizationsoftype:twoisolatedvertices:case:lvertu1v1rvertbackslashp1Sequals1}.\ref{numberofmatrixrealizationsoftype:twoisolatedvertices:case:lvertu1v1rvertbackslashp1Sequals1:u1doesnotequalv1andbracesu1v1bracesintersectp1Sisasingleton}.\ref{numberofmatrixrealizationsoftype:twoisolatedvertices:case:lvertu1v1rvertbackslashp1Sequals1:u1doesnotequalv1andbracesu1v1bracesintersectp1Sisasingletoncontainingv1}.\ref{numberofmatrixrealizationsoftype:twoisolatedvertices:case:lvertu1v1rvertbackslashp1Sequals1:u1doesnotequalv1andbracesu1v1bracesintersectp1Sisasingletoncontainingu1:v1equaltoa1} we found exactly $2\cdot (a_1-1)\cdot ((n-1)-2)$ realizations of type \ref{item:comparisonproofcaseLhassize6:twoisolatedvertices} by $B\mid_{\{u,v\}}$, it follows that 
right now there are exactly $8\cdot (a_1-1)\cdot ((n-1)-2)$ realizations of 
type \ref{item:comparisonproofcaseLhassize6:C4intersectingtwoedgesinseparateadjacentvertices}. 

If \ref{numberofmatrixrealizationsoftype:twoisolatedvertices:case:lvertu1v1rvertbackslashp1Sequals1}.\ref{numberofmatrixrealizationsoftype:twoisolatedvertices:case:lvertu1v1rvertbackslashp1Sequals1:u1doesnotequalv1andbracesu1v1bracesintersectp1Sisasingleton}.\ref{numberofmatrixrealizationsoftype:twoisolatedvertices:case:lvertu1v1rvertbackslashp1Sequals1:u1doesnotequalv1andbracesu1v1bracesintersectp1Sisasingletoncontainingv1}.\ref{numberofmatrixrealizationsoftype:twoisolatedvertices:case:lvertu1v1rvertbackslashp1Sequals1:u1doesnotequalv1andbracesu1v1bracesintersectp1Sisasingletoncontainingu1:v1equaltoc1}, then one more time we have to distinguish the anonymous subcases. 
If \ref{numberofmatrixrealizationsoftype:twoisolatedvertices:case:lvertu1v1rvertbackslashp1Sequals1}.\ref{numberofmatrixrealizationsoftype:twoisolatedvertices:case:lvertu1v1rvertbackslashp1Sequals1:u1doesnotequalv1andbracesu1v1bracesintersectp1Sisasingleton}.\ref{numberofmatrixrealizationsoftype:twoisolatedvertices:case:lvertu1v1rvertbackslashp1Sequals1:u1doesnotequalv1andbracesu1v1bracesintersectp1Sisasingletoncontainingv1}.\ref{numberofmatrixrealizationsoftype:twoisolatedvertices:case:lvertu1v1rvertbackslashp1Sequals1:u1doesnotequalv1andbracesu1v1bracesintersectp1Sisasingletoncontainingu1:v1equaltoc1}, and the first clause of \eqref{eq:proofofcomparativecountingtheoremcase3secondeq} is true, 
then we know $u_1\neq v_1$, $u_1\notin\upp_1(S)$, $v_1\in\upp_1(S)$ and 
$u_2 = v_2 \notin \upp_2(S)$, with $u_2 = v_2$ ruling out both \ref{item:comparisonproofcaseLhassize6:C4intersectingtwoedgesinseparatenonadjacentvertices} and \ref{item:comparisonproofcaseLhassize6:C4intersectingtwoedgesinseparateadjacentvertices}. 
Moreover, combining $u_1\notin\upp_1(S)$ with $u_2\notin\upp_2(S)$ proves the first 
clause of the conjuction \ref{characterizationoftype:C4intersectingatwopathinitsinnervertex:property4} to be false. Again, for each decision the information in the 
first clause of \eqref{eq:proofofcomparativecountingtheoremcase3secondeq} was used. 
Now only \ref{item:comparisonproofcaseLhassize6:C4intersectingatwopathinanendvertex} 
is left and indeed all properties \ref{characterizationoftype:C4intersectingatwopathinanendvertex:property1}--\ref{characterizationoftype:C4intersectingatwopathinanendvertex:property4} are true (with the disjunction \ref{characterizationoftype:C4intersectingatwopathinanendvertex:property4} satisfied only by way of its second clause). Since in \ref{numberofmatrixrealizationsoftype:twoisolatedvertices:case:lvertu1v1rvertbackslashp1Sequals1}.\ref{numberofmatrixrealizationsoftype:twoisolatedvertices:case:lvertu1v1rvertbackslashp1Sequals1:u1doesnotequalv1andbracesu1v1bracesintersectp1Sisasingleton}.\ref{numberofmatrixrealizationsoftype:twoisolatedvertices:case:lvertu1v1rvertbackslashp1Sequals1:u1doesnotequalv1andbracesu1v1bracesintersectp1Sisasingletoncontainingv1}.\ref{numberofmatrixrealizationsoftype:twoisolatedvertices:case:lvertu1v1rvertbackslashp1Sequals1:u1doesnotequalv1andbracesu1v1bracesintersectp1Sisasingletoncontainingu1:v1equaltoc1} we found that in the first subcase there are exactly $(c_1-2)\cdot ((n-1)-2)$ realizations of 
type \ref{item:comparisonproofcaseLhassize6:twoisolatedvertices} by $B\mid_{\{u,v\}}$, 
it follows for our present situation that there are 
exactly $4\cdot (c_1-2)\cdot ((n-1)-2)$  realizations of \ref{item:comparisonproofcaseLhassize6:C4intersectingatwopathinanendvertex} by $B\mid_{\{u,v\}}$. 

If \ref{numberofmatrixrealizationsoftype:twoisolatedvertices:case:lvertu1v1rvertbackslashp1Sequals1}.\ref{numberofmatrixrealizationsoftype:twoisolatedvertices:case:lvertu1v1rvertbackslashp1Sequals1:u1doesnotequalv1andbracesu1v1bracesintersectp1Sisasingleton}.\ref{numberofmatrixrealizationsoftype:twoisolatedvertices:case:lvertu1v1rvertbackslashp1Sequals1:u1doesnotequalv1andbracesu1v1bracesintersectp1Sisasingletoncontainingv1}.\ref{numberofmatrixrealizationsoftype:twoisolatedvertices:case:lvertu1v1rvertbackslashp1Sequals1:u1doesnotequalv1andbracesu1v1bracesintersectp1Sisasingletoncontainingu1:v1equaltoc1}, and the second clause of \eqref{eq:proofofcomparativecountingtheoremcase3secondeq} 
is true, then we know that $u_1\neq v_1$, $u_1\notin\upp_1(S)$, $v_1\in\upp_1(S)$ 
and $u_2\in\upp_2(S)$ and $v_2\notin\upp_2(S)$. Since then $u_2\neq v_2$ it follows 
$\{u_1,u_2\}\cap \{v_1,v_2\} = \emptyset$, contradicting both 
\ref{characterizationoftype:C4intersectingatwopathinanendvertex:property3} and 
\ref{characterizationoftype:C4intersectingatwopathinitsinnervertex:property3}. 
Combining $u_1\notin\upp_1(S)$ and $v_2\notin\upp_2(S)$ we see that both clauses 
of the disjunction \ref{characterizationoftype:C4intersectingtwoedgesinseparatenonadjacentvertices:property4} are false. We are left with type \ref{item:comparisonproofcaseLhassize6:C4intersectingtwoedgesinseparateadjacentvertices} and indeed, all properties 
\ref{characterizationoftype:C4intersectingtwoedgesinseparateadjacentvertices:property1}--\ref{characterizationoftype:C4intersectingtwoedgesinseparateadjacentvertices:property4} are satisfied (for the disjunction \ref{characterizationoftype:C4intersectingtwoedgesinseparateadjacentvertices:property4} it is only the second clause, which is).  
Since in \ref{numberofmatrixrealizationsoftype:twoisolatedvertices:case:lvertu1v1rvertbackslashp1Sequals1}.\ref{numberofmatrixrealizationsoftype:twoisolatedvertices:case:lvertu1v1rvertbackslashp1Sequals1:u1doesnotequalv1andbracesu1v1bracesintersectp1Sisasingleton}.\ref{numberofmatrixrealizationsoftype:twoisolatedvertices:case:lvertu1v1rvertbackslashp1Sequals1:u1doesnotequalv1andbracesu1v1bracesintersectp1Sisasingletoncontainingv1}.\ref{numberofmatrixrealizationsoftype:twoisolatedvertices:case:lvertu1v1rvertbackslashp1Sequals1:u1doesnotequalv1andbracesu1v1bracesintersectp1Sisasingletoncontainingu1:v1equaltoc1} we found that there exist exactly $2\cdot (c_1-2)\cdot ((n-1)-2)$ realizations 
of type \ref{item:comparisonproofcaseLhassize6:twoisolatedvertices} by $B\mid_{\{u,v\}}$, 
it follows that in the present situation there are exactly 
$8\cdot (c_1-2)\cdot ((n-1)-2)$ realizations of type \ref{item:comparisonproofcaseLhassize6:C4intersectingtwoedgesinseparateadjacentvertices} by $B\mid_{\{u,v\}}$. 

The next case to reexamine is \ref{numberofmatrixrealizationsoftype:twoisolatedvertices:case:lvertu1v1rvertbackslashp1Sequals2}. This case is symmetric to 
\ref{numberofmatrixrealizationsoftype:twoisolatedvertices:case:lvertu1v1rvertbackslashp1Sequals0} under swapping the subscripts $1$ and $2$. Therefore, we can analyse its 
subcases by reexaming the analysis of \ref{numberofmatrixrealizationsoftype:twoisolatedvertices:case:lvertu1v1rvertbackslashp1Sequals0} with this swap in mind. 
First of all, if \ref{numberofmatrixrealizationsoftype:twoisolatedvertices:case:lvertu1v1rvertbackslashp1Sequals2}, then $\{u_2,v_2\}\subseteq \upp_2(S)$, and this 
renders both clauses of the disjunction \ref{characterizationoftype:C4intersectingatwopathinanendvertex:property4} false. 

Reading  \ref{numberofmatrixrealizationsoftype:twoisolatedvertices:case:lvertu1v1rvertbackslashp1Sequals0}.\ref{numberofmatrixrealizationsoftype:twoisolatedvertices:case:lvertu1v1rvertbackslashp1Sequals0:case:u1equalsv1} this way implies that we know 
$u_2=v_2 \in \upp_2(S)$, and $u_2=v_2$ contradicts both \ref{characterizationoftype:C4intersectingtwoedgesinseparatenonadjacentvertices:property3} and \ref{characterizationoftype:C4intersectingtwoedgesinseparateadjacentvertices:property3}. The only type 
remaining is \ref{item:comparisonproofcaseLhassize6:C4intersectingatwopathinitsinnervertex} and indeed, all properties \ref{characterizationoftype:C4intersectingatwopathinitsinnervertex:property1}--\ref{characterizationoftype:C4intersectingatwopathinitsinnervertex:property4} are satisfied. Since in \ref{numberofmatrixrealizationsoftype:twoisolatedvertices:case:lvertu1v1rvertbackslashp1Sequals0}.\ref{numberofmatrixrealizationsoftype:twoisolatedvertices:case:lvertu1v1rvertbackslashp1Sequals0:case:u1equalsv1} we found 
that there are exactly $2\cdot \binom{(n-1)-2}{2}$ realizations of type \ref{item:comparisonproofcaseLhassize6:twoisolatedvertices} by $B\mid_{\{u,v\}}$, 
it follows that here there exist exactly $8\cdot \binom{(n-1)-2}{2}$ realizations 
of type \ref{item:comparisonproofcaseLhassize6:C4intersectingatwopathinitsinnervertex} by $B\mid_{\{u,v\}}$. 

Reading \ref{numberofmatrixrealizationsoftype:twoisolatedvertices:case:lvertu1v1rvertbackslashp1Sequals0}.\ref{numberofmatrixrealizationsoftype:twoisolatedvertices:case:lvertu1v1rvertbackslashp1Sequals0:case:u1doesnotequalv1} this way implies that we know 
$u_2\neq v_2$, $\{u_2,v_2\}\cap\upp_2(S) = \emptyset$, $u_1\neq v_1$ 
and $\{u_1,v_1\}\subseteq \upp_1(S)$. The properties $u_1\neq v_1$ and $u_2\neq v_2$ 
taken together contradict both \ref{characterizationoftype:C4intersectingatwopathinanendvertex:property3} and \ref{characterizationoftype:C4intersectingatwopathinitsinnervertex:property3}. The property $\{  u_2 , v_2\}\cap \upp_2(S) = \emptyset$ alone renders 
both clauses of the disjunction \ref{characterizationoftype:C4intersectingtwoedgesinseparateadjacentvertices:property4} false. What we are left with is type \ref{item:comparisonproofcaseLhassize6:C4intersectingtwoedgesinseparatenonadjacentvertices} and 
indeed all properties \ref{characterizationoftype:C4intersectingtwoedgesinseparatenonadjacentvertices:property1}--\ref{characterizationoftype:C4intersectingtwoedgesinseparatenonadjacentvertices:property4} are satisfied. Since in \ref{numberofmatrixrealizationsoftype:twoisolatedvertices:case:lvertu1v1rvertbackslashp1Sequals0}.\ref{numberofmatrixrealizationsoftype:twoisolatedvertices:case:lvertu1v1rvertbackslashp1Sequals0:case:u1doesnotequalv1} we found that there are exactly 
$2\cdot \binom{(n-1)-2}{2}$ realizations of type 
\ref{item:comparisonproofcaseLhassize6:twoisolatedvertices} by $B\mid_{\{u,v\}}$, it follows 
by symmetry that here, too, there exist exactly $8\cdot \binom{(n-1)-2}{2}$ 
realizations of type \ref{item:comparisonproofcaseLhassize6:C4intersectingtwoedgesinseparatenonadjacentvertices} by $B\mid_{\{u,v\}}$. 

We have now completed reexamining  the analysis of \ref{numberofmatrixrealizationsoftype:twoisolatedvertices} and we may now add up (separately for each of the four types 
\ref{item:comparisonproofcaseLhassize6:C4intersectingtwoedgesinseparatenonadjacentvertices}--\ref{item:comparisonproofcaseLhassize6:C4intersectingatwopathinitsinnervertex} 
the number of realizations  we found  during the reexamination. 

For \ref{item:comparisonproofcaseLhassize6:C4intersectingtwoedgesinseparatenonadjacentvertices} we found exactly $8\cdot\binom{(n-1)-2}{2} + 8\cdot\binom{(n-1)-2}{2} 
= 16 \cdot \binom{(n-1)-2}{2}$  realizations by $B\mid_{\{u,v\}}$. Now  \ref{numberofmatrixrealizationsoftype:C4intersectingtwoedgesinseparatenonadjacentvertices} is proved.

For \ref{item:comparisonproofcaseLhassize6:C4intersectingtwoedgesinseparateadjacentvertices} we found exactly 
$
  8\cdot (n-1-a_1-1)\cdot ((n-1)-2) 
+ 8 \cdot (n-1-c_1)\cdot ((n-1)-2) 
+ 8\cdot (a_1-1)\cdot ((n-1)-2) 
+ 8\cdot (c_1-2)\cdot ((n-1)-2) 
=
  16\cdot ((n-1)-2)^2
$ realizations by $B\mid_{\{u,v\}}$. Now  \ref{numberofmatrixrealizationsoftype:C4intersectingtwoedgesinseparateadjacentvertices} is proved.

For \ref{item:comparisonproofcaseLhassize6:C4intersectingatwopathinanendvertex} 
we found exactly  $
  8 \cdot ((n-1)-2)^2 
+ 4 \cdot (n-1-a_1-1)\cdot ((n-1)-2) 
+ 4 \cdot (n-1-c_1)\cdot ((n-1)-2) 
+ 4 \cdot (a_1-1)\cdot ((n-1)-2)  
+ 4 \cdot (c_1-2)\cdot ((n-1)-2)
= 
  8\cdot ((n-1)-2)^2 
+ 8\cdot ((n-1)-2)^2
= 16 \cdot ((n-1)-2)^2
$ realizations by $B\mid_{\{u,v\}}$. Now 
\ref{numberofmatrixrealizationsoftype:C4intersectingatwopathinanendvertex} is proved. 

For \ref{item:comparisonproofcaseLhassize6:C4intersectingatwopathinitsinnervertex} we 
found exactly 
 $8\cdot\binom{(n-1)-2}{2} + 8\cdot\binom{(n-1)-2}{2} = 16 \cdot \binom{(n-1)-2}{2}$ 
realizations by $B\mid_{\{u,v\}}$. Now \ref{numberofmatrixrealizationsoftype:C4intersectingatwopathinitsinnervertex} is proved.

As to \ref{numberofmatrixrealizationsoftype:threeisolatedvertices}, let us first 
note that Definition \ref{def:XBandecXB} implies:
\begin{lemma}
For every $B\in\{0,\pm\}^I$ with $I\in\binom{[n-1]^2}{6}$, $I = S\sqcup \{u,v\}$ and 
$\upX_{B\mid_{S}} \cong C^4$ we have $\upX_{B} \cong \ref{item:comparisonproofcaseLhassize6:threeisolatedvertices}$ if and only if 

{\scriptsize
\begin{minipage}[b]{0.3\linewidth}
\begin{enumerate}[label={\rm(P.\ref{item:comparisonproofcaseLhassize6:threeisolatedvertices}.\arabic{*})}]
\item\label{characterizationoftype:threeisolatedvertices:property1} 
$B[u]=B[v]=0$,
\end{enumerate}
\end{minipage}
\begin{minipage}[b]{0.6\linewidth}
\begin{enumerate}[label={\rm(P.\ref{item:comparisonproofcaseLhassize6:threeisolatedvertices}.\arabic{*})},start=2]
\item\label{characterizationoftype:threeisolatedvertices:property2} 
$\lvert \{ u_1, v_1\} \setminus \upp_1(S) \rvert +
\lvert \{ u_2, v_2\} \setminus \upp_2(S) \rvert = 3$\quad .
\end{enumerate}
\end{minipage}
}
\end{lemma}

We now distinguish cases according to how \ref{characterizationoftype:threeisolatedvertices:property2} is satisfied. 
\begin{enumerate}[label={\rm(C.\ref{numberofmatrixrealizationsoftype:threeisolatedvertices}.\arabic{*})}]
\item\label{numberofmatrixrealizationsoftype:threeisolatedvertices:case:lvertu1v1rvertbackslashp1Sequals0} $\lvert \{ u_1, v_1\} \setminus \upp_1(S) \rvert = 0$. Then 
$\lvert \{ u_2, v_2\} \setminus \upp_2(S)\rvert = 3$ by 
\ref{characterizationoftype:threeisolatedvertices:property2}, which is impossible. 
Hence Case 1 does not occur. 
\item\label{numberofmatrixrealizationsoftype:threeisolatedvertices:case:lvertu1v1rvertbackslashp1Sequals1} $\lvert \{ u_1, v_1\} \setminus \upp_1(S) \rvert = 1$. Then 
$\lvert \{ u_2, v_2\} \setminus \upp_2(S)\rvert = 2$ by 
\ref{characterizationoftype:threeisolatedvertices:property2}, 
which is equivalent to 
\begin{equation}\label{eq:proofofcomparativecountingtheoremcase2firsteq}
\text{$u_2\neq v_2$ and $\{ u_2, v_2 \} \cap \upp_2(S) = \emptyset$} \quad . 
\end{equation}
Since the condition defining \ref{numberofmatrixrealizationsoftype:threeisolatedvertices:case:lvertu1v1rvertbackslashp1Sequals1} is equivalent to 
\begin{equation}\label{eq:proofofcomparativecountingtheoremcase2secondeq}
\text{
($u_1 = v_1$ and $\{u_1,v_1\} \cap \upp_1(S) = \emptyset$) or 
($u_1 \neq v_1$ and $\lvert \{u_1,v_1\} \cap \upp_1(S) \rvert = 1$)
}\quad ,
\end{equation}
there are two further cases. 
\begin{enumerate}[label={\rm(\arabic{*})}]
\item\label{numberofmatrixrealizationsoftype:threeisolatedvertices:case:lvertu1v1rvertbackslashp1Sequals1:u1equalsv1andbracesu1v1bracesintersectp1Sisempty} 
$u_1 = v_1$ and $\{u_1,v_1\} \cap \upp_1(S) = \emptyset$. Then there are exactly 
$(n-1)-2$ such $u_1=v_1$. Combining \eqref{eq:proofofcomparativecountingtheoremcase2firsteq} with $u\prec v$ it follows that $u_2<v_2$, therefore in the present case each 
of the $\binom{(n-1)-2}{2}$ different sets $\{u_2,v_2\}$ satisfying 
\eqref{eq:proofofcomparativecountingtheoremcase2firsteq} determines the two 
pairs $u$ and $v$. Therefore there are exactly $((n-1)-2) \cdot \binom{(n-1)-2}{2}$ 
realizations of \ref{numberofmatrixrealizationsoftype:threeisolatedvertices:case:lvertu1v1rvertbackslashp1Sequals1:u1equalsv1andbracesu1v1bracesintersectp1Sisempty}.
\item\label{numberofmatrixrealizationsoftype:threeisolatedvertices:case:lvertu1v1rvertbackslashp1Sequals1:u1doesnotequalv1andbracesu1v1bracesintersectp1Sisasingleton} 
$u_1 \neq v_1$ and $\lvert \{u_1,v_1\} \cap \upp_1(S) \rvert = 1$. From 
$u_1 \neq v_1$ and the assumption $u\prec v$ it follows that $u_1 < v_1$. By 
\eqref{eq:proofofcomparativecountingtheoremcase2firsteq}, $u_2<v_2$ or $v_2<u_2$, 
but nothing more is known about $u_2$ and $v_2$. Therefore, both possibilities must 
be taken into account. Because of $\upp_1(S) = \{a_1,b_1,c_1,d_1\} = \{a_1,c_1\}$ 
and $u_1<v_1$ there are exactly four possibilities for 
$\lvert \{u_1,v_1\} \cap \upp_1(S) \rvert = 1$ to be true: 
\begin{enumerate}[label={\rm(\arabic{*})}]
\item\label{numberofmatrixrealizationsoftype:threeisolatedvertices:case:lvertu1v1rvertbackslashp1Sequals1:u1doesnotequalv1andbracesu1v1bracesintersectp1Sisasingletonu1equalsa1}  $u_1 = a_1 = b_1$. Then because of $u_1 < v_1$ and $v_1 \neq c_1 = d_1$ it 
follows that there are exactly $n-1 - a_1 - 1$ different $v_1$ with 
$v_1\notin \upp_1(S)$ in this case. For each of them,  there exist exactly 
$\binom{(n-1)-2}{2}$ different \emph{sets} $\{u_2,v_2\}$ satisfying 
\eqref{eq:proofofcomparativecountingtheoremcase2firsteq}. Now the two pairs $u$ and 
$v$ are not determined by them: each of the sets can be realized in exactly two 
ways, both by $u_2<v_2$ and by $v_2<u_2$. Therefore, there are exactly 
$(n-1-a_1-1)\cdot 2\cdot \binom{(n-1)-2}{2}$ realizations of 
\ref{numberofmatrixrealizationsoftype:threeisolatedvertices:case:lvertu1v1rvertbackslashp1Sequals1:u1doesnotequalv1andbracesu1v1bracesintersectp1Sisasingletonu1equalsa1} 
by $u$ and $v$. 
\item\label{numberofmatrixrealizationsoftype:threeisolatedvertices:case:lvertu1v1rvertbackslashp1Sequals1:u1doesnotequalv1andbracesu1v1bracesintersectp1Sisasingletonu1equalsc1}  $u_1 = c_1 = d_1$. Then because of $u_1 < v_1$ it follows that there are 
exactly $n-1 - c_1$ different $v_1$ with $v_1\notin \upp_1(S)$. As in the preceding 
case, for each of these $v_1$ there exist  exactly $\binom{(n-1)-2}{2}$ different 
sets $\{u_2,v_2\}$ satisfying 
\eqref{eq:proofofcomparativecountingtheoremcase2firsteq}, hence exactly 
$2\cdot\binom{(n-1)-2}{2}$ different $u$ and $v$. Therefore there exist exactly 
$(n-1 - c_1) \cdot 2 \cdot \binom{(n-1)-2}{2}$ different realizations of type 
\ref{item:comparisonproofcaseLhassize6:threeisolatedvertices} by $B\mid_{\{u,v\}}$. 
\item\label{numberofmatrixrealizationsoftype:threeisolatedvertices:case:lvertu1v1rvertbackslashp1Sequals1:u1doesnotequalv1andbracesu1v1bracesintersectp1Sisasingletonv1equalsa1}  $v_1 = a_1 = b_1$. Then because of $u_1 < v_1$ it follows that there are 
exactly $a_1 - 1$ different $u_1$ with $u_1\notin \upp_1(S)$  in this case. For the 
same reasons as in the preceding two cases we know that here there exist exactly 
$(a_1 - 1) \cdot 2 \cdot \binom{(n-1)-2}{2}$ different realizations of type 
\ref{item:comparisonproofcaseLhassize6:threeisolatedvertices} by $B\mid_{\{u,v\}}$. 
\item\label{numberofmatrixrealizationsoftype:threeisolatedvertices:case:lvertu1v1rvertbackslashp1Sequals1:u1doesnotequalv1andbracesu1v1bracesintersectp1Sisasingletonv1equalsc1} $v_1 = c_1 = d_1$. Then because of $u_1 < v_1$ and $u_1 \neq a_1 = b_1$ it 
follows that there are $c_1 - 1 - 1$ different $u_1$ with $u_1\notin \upp_1(S)$  in 
this case. For the same reasons as in the preceding three cases we know that here 
there exist exactly $(c_1 - 1 - 1)\cdot 2 \cdot \binom{(n-1)-2}{2}$ different 
realizations of type \ref{item:comparisonproofcaseLhassize6:threeisolatedvertices} 
by $B\mid_{\{u,v\}}$. 
\end{enumerate}

It follows that if \ref{numberofmatrixrealizationsoftype:threeisolatedvertices:case:lvertu1v1rvertbackslashp1Sequals1}.\ref{numberofmatrixrealizationsoftype:threeisolatedvertices:case:lvertu1v1rvertbackslashp1Sequals1:u1doesnotequalv1andbracesu1v1bracesintersectp1Sisasingleton}, then there exist exactly 
$
\bigl ( (n-1-a_1-1) + (n-1-c_1) + (a_1-1) + (c_1-1-1)\bigr) \cdot 
2\cdot \binom{(n-1)-2}{2} =  4\cdot ((n-1)-2) \cdot \binom{(n-1)-2}{2}
$
different realizations of type \ref{item:comparisonproofcaseLhassize6:threeisolatedvertices} by $B\mid_{\{u,v\}}$. 
\end{enumerate}
It follows that if \ref{numberofmatrixrealizationsoftype:threeisolatedvertices:case:lvertu1v1rvertbackslashp1Sequals1}, then there are exactly 
$( (n-1)-2  + 4\cdot((n-1)-2))\cdot \binom{(n-1)-2}{2} = 
5\cdot ((n-1)-2) \cdot \binom{(n-1)-2}{2}$ realizations of type \ref{item:comparisonproofcaseLhassize6:threeisolatedvertices} by $B\mid_{\{u,v\}}$. 

\item\label{numberofmatrixrealizationsoftype:threeisolatedvertices:case:lvertu1v1rvertbackslashp1Sequals2} $\lvert \{ u_1, v_1\} \setminus \upp_1(S) \rvert = 2$. This is 
equivalent to 
\begin{equation}\label{eq:proofofcomparativecountingtheoremcase3firsteq}
\text{$u_1\neq v_1$ and $\{ u_1, v_1 \} \cap \upp_1(X) = \emptyset$}\quad . 
\end{equation}
Equation \ref{characterizationoftype:threeisolatedvertices:property2} implies $\lvert \{ u_2, v_2\} \setminus \upp_2(S)\rvert = 1$, 
which is equivalent to 
\begin{equation}\label{eq:proofofcomparativecountingtheoremcase3secondeq}
\text{
($u_2 = v_2$ and $\{u_2,v_2\} \cap \upp_2(S) = \emptyset$) or 
($u_2 \neq v_2$ and $\lvert \{u_2,v_2\} \cap \upp_2(S) \rvert = 1$)
}\quad .
\end{equation}
By swapping the subscripts $1$ and $2$ in the argument given for Case 2 it now 
follows that if \ref{numberofmatrixrealizationsoftype:threeisolatedvertices:case:lvertu1v1rvertbackslashp1Sequals2}, then there are 
exactly $5\cdot ((n-1)-2) \cdot \binom{(n-1)-2}{2}$ different realizations of 
type \ref{item:comparisonproofcaseLhassize6:threeisolatedvertices} by $B\mid_{\{u,v\}}$. 

\item\label{numberofmatrixrealizationsoftype:threeisolatedvertices:case:lvertu1v1rvertbackslashp1Sequals3} $\lvert \{ u_1, v_1\} \setminus \upp_1(S) \rvert = 3$. This is 
impossible, hence \ref{numberofmatrixrealizationsoftype:threeisolatedvertices:case:lvertu1v1rvertbackslashp1Sequals3} does not occur.  
\end{enumerate}

It follows that for every fixed $S$ there are exactly 
$10\cdot ((n-1)-2) \cdot \binom{(n-1)-2}{2}$ possibilities to position the two zeros 
indexed by $I\setminus S$ such that $\upX_{B} \cong \ref{item:comparisonproofcaseLhassize6:threeisolatedvertices}$. This completes the proof of \ref{numberofmatrixrealizationsoftype:threeisolatedvertices}.

As to \ref{numberofmatrixrealizationsoftype:oneadditionaledgeintersectingC4andtwoisolatedvertices}--\ref{numberofmatrixrealizationsoftype:twoadditionalnondisjointedgesdisjointfromC4} we begin by noting that for each of the four isomorphism types 
\ref{item:comparisonproofcaseLhassize6:oneadditionaledgeintersectingC4andtwoisolatedvertices}--\ref{item:comparisonproofcaseLhassize6:twoadditionalnondisjointedgesdisjointfromC4}, a necessary condition is that $\lvert \upV(\upX_{B})\setminus \upV(\upX_{B\mid_{S}}) \rvert = 3$. 
We can therefore prove \ref{numberofmatrixrealizationsoftype:oneadditionaledgeintersectingC4andtwoisolatedvertices}--\ref{numberofmatrixrealizationsoftype:twoadditionalnondisjointedgesdisjointfromC4} during one reexamination of the proof of \ref{numberofmatrixrealizationsoftype:threeisolatedvertices}. 

Since \ref{numberofmatrixrealizationsoftype:threeisolatedvertices:case:lvertu1v1rvertbackslashp1Sequals0} and \ref{numberofmatrixrealizationsoftype:threeisolatedvertices:case:lvertu1v1rvertbackslashp1Sequals3} are impossible, we only have to 
consider \ref{numberofmatrixrealizationsoftype:threeisolatedvertices:case:lvertu1v1rvertbackslashp1Sequals1} and \ref{numberofmatrixrealizationsoftype:threeisolatedvertices:case:lvertu1v1rvertbackslashp1Sequals2}. 
If \ref{numberofmatrixrealizationsoftype:threeisolatedvertices:case:lvertu1v1rvertbackslashp1Sequals1}, we know \eqref{eq:proofofcomparativecountingtheoremcase2firsteq} but 
this is not sufficient to rule out any of the types \ref{item:comparisonproofcaseLhassize6:oneadditionaledgeintersectingC4andtwoisolatedvertices}--\ref{item:comparisonproofcaseLhassize6:twoadditionalnondisjointedgesdisjointfromC4}. 

If \ref{numberofmatrixrealizationsoftype:threeisolatedvertices:case:lvertu1v1rvertbackslashp1Sequals1}.\ref{numberofmatrixrealizationsoftype:threeisolatedvertices:case:lvertu1v1rvertbackslashp1Sequals1:u1equalsv1andbracesu1v1bracesintersectp1Sisempty} we 
know that 
\begin{equation}\label{eq:knowledgeincase:numberofmatrixrealizationsoftype:threeisolatedvertices:case:lvertu1v1rvertbackslashp1Sequals1:u1equalsv1andbracesu1v1bracesintersectp1Sisempty}
u_1 = v_1,\; \{u_1,v_1\}\cap\upp_1(S) = \emptyset,\; 
u_2\neq v_2,\; \{u_2,v_2\}\cap\upp_2(S) = \emptyset\quad ,
\end{equation}
and will now consider the consequences of this for \ref{numberofmatrixrealizationsoftype:oneadditionaledgeintersectingC4andtwoisolatedvertices}--\ref{numberofmatrixrealizationsoftype:twoadditionalnondisjointedgesdisjointfromC4}.

\begin{enumerate}[label={\rm(\arabic{*})}]
\item Concerning contributions  to \ref{numberofmatrixrealizationsoftype:oneadditionaledgeintersectingC4andtwoisolatedvertices}, note that properties $\{u_1,v_1\}\cap\upp_1(S) = \emptyset$ and $\{u_2,v_2\}\cap\upp_2(S) = \emptyset$ make an edge intersecting $\upX_{B\mid_{S}} \cong C^4$ impossible, hence the case \ref{numberofmatrixrealizationsoftype:threeisolatedvertices:case:lvertu1v1rvertbackslashp1Sequals1}.\ref{numberofmatrixrealizationsoftype:threeisolatedvertices:case:lvertu1v1rvertbackslashp1Sequals1:u1equalsv1andbracesu1v1bracesintersectp1Sisempty} does not contribute\footnote{The fact that neither \ref{numberofmatrixrealizationsoftype:threeisolatedvertices:case:lvertu1v1rvertbackslashp1Sequals1}.\ref{numberofmatrixrealizationsoftype:threeisolatedvertices:case:lvertu1v1rvertbackslashp1Sequals1:u1equalsv1andbracesu1v1bracesintersectp1Sisempty} nor the corresponding subcase of \ref{numberofmatrixrealizationsoftype:threeisolatedvertices:case:lvertu1v1rvertbackslashp1Sequals2} (which due to symmetry was not spelled out in the proof of \ref{numberofmatrixrealizationsoftype:threeisolatedvertices} and therefore does not have a name) 
contribute to $\lvert(\upXul^{6,n,n})^{-1}\ref{item:comparisonproofcaseLhassize6:oneadditionaledgeintersectingC4andtwoisolatedvertices}\rvert$ is a reason why $\lvert(\upXul^{6,n,n})^{-1}\ref{item:comparisonproofcaseLhassize6:oneadditionaledgeintersectingC4andtwoisolatedvertices}\rvert$ is larger but not twice as large as 
$\lvert(\upXul^{6,n,n})^{-1}\ref{item:comparisonproofcaseLhassize6:threeisolatedvertices}\rvert$ even though in the cases where \ref{item:comparisonproofcaseLhassize6:oneadditionaledgeintersectingC4andtwoisolatedvertices} \emph{can} be realized the number of realizations is twice as large as for \ref{item:comparisonproofcaseLhassize6:threeisolatedvertices}.} to $\lvert(\upXul^{6,n,n})^{-1}\ref{item:comparisonproofcaseLhassize6:oneadditionaledgeintersectingC4andtwoisolatedvertices}\rvert$. 
\item Concerning contributions to \ref{numberofmatrixrealizationsoftype:oneadditionaldisjointedgeandoneisolatedvertex}, note that \eqref{eq:knowledgeincase:numberofmatrixrealizationsoftype:threeisolatedvertices:case:lvertu1v1rvertbackslashp1Sequals1:u1equalsv1andbracesu1v1bracesintersectp1Sisempty} implies that $\upX_{B} \cong \ref{item:comparisonproofcaseLhassize6:oneadditionaldisjointedgeandoneisolatedvertex}$ if and only if 
either ($B[u]\in \{\pm\}$ and $B[v] = 0$) or ($B[u] = 0$ and $B[v] \in \{\pm\}$). 
Each of these clauses corresponds to $2$ different $B$. It follows that if 
\ref{numberofmatrixrealizationsoftype:threeisolatedvertices:case:lvertu1v1rvertbackslashp1Sequals1}.\ref{numberofmatrixrealizationsoftype:threeisolatedvertices:case:lvertu1v1rvertbackslashp1Sequals1:u1equalsv1andbracesu1v1bracesintersectp1Sisempty}, then 
there are $4$-times as many realizations of type 
\ref{item:comparisonproofcaseLhassize6:oneadditionaldisjointedgeandoneisolatedvertex} as there are of type 
\ref{item:comparisonproofcaseLhassize6:threeisolatedvertices}. Therefore, 
if \ref{numberofmatrixrealizationsoftype:threeisolatedvertices:case:lvertu1v1rvertbackslashp1Sequals1}.\ref{numberofmatrixrealizationsoftype:threeisolatedvertices:case:lvertu1v1rvertbackslashp1Sequals1:u1equalsv1andbracesu1v1bracesintersectp1Sisempty}, 
there are exactly $4\cdot ((n-1)-2)\cdot\binom{(n-1)-2}{2}$ realizations of type 
\ref{item:comparisonproofcaseLhassize6:oneadditionaldisjointedgeandoneisolatedvertex} 
by $B\mid_{\{u,v\}}$. 
\item Concerning contributions to \ref{numberofmatrixrealizationsoftype:twoadditionaledgesonlyoneofthemdisjoint},  since properties $\{u_1,v_1\}\cap\upp_1(S) = \emptyset$ and $\{u_2,v_2\}\cap\upp_2(S) = \emptyset$ make an edge intersecting $\upX_{B\mid_{S}} \cong C^4$ impossible, the case \ref{numberofmatrixrealizationsoftype:threeisolatedvertices:case:lvertu1v1rvertbackslashp1Sequals1}.\ref{numberofmatrixrealizationsoftype:threeisolatedvertices:case:lvertu1v1rvertbackslashp1Sequals1:u1equalsv1andbracesu1v1bracesintersectp1Sisempty} does not contribute to \ref{numberofmatrixrealizationsoftype:twoadditionaledgesonlyoneofthemdisjoint}.
\item Concerning contributions to \ref{numberofmatrixrealizationsoftype:twoadditionalnondisjointedgesdisjointfromC4}, we see from \eqref{eq:knowledgeincase:numberofmatrixrealizationsoftype:threeisolatedvertices:case:lvertu1v1rvertbackslashp1Sequals1:u1equalsv1andbracesu1v1bracesintersectp1Sisempty} that $\upX_{B} \cong \ref{item:comparisonproofcaseLhassize6:twoadditionalnondisjointedgesdisjointfromC4}$ if and only if 
($B[u]\in \{\pm\}$ and  $B[v]\in\{\pm\}$), and there are $4$ different 
$B\mid_{\{u,v\}}\in\{0,\pm\}^{\{u,v\}}$ satisfying this. Therefore, if 
\ref{numberofmatrixrealizationsoftype:threeisolatedvertices:case:lvertu1v1rvertbackslashp1Sequals1}.\ref{numberofmatrixrealizationsoftype:threeisolatedvertices:case:lvertu1v1rvertbackslashp1Sequals1:u1equalsv1andbracesu1v1bracesintersectp1Sisempty}, 
there are exactly $4\cdot ((n-1)-2)\cdot\binom{(n-1)-2}{2}$ realizations of type 
\ref{item:comparisonproofcaseLhassize6:twoadditionalnondisjointedgesdisjointfromC4} 
by $B\mid_{\{u,v\}}$. 
\end{enumerate}

If \ref{numberofmatrixrealizationsoftype:threeisolatedvertices:case:lvertu1v1rvertbackslashp1Sequals1}.\ref{numberofmatrixrealizationsoftype:threeisolatedvertices:case:lvertu1v1rvertbackslashp1Sequals1:u1doesnotequalv1andbracesu1v1bracesintersectp1Sisasingleton}, then we know 
\begin{equation}\label{eq:knowledgeincase:numberofmatrixrealizationsoftype:threeisolatedvertices:case:lvertu1v1rvertbackslashp1Sequals1:u1doesnotequalv1andbracesu1v1bracesintersectp1Sisasingleton}
u_1\neq v_1,\; \lvert \{u_1,v_1\}\cap\upp_1(S) \rvert = 1,\; 
u_2\neq v_2,\; \{ u_2,v_2\}\cap\upp_2(S) = \emptyset\quad .
\end{equation}
and will now consider the consequences of this for \ref{numberofmatrixrealizationsoftype:oneadditionaledgeintersectingC4andtwoisolatedvertices}--\ref{numberofmatrixrealizationsoftype:twoadditionalnondisjointedgesdisjointfromC4}.
\begin{enumerate}[label={\rm(\arabic{*})}]
\item  Concerning contributions  to \ref{numberofmatrixrealizationsoftype:oneadditionaledgeintersectingC4andtwoisolatedvertices}, we can argue as follows: 
If $\lvert \{u_1,v_1\}\cap\upp_1(S) \rvert = 1$ is true as $u_1\in\upp_1(S)$, 
then there are exactly two $B\mid_{\{u,v\}}\in\{0,\pm\}^{\{u,v\}}$ with $\upX_{B}\cong \ref{item:comparisonproofcaseLhassize6:oneadditionaledgeintersectingC4andtwoisolatedvertices}$, 
namely those which satisfy ($B[u]\in\{\pm\}$ and $B[v] = 0$). If it is true 
as $v_1\in\upp_1(S)$, then again there are exactly two such $B\mid_{\{u,v\}}$, namely those 
which satisfy ($B[u] = 0$ and $B[v] \in \{\pm\}$). It follows that 
without having to reexamine the subcases \ref{numberofmatrixrealizationsoftype:threeisolatedvertices:case:lvertu1v1rvertbackslashp1Sequals1}.\ref{numberofmatrixrealizationsoftype:threeisolatedvertices:case:lvertu1v1rvertbackslashp1Sequals1:u1doesnotequalv1andbracesu1v1bracesintersectp1Sisasingleton}.\ref{numberofmatrixrealizationsoftype:threeisolatedvertices:case:lvertu1v1rvertbackslashp1Sequals1:u1doesnotequalv1andbracesu1v1bracesintersectp1Sisasingletonu1equalsa1}--\ref{numberofmatrixrealizationsoftype:threeisolatedvertices:case:lvertu1v1rvertbackslashp1Sequals1}.\ref{numberofmatrixrealizationsoftype:threeisolatedvertices:case:lvertu1v1rvertbackslashp1Sequals1:u1doesnotequalv1andbracesu1v1bracesintersectp1Sisasingleton}.\ref{numberofmatrixrealizationsoftype:threeisolatedvertices:case:lvertu1v1rvertbackslashp1Sequals1:u1doesnotequalv1andbracesu1v1bracesintersectp1Sisasingletonv1equalsc1} we know that there are twice as many realizations of \ref{item:comparisonproofcaseLhassize6:oneadditionaledgeintersectingC4andtwoisolatedvertices}  by $B\mid_{\{u,v\}}$ in the case 
\ref{numberofmatrixrealizationsoftype:threeisolatedvertices:case:lvertu1v1rvertbackslashp1Sequals1}.\ref{numberofmatrixrealizationsoftype:threeisolatedvertices:case:lvertu1v1rvertbackslashp1Sequals1:u1doesnotequalv1andbracesu1v1bracesintersectp1Sisasingleton} 
than of \ref{numberofmatrixrealizationsoftype:threeisolatedvertices}. Therefore, 
if \ref{numberofmatrixrealizationsoftype:threeisolatedvertices:case:lvertu1v1rvertbackslashp1Sequals1}.\ref{numberofmatrixrealizationsoftype:threeisolatedvertices:case:lvertu1v1rvertbackslashp1Sequals1:u1doesnotequalv1andbracesu1v1bracesintersectp1Sisasingleton},
then there are exactly $8\cdot ((n-1)-2)\cdot \binom{(n-1)-2}{2}$ realizations 
of \ref{item:comparisonproofcaseLhassize6:oneadditionaledgeintersectingC4andtwoisolatedvertices}  by $B\mid_{\{u,v\}}$. 
\item Concerning contributions to \ref{numberofmatrixrealizationsoftype:oneadditionaldisjointedgeandoneisolatedvertex}, we have to distinguish in what way property 
$\lvert \{ u_1,v_1 \} \cap \upp_1(S) \rvert = 1 $ in 
\eqref{eq:knowledgeincase:numberofmatrixrealizationsoftype:threeisolatedvertices:case:lvertu1v1rvertbackslashp1Sequals1:u1doesnotequalv1andbracesu1v1bracesintersectp1Sisasingleton} is satisfied. If $u_1\in\upp_1(S)$ but $v_1\notin\upp_1(S)$, then 
$\upX_{B} \cong \ref{numberofmatrixrealizationsoftype:oneadditionaldisjointedgeandoneisolatedvertex}$ if and only if $B[u]=0$ and $B[v]\in\{\pm\}$, hence in this case 
there exist  $2$ different $B\mid_{\{u,v\}}$ with $\upX_{B} \cong \ref{numberofmatrixrealizationsoftype:oneadditionaldisjointedgeandoneisolatedvertex}$. If $u_1\notin\upp_1(S)$ but $v_1\in\upp_1(S)$, then $\upX_{B} \cong \ref{numberofmatrixrealizationsoftype:oneadditionaldisjointedgeandoneisolatedvertex}$ if and only if $B[u]\in\{\pm\}$ 
and $B[v] = 0$, hence in this case there again exist  $2$ different $B\mid_{\{u,v\}}$ 
with $\upX_{B} \cong \ref{numberofmatrixrealizationsoftype:oneadditionaldisjointedgeandoneisolatedvertex}$. It follows that if \ref{numberofmatrixrealizationsoftype:threeisolatedvertices:case:lvertu1v1rvertbackslashp1Sequals1}.\ref{numberofmatrixrealizationsoftype:threeisolatedvertices:case:lvertu1v1rvertbackslashp1Sequals1:u1doesnotequalv1andbracesu1v1bracesintersectp1Sisasingleton}, then there there are $2$-times as many realizations of \ref{item:comparisonproofcaseLhassize6:oneadditionaldisjointedgeandoneisolatedvertex} than of \ref{item:comparisonproofcaseLhassize6:threeisolatedvertices} by $B\mid_{\{u,v\}}$. Therefore, if \ref{numberofmatrixrealizationsoftype:threeisolatedvertices:case:lvertu1v1rvertbackslashp1Sequals1}.\ref{numberofmatrixrealizationsoftype:threeisolatedvertices:case:lvertu1v1rvertbackslashp1Sequals1:u1doesnotequalv1andbracesu1v1bracesintersectp1Sisasingleton}, then there are exactly 
$2\cdot 4\cdot ((n-1)-2)\cdot \binom{(n-1)-2}{2} = 8\cdot (n-3)\cdot \binom{n-3}{2}$ 
realizations of \ref{item:comparisonproofcaseLhassize6:oneadditionaldisjointedgeandoneisolatedvertex} by $B\mid_{\{u,v\}}$. 
\item Concerning contributions to \ref{numberofmatrixrealizationsoftype:twoadditionaledgesonlyoneofthemdisjoint}, note that no matter how \eqref{eq:knowledgeincase:numberofmatrixrealizationsoftype:threeisolatedvertices:case:lvertu1v1rvertbackslashp1Sequals1:u1doesnotequalv1andbracesu1v1bracesintersectp1Sisasingleton} is satisfied, we have 
$\upX_{B} \cong \ref{numberofmatrixrealizationsoftype:twoadditionaledgesonlyoneofthemdisjoint}$ if and only if ($B[u]\in\{\pm\}$ and $B[v]\in\{\pm\}$). Hence, if \ref{numberofmatrixrealizationsoftype:threeisolatedvertices:case:lvertu1v1rvertbackslashp1Sequals1}.\ref{numberofmatrixrealizationsoftype:threeisolatedvertices:case:lvertu1v1rvertbackslashp1Sequals1:u1doesnotequalv1andbracesu1v1bracesintersectp1Sisasingleton}, 
then there are $4$-times as many realizations of \ref{numberofmatrixrealizationsoftype:twoadditionaledgesonlyoneofthemdisjoint} than there are of \ref{numberofmatrixrealizationsoftype:threeisolatedvertices}, that is, if \ref{numberofmatrixrealizationsoftype:threeisolatedvertices:case:lvertu1v1rvertbackslashp1Sequals1}.\ref{numberofmatrixrealizationsoftype:threeisolatedvertices:case:lvertu1v1rvertbackslashp1Sequals1:u1doesnotequalv1andbracesu1v1bracesintersectp1Sisasingleton}, then there are exactly 
$4\cdot 4\cdot ( (n-1)-2)\cdot \binom{(n-1)-2}{2} = 16\cdot (n-3)\cdot \binom{n-3}{2}$ 
realizations of type \ref{numberofmatrixrealizationsoftype:twoadditionaledgesonlyoneofthemdisjoint} by $B\mid_{\{u,v\}}$.
\item Concerning contributions to \ref{numberofmatrixrealizationsoftype:twoadditionalnondisjointedgesdisjointfromC4}, note that \eqref{eq:knowledgeincase:numberofmatrixrealizationsoftype:threeisolatedvertices:case:lvertu1v1rvertbackslashp1Sequals1:u1doesnotequalv1andbracesu1v1bracesintersectp1Sisasingleton} says that 
$u_1\neq v_1$ and $u_2\neq v_2$, and this makes it impossible to create a 
$2$-path outside of $\upX_{B\mid_S}\cong C^4$. Therefore, if \ref{numberofmatrixrealizationsoftype:threeisolatedvertices:case:lvertu1v1rvertbackslashp1Sequals1}.\ref{numberofmatrixrealizationsoftype:threeisolatedvertices:case:lvertu1v1rvertbackslashp1Sequals1:u1doesnotequalv1andbracesu1v1bracesintersectp1Sisasingleton}, there is no contribution to \ref{numberofmatrixrealizationsoftype:twoadditionalnondisjointedgesdisjointfromC4}. 
\end{enumerate}
We now take stock of what we found in the subcases 
\ref{numberofmatrixrealizationsoftype:threeisolatedvertices:case:lvertu1v1rvertbackslashp1Sequals1}.\ref{numberofmatrixrealizationsoftype:threeisolatedvertices:case:lvertu1v1rvertbackslashp1Sequals1:u1equalsv1andbracesu1v1bracesintersectp1Sisempty}
and \ref{numberofmatrixrealizationsoftype:threeisolatedvertices:case:lvertu1v1rvertbackslashp1Sequals1}.\ref{numberofmatrixrealizationsoftype:threeisolatedvertices:case:lvertu1v1rvertbackslashp1Sequals1:u1doesnotequalv1andbracesu1v1bracesintersectp1Sisasingleton}
in order to know what the entire case \ref{numberofmatrixrealizationsoftype:threeisolatedvertices:case:lvertu1v1rvertbackslashp1Sequals1} contributes to \ref{numberofmatrixrealizationsoftype:oneadditionaledgeintersectingC4andtwoisolatedvertices}--\ref{numberofmatrixrealizationsoftype:twoadditionalnondisjointedgesdisjointfromC4}.

Since \ref{numberofmatrixrealizationsoftype:threeisolatedvertices:case:lvertu1v1rvertbackslashp1Sequals1}.\ref{numberofmatrixrealizationsoftype:threeisolatedvertices:case:lvertu1v1rvertbackslashp1Sequals1:u1equalsv1andbracesu1v1bracesintersectp1Sisempty} 
did not contribute to  $\lvert(\upXul^{6,n,n})^{-1}\ref{item:comparisonproofcaseLhassize6:oneadditionaledgeintersectingC4andtwoisolatedvertices}\rvert$ but \ref{numberofmatrixrealizationsoftype:threeisolatedvertices:case:lvertu1v1rvertbackslashp1Sequals1}.\ref{numberofmatrixrealizationsoftype:threeisolatedvertices:case:lvertu1v1rvertbackslashp1Sequals1:u1doesnotequalv1andbracesu1v1bracesintersectp1Sisasingleton} did contribute 
$8\cdot (n-3)\cdot \binom{n-3}{2}$, it follows that if \ref{numberofmatrixrealizationsoftype:threeisolatedvertices:case:lvertu1v1rvertbackslashp1Sequals1}, then there are 
exactly $8\cdot (n-3)\cdot \binom{n-3}{2}$ realizations of type \ref{item:comparisonproofcaseLhassize6:oneadditionaledgeintersectingC4andtwoisolatedvertices} 
by $B\mid_{\{u,v\}}$. 

Since \ref{numberofmatrixrealizationsoftype:threeisolatedvertices:case:lvertu1v1rvertbackslashp1Sequals1}.\ref{numberofmatrixrealizationsoftype:threeisolatedvertices:case:lvertu1v1rvertbackslashp1Sequals1:u1equalsv1andbracesu1v1bracesintersectp1Sisempty} 
contributed $4\cdot (n-3) \cdot\binom{n-3}{2}$ to $\lvert(\upXul^{6,n,n})^{-1}\ref{item:comparisonproofcaseLhassize6:oneadditionaldisjointedgeandoneisolatedvertex}\rvert$ while \ref{numberofmatrixrealizationsoftype:threeisolatedvertices:case:lvertu1v1rvertbackslashp1Sequals1}.\ref{numberofmatrixrealizationsoftype:threeisolatedvertices:case:lvertu1v1rvertbackslashp1Sequals1:u1doesnotequalv1andbracesu1v1bracesintersectp1Sisasingleton} contributed 
$8\cdot (n-3)\cdot\binom{n-3}{2}$, it follows that if \ref{numberofmatrixrealizationsoftype:threeisolatedvertices:case:lvertu1v1rvertbackslashp1Sequals1}, then there are 
exactly $12\cdot (n-3) \cdot\binom{n-3}{2}$ realizations of type \ref{item:comparisonproofcaseLhassize6:oneadditionaldisjointedgeandoneisolatedvertex} by $B\mid_{\{u,v\}}$. 

Since \ref{numberofmatrixrealizationsoftype:threeisolatedvertices:case:lvertu1v1rvertbackslashp1Sequals1}.\ref{numberofmatrixrealizationsoftype:threeisolatedvertices:case:lvertu1v1rvertbackslashp1Sequals1:u1equalsv1andbracesu1v1bracesintersectp1Sisempty} 
did not contribute to  $\lvert(\upXul^{6,n,n})^{-1}\ref{item:comparisonproofcaseLhassize6:twoadditionaledgesonlyoneofthemdisjoint}\rvert$ but 
\ref{numberofmatrixrealizationsoftype:threeisolatedvertices:case:lvertu1v1rvertbackslashp1Sequals1}.\ref{numberofmatrixrealizationsoftype:threeisolatedvertices:case:lvertu1v1rvertbackslashp1Sequals1:u1doesnotequalv1andbracesu1v1bracesintersectp1Sisasingleton} did contribute $16 \cdot (n-3) \cdot \binom{n-3}{2}$, it follows that if 
\ref{numberofmatrixrealizationsoftype:threeisolatedvertices:case:lvertu1v1rvertbackslashp1Sequals1}, then there are exactly $16 \cdot (n-3) \cdot \binom{n-3}{2}$ realizations of type \ref{item:comparisonproofcaseLhassize6:twoadditionaledgesonlyoneofthemdisjoint} by $B\mid_{\{u,v\}}$. 

Since \ref{numberofmatrixrealizationsoftype:threeisolatedvertices:case:lvertu1v1rvertbackslashp1Sequals1}.\ref{numberofmatrixrealizationsoftype:threeisolatedvertices:case:lvertu1v1rvertbackslashp1Sequals1:u1equalsv1andbracesu1v1bracesintersectp1Sisempty} 
contributed $4\cdot (n-3)\cdot\binom{n-3}{2}$ to $\lvert(\upXul^{6,n,n})^{-1}\ref{item:comparisonproofcaseLhassize6:twoadditionalnondisjointedgesdisjointfromC4}\rvert$ while \ref{numberofmatrixrealizationsoftype:threeisolatedvertices:case:lvertu1v1rvertbackslashp1Sequals1}.\ref{numberofmatrixrealizationsoftype:threeisolatedvertices:case:lvertu1v1rvertbackslashp1Sequals1:u1doesnotequalv1andbracesu1v1bracesintersectp1Sisasingleton} did not contribute 
anything, it follows that if \ref{numberofmatrixrealizationsoftype:threeisolatedvertices:case:lvertu1v1rvertbackslashp1Sequals1}, then there are exactly 
$4\cdot (n-3)\cdot\binom{n-3}{2}$ realizations of type \ref{item:comparisonproofcaseLhassize6:twoadditionalnondisjointedgesdisjointfromC4} by $B\mid_{\{u,v\}}$. 

Since the case \ref{numberofmatrixrealizationsoftype:threeisolatedvertices:case:lvertu1v1rvertbackslashp1Sequals2}  is symmetric to the case \ref{numberofmatrixrealizationsoftype:threeisolatedvertices:case:lvertu1v1rvertbackslashp1Sequals1} via interchanging 
the subscripts $1$ and $2$, we will get the same contributions to \ref{numberofmatrixrealizationsoftype:oneadditionaledgeintersectingC4andtwoisolatedvertices}--\ref{numberofmatrixrealizationsoftype:twoadditionalnondisjointedgesdisjointfromC4} as in the 
case \ref{numberofmatrixrealizationsoftype:threeisolatedvertices:case:lvertu1v1rvertbackslashp1Sequals1}. We therefore have to double each of the four results found for 
\ref{numberofmatrixrealizationsoftype:threeisolatedvertices:case:lvertu1v1rvertbackslashp1Sequals1} to get the correct numbers of realizations of 
types \ref{item:comparisonproofcaseLhassize6:oneadditionaledgeintersectingC4andtwoisolatedvertices}--\ref{item:comparisonproofcaseLhassize6:twoadditionalnondisjointedgesdisjointfromC4}. This proves \ref{numberofmatrixrealizationsoftype:oneadditionaledgeintersectingC4andtwoisolatedvertices}--\ref{numberofmatrixrealizationsoftype:twoadditionalnondisjointedgesdisjointfromC4}.

As to \ref{numberofmatrixrealizationsoftype:fourisolatedvertices}, it suffices 
to note that the  two values of $B\mid_{\{u,v\}}$ are determined: since there does 
not exist an edge outside $\upX_{B\mid_S}$ $\cong$ $C^4$, they both must be zero. 
Therefore \ref{numberofmatrixrealizationsoftype:fourisolatedvertices} is the number 
of $\{u,v\}\in\binom{[n-1]^2}{2}$ such that $\upX_{B\mid_{S}\sqcup \{0\}^{\{u,v\}}} \cong 
\ref{item:comparisonproofcaseLhassize6:fourisolatedvertices}$. By definition of $S$ 
the latter is equivalent to saying that $\upX_{B\mid_{S}\sqcup \{0\}^{\{u,v\}}}$ has exactly 
eight vertices. It follows from Definition \ref{def:XBandecXB} that this is the case 
if and only if simultaneously
\begin{equation}\label{eq:conditioninproofof:numberofmatrixrealizationsoftype:fourisolatedvertices}
\lvert \{ u_1, v_1 \} \setminus \upp_1(S) \rvert = 2
\quad \text{and}\quad 
\lvert \{ u_2, v_2 \} \setminus \upp_2(S) \rvert = 2\quad .
\end{equation} 
Due to $u_1<v_1$, the number of $\{u_1,v_1\}\subseteq [n-1]$ 
with $\lvert \{ u_1, v_1 \}\setminus \upp_1(S) \rvert = 2$ is $\binom{n-3}{2}$. 
Since we only assume $u\prec v$ and hence both $u_2<v_2$ and $u_2>v_2$ are possible, 
for each of these $\binom{n-3}{2}$ different $\{u_1,v_1\}$ there are 
$2\cdot \binom{n-3}{2}$ different $\{u_2,v_2\}$ with 
$\lvert \{ u_2, v_2 \} \setminus \upp_2(S) \rvert = 2$. This proves 
\ref{numberofmatrixrealizationsoftype:fourisolatedvertices}.

As to \ref{numberofmatrixrealizationsoftype:oneadditionaldisjointedgeandtwoisolatedvertices} and \ref{numberofmatrixrealizationsoftype:twoadditionaldisjointedgesdisjointfromC4}, note that for both ismorphism types \ref{item:comparisonproofcaseLhassize6:oneadditionaldisjointedgeandtwoisolatedvertices} and \ref{item:comparisonproofcaseLhassize6:twoadditionaldisjointedgesdisjointfromC4} it is necessary 
that $X_{ B\mid_{S} \sqcup \{0\}^{\{u,v\}} } \cong \ref{item:comparisonproofcaseLhassize6:fourisolatedvertices}$. We can therefore prove \ref{numberofmatrixrealizationsoftype:oneadditionaldisjointedgeandtwoisolatedvertices} and \ref{numberofmatrixrealizationsoftype:twoadditionaldisjointedgesdisjointfromC4} by reexamining the proof of \ref{numberofmatrixrealizationsoftype:fourisolatedvertices}. In whatever way 
\eqref{eq:conditioninproofof:numberofmatrixrealizationsoftype:fourisolatedvertices} 
is satisfied,  there are exactly $4$ different $B\mid_{\{u,v\}}$ with 
$\upX_{B} \cong \ref{item:comparisonproofcaseLhassize6:oneadditionaldisjointedgeandtwoisolatedvertices}$, namely those satisfying 
\begin{equation}\label{eq:conditioninproofof:numberofmatrixrealizationsoftype:oneadditionaldisjointedgeandtwoisolatedvertices}
(B[u]\in\{\pm\}\; \text{and}\; B[v]=0) \quad \text{or}\quad 
(B[u]=0\; \text{and}\; B[v]\in\{\pm\})\quad , 
\end{equation}
but there are also $4$ different $B\mid_{\{u,v\}}$ with 
$\upX_{B} \cong \ref{item:comparisonproofcaseLhassize6:twoadditionaldisjointedgesdisjointfromC4}$, namely those satisfying 
\begin{equation}\label{eq:conditioninproofof:numberofmatrixrealizationsoftype:numberofmatrixrealizationsoftype:twoadditionaldisjointedgesdisjointfromC4}
B[u]\in\{\pm\}\; \text{and}\; B[v]\in\{\pm\} \quad .
\end{equation}
This proves both $\lvert(\upXul^{6,n,n})^{-1}\ref{item:comparisonproofcaseLhassize6:oneadditionaldisjointedgeandtwoisolatedvertices}\rvert =\lvert(\upXul^{6,n,n})^{-1}\ref{item:comparisonproofcaseLhassize6:twoadditionaldisjointedgesdisjointfromC4}\rvert = 4 \cdot \lvert(\upXul^{6,n,n})^{-1}\ref{item:comparisonproofcaseLhassize6:fourisolatedvertices}\rvert$, and therefore both \ref{numberofmatrixrealizationsoftype:oneadditionaldisjointedgeandtwoisolatedvertices} and \ref{numberofmatrixrealizationsoftype:twoadditionaldisjointedgesdisjointfromC4}. The proof of \ref{item:numberofrealizationsofnonforestswhenkequals6} is now complete. 
\end{proof}

The relations \ref{linearrelationbetweent2andt3:kequals5}--\ref{linearrelationbetweent16tot18:kequals6} in Lemma \ref{lem:relationsamongthenumbersofmatrixrealizations} 
give us a plausibility check (i.e. necessary conditions) for the explicit formulas 
$\lvert(\upXul^{6,n,n})^{-1}\ref{item:comparisonproofcaseLhassize6:oneisolatedvertex}\rvert$, $\dotsc$, $\lvert(\upXul^{6,n,n})^{-1}\ref{item:comparisonproofcaseLhassize6:twoadditionaldisjointedgesdisjointfromC4}\rvert$ that we found in \ref{numberofmatrixrealizationsoftype:oneisolatedvertex:kequals5}--\ref{numberofmatrixrealizationsoftype:twoadditionaldisjointedgesdisjointfromC4}. For brevity let $x:= n-3$ and 
$y := \binom{n-3}{2}$. Then, indeed,  the explicit 
formulas that we found in \ref{item:numberofrealizationsofnonforestswhenkequals5} 
and \ref{item:numberofrealizationsofnonforestswhenkequals6} pass the test: the 
formulas in \ref{item:numberofrealizationsofnonforestswhenkequals5} evidently 
satisfy \ref{linearrelationbetweent2andt3:kequals5} and \ref{linearrelationbetweent5andt7:kequals5}. Moreover, since  $(3^2 - 1) \cdot \bigl ( 8 x^2 + 8 y \bigr ) 
= 24 x^2 + 32 y + 8 x^2 + 16 y + 16 x^2 + 16 x^2 + 16 y$, the formulas 
\ref{numberofmatrixrealizationsoftype:twoisolatedvertices}--\ref{numberofmatrixrealizationsoftype:C4intersectingatwopathinitsinnervertex} satisfy \ref{linearrelationbetweent4tot10:kequals6}. Since $(3^2 - 1) \cdot 10 xy = 16 x y + 24 x y + 32 x y + 8 x y$, 
the formulas in \ref{numberofmatrixrealizationsoftype:threeisolatedvertices}--\ref{numberofmatrixrealizationsoftype:twoadditionalnondisjointedgesdisjointfromC4} 
satisfy \ref{linearrelationbetweent11tot15:kequals6}. Since 
$(3^2-1)\cdot 2y^2 = 8y^2 + 8y^2$, the formulas in \ref{numberofmatrixrealizationsoftype:fourisolatedvertices}--\ref{numberofmatrixrealizationsoftype:twoadditionaldisjointedgesdisjointfromC4} satisfy \ref{linearrelationbetweent16tot18:kequals6}.

\subsection{Counting failures of equality of $\Prob_{\chio}$ and $\Prob_{\lcf}$} 

While determining an absolute cardinality 
$\lvert(\upXul^{k,n,n})^{-1}(\mathfrak{X})\rvert$ seems to necessitate  work 
specifically depending on the isomorphism type $\mathfrak{X}$, the ratio of 
all \emph{balanced} matrix realizations to \emph{all} realizations 
is easy to compute since it is determined by the Betti number of $\mathfrak{X}$ 
alone. This is the content of \ref{eq:ratioofbalancedmatrixrealizationsofanisomorphismtypetoallrealizations}  in the following lemma: 

\begin{lemma}\label{lem:forfixedisomorphismtyperatioofbalancedrealizationsdependsonthebettinumberalone}
For every $(s,t)\in\Z_{\geq 2}^2$, every $0\leq k \leq (s-1)(t-1)$, every unlabelled 
bipartite graph $\mathfrak{X}$ and every $\beta\in \Z_{\geq 1}$, 
\begin{enumerate}[label={\rm(E\arabic{*})}]
\item\label{eq:ratioofbalancedmatrixrealizationsofanisomorphismtypetoallrealizations} $\lvert\{ B\in(\upXul^{k,s,t})^{-1}( \mathfrak{X}) \colon 
(\upX_B,\sigma_B)\;\mathrm{balanced}\} \rvert = (\tfrac12)^{\beta_1(\mathfrak{X})}
\cdot\lvert(\upXul^{k,s,t})^{-1} (\mathfrak{X})\rvert$ \quad , 
\item\label{eq:expansionofamatrixfailuresetwithapositiveratiobyisomorphismtypes} 
$\lvert\mathcal{F}_{\cdot 2^\beta}^{\mathrm{M}}(k,s,t)\rvert = 
\sum_{\mathfrak{X}\in\im(\upXul^{k,s,t})\colon \beta_1(\mathfrak{X}) = \beta}\; (\tfrac12)^\beta \cdot 
\bigl \lvert \bigl( \upXul^{k,s,t}\bigr)^{-1} (\mathfrak{X})\bigr \rvert$\quad ,
\item\label{eq:expansionofamatrixfailuresetwithzeroreatiobyisomorphismtypes} 
$\lvert\mathcal{F}_{\cdot 0}^{\mathrm{M}}(k,s,t)\rvert = 
\sum_{\mathfrak{X}\in\im(\upXul^{k,s,t})\colon \beta_1(\mathfrak{X}) \geq 1}\; 
\bigl (1-(\tfrac12)^{\beta_1(\mathfrak{X})} \bigr ) 
\cdot\lvert(\upXul^{k,s,t})^{-1} (\mathfrak{X})\rvert$ \quad . 
\end{enumerate} 
\end{lemma}
\begin{proof}
If $\mathcal{M}$ is a set of matrices, let us define 
$\Dom(\mathcal{M}) := \{ \Dom(B)\colon B\in\mathcal{M}\}$ and 
$\Supp(\mathcal{M}) := \{ \Supp(B)\colon B\in\mathcal{M}\}$. Moreover, 
if $S$ is a set, $\mathcal{S}\subseteq \mathfrak{P}(S)$ a set of subsets and 
$U\in \mathfrak{P}(S)$ a subset, then $U\cap\mathcal{S} := 
\{ U\cap S\colon S\in\mathcal{S}\}$. Using these notations, we can prove 
\ref{eq:ratioofbalancedmatrixrealizationsofanisomorphismtypetoallrealizations} by 
the following calculation: for every unlabelled $\mathfrak{X}$ we have 
$\lvert\{$ $B$ $\in$ $(\upXul^{k,s,t})^{-1}$ $(\mathfrak{X})\colon$ 
$(\upX_B,\sigma_B)$ balanced $\}\rvert$ $=$ 
$\sum_{J\in \Supp( (\upXul^{k,s,t})^{-1}(\mathfrak{X}))}$ 
$\lvert \{$ $B$ $\in$  $(\upXul^{k,s,t})^{-1}(\mathfrak{X})\colon$ 
$(\upX_B,\sigma_B)$ balanced, $\Supp(B)$ $=$ $J$ $\} \rvert$
$=$ (\ref{item:numberofbalancedsignfunctions} in Lemma \ref{lem:equivalenceofexistenceofbconstantrpropervertex2coloringandcyclicallyrevenness}) $=$ 
$\sum_{J\in\Supp( (\upXul^{k,s,t})^{-1}(\mathfrak{X}))}$ 
$2^{\lvert J \rvert - \beta_1(\mathfrak{X})}$ $=$ $(\tfrac12)^{\beta_1(\mathfrak{X})}$
$\sum_{I\in\Dom((\upXul^{k,s,t})^{-1}(\mathfrak{X}))}$ 
$\sum_{J\in I\cap\Supp( (\upXul^{k,s,t})^{-1}(\mathfrak{X}))}$ $2^{\lvert J \rvert}$ $=$ (directly 
from the definitions) $=$ 
$(\tfrac12)^{\beta_1(\mathfrak{X})}$ $\lvert (\upXul^{k,s,t})^{-1}(\mathfrak{X})\rvert$. 
As to \ref{eq:expansionofamatrixfailuresetwithapositiveratiobyisomorphismtypes}, 
this is true since $\lvert\mathcal{F}_{\cdot 2^\beta}^{\mathrm{M}}(k,s,t)\rvert$ 
$=$ (\ref{relationbeweenchiomeasureandlazycoinflipmeasuregovernedbyfirstbettinumber} in Theorem \ref{thm:graphtheoreticalcharacterizationofthechiomeasure}) $=$ 
$\sum_{\mathfrak{X}\in\im(\upXul^{k,s,t})\colon\beta_1(\mathfrak{X})=\beta}$ 
$\lvert \{$ $B\in(\upXul^{k,s,t})^{-1}(\mathfrak{X})\colon$ $(\upX_B,\sigma_B)$ 
$\mathrm{balanced}$ $\}\rvert$ $=$ (by \ref{eq:ratioofbalancedmatrixrealizationsofanisomorphismtypetoallrealizations}) $=$ 
$\sum_{\mathfrak{X}\in\im(\upXul^{k,s,t})\colon\beta_1(\mathfrak{X})=\beta}$ 
$(\tfrac12)^\beta\cdot\bigl\lvert(\upXul^{k,s,t}\bigr)^{-1}(\mathfrak{X})\bigr\rvert$.
As to \ref{eq:expansionofamatrixfailuresetwithzeroreatiobyisomorphismtypes}, note 
that $\lvert\mathcal{F}_{\cdot 0}^{\mathrm{M}}(k,s,t)\rvert$ $=$ 
(\ref{characterizationofwhenchiomeasureispositive} in Theorem 
\ref{thm:graphtheoreticalcharacterizationofthechiomeasure}) $=$ 
$\sum_{\mathfrak{X}\in\im(\upXul^{k,s,t})\colon\beta_1(\mathfrak{X})\geq 1}$ $\lvert\{$ 
$B$ $\in$ $(\upXul^{k,s,t})^{-1}$ $(\mathfrak{X})\colon$ $(\upX_B,$ $\sigma_B)$ 
$\mathrm{not}$ $\mathrm{balanced}$ $\}\rvert$ $=$ 
(using \ref{eq:ratioofbalancedmatrixrealizationsofanisomorphismtypetoallrealizations}) $=$ $\sum_{\mathfrak{X}\in\im(\upXul^{k,s,t})\colon\beta_1(\mathfrak{X})\geq 1}$ 
$\bigl (1-(\tfrac12)^{\beta_1(\mathfrak{X})}\bigr) \cdot \lvert(\upXul^{k,s,t})^{-1} (\mathfrak{X})\rvert$. 
\end{proof}

The fewer the number $\dom(B)$ of entries specified, the larger an 
entry-specification event $\mathcal{E}_B^{[n-1]^2}$ is (as a set). Any two 
probability measures by definition agree on the largest possible event, the 
entire sample space. The following theorem explores to what extent 
$\Prob_{\chio}$ and $\Prob_{\lcf}$ agree on successively smaller entry-specification 
events, descending down to as much as six specifications. 

\begin{theorem}[number of exceptions to equality of 
$\Prob_{\chio}$ and $\Prob_{\lcf}$ on large entry-specification events]
\label{thm:comparativecountingtheorem}
In the following statements let $\emptyset\subseteq I \subseteq [n-1]^2$, 
$B\in \{0,\pm\}^I$ and $\mathcal{E}_B:=\mathcal{E}_B^{[n-1]^2}$.  
\begin{enumerate}[label={\rm(Ex\arabic{*})}, start=3]
\item\label{comparativecountingtheorem:item:uptto3entries} For each of 
the $\sum_{0\leq k \leq 3} 3^k \cdot \binom{(n-1)^2}{k}\sim \frac92 \cdot n^6$ 
possible events $\mathcal{E}_B$ with $0\leq \dom(B) \leq 3$, it 
is true that $\Prob_{\chio}[\mathcal{E}_B] = \Prob_{\lcf} [\mathcal{E}_B] 
= (\frac12)^{\dom(B)+\supp(B)}$. \\
\item\label{comparativecountingtheorem:item:fourentriesspecified}
Among the $3^4 \cdot \binom{(n-1)^2}{4} \sim \frac{27}{8} \cdot n^8$ possible events 
$\mathcal{E}_B$ with $\dom(B)=4$, there  are precisely  
$\lvert \mathcal{F}^{\mathrm{M}} (4,n) \rvert 
= 2^4 \cdot \binom{n-1}{2}\cdot \binom{n-1}{2} \sim 4\cdot n^4$ events 
for which $\Prob_{\chio}[\mathcal{E}_B] = \Prob_{\lcf} [\mathcal{E}_B]$ does not 
hold. Of these, we have $\tfrac{\lvert\mathcal{F}_{\cdot 0}^{\mathrm{M}}(4,n)\rvert}{\lvert\mathcal{F}^{\mathrm{M}}(4,n)\rvert} = \tfrac{\lvert\mathcal{F}_{\cdot 2}^{\mathrm{M}}(4,n)\rvert}{\lvert\mathcal{F}^{\mathrm{M}}(4,n)\rvert} = \tfrac12$.  
\item\label{comparativecountingtheorem:item:fiveentriesspecified}
Among the $3^5 \cdot \binom{(n-1)^2}{5} \sim \frac{81}{40} \cdot n^{10}$ 
different events $\mathcal{E}_B$ with $\dom(B) = 5$, there are precisely  
$\lvert \mathcal{F}^{\mathrm{M}}(5,n) \rvert = 
48\cdot ((n-1)^2 - 4) \cdot \binom{n-1}{2}\cdot \binom{n-1}{2}\sim 12\cdot n^6$ 
events for which $\Prob_{\chio}[\mathcal{E}_B] = \Prob_{\lcf} [\mathcal{E}_B]$ does 
not hold. Of these, we have 
$\tfrac{\lvert \mathcal{F}_{\cdot 0}^{\mathrm{M}}(5,n) \rvert}{\lvert \mathcal{F}^{\mathrm{M}}(5,n) \rvert} = \tfrac{\lvert \mathcal{F}_{\cdot 2}^{\mathrm{M}}(5,n) \rvert}{\lvert \mathcal{F}^{\mathrm{M}}(5,n) \rvert} = \tfrac12$. 
\item\label{comparativecountingtheorem:item:sixentriesspecified} Among the 
$3^6 \cdot \binom{(n-1)^2}{6} \sim \frac{81}{80} n^{12}$ different events 
$\mathcal{E}_B$ with $\dom(B) = 6$, there are precisely 
$\lvert \mathcal{F}^{\mathrm{M}}(6,n) \rvert = 
18 n^8 - 180 n^7 + \frac{1868}{3} n^6 - \frac{2176}{3} n^5 - 
\frac{754}{3} n^4  + \frac{428}{3} n^3 + \frac{8144}{3} n^2 - 
\frac{11536}{3} n + 1504 \sim 18n^8$ events for which 
$\Prob_{\chio}[\mathcal{E}_B] = \Prob_{\lcf} [\mathcal{E}_B]$ does not 
hold. Of these, we have $\lvert \mathcal{F}_{\cdot 0}^{\mathrm{M}}(6,n) \rvert = 
9 n^8 - 90 n^7 + \tfrac{934}{3} n^6 - 360 n^5 - \tfrac{449}{3} n^4 + 154 n^3 + 
\tfrac{3664}{3} n^2 - 1816 n + 720 \sim 9n^8$, 
$\lvert \mathcal{F}_{\cdot 2}^{\mathrm{M}}(6,n) \rvert =
9 n^8 - 90 n^7 + \tfrac{934}{3} n^6 - 368 n^5  - \tfrac{233}{3} n^4 - 94 n^3 + 
\tfrac{4888}{3} n^2 - 2136 n + 816 \sim 9n^8$, and 
$\lvert \mathcal{F}_{\cdot 4}^{\mathrm{M}}(6,n) \rvert$ $=$ 
$\tfrac83 n^5 - 24 n^4 + \tfrac{248}{3} n^3 - 136 n^2 + \tfrac{320}{3} n - 32 \sim 
\tfrac83 n^5$, hence in particular 
$\lim_{n\rightarrow\infty} \tfrac{\lvert \mathcal{F}_{\cdot 0}^{\mathrm{M}}(6,n) \rvert}{\lvert \mathcal{F}^{\mathrm{M}}(6,n) \rvert} = 
\lim_{n\rightarrow\infty} \tfrac{\lvert \mathcal{F}_{\cdot 2}^{\mathrm{M}}(6,n) \rvert}{\lvert \mathcal{F}^{\mathrm{M}}(6,n) \rvert} = \tfrac12$ and  
$\lim_{n\rightarrow\infty} \tfrac{\lvert \mathcal{F}_{\cdot 4}^{\mathrm{M}}(6,n) \rvert}{\lvert \mathcal{F}^{\mathrm{M}}(6,n) \rvert} = 0$.
\end{enumerate}
\end{theorem}
\begin{proof}
The total numbers of entry specification events mentioned at the beginning of 
\ref{comparativecountingtheorem:item:fourentriesspecified}--\ref{comparativecountingtheorem:item:sixentriesspecified}, and all the asymptotic equalities are easily 
checked, so we do not have to say more about them. 

As to \ref{comparativecountingtheorem:item:uptto3entries}, this follows 
immediately from \ref{item:failuresetofisomorphismtypesk3} in Corollary 
\ref{cor:isomorphismtypesforwhichequalityofmeasuresofentryspecificationeventsfails}. 
As to \ref{comparativecountingtheorem:item:fourentriesspecified}, the claimed 
value of $\lvert \mathcal{F}^{\mathrm{M}}(4,n)\rvert$ is true by 
\ref{it:partitionoffailuresets:kequals4} in Corollary 
\ref{cor:partitionsoffailuresets} combined with Lemma \ref{lem:numberofmatrixcircuitsofgivenlengthwithingivencartesianproduct} and the obvious fact that 
$\lvert(\upXul^{4,n,n})^{-1}\ref{item:comparisonproofcaseLhassize4:4circuit}\rvert = 
2^4\cdot \lvert \Cir(4,n)\rvert$. The claimed ratios can be deduced as follows: note 
that $\{ \mathfrak{X}\in\im(\upXul^{4,n,n})\colon\beta_1(\mathfrak{X})\geq 1\} = 
\{ \ref{item:comparisonproofcaseLhassize4:4circuit} \}$ by Corollary \ref{cor:isomorphismtypesforwhichequalityofmeasuresofentryspecificationeventsfails}, hence $\lvert\mathcal{F}_{\cdot 0}^{\mathrm{M}}(4,n)\rvert$ $=$ (by \ref{eq:expansionofamatrixfailuresetwithzeroreatiobyisomorphismtypes}) $=$ $\bigl(1-(\tfrac12)^{\beta_1\ref{item:comparisonproofcaseLhassize4:4circuit}}\bigr) \cdot 
\lvert (\upXul^{4,n,n})^{-1}\ref{item:comparisonproofcaseLhassize4:4circuit} \rvert$ 
$=$ $\tfrac12$ $\cdot$ $\lvert (\upXul^{4,n,n})^{-1}\ref{item:comparisonproofcaseLhassize4:4circuit} \rvert$ $=$ (by \ref{it:partitionoffailuresets:kequals4} in 
Corollary \ref{cor:partitionsoffailuresets}) $=$ 
$\tfrac12\cdot \lvert\mathcal{F}^{\mathrm{M}}(4,n)\rvert$, which together with 
the equation $\mathcal{F}^{\mathrm{M}}(4,n) = \mathcal{F}_{\cdot 0}^{\mathrm{M}}(4,n) 
\sqcup \mathcal{F}_{\cdot 2}^{\mathrm{M}}(4,n)$ from 
\ref{it:decompositionsofmatrixfailuresets} in Corollary \ref{cor:ratiosandvaluesofpchioandplcffortheisomorphismtypesforwhichequalityfails} proves both 
$\lvert\mathcal{F}_{\cdot 0}^{\mathrm{M}}(4,n)\rvert / \lvert\mathcal{F}^{\mathrm{M}}(4,n)\rvert$ $=$ $\tfrac12$ and $\lvert\mathcal{F}_{\cdot 2}^{\mathrm{M}}(4,n)\rvert /
\lvert\mathcal{F}^{\mathrm{M}}(4,n)\rvert$ $=$ $\tfrac12$.  

As to \ref{comparativecountingtheorem:item:fiveentriesspecified}, the number 
stated first is obvious. The claimed value of 
$\lvert \mathcal{F}^{\mathrm{G}}(5,n)\rvert$ can be deduced as follows: by 
\ref{relationbeweenchiomeasureandlazycoinflipmeasuregovernedbyfirstbettinumber} in 
Theorem \ref{thm:graphtheoreticalcharacterizationofthechiomeasure} we have 
$\Prob_{\chio}[\mathcal{E}_B] \neq \Prob_{\lcf}[\mathcal{E}_B]$ if and only if 
$\beta_1(\upX_{B}) > 0$. Since due to Definition \ref{def:XBandecXB} we have 
$0\leq f_1(\upX_{B}) \leq \dom(B) = \lvert I \rvert = 5$, it is 
easy to see that $\beta_1(\upX_B) \leq 1$. Therefore 
$\Prob_{\chio}[\mathcal{E}_B] \neq \Prob_{\lcf}[\mathcal{E}_B]$ if and only if 
$C^4\hookrightarrow \upX_{B}$. The latter property is equivalent to the existence 
of a matrix-$4$-circuit $S\subseteq I$ with $S\subseteq\Supp(B)$. Note that every 
$I\in\binom{[n-1]^2}{5}$ contains at most one matrix-$4$-circuit. Therefore, the 
number of all $I\in\binom{[n-1]^2}{5}$ with 
$\Prob_{\chio}[\mathcal{E}_B] \neq \Prob_{\lcf}[\mathcal{E}_B]$ is equal to the 
number of all matrix-$4$-circuits $S\in \binom{[n-1]^2}{4}$ with 
$S\subseteq\Supp(B)$, multiplied by the number of possibilities to choose an 
arbitrary position $u\in [n-1]^2\setminus S$ and an arbitrary $B[u]\in \{0,\pm\}$, 
i.e. $2^4\cdot \binom{n-1}{2}^2 \cdot 3 \cdot ((n-1)^2-4) 
= 48 \cdot ((n-1)^2-4)\cdot\binom{n-1}{2}^2$. This proves the second claim 
in \ref{comparativecountingtheorem:item:fiveentriesspecified}. The claimed ratios 
can be deduced as follows: note that 
$\{ \mathfrak{X}\in\im(\upXul^{5,n,n})\colon\beta_1(\mathfrak{X})\geq 1\} = 
\{ \ref{item:comparisonproofcaseLhassize6:oneisolatedvertex}, 
\ref{item:comparisonproofcaseLhassize6:oneadditionaledgeintersectingC4}, 
\ref{item:comparisonproofcaseLhassize6:twoisolatedvertices}, 
\ref{item:comparisonproofcaseLhassize6:C4withoneadditionaldisjointedge}
\}$ by Corollary \ref{cor:isomorphismtypesforwhichequalityofmeasuresofentryspecificationeventsfails}, hence $\lvert\mathcal{F}_{\cdot 0}^{\mathrm{M}}(5,n)\rvert$ $=$ (by \ref{eq:expansionofamatrixfailuresetwithzeroreatiobyisomorphismtypes}) $=$ 
$\bigl(1-(\tfrac12)^{\beta_1\ref{item:comparisonproofcaseLhassize6:oneisolatedvertex}}\bigr) \cdot 
\lvert (\upXul^{5,n,n})^{-1}\ref{item:comparisonproofcaseLhassize6:oneisolatedvertex} \rvert$ $+$ 
$\bigl(1-(\tfrac12)^{\beta_1\ref{item:comparisonproofcaseLhassize6:oneadditionaledgeintersectingC4}}\bigr) \cdot \lvert (\upXul^{5,n,n})^{-1}\ref{item:comparisonproofcaseLhassize6:oneadditionaledgeintersectingC4} \rvert$ 
$+$ 
$\bigl(1-(\tfrac12)^{\beta_1\ref{item:comparisonproofcaseLhassize6:twoisolatedvertices}}\bigr) \cdot \lvert (\upXul^{5,n,n})^{-1}\ref{item:comparisonproofcaseLhassize6:twoisolatedvertices} \rvert$ $+$
$\bigl(1-(\tfrac12)^{\beta_1\ref{item:comparisonproofcaseLhassize6:C4withoneadditionaldisjointedge}}\bigr) \cdot \lvert (\upXul^{5,n,n})^{-1}\ref{item:comparisonproofcaseLhassize6:C4withoneadditionaldisjointedge} \rvert$  
$=$ 
$\tfrac12$ $\bigl ($ $\lvert (\upXul^{5,n,n})^{-1}\ref{item:comparisonproofcaseLhassize6:oneisolatedvertex} \rvert$ $+$
$\lvert (\upXul^{5,n,n})^{-1}\ref{item:comparisonproofcaseLhassize6:oneadditionaledgeintersectingC4} \rvert$ $+$
$\lvert (\upXul^{5,n,n})^{-1}\ref{item:comparisonproofcaseLhassize6:twoisolatedvertices}\rvert$ $+$
$\lvert (\upXul^{5,n,n})^{-1}\ref{item:comparisonproofcaseLhassize6:C4withoneadditionaldisjointedge} \rvert$ $\bigr)$ $=$ (by \ref{it:partitionoffailuresets:kequals5} in 
Corollary \ref{cor:partitionsoffailuresets}) $=$ 
$\tfrac12\cdot \lvert\mathcal{F}^{\mathrm{M}}(5,n)\rvert$, which together with 
the equation $\mathcal{F}^{\mathrm{M}}(5,n) = \mathcal{F}_{\cdot 0}^{\mathrm{M}}(5,n) 
\sqcup \mathcal{F}_{\cdot 2}^{\mathrm{M}}(5,n)$ from 
\ref{it:decompositionsofmatrixfailuresets} in Corollary \ref{cor:ratiosandvaluesofpchioandplcffortheisomorphismtypesforwhichequalityfails} proves both 
$\lvert\mathcal{F}_{\cdot 0}^{\mathrm{M}}(5,n)\rvert / \lvert\mathcal{F}^{\mathrm{M}}(5,n)\rvert$ $=$ $\tfrac12$ and $\lvert\mathcal{F}_{\cdot 2}^{\mathrm{M}}(5,n)\rvert /
\lvert\mathcal{F}^{\mathrm{M}}(5,n)\rvert$ $=$ $\tfrac12$.

As to \ref{comparativecountingtheorem:item:sixentriesspecified}, the claimed 
value of $\lvert \mathcal{F}^{\mathrm{G}}(6,n)\rvert$ can be deduced as follows: 
using \ref{item:failuresetofisomorphismtypesk6} in Corollary 
\ref{cor:isomorphismtypesforwhichequalityofmeasuresofentryspecificationeventsfails}, 
and inspecting the list of isomorphism types in Lemma 
\ref{lem:bipartitenonforestsorderedbytheirfvectors}, we know that equality of the 
measures fails if and only  $C^4\hookrightarrow \upX_B$ or 
$C^6\hookrightarrow \upX_B$. To count the events for which this is true let us define 
$h_{C^6}(n):=\lvert \{ B \in \{0,\pm\}^I\colon 
I\in \binom{[n-1]^2}{6},\quad C^6 \hookrightarrow \upX_B\} \rvert$, 
$h_{K^{2,3}}(n) := \lvert \{ B \in \{0,\pm\}^I\colon 
I\in \binom{[n-1]^2}{6},\quad K^{2,3}\hookrightarrow \upX_{B}\}\rvert$ and 
$h_{C^4, \neg K^{2,3}}(n) := \lvert \{ B \in \{0,\pm\}^I\colon 
I\in \binom{[n-1]^2}{6},\quad C^4\hookrightarrow \upX_{B},\quad 
K^{2,3} \not\hookrightarrow \upX_{B}\}\rvert$. 
Since for a graph $X$ with at most six edges the properties 
$C^6\hookrightarrow X$, $K^{2,3}\hookrightarrow X$, and $C^4\hookrightarrow X$ 
but $K^{2,3} \not\hookrightarrow X$ are mutually exclusive, it follows that  
\begin{equation}
\label{eq:numberofentryspecificationeventswithunequalmeasuressizeLis6}
\lvert \mathcal{F}^{\mathrm{M}}(6,n)\rvert := 
h_{C^6}(n) + h_{C^4,\ \neg K^{2,3}}(n) + h_{K^{2,3}}(n)\quad .
\end{equation}
Lemma \ref{lem:numberofmatrixcircuitsofgivenlengthwithingivencartesianproduct} 
implies that $h_{C^6}(n) = 2^6\cdot 3!\cdot \binom{n-1}{3}^2$, the factor of $2^6$ 
accounting for the fact that the property $C^6\hookrightarrow \upX_{B}$ is 
indifferent to the choice of the six signs in $B$. Moreover, evidently, 
$h_{K^{2,3}}(n) = 2\cdot 2^6 \cdot \binom{n-1}{3} \cdot \binom{n-1}{2}$, 
where the first $2$ accounts for the two possibilities 
$\lvert \upp_1(I) \rvert = 2$ and $\lvert \upp_2(I) \rvert = 3$ 
or 
$\lvert \upp_1(I) \rvert = 3$ and $\lvert \upp_2(I) \rvert = 2$. 

In order to compute $h_{C^4,\ \neg K^{2,3}}(n)$, we will employ a simple 
inclusion-exclusion-argument: for any of the $\binom{n-1}{2}^2$ choices 
$1\leq i_1 < i_2 \leq n-1$ and $1\leq j_1 < j_2 \leq n-1$ for the position of 
a matrix-$4$-circuit $S = \{(i_1,j_1), (i_1,j_2), (i_2,j_1), (i_2,j_2)\}$, the number 
of distinct $I\subseteq [n-1]^2$ with $\lvert I\rvert =6$ such that $B$ has an 
matrix-$4$-circuit \emph{at least} at the positions $(i_1,j_1)$, $(i_1,j_2)$, 
$(i_2,j_1)$ and $(i_2,j_2)$ is $ 2^4 \cdot  3^2 \cdot \binom{(n-1)^2-4}{2}$, 
where the factor $2^4$ accounts for the mandatory ($\pm$)-values of the $4$-circuit 
entries, the factor $3^2$ accounts for the arbitrary $\{0,\pm\}$-values of the 
two non-$4$-circuit-entries and the factor $\binom{(n-1)^2-4}{2}$ accounts for the 
arbitrary positions of the two non-$4$-circuit-entries in $[n-1]^2\setminus S$. 
Summing this expression over all $\binom{n-1}{2}^2$ possible choices of $S$, we 
obtain $h_{\geq}(n) := 2^4 \cdot \binom{n-1}{2}^2 \cdot  3^2 \cdot 
\binom{(n-1)^2-4}{2}$. But this is not $h_{C^4,\ \neg K^{2,3}}(n)$ yet: from the list in 
Lemma \ref{lem:bipartitenonforestsorderedbytheirfvectors} we see that there is 
precisely one type which contains more than one copy of $C^4$, namely $K^{2,3}$, 
which contains exactly three copies. Therefore, in $h_{\geq}(n)$ every $I$ with 
$\upX_{B} \cong K^{2,3}$ has been counted exactly three times, and we did not 
overcount any of the realizations of the other isomorphism types. Thus, in order to 
arrive at $h_{C^4,\ \neg K^{2,3}}(n)$, we have to subtract three times $h_{K^{2,3}}(n)$. 
Hence, according to \cite{mathematica6.0}, 
{\scriptsize
\begin{equation}\label{eq:numberofmatriceswith6elementdomainandassociatedgraphwithC4butnotK23:viaexclusioninclusion}
h_{C^4,\ \neg K^{2,3}}(n) 
= h_{\geq}(n) - 3\cdot h_{K^{2,3}}(n) = 18 n^8 - 180 n^7 + 612 n^6 - 608 n^5 
- 774 n^4 + 1348 n^3 + 1200 n^2 - 2864 n + 1248. 
\end{equation} 
}
Substituting this into 
\eqref{eq:numberofentryspecificationeventswithunequalmeasuressizeLis6}, one indeed 
arrives at the claimed value of $\lvert \mathcal{F}^{\mathrm{M}}(6,n)\rvert$.

As to $\lvert\mathcal{F}_{\cdot 0}^{\mathrm{M}}(6,n)\rvert$, we can use 
\ref{eq:expansionofamatrixfailuresetwithapositiveratiobyisomorphismtypes} 
in Lemma \ref{lem:forfixedisomorphismtyperatioofbalancedrealizationsdependsonthebettinumberalone} to calculate $\lvert \mathcal{F}_{\cdot 0}^{\mathrm{M}}(6,n)\rvert$ $=$ 
$\sum_{\mathfrak{X}\in \im(\upXul^{6,n,n})\colon \beta_1(\mathfrak{X}) = 1}$ 
$(1-(\tfrac12)^1)$ $\cdot$ $\lvert (\upX^{6,n,n})^{-1}(\mathfrak{X}) \rvert$ $+$ 
$\sum_{\mathfrak{X}\in \im(\upXul^{6,n,n})\colon\beta_1(\mathfrak{X}) = 2}$ 
$(1-(\tfrac12)^2)$ $\cdot$ $\lvert (\upX^{6,n,n})^{-1}(\mathfrak{X}) \rvert$
$=$ (from the list in Lemma \ref{lem:bipartitenonforestsorderedbytheirfvectors}) $=$ 
 $\tfrac12$ $\cdot$ $\sum_{\mathfrak{X}\in \{\ref{item:comparisonproofcaseLhassize6:oneisolatedvertex},\dotsc,\ref{item:comparisonproofcaseLhassize6:twoadditionaldisjointedgesdisjointfromC4}\}\setminus\{\ref{item:comparisonproofcaseLhassize6:XBLisomorphictoK23}\}}$ $\lvert (\upX^{6,n,n})^{-1}(\mathfrak{X}) \rvert$ $+$ $\tfrac34$ $\cdot$ $\lvert (\upX^{6,n,n})^{-1}\ref{item:comparisonproofcaseLhassize6:XBLisomorphictoK23} \rvert$ $=$ $\tfrac12 \cdot(h_{C^4,\, \neg K^{2,3}}(n) + 
h_{C^6}(n)) + \tfrac34\cdot h_{K^{2,3}}(n)$, and it can be checked that this equals 
the claimed value of $\lvert \mathcal{F}_{\cdot 0}^{\mathrm{M}}(6,n)\rvert$.

Similarly, $\lvert \mathcal{F}_{\cdot 2}^{\mathrm{M}}(6,n)\rvert$ $=$ 
$\sum_{\mathfrak{X}\in \im(\upXul^{6,n,n})\colon \beta_1(\mathfrak{X}) = 1}$ $(\tfrac12)^1$ $\cdot$ 
$\lvert (\upX^{6,n,n})^{-1}(\mathfrak{X}) \rvert$ $=$ (from the list in 
Lemma \ref{lem:bipartitenonforestsorderedbytheirfvectors}) $=$ 
 $(\tfrac12)$ $\cdot$ $\sum_{\mathfrak{X}\in \{\ref{item:comparisonproofcaseLhassize6:oneisolatedvertex},\dotsc,\ref{item:comparisonproofcaseLhassize6:twoadditionaldisjointedgesdisjointfromC4}\}\setminus\{\ref{item:comparisonproofcaseLhassize6:XBLisomorphictoK23}\}}$ $\lvert (\upX^{6,n,n})^{-1}(\mathfrak{X}) \rvert$ $=$ $\tfrac12$ $($ $h_{C^4,\, \neg K^{2,3}}(n)$ $+$ $h_{C^6}(n)$ $)$, 
and it can be checked that this is equal to the value of 
$\lvert \mathcal{F}_{\cdot 2}^{\mathrm{M}}(6,n)\rvert$ which is claimed in 
\ref{comparativecountingtheorem:item:sixentriesspecified}. Finally, 
$\lvert\mathcal{F}_{\cdot 4}^{\mathrm{M}}(6,n)\rvert$ $=$ 
$\sum_{\mathfrak{X}\in \im(\upXul^{6,n,n})\colon\beta_1(\mathfrak{X})=2)}$ $(\tfrac12)^2$ $\cdot$ 
$\lvert (\upX^{6,n,n})^{-1}(\mathfrak{X}) \rvert$ $=$ (from the list in 
Lemma \ref{lem:bipartitenonforestsorderedbytheirfvectors}) $=$ 
 $(\tfrac14)$ $\cdot$ $\lvert (\upX^{6,n,n})^{-1}\ref{item:comparisonproofcaseLhassize6:XBLisomorphictoK23} \rvert$ $=$ $\tfrac14\cdot h_{K^{2,3}}$, and this equals the 
value of $ \lvert \mathcal{F}_{\cdot 4}^{\mathrm{M}}(6,n)\rvert$ claimed in 
\ref{comparativecountingtheorem:item:sixentriesspecified}. 
\end{proof}

\subsubsection{Alternative checks using Theorem \ref{thm:countingmatrixrealizations}}\label{subsubsec:alternativechecks}

We did not need Theorem \ref{thm:countingmatrixrealizations} in our proof of 
Theorem \ref{thm:comparativecountingtheorem}. It can, nevertheless, provide 
additional security since via Corollary \ref{cor:partitionsoffailuresets} and 
Theorem \ref{thm:countingmatrixrealizations} one may take an 
inclusion-exclusion-free (but, all told, much more laborious) alternative route 
to the claimed values of $\lvert \mathcal{F}^{\mathrm{M}}(5,n)\rvert$ and 
$\lvert \mathcal{F}^{\mathrm{M}}(6,n)\rvert$. As to the claimed value 
of $\lvert\mathcal{F}^{\mathrm{G}}(5,n)\rvert$, by 
\ref{it:partitionoffailuresets:kequals5} in Corollary 
\ref{cor:partitionsoffailuresets} combined with Theorem 
\ref{thm:countingmatrixrealizations} we have 
$\lvert\mathcal{F}^{\mathrm{G}}(5,n)\rvert$ $=$ 
\ref{numberofmatrixrealizationsoftype:oneisolatedvertex:kequals5} $+$ 
\ref{numberofmatrixrealizationsoftype:oneadditionaledgeintersectingC4:kequals5} 
$+$ \ref{numberofmatrixrealizationsoftype:twoisolatedvertices:kequals5} $+$ 
\ref{numberofmatrixrealizationsoftype:C4withoneadditionaldisjointedge:kequals5} 
$=$ $2^4\cdot \binom{n-1}{2}^2\cdot 
\bigl (  4\cdot (n-3) + 8 \cdot (n-3) + 1\cdot (n-3)^2 + 2\cdot (n-3)^2\bigr)$ $=$ 
$48\cdot ((n-1)^2 - 4)\cdot \binom{n-1}{2} \cdot \binom{n-1}{2}$. 

As to the claimed value of $\lvert\mathcal{F}^{\mathrm{G}}(6,n)\rvert$, by 
\ref{it:partitionoffailuresets:kequals6} in Corollary \ref{cor:partitionsoffailuresets}, the function  $\lvert\mathcal{F}^{\mathrm{G}}(6,n)\rvert$ equals the sum of the 
nineteen functions which were found in \ref{numberofmatrixrealizationsoftype:oneisolatedvertex}--\ref{numberofmatrixrealizationsoftype:twoadditionaldisjointedgesdisjointfromC4} of \ref{item:numberofrealizationsofnonforestswhenkequals6} in Theorem \ref{thm:countingmatrixrealizations}, and 
one can check (e.g. with \cite{mathematica6.0}) that indeed 
$\lvert\mathcal{F}^{\mathrm{G}}(6,n)\rvert = 
\sum_{2\leq k \leq 20} (\mathrm{m}6.\mathrm{t}k) = 
18n^8 - 180n^7 + \tfrac{1868}{3}n^6 - \tfrac{2176}{3}n^5 - \tfrac{754}{3}n^4 + 
\tfrac{428}{3}n^3 - \tfrac{8144}{3}n^2 - \tfrac{11536}{3}n + 1504$. 

Incidentally, let us note that by summing all functions in 
\ref{numberofmatrixrealizationsoftype:oneisolatedvertex}--\ref{numberofmatrixrealizationsoftype:twoadditionaldisjointedgesdisjointfromC4} except 
\ref{numberofmatrixrealizationsoftype:XBLisomorphictoK23} and 
\ref{numberofmatrixrealizationsoftype:threeisolatedvertices} 
(i.e. by summing seventeen functions) one may also check the equation
$h_{C^4,\ \neg K^{2,3}}(n) = h_{\geq}(n) - 3\cdot h_{K^{2,3}}(n) = 
18 n^8 - 180 n^7 + 612 n^6 - 608 n^5 - 774 n^4 + 1348 n^3 + 1200 n^2 - 2864 n + 1248$ 
claimed in the proof above.

\subsubsection{Quantitatively dominant graph-theoretical reasons for 
$\Prob_{\chio}\neq\Prob_{\lcf}$}

Let us note that Theorem \ref{thm:countingmatrixrealizations} 
tells us that of the $\lvert \mathcal{F}^{\mathrm{M}}(6,n) \rvert\in 
\Omega_{n\to\infty} (n^8)$ six-element-entry-specifications which cause non-agreement 
of $\Prob_{\chio}$ and $\Prob_{\lcf}$, most of the failures are concentrated at only 
three out of the nineteen isomorphism types in 
\ref{item:numberofrealizationsofnonforestswhenkequals6}: only the types \ref{item:comparisonproofcaseLhassize6:fourisolatedvertices}, \ref{item:comparisonproofcaseLhassize6:oneadditionaldisjointedgeandtwoisolatedvertices} and 
\ref{item:comparisonproofcaseLhassize6:twoadditionaldisjointedgesdisjointfromC4} 
have a preimage under $\upXul^{6,n,n}$ which is of size $\Omega_{n\to\infty}(n^8)$.  
The quantitative domination of these isomorphism types is, however, a rather slow 
one in that 
\[ \tfrac{\sum_{2\leq k \leq 17} \lvert(\upXul^{6,n,n})^{-1}(\mathrm{t}k)\rvert}{\lvert(\upXul^{6,n,n})^{-1}\ref{item:comparisonproofcaseLhassize6:fourisolatedvertices}\rvert + \lvert(\upXul^{6,n,n})^{-1}\ref{item:comparisonproofcaseLhassize6:oneadditionaldisjointedgeandtwoisolatedvertices}\rvert + \lvert(\upXul^{6,n,n})^{-1}\ref{item:comparisonproofcaseLhassize6:twoadditionaldisjointedgesdisjointfromC4}\rvert} \in \Theta(n^{-1})\quad .\]

\subsubsection{Estimate of the number of failures of equality of $\Prob_{\chio}$ and 
$\Prob_{\lcf}$ for events $\mathcal{E}_B$ with $B\in\{0,\pm\}^I$ and 
$I\in\binom{[n-1]^2}{k}$ and $k$ general}
 
\begin{proposition}[for fixed $k$ the measures 
$\Prob_{\chio}$ and $\Prob_{\lcf}$ agree for almost all entry-specifications]
\label{prop:roughestimationofnumberoffailureevents}
For every fixed $k\geq 1$ we have 
$\lvert \mathcal{F}^{\mathrm{M}}(k,n)\rvert$ $/$ $\lvert \{ B\in \{0,\pm\}^I,\; 
I\in\binom{[n-1]^2}{k}\}\rvert$ $\in$ 
$\mathcal{O}_{n\to\infty}(n^{-2})$\quad .
\end{proposition}
\begin{proof}
We will estimate numerator and denominator of this fraction separately. The 
denominator is equal to $3^k\cdot \binom{(n-1)^2}{k} \in \Omega_{n\to\infty}(n^{2k})$. 
Moreover, a very rough estimate suffices to obtain a bound on the numerator which 
nevertheless is sufficiently small to prove that the ratio vanishes: 
$\lvert \mathcal{F}^{\mathrm{M}}(k,n) \rvert$ 
$\By{Theorem \ref{thm:graphtheoreticalcharacterizationofthechiomeasure}.\ref{relationbeweenchiomeasureandlazycoinflipmeasuregovernedbyfirstbettinumber}}{=}$ 
$\lvert \{$ $B$ $\in$ $\{0,\pm\}^I\colon$ 
$I\in\binom{[n-1]^2}{k}$, $B$ contains a matrix-circuit $\} \rvert$  
$=$ $\bigl \lvert$ $\bigcup_{1\leq j \leq \lfloor \tfrac{k}{2} \rfloor}$ 
$\bigcup_{L\in\Cir(2j,n)}$ $\{$ $B$ $\in$ $\{0,\pm\}^I\colon$ $I\in\binom{[n-1]^2}{k}$, 
$L\subseteq \Supp(B)$ $\}$ $\bigr \rvert$ $\leq$ 
$\sum_{1\leq j \leq \lfloor \tfrac{k}{2} \rfloor}$ $\sum_{L\in\Cir(2j,n)}$ 
$\lvert \{$ $B$ $\in$ $\{0,\pm\}^I\colon$ $I$ $\in$ $\binom{[n-1]^2}{k}$, 
$L$ $\subseteq$ $\Supp(B)$ $\} \rvert$ $=$ $\sum_{1\leq j \leq \lfloor \tfrac{k}{2} \rfloor}$ 
$2^{2j}$ $\cdot$ $3^{k-2j}$ $\cdot$ $\binom{(n-1)^2-2j}{k-2j}$ $\cdot$ 
$\lvert \Cir(2j,n) \rvert$ $\By{Lemma \ref{lem:numberofmatrixcircuitsofgivenlengthwithingivencartesianproduct}}{=}$ $\sum_{1\leq j \leq \lfloor \tfrac{k}{2} \rfloor}$  $2^{2j}$ 
$\cdot$ $3^{k-2j}$ $\cdot$ $\binom{(n-1)^2-2j}{k-2j}$ $\cdot$ $\binom{n-1}{j}^2$ 
$\cdot$ $\frac{j!(j-1)!}{2}$ $\in$ $\sum_{1\leq j \leq \lfloor \tfrac{k}{2} \rfloor}$ 
$\mathcal{O}_{n\to\infty}(1)$ $\cdot$ $\mathcal{O}_{n\to\infty}(n^{2k-4j})$ $\cdot$ 
$\mathcal{O}_{n\to\infty}(n^{2j})$ $\cdot$ $\mathcal{O}_{n\to\infty}(1)$  $\subseteq$ 
$\sum_{1\leq j \leq \lfloor \tfrac{k}{2} \rfloor}$ $\mathcal{O}_{n\to\infty}(n^{2k-2j})$ 
$\subseteq$ $\mathcal{O}_{n\to\infty}(n^{2k-2})$.
\end{proof}

\section{Connection to counting singular $\{\pm\}$-matrices}

\subsection{Basic connections}\label{subsec:basicconnections}

The Lemmas \ref{lem:chiocondensationaffectsrankintheleastpossibleway} and 
\ref{lem:uniformmeasureofranksetaschiomeasureofranksetoneranklower}, which are 
consequences of Chio's identity \ref{lem:chioidentity}, are the basic reason why 
the measure $\Prob_{\chio}$ is relevant for the study of singular $\pm$-matrices.

\begin{lemma}[Chio condensation affects rank to the least possible degree]
\label{lem:chiocondensationaffectsrankintheleastpossibleway}
For every integral domain $R$, every $(s,t)\in \Z_{\geq 2}^2$ and 
every $A\in R^{[s]\times [t]}$ with $a_{s,t}\neq 0$ we have 
$\rk(\tfrac12\upC_{(s,t)}(A)) = \rk(A)-1$.
\end{lemma}
\begin{proof}
If $\rk(A) = 1$, then obviously $\upC_{(s,t)}(A) = \{0\}^{[s-1]\times [t-1]}$ and the 
claim is true.  We may therefore assume that $r:=\rk(A)\geq 2$. By the equality 
of rank and determinantal rank over integral domains 
(cf. e.g. \cite[Corollary 2.29(2)]{MR1181420}) there 
exists $S\in\binom{[s]}{r}$ and $T\in \binom{[t]}{r}$ such that 
$\det(A\mid_{S\times T}) \neq 0$. If $s\notin S$, then by temporarily passing to the 
field of fractions of $R$ we may appeal to Steinitz' exchange lemma for vector spaces 
to prove the existence of at least one $i_0\in S$ such that 
$\det\bigl ( A\mid_{((S\setminus \{i_0\})\sqcup \{s\})\times T} \bigr )  \neq 0$. 
Analogously for $t\notin T$. Therefore we may assume that $s\in S$ and $t\in T$. Hence 
$S\times T = ((S\setminus\{s\})\times (T\setminus\{t\}))^{\breve{}}$ and therefore 
$\upC_{(s,t)}(A\mid_{S\times T})$ is defined. By Lemma \ref{lem:chioidentity} we know 
that $\det(\upC_{(s,t)}(A\mid_{S\times T})) = 
a_{s,t}^{r-2} \cdot \det(A\mid_{S\times T}) \neq 0$, the latter since $R$ is an 
integral domain and $a_{s,t}\neq 0$ by assumption. Since $\upC_{(s,t)}(A\mid_{S\times T}) = 
\upC_{(s,t)}(A)\mid_{(S\setminus\{s\})\times(T\setminus\{t\})} \in R^{(r-1)\times (r-1)}$, and 
by the equality of rank and determinantal rank, this implies 
$\rk(\upC_{(s,t)}(A)) \geq r-1$. On the other hand we also have 
$\rk(\upC_{(s,t)}(A)) \leq r-1$. To see this, it suffices to note that every 
$r\times r$ submatrix of $\upC_{(s,t)}(A)$ is the Chio condensate of an 
$(r+1)\times(r+1)$ submatrix of $A$, hence by Chio's identity a 
nonvanishing $r\times r$ minor of $\upC_{(s,t)}(A)$ would imply a nonvanishing 
$(r+1)\times (r+1)$ minor of $A$, contrary to the assumption of $\rk(A) = r$. 
\end{proof}

\begin{lemma}\label{lem:uniformmeasureofranksetaschiomeasureofranksetoneranklower}
$\Prob\bigl [ \Ra_r(\{\pm\}^{[s]\times[t]})\bigr ] = 
\Prob_{\chio} \bigl [ \Ra_{r-1}(\{0,\pm\}^{[s-1]\times[t-1]})\bigr ]$ for every 
$(s,t)\in\Z_{\geq 2}^2$ and $1\leq r \leq \min(s,t)$. 
\end{lemma}
\begin{proof}
This follows from the calculation 
$\Prob$ $\bigl [$ $\Ra_r($ $\{\pm\}^{[s]\times[t]})$ $\bigr ]$ $=$ $\frac{1}{2^{s\cdot t}}$ 
$\lvert\{$ $A$ $\in$ $\{\pm\}^{[s]\times [t]}\colon$ $\rk(A)$ $=$ $r$ $\}\rvert$ 
$\By{Lemma \eqref{lem:chiocondensationaffectsrankintheleastpossibleway}}{=}$ 
$\frac{1}{2^{s\cdot t}}$ $\sum_{B\in\{0,\pm\}^{[s-1]\times[t-1]}\colon\rk(B) = r-1}$ $\lvert ($ 
$\tfrac12\upC_{(s,t)}$ $)^{-1}$ $(B)$ $\rvert$ 
$\By{Definition \ref{eq:definitionmeasurePchio}}{=}$ 
$\Prob_{\chio}$ $[$ $\Ra_{r-1}($ $\{$ $0,$ $\pm$ $\}^{[s-1]\times [t-1]}$ $)$ $]$. Note, 
incidentally, that with the third equality sign, many zero-summands are introduced. 
\end{proof}

\begin{corollary}\label{cor:uniformmeasureofranksublevelsetaschiomeasureofranksublevelsetoneranklower}
$\Prob\bigl [ \Ra_{\mathcal{R}}(\{\pm\}^{[s]\times[t]})\bigr ] = 
\Prob_{\chio} \bigl [ \Ra_{\mathcal{R}}(\{0,\pm\}^{[s-1]\times[t-1]})\bigr ]$ for every 
$(s,t)\in\Z_{\geq 2}^2$ and every $\mathcal{R}\in\mathfrak{P}([\min(s,t)]\sqcup\{0\})$.
\end{corollary}
\begin{proof}
Immediate from Lemma \ref{lem:uniformmeasureofranksetaschiomeasureofranksetoneranklower} and Definition \ref{def:ranksets}.
\end{proof}

\begin{corollary}\label{cor:implicationofbourgainvuwoodtheoremonchiomeasure}
$\Prob_{\chio} [ \Ra_{<n-1}(\{0,\pm\}^{[n-1]^2}) ] \leq (1/\sqrt{2} + o(1))^n$ 
for $n\rightarrow \infty$.
\end{corollary}
\begin{proof}
By combining Lemma \ref{cor:uniformmeasureofranksublevelsetaschiomeasureofranksublevelsetoneranklower} with \eqref{thm:bourgainvuwood:uniformdistribution} in 
Theorem \ref{thm:bourgainvuwood}. 
\end{proof}

\subsection{Sign functions which are both singular and balanced}
\label{subsec:countingsingularandcyclicallyrevenedgecolourings}

\begin{definition}[$G_{s,t}$]
For every $(s,t)\in\Z_{\geq 2}^2$ define $G_{s,t} := \bigl(\bigoplus_{1\leq i \leq s-1}\Z/2\bigr) \oplus \bigl(\bigoplus_{1\leq j\leq t-1}\Z/2\bigr)$.
\end{definition}

We will use the following group actions. Informally, $\alpha_X$ is the action by 
switching signs of edges simultaneously in all `stars' centered at those vertices 
for which $g$ has nonzero components. 

\begin{definition}\label{def:thetwoswitchinggroupactions}
For every $(s,t)\in\Z_{\geq 2}^2$ define the group action 
$\alpha_{s,t}\colon G_{s,t} \longrightarrow \Sym(\{0,\pm\}^{[s-1]\times[t-1]})$, 
$((g_i)_{i\in [s-1]} , (g_j)_{j\in [t-1]} ) \longmapsto 
\binom{\{0,\pm\}^{[s-1]\times [t-1]}\to\{0,\pm\}^{[s-1]\times [t-1]}}{(b_{i,j})_{(i,j)\in {[s-1]\times [t-1]}} \mapsto (-1)^{g_i}\cdot (-1)^{g_j} \cdot b_{i,j}}$. For every $X\in \BG_{s,t}$ 
define the group action $\alpha_X\colon G_{s,t} \longrightarrow 
\Sym( \{\pm\}^{\upE(X)})$ which is defined by $(\alpha_X(g)(\sigma))(e) := 
(-1)^{g_i}\cdot (-1)^{g_j} \cdot \sigma(e)$ for every $\sigma\in\{\pm\}^{\mathrm{E}(X)}$ 
and every $e = \{ (i,t), (s,j)\} \in \upE(X)$. 
\end{definition}

Note that neither $\alpha_{s,t}$ nor $\alpha_X$ are faithful group actions. More 
precisely, not only is both $\ker(\alpha_{s,t})$ and $\ker(\alpha_X)$ a 
$2$-element set, but $\alpha_{s,t}$ and $\alpha_X$ are both double-covers onto their 
images. We could construct a faithful action by making an arbitrary choice of 
a $\iota\in [s-1]\cup [t-1]$ and then refraining from switching at this index 
(analogously, by  making an arbitrary choice of a star in $X$ and then refraining 
from switching that particular star). 

Balancedness is a very `rigid' property of an edge-signing in that it is 
determined by the signing of an arbitrary spanning tree: 

\begin{lemma}[rigidity of balanced edge signings]\label{lem:setofbalancedsigningsdeterminedbysigningsofarbitraryspanningtree}
For every connected graph $X$ and every spanning tree $T$ of $X$, there is a 
bijection $\{\pm\}^{\upE(T)} \leftrightarrow 
\{ \sigma\in \{\pm\}^{\upE(X)}\colon \text{$(X,\sigma)$ balanced}\}$. 
\end{lemma}
\begin{proof}[Sketch of proof] 
Since the balance-preserving sign of every edge $e\in\upE(X)\setminus\upE(T)$ is 
determined by the unique circuit in $\upE(T)\cup\{e\}$, for every given 
$\sigma\in \{\pm\}^{\upE(T)}$, there is at most one balanced extension of $\sigma$, 
i.e. at most one $\tilde{\sigma}\in \{\pm\}^{\upE(X)}$ with 
$\sigma = \tilde{\sigma}\mid_{\upE(T)}$ and $(X,\tilde{\sigma})$  balanced. Moreover, 
this extension can be constructed in the obvious `greedy' way by 
successively adding in the elements of 
$\upE(X)\setminus\upE(T)$ in an arbitrary order while at each step of the 
construction choosing the sign of the added edge so as to avoid non-balanced 
circuits. That this is indeed possible can be proved by an induction on the 
number $\lvert\upE(X)\setminus\upE(T)\rvert$ of edges to be added. A key 
observation (routine to prove and known since at least \cite[Theorem 2]{MR0067468}) 
is that at each step of the construction, for each pair of vertices 
either \emph{all} paths within the partially constructed graphs 
which have these two vertices as endvertices have sign $(-)$ or \emph{all} such 
paths have sign $(+)$, and therefore the greedy construction never stalls. 
\end{proof}

\begin{lemma}\label{lem:starswitchingistransitiveonbalancedsignings}
For every graph $X$ the restriction $\im(\alpha_X)\mid_{S_{\bal}(X)}$ is a 
transitive permutation group on $S_{\bal}(X)$. 
\end{lemma}
\begin{proof}[Sketch of proof]
One way to look at this is as `making use of the rigidity of balanced signings': we 
can choose an arbitrary spanning tree $T_i$ for each connected component $X_i$ 
of $X$, then show that $\im(\alpha_X)\mid_{S_{\bal}(X)}$ is transitive on the set 
$\{\pm\}^{\upE(T_i)}$ of all edge-signings of $T_i$, and then appeal to 
Lemma \ref{lem:setofbalancedsigningsdeterminedbysigningsofarbitraryspanningtree} 
which says that this transitivity already implies transitivity on the full 
set $S_{\bal}(X)$. 
\end{proof}

Given a $\{0,1\}$-matrix, it can be possible to increase its $\Z$-rank by choosing 
signs for the entries.  If we require the signed matrix to be \emph{balanced}, 
however, the rank must stay the same. This follows quickly from the 
graph-theoretical considerations above: 

\begin{proposition}[all balanced signings of a $\{0,1\}$-matrix have the 
same rank]\label{prop:allbalancedsigningshavethesamerank}
Let $B\in \{0,1\}^{[s-1]\times [t-1]}$. Let $\tilde{B}$ be an arbitrary 
`balanced signing of $B$', i.e. $\tilde{B} \in\{0,\pm\}^{[s-1]\times [t-1]}$, 
$\Supp(\tilde{B}) = \Supp(B)$ and 
$(\upX_{\tilde{B}},\sigma_{\tilde{B}}) = (\upX_{B},\sigma_{\tilde{B}})$ is a balanced 
signed graph. Then $\rk(\tilde{B}) = \rk(B)$. 
\end{proposition}
\begin{proof}
Since both $(\upX_B,\sigma_B)$ and $(\upX_{\tilde{B}},\sigma_{\tilde{B}})$ are balanced, 
by Lemma \ref{lem:starswitchingistransitiveonbalancedsignings} there exists 
$g\in G_{s,t}$ such that $\alpha_X(g)(\sigma_{\tilde{B}}) = \sigma_B$. In view of 
Definition \ref{def:thetwoswitchinggroupactions} and 
Definition \ref{def:XBandecXB},  this implies $\alpha_{s,t}(g)(\tilde{B}) = B$. 
Since $\alpha_{s,t}$ obviously keeps the rank invariant, the claim is proved. 
\end{proof}

We will now use the knowledge established so far to analyse the 
tempting `absolute' route of using Corollary \ref{cor:uniformmeasureofranksublevelsetaschiomeasureofranksublevelsetoneranklower} and then partitioning according to isomorphism type of the associated bipartite graph. The conclusion is that this will lead us 
onto a well-beaten path (counting singular $\{0,1\}$-matrices):

\begin{proposition}[on rank-level-sets, the Chio measure agrees with the uniform 
measure after forgetting the signs]\label{prop:chiomeasureasuniformmeasureonchiosets}
Let $(s,t)\in\Z_{\geq 2}^2$ and $\mathcal{R}\in\mathfrak{P}(\{1,\dotsc,\min(s,t)\})$. 
Then 
\begin{equation}
\Prob_{\chio}[\Ra_{\mathcal{R}}(\{0,\pm\}^{[s-1]\times [t-1]})] = 
\Prob[\Ra_{\mathcal{R}}(\{0,1\}^{[s-1]\times [t-1]})] \quad .
\end{equation}
\end{proposition}
\begin{proof}
This follows from the calculation
\begin{align}
\Prob_{\chio}[\Ra_{\mathcal{R}}(\{0,\pm\}^{[s-1]\times [t-1]})] 
& \By{\text{\tiny (\ref{characterizationofwhenchiomeasureispositive} in 
Theorem \ref{thm:graphtheoreticalcharacterizationofthechiomeasure})}}{=}
\Prob_{\chio} \bigl [\bigl \{ B\in \{0,\pm\}^{[s-1]\times [t-1]}\colon 
\parbox{0.15\linewidth}{\tiny $\rk(B) \in \mathcal{R}$, 
$(\upX_B,\sigma_B)$ balanced} \bigr \} \bigr ] \notag \\
& = \sum_{\mathfrak{X}\in\ul(\BG_{s,t})} \Prob_{\chio} \bigl [\bigl 
\{ B\in \{0,\pm\}^{[s-1]\times [t-1]}\colon 
\parbox{0.16\linewidth}{\tiny $\rk(B) \in \mathcal{R}$, 
$\upX_B\cong \mathfrak{X}$, \\ $(\upX_B,\sigma_B)$ balanced} \bigr \} \bigr ] \notag \\
\text{\tiny (by \ref{cor:chiomeasureofasinglematrix} in Corollary \ref{cor:quickconsequencesofthecharacterizations})} 
& = \sum_{\mathfrak{X}\in\ul(\BG_{s,t})} 2^{-st + \beta_0(\mathfrak{X}) + 1} \cdot 
\lvert \{ \text{ \tiny $B\in \{0,\pm\}^{[s-1]\times [t-1]}\colon$ } 
\parbox{0.16\linewidth}{\tiny $\rk(B)\in\mathcal{R}$, $\upX_B\cong \mathfrak{X}$, \\ 
$(\upX_B,\sigma_B)$ balanced} \bigr \} \rvert \notag \\
\parbox{0.2\linewidth}{\tiny (by \ref{item:numberofbalancedsignfunctions} in 
Lemma \ref{lem:equivalenceofexistenceofbconstantrpropervertex2coloringandcyclicallyrevenness} \\ and Proposition \ref{prop:allbalancedsigningshavethesamerank})} & = 
\sum_{\mathfrak{X}\in\ul(\BG_{s,t})} 2^{f_0(\mathfrak{X})-st + 1} \cdot \lvert 
\{ \text{ \scriptsize $B\in \{0,1\}^{[s-1]\times [t-1]}\colon 
\rk(B)\in\mathcal{R},\; \upX_B\cong \mathfrak{X}$ }\} \rvert \notag \\
\parbox{0.2\linewidth}{\tiny ( since $f_0(\mathfrak{X})$ is equal \\ 
to $(s-1)+(t-1)$ for \\ 
every $\mathfrak{X}\in\ul(\BG_{s,t})$ )} & = \sum_{\mathfrak{X}\in\ul(\BG_{s,t})} 
2^{-(s-1)(t-1)} \cdot \lvert \{ \text{ \scriptsize $B\in \{0,1\}^{[s-1]\times [t-1]}\colon 
\rk(B)\in\mathcal{R},\; \upX_B\cong \mathfrak{X}$ }\} \rvert \notag  \\ 
& = (\tfrac12)^{(s-1)(t-1)}\cdot \lvert \{ B\in \{0,1\}^{[s-1]\times[t-1]}\colon
\rk(B)\in\mathcal{R} \}\rvert  \notag \\
 & = \Prob[\Ra_{\mathcal{R}}(\{0,1\}^{[s-1]\times[t-1]})]\quad .\notag 
\end{align}
The proof of Proposition \ref{prop:chiomeasureasuniformmeasureonchiosets} is now 
complete. 
\end{proof}

It should be noted that the special case 
$\Prob[\Ra_{<n}(\{\pm\}^{[n]^2})] = \Prob[\Ra_{<n-1}(\{0,1\}^{[n-1]^2})]$ seems well-known 
(the author does not have an explicit reference corroborating this, but there are 
articles in which this is implicit (e.g. \cite{MR2216479}).

\subsection{A relative point of view}

What really appears to promise progress on Conjecture \ref{conj:socalledfolkloreconjecture} is to use the theorem of Bourgain--Vu--Wood to reach a more relative 
vantage point: 

\begin{proposition}[relative formulations of Conjecture \ref{conj:socalledfolkloreconjecture}]\label{prop:equivalentnewformulationofoldconjecture}
The following statements are equivalent: 
{\scriptsize
\begin{enumerate}[label={\rm(Q\arabic{*})}]
\item\label{lem:equivalentproblem:it:folkloreconjecture} 
$\Prob\bigl [ \Ra_{<n}(\{\pm\}^{[n]^2}) \bigr ] 
\leq (\frac12+o_{n\to\infty}(1))^n$
\item\label{eq:equivalentformulationoftheconjecturewithabbreviations}
$\Prob_{\chio}\bigl [ \Ra_{<n-1} ( \{0,\pm\}^{[n-1]^2} )\bigr ] 
\leq (\frac12+o_{n\to\infty}(1)) \cdot \Prob_{\lcf} \bigl [ \Ra_{<n-1}( \{0,\pm\}^{[n-1]^2} )\bigr ]$
\item\label{conj:relativizedconjecture}
$\sum_{ B'\ \in \  \Ra_{<n-1}( \{0,\pm\}^{[n-1]^2} ) } \Prob_{\chio}[B'] \leq  
\bigl(\frac12+o_{n\to\infty}(1)\bigr) \cdot \sum_{ B''\ \in \  \Ra_{<n-1}( \{0,\pm\}^{[n-1]^2} ) } 
\Prob_{\lcf} \bigl [ B'' \bigr ]$
\item\label{arelativizedformulation} 
$\lvert \{ B'\in\{0,1\}^{[n-1]^2}\colon \rk(B') < n-1 \} \rvert \leq 
(\tfrac12 + o_{n\to\infty}(1)) \cdot \sum_{B\in\{0,1\}^{[n-1]^2}} (\tfrac12)^{\supp(B)} \cdot 
\Biggl \lvert  \Biggl \{ \parbox{0.2\linewidth}{$B''\in\{0,\pm\}^{[n-1]^2}\colon$ \\ 
$\Supp(B'')=\Supp(B)$, $\rk(B'') < n-1$ }  \Biggr \} \Biggr \rvert$
\end{enumerate}
}
\end{proposition}
\begin{proof}
As to the equivalence 
\ref{lem:equivalentproblem:it:folkloreconjecture} $\Leftrightarrow$ 
\ref{eq:equivalentformulationoftheconjecturewithabbreviations}, 
if \ref{lem:equivalentproblem:it:folkloreconjecture}, then 
$\Prob_{\chio}\bigl [ \Ra_{<n-1} ( \{0,\pm\}^{[n-1]^2} )\bigr ]$ 
$\By{Corollary \eqref{cor:uniformmeasureofranksublevelsetaschiomeasureofranksublevelsetoneranklower}}{=}$ 
$\Prob \bigl [ \Ra_{<n} ( \{\pm\}^{[n]^2} )\bigr ]$ 
$\By{\ref{lem:equivalentproblem:it:folkloreconjecture}}{\leq} $
$(\frac12 + o_{n\to\infty}(1))^n$ $=$ $(\frac12 + o_{n\to\infty}(1))\cdot$ 
$(\frac12 + o_{n\to\infty}(1))^{n-1}$ $\By{Theorem \ref{thm:bourgainvuwood}}{\sim}$ 
$(\frac12 + o_{n\to\infty}(1))$ $\cdot$ $\Prob_{\lcf} \bigl[\Ra_{<n-1}( \{0,\pm\}^{[n-1]^2} )\bigr ]$, 
which is \ref{eq:equivalentformulationoftheconjecturewithabbreviations}. As to the 
converse, \ref{eq:equivalentformulationoftheconjecturewithabbreviations} implies 
$\Prob \bigl [ \Ra_{<n} ( \{\pm\}^{[n]^2} )\bigr ]$ 
$\By{Corollary \eqref{cor:uniformmeasureofranksublevelsetaschiomeasureofranksublevelsetoneranklower}}{=}$ $\Prob_{\chio}\bigl [ \Ra_{<n-1} ( \{0,\pm\}^{[n-1]^2} )\bigr ]$ 
$\By{\ref{eq:equivalentformulationoftheconjecturewithabbreviations}}{\leq}$ 
$(\frac12 + o_{n\to\infty}(1))$ $\cdot$ $\Prob_{\lcf} \bigl [ \Ra_{<n-1}( \{0,\pm\}^{[n-1]^2} )\bigr]$ 
$\By{Theorem \ref{thm:bourgainvuwood}}{\sim}$ $(\frac12 + o_{n\to\infty}(1))\cdot$ 
$(\frac12 + o_{n\to\infty}(1))^{n-1}$ $=$  $(\frac12 + o_{n\to\infty}(1))^n$, which is \ref{lem:equivalentproblem:it:folkloreconjecture}.
The equivalence \ref{eq:equivalentformulationoftheconjecturewithabbreviations} 
$\Leftrightarrow$ \ref{conj:relativizedconjecture} is obvious. As to 
\ref{conj:relativizedconjecture} $\Leftrightarrow$ \ref{arelativizedformulation}, 
note that 
\ref{conj:relativizedconjecture} $\By{Proposition \ref{prop:chiomeasureasuniformmeasureonchiosets}}{\Leftrightarrow}$
$\sum$ $\llbracket$ $B'$ $\in$ $\{0,1\}^{[n-1]^2}\colon$ $\rk(B')$ $<$ $n-1$ 
$\rrbracket$ $\Prob[B']$ $\leq$ $\bigl(\frac12+o_{n\to\infty}(1)\bigr)$ $\cdot$ 
$\sum$ $\llbracket$ $B''$ $\in$ $\Ra_{<n-1}( \{0,\pm\}^{[n-1]^2} )$ $\rrbracket$ 
$\Prob_{\lcf}$ $\bigl [$ $B''$ $\bigr ]$ 
$\Leftrightarrow$ $\lvert \{$ $B'$ $\in$ $\{0,1\}^{[n-1]^2}\colon$ $\rk(B')$ 
$<$ $n-1$ $\} \rvert$ 
$\leq$ $(\tfrac12 + o_{n\to\infty}(1))$ $\cdot$ $\sum_{B''\in\{0,\pm\}^{[n-1]^2}\colon\rk(B')<n-1}$ 
$(\tfrac12)^{\supp(B'')}$
$\Leftrightarrow$
$\lvert \{$ $B'$ $\in$ $\{0,1\}^{[n-1]^2}\colon$ $\rk(B')$ $<$ $n-1$ $\} \rvert$ $\leq$ 
$(\tfrac12 + o_{n\to\infty}(1))$ $\cdot$ $\sum_{B\in\{0,1\}^{[n-1]^2}}$ 
$\sum$ $\llbracket$ $B''$ $\in$ $\{0,\pm\}^{[n-1]^2}\colon$ $\Supp(B'')$ $=$ 
$\Supp(B)$, $\rk(B'')$ $<$ $n-1$ $\rrbracket$  $(\tfrac12)^{\supp(B'')}$ 
$\Leftrightarrow$
$\lvert\{$ $B'$ $\in$ $\{0,1\}^{[n-1]^2}\colon$ $\rk(B') < n-1$ $\}$ $\rvert$ 
 $\leq$ $(\tfrac12 + o_{n\to\infty}(1))$ $\cdot$ $\sum_{B\in\{0,1\}^{[n-1]^2}}$ 
$(\tfrac12)^{\supp(B)}$ 
$\cdot$ $\lvert$ $\{$ $B''$ $\in$ $\{0,\pm\}^{[n-1]^2}\colon$ $\Supp(B'')=\Supp(B)$, 
$\rk(B'')$ $<$ $n-1$ $\}\rvert$ $\Leftrightarrow$ \ref{arelativizedformulation}.
\end{proof}

Note the `relativizing' effect of having two sums over the same index set on either 
side of an (conjectured) inequality: thanks to commutativity of addition one may 
go about pitting (collections of) unequally indexed summands on both sides of 
\ref{conj:relativizedconjecture} against one another, in the hope of finding a 
rearragement that allows one to prove the inequality without any a priori knowledge 
about the size of the index set of the sums. Of course, \emph{if} 
\ref{lem:equivalentproblem:it:folkloreconjecture} is true, \emph{then} the 
inequality is true for \emph{every} permutation of the summands but the point is 
that this is not known and that it would suffice to prove the existence of only 
one suitable rearrangement of the summands to prove (or maybe disprove) 
Conjecture \ref{lem:equivalentproblem:it:folkloreconjecture}. 

\subsubsection{The inequality \ref{eq:equivalentformulationoftheconjecturewithabbreviations} must fail on entry-specification events} 

Let us remark that in view of the formula 
\ref{relationbeweenchiomeasureandlazycoinflipmeasuregovernedbyfirstbettinumber} in 
Theorem \ref{thm:graphtheoreticalcharacterizationofthechiomeasure} we find ourselves 
in a slightly ironic situation: while 
\ref{eq:equivalentformulationoftheconjecturewithabbreviations}, which speaks about 
the $\Prob_{\chio}$-measure of the (rather mysterious) event 
$\Ra_{<n-1}(\{0,\pm\}^{[n-1]^2})$, may well be true, it cannot possibly be true in a 
non-trivial way (left-hand side nonzero) on any of the (rather simple) 
entry-specification events. 

\subsubsection{Worst possible failure of \ref{eq:equivalentformulationoftheconjecturewithabbreviations} for singleton events}\label{worstpossiblefailure}
Already in Corollary \ref{cor:recipe} we have seen examples that 
the inequality \ref{eq:equivalentformulationoftheconjecturewithabbreviations} can fail when the 
event $\Ra_{<n-1}(\{0,\pm\}^{[n-1]^2})$ is replaced by other events---the failure 
seeming more likely and more severe as the events get smaller. We will now see  
that \ref{eq:equivalentformulationoftheconjecturewithabbreviations} fails arbitrarily 
badly on every atom of the measure space we are dealing with (i.e. a singleton 
event $\{B\}$ with $B\in \{0,\pm\}^{[n-1]^2}$ and $\Prob_{\chio}[B] > 0$). For such 
events the ratio of $\Prob_{\chio}$ and $\Prob_{\lcf}$ diverges as quickly 
as $2^{n^2}$ when $n\rightarrow \infty$, 
while \ref{eq:equivalentformulationoftheconjecturewithabbreviations} asserts a 
bounded ratio as $n\rightarrow \infty$. 

By \ref{relationbeweenchiomeasureandlazycoinflipmeasuregovernedbyfirstbettinumber} 
in Theorem \ref{thm:graphtheoreticalcharacterizationofthechiomeasure}, we know
$\Prob_{\chio}[ \mathcal{E}_B^J ]/\Prob_{\lcf} [\mathcal{E}_B^J] = 2^{\beta_1( \upX_B )}$. 
Therefore, to determine the maximum of $\Prob_{\chio}[ \mathcal{E}_B^J ]/\Prob_{\lcf} [\mathcal{E}_B^J]$ over all entry specification events $\mathcal{E}_B^J$ it suffices to 
determine the maximum of $\beta_1( \upX_B )$ over all bipartite graphs 
$\upX_B\in \BG_{n,n}$ with $B\in \{0,\pm\}^{[n-1]^2}$. The Betti number 
$\beta_1(X) = f_1(X) - f_0(X) + \beta_0(X)$ as a function of $X\in\BG_{n,n}$ attains 
a unique maximum at $X = K^{n-1,n-1}$. The corresponding value is 
$(n-1)^2 - 2(n-1) + 1 = (n-2)^2$. Since $K^{n-1,n-1}$ can indeed occur as $\upX_B$ with 
$B$ $\in$ $\Ra_{<n-1}(\{0,\pm\}^{[n-1]^2})$ $\cap$ 
$\im( \tfrac12 \upC_{(n,n)}\colon$ $\{\pm\}^{[n]^2}$ $\rightarrow$ 
$\{0,\pm\}^{[n-1]^2} )$, it follows that for every fixed $n$, the maximum of 
$\Prob_{\chio}[ \mathcal{E}_B^J ] / \Prob_{\lcf} [\mathcal{E}_B^J]$ over all 
$\mathcal{E}_B^J$ with $\emptyset\neq I \subseteq J \subseteq [n-1]^2$ and 
$B\in\{0,\pm\}^I \cap \im( \tfrac12 \upC_{(n,n)}\colon 
\{\pm\}^{[n]^2} \rightarrow \{0,\pm\}^{[n-1]^2} )$ is $2^{(n-2)^2}\sim 2^{n^2}$. 
Since Proposition \ref{prop:allbalancedsigningshavethesamerank} implies that every 
$B\in\{0,\pm\}^I \cap \im( \tfrac12 \upC_{(n,n)}\colon 
\{\pm\}^{[n]^2} \rightarrow \{0,\pm\}^{[n-1]^2} )$ which realizes the maximum, 
i.e. $\upX_B\cong K^{n-1,n-1}$, has rank $1$, it follows that $2^{(n-2)^2}$ is also the 
maximum of $\Prob_{\chio}[ \mathcal{E}_B^J ] / \Prob_{\lcf} [\mathcal{E}_B^J]$ over all 
$\mathcal{E}_B^J$ with $\emptyset\neq I \subseteq J \subseteq [n-1]^2$ and 
$B\in\Ra_{<n-1}(\{0,\pm\}^{[n-1]^2}) \cap \im( \tfrac12 \upC_{(n,n)}\colon 
\{\pm\}^{[n]^2} \rightarrow \{0,\pm\}^{[n-1]^2} )$.

\subsubsection{Extent of failure of \ref{eq:equivalentformulationoftheconjecturewithabbreviations} on those $B$ 
which are Chio condensates of random $A\in\{\pm\}^{[n]^2}$} 

Let us have a quick informal look at the typical value of the ratio 
$\Prob_{\chio}[B]$ and $\Prob_{\lcf}[B]$ if $\{0,\pm\}^{[n-1]^2}\ni B := \upC_{(n,n)}(A)$ 
with $A\in\{\pm\}^{[n]^2}$ chosen uniformly at random. Of course, such a $B$ is 
(by Theorem \ref{thm:bourgainvuwood} and Lemma \ref{lem:chiocondensationaffectsrankintheleastpossibleway}) asymptotically almost surely \emph{not} an element of 
$\Ra_{<n-1}(\{0,\pm\}^{[n-1]^2})$. 

By Corollary \ref{cor:whatrandombipartitegraphisX12CnnAforrandomA}, 
for  $A\in\{\pm\}^{[n]^2}$ chosen uniformly at random, the graph $\upX_B$ is a 
random bipartite graph with partition classes of $n-1$ vertices on either side 
and having i.i.d. edges with probability $\frac12$. Since this is very far above 
the threshold proved by Frieze \cite{MR829352} for hamiltonicity of a random 
bipartite graph, it follows that (for \emph{very} strong reasons) a.a.s. 
$\beta_0(\upX_B) = 1$. As to $f_1( \upX_B )$, a standard argument using Chernoff's 
bound shows that for every $\epsilon>0$ a.a.s. (and approaching $1$ exponentially 
fast) 
$(\frac12-\epsilon)\cdot (n-1)^2 \leq f_1( \upX_B ) \leq (\frac12 + \epsilon) \cdot 
(n-1)^2$. For simplicity let us pretend that $f_1( \upX_B ) = \frac12 (n-1)^2$ 
exactly. Then $\Prob_{\chio}[ B ] / \Prob_{\lcf} [B] = 2^{\beta_1(\upX_B)} = 
2^{f_1(\upX_B) - f_0(\upX_B) + \beta_0(\upX_B)} = 2^{\frac12 n^2 - 3n + \frac32} \sim 2^{\frac12 n^2} = 
\sqrt{2^{n^2}} \rightarrow \infty$ as $n\rightarrow \infty$, and we have learned that 
the worst-case failure-ratio of \ref{eq:equivalentformulationoftheconjecturewithabbreviations} found 
in \ref{worstpossiblefailure} arises roughly by squaring the failure-ratio for a 
$B = \tfrac12\upC_{(n,n)}(A)$ with $A\in\{\pm\}^{[n]^2}$ random.

\section{Concluding questions}

Let us close with three questions: 

\subsection{What to make of the $k$-wise independence?}

Note that Theorem \ref{thm:comparativecountingtheorem} in particular says that 
one application of $\tfrac12 \upC_{(n,n)}$ to a $\{\pm\}$-valued $n\times n$-matrix 
yields a $\{0,\pm\}$-matrix whose entries are distributed 
as $\frac14,\frac12,\frac14$ but are merely $3$-wise stochastically independent 
(in the sense of \cite[Definition 2.4, p.209]{MR2410304}) (and `almost' $k$-wise, 
the `almost' quantified exactly for $k\in \{4,5,6\}$ in Theorem \ref{thm:comparativecountingtheorem} and quantified roughly for general $k$ in 
Proposition \ref{prop:roughestimationofnumberoffailureevents}). There appears to be 
ongoing research on how partial independence relates to full independence, a keyword 
being `$k$-wise independence'. To give only 
one very recent (the topic has been studied at least since the 1980s) example, which 
appears to summarize well the overall spirit, here is a quote from the abstract 
of \cite{BenjaminiGurelGurevichPeled}: 
{\scriptsize
\begin{quote}
We pursue a systematic study of the following problem. Let 
$f\colon \{0,1\}^n\rightarrow \{0,1\}$ be a (usually monotone) boolean function 
whose behaviour is well understood when the input bits are identically independently 
distributed. What can be said about the behaviour of the function when the input 
bits are not completely independent, but only $k$-wise independent [...]? How high 
should $k$ be so that any $k$-wise independent distribution ``fools'' the function, 
i.e. causes it to behave nearly the same as when the bits are completely independent?
\end{quote}
}
The author wonders whether the theory of $k$-wise stochastical independence has any 
bearing on the present problem. In particular, can the proof of Bourgain--Vu--Wood 
for $\Prob_{\lcf}$ be deconstructed and somehow reassembled for $\Prob_{\chio}$, aided 
by knowledge about $k$-wise independence?

\subsection{More accurate estimations?} 

Note that Theorem \ref{thm:comparativecountingtheorem}.\ref{comparativecountingtheorem:item:fourentriesspecified}---\ref{comparativecountingtheorem:item:sixentriesspecified} teaches us that, asymptotically, the ratio 
$\lvert \mathcal{F}^{\mathrm{M}}(k,n)\rvert$ $/$ $\lvert \{ B\in \{0,\pm\}^I,\; 
I\in\binom{[n-1]^2}{k}\}\rvert$ for $k\in \{4,5,6\}$ takes values 
$4 n^4  / \frac{27}{8} n^8$ $=$ $\frac{32}{27}$ $n^{-4}$ $\in$ 
$[$ $1.1 n^{-4},$ $1.2 n^{-4}$ $]$, 
$12 n^6 / \frac{81}{40} n^{10}$ $=$ $\frac{480}{81}$ $n^{-4}$ $\in$ $[$ $5.9 n^{-4},$ 
$6.0$ $n^{-4}$ $]$ and 
$18 n^8 / \frac{81}{80} n^{12}$ $=$ $\frac{1440}{81}$ $n^{-4}$ $\in$ $[$ $17.7$ 
$n^{-4},$ $17.8$ $n^{-4}$ $]$ which are all---although of course growing 
with $k$---vanishing with the same speed 
$\mathcal{O}_{n\to\infty}(n^{-4})$. This shows that the rough bound in Proposition 
\ref{prop:roughestimationofnumberoffailureevents} is not an asymptotically tight 
one. This of course raises two questions: \emph{Is it true that 
$\lvert \mathcal{F}^{\mathrm{M}}(k,n)\rvert$ $/$ $\lvert \{ B\in \{0,\pm\}^I,\; 
I\in\binom{[n-1]^2}{k}\}\rvert$ $\in\mathcal{O}_{n\to\infty}(n^{-4})$ for 
every fixed $k$?} Moreover, note that while Proposition 
\ref{prop:roughestimationofnumberoffailureevents} shows that 
$\lvert \mathcal{F}^{\mathrm{M}}(k,n)\rvert$ $/$ $\lvert \{ B\in \{0,\pm\}^I,\; 
I\in\binom{[n-1]^2}{k}\}\rvert$ vanishes as $n\to\infty$ for every fixed $k$, 
the discussion in \ref{worstpossiblefailure} shows that for the extreme case 
of $k=(n-1)^2$, i.e. if the entire matrix $B\in \{0,\pm\}^{[n-1]^2}$ is specified, 
$\Prob_{\chio}[\mathcal{E}_B^{[n-1]^2}] \neq \Prob_{\lcf}[\mathcal{E}_B^{[n-1]^2}]$ for the 
vast majority of $B\in \{0,\pm \}^{[n-1]}$. In between these two 
extremes, i.e. almost sure agreement as apposed to almost sure non-agreement of 
$\Prob_{\chio}$ and $\Prob_{\lcf}$, there should be a tipping point. This raises the 
question: \emph{For what order of growth $k=k(n)$ does 
$\lvert \mathcal{F}^{\mathrm{M}}(k,n)\rvert$ $/$ $\lvert \{ B\in \{0,\pm\}^I,\; 
I\in\binom{[n-1]^2}{k}\}\rvert$ first become bounded away from zero? And for what 
order does it first tilt in favour of the non-agreement events?} 

\subsection{Hidden connections to the Guralnick--Mar{\'o}ti-theorem?}

If $\sigma\colon G\rightarrow \Aut_K(V)$ is a representation of a group $G$ on 
a $K$-vector space $V$, then for every $g\in G$ let $\Fix_V(g)$ denote the 
\emph{fixed-point space of $g$}, i.e. the $K$-linear subspace 
$\{v\in V\colon \sigma(g)(v) = v \}$. 

In recent times there have been advances (\cite{MR1726791}, \cite{MR2231893}, 
\cite{MR2669683}) concerning the problem of bounding averages of dimensions of 
fixed-point spaces by a fraction of the dimension of the representation, leading 
to a full proof (and in more general form) by R. M. Guralnick and 
A. Mar{\'o}ti \cite{MR2735760} of a 1966 conjecture of P. M. Neumann: 

\begin{theorem}[Guralnick--Mar{\'o}ti \cite{MR2735760}, Theorem 1.1]
\label{thm:guralnickmarotitheorem}
For every finite group $G$ with smallest prime factor of $\lvert G \rvert$ 
denoted by $p$, every field $K$, every finite-dimensional $K$-vector space $V$, 
every homomorphism $\sigma\colon G\rightarrow \Aut_K(V)$, and 
every normal subgroup $N$ of $G$ which does not have a trivial composition 
factor on $V$, 
\begin{equation}\label{eq:guralnickmarotibound}
\frac{1}{\lvert N \rvert} \ \sum_{\tilde{g} \in N \cdot g}\ \dim_K(\Fix_V(\tilde{g}))
\leq \frac{1}{p}\ \dim_K(V) \quad .
\end{equation}
\end{theorem}
Although the resemblance is likely to be merely superficial, the author cannot 
help being intrigued by the similarity of this inequality to 
\ref{conj:relativizedconjecture}, together with the fact that both in 
\ref{conj:relativizedconjecture} and in \eqref{eq:guralnickmarotibound} there can be 
zero-summands on the left-hand side. Moreover, when trying to combine Chio 
condensation of sign matrices with group actions, one gets the impression that 
groups of even order (i.e. $p=2$) play a natural role. One goal along these lines 
is to discover a vector space avatar of the lazy coin flip  measure. Via 
Theorem \ref{thm:graphtheoreticalcharacterizationofthechiomeasure} the author 
found the following formula (which is of course easy to check directly): for 
every $B\in \{0,\pm\}^{[n-1]^2}$ we have
\begin{equation}
\Prob_{\lcf}[B] = (\tfrac12)^{(n-1)^2}\cdot 
(\tfrac12)^{\dim_{\Z/2}( \upB^1(\upX_B;\; \Z/2) \oplus \upZ_1(\upX_B;\; \Z/2) ) } \quad . 
\end{equation} 
This suggests studying group actions on the direct sum  
$\upB^1(\upX_{\frac12\upC_{(n,n)}(A)};\Z/2) \oplus \upZ_1(\upX_{\frac12\upC_{(n,n)}(A)};\Z/2)$. 
A sensible first choice are those actions which are induced by the `natural' 
$\lvert \det(\cdot)\rvert$-preserving (hence intransitive) group actions on 
$\{\pm\}^{[n]^2}$ (a merit of those `standard actions' is that Chio condensation 
commutes with the corresponding actions on $\{0,\pm\}^{[n-1]^2}$). So far the 
author did not detect much resonance with Theorem \ref{thm:guralnickmarotitheorem}. 
Mirage or more? 

\bibliographystyle{amsplain} 
\bibliography{HEINIGchiocondensationandrandomsignmatrices20110808}

\end{document}